\documentclass[10.5pt]{article}

\usepackage[top=3cm, bottom=3cm, inner=3cm, outer=3cm, left=2.5cm, right=2.5cm, includehead]{geometry}
\usepackage{fancyhdr}
\usepackage[english]{babel}
\usepackage[authoryear]{natbib}
\usepackage[colorlinks=true, allcolors=blue]{hyperref}
\usepackage{doi}
\usepackage{multibib}
\newcites{SM}{References for the supplementary material}
\usepackage[utf8]{inputenc}
\usepackage[T1]{fontenc}

\usepackage{amsthm}
\usepackage{amsmath}
\usepackage{amsfonts}
\usepackage{amssymb}

\usepackage{graphicx} 
\usepackage{graphbox}
\usepackage{float}
\usepackage{booktabs}
\usepackage{multirow}
\usepackage{rotating}
\usepackage{array}
\usepackage{xcolor}
\definecolor{forestgreen}{rgb}{0.0, 0.43, 0.13}
\usepackage{subcaption}
\usepackage[ruled]{algorithm2e}
\usepackage{comment}

\usepackage{stackrel}

\usepackage{breakcites}

\usepackage{nowidow}

\usepackage{enumitem}

\usepackage{ifplatform}

\usepackage{textcomp}
\usepackage{bbm}

\ifwindows
\fi

\allowdisplaybreaks

\newcommand{\lp}{\left(}
\newcommand{\rp}{\right)}

\newcommand{\G}{\mathbb{G}}
\newcommand{\R}{\mathbb{R}}

\newcommand{\Sp}{\mathbb{S}}

\newcommand{\rd}{\mathrm{d}}

\newcommand{\cY}{\mathcal{Y}}
\newcommand{\cF}{\mathcal{F}}
\newcommand{\bx}{\boldsymbol{x}}
\newcommand{\be}{\boldsymbol{e}}
\newcommand{\bm}{\boldsymbol{m}}

\newcommand{\bq}{\boldsymbol{q}}
\newcommand{\by}{\boldsymbol{y}}
\newcommand{\bX}{\boldsymbol{X}}

\newcommand{\bZ}{\boldsymbol{Z}}
\newcommand{\bJ}{\boldsymbol{J}}
\newcommand{\bU}{\boldsymbol{U}}
\newcommand{\bR}{\boldsymbol{R}}

\newcommand{\bmu}{\boldsymbol\mu}
\newcommand{\bu}{\boldsymbol{u}}
\newcommand{\bv}{\boldsymbol{v}}

\newcommand{\bpsi}{\boldsymbol{\psi}}
\newcommand{\one}{\mathbf{1}}

\newcommand{\bxi}{\boldsymbol\xi}

\newcommand{\bbeta}{\boldsymbol\eta}

\newcommand{\bzeta}{\boldsymbol\zeta}

\newcommand{\btheta}{\boldsymbol\theta}

\newcommand{\bTheta}{\boldsymbol\Theta}

\newcommand{\bSigma}{\boldsymbol\Sigma}

\newcommand{\bomega}{\boldsymbol\omega}

\newcommand{\blambda}{\boldsymbol\lambda}
\newcommand{\bB}{\boldsymbol{B}}

\newcommand{\bA}{\boldsymbol{A}}
\newcommand{\bV}{\boldsymbol{V}}
\newcommand{\bI}{\boldsymbol{I}}

\newcommand{\bS}{\boldsymbol{S}}
\newcommand{\bP}{\boldsymbol{P}}
\newcommand{\bQ}{\boldsymbol{Q}}
\newcommand{\bT}{\boldsymbol{T}}

\newcommand{\bW}{\boldsymbol{W}}

\newcommand{\lrb}[1]{\left\{#1\right\}}

\newcommand{\Cov}[2]{\mathsf{Cov}(#1,#2)}
\newcommand{\supp}[1]{\mathrm{supp}\lp #1\rp}

\DeclareFontFamily{OT1}{pzc}{}
\DeclareFontShape{OT1}{pzc}{m}{it}{<-> s * [1.10] pzcmi7t}{}
\DeclareMathAlphabet{\mathpzc}{OT1}{pzc}{m}{it}

\theoremstyle{plain}
\newtheorem{theorem}{Theorem}[section]
\newtheorem{corollary}{Corollary}[section]
\newtheorem{proposition}{Proposition}[section]
\newtheorem{lemma}{Lemma}[section]

\theoremstyle{definition}

\newtheorem{remark}{Remark}[section]

\newtheorem{example}{Example}[section]

\newcommand{\defin}{:=}

\DeclareMathOperator*{\arccosh}{arccosh}

\newcommand{\convl}{\stackrel{\mathcal{L}}{\rightarrow}}
\newcommand{\probconvarrow}{\mathrel{\mathpalette\probconvarrowaux\relax}}
\newcommand{\probconvarrowaux}[2]{\scalebox{1.15}[1]{\ensuremath{#1\rightarrow}}}
\newcommand{\convprob}[1]{\overset{#1}{\probconvarrow}}
\newcommand{\convp}{\convprob{\mathsf{P}}}
\newcommand{\convpboot}{\stackrel[*]{\mathsf{P}}{\probconvarrow}}
\newcommand{\convas}{\stackrel{a.s.}{\longrightarrow}}

\newcommand{\Prob}[1]{\mathsf{P}\left(#1\right)}
\newcommand{\Probbig}[1]{\mathsf{P}\big(#1\big)}

\newcommand{\E}[1]{\mathsf{E}( #1 )}

\newcommand{\Ebig}[1]{\mathsf{E}\big( #1 \big)}

\newcommand{\Var}[1]{\mathsf{Var}( #1 )}

\newif\ifincludeimages
\includeimagestrue

\title{Goodness-of-fit for distributions on metric spaces}
\author{Diego Serrano\textsuperscript{1}, Eduardo García-Portugués\textsuperscript{1}, and Ingrid Van Keilegom\textsuperscript{2}\\[0.3em]
\small \textsuperscript{1}Department of Statistics, Universidad Carlos III de Madrid\\
\small \textsuperscript{2}ORSTAT, KU Leuven}
\date{}

\begin{document}
\maketitle
\begin{abstract}
We propose a general goodness-of-fit framework for distributions on separable metric spaces. Under suitable identifiability conditions, probability distributions are characterized by distance profiles, which motivates their use in goodness-of-fit testing, for simple and composite null hypotheses. For composite null hypotheses, parameter estimation is incorporated via a Bahadur-type expansion, and the asymptotic distribution of the empirical process for distance profiles is obtained under the null. We define test statistics based on this empirical process and derive their asymptotic null distributions. We further study the behavior of the proposed tests under fixed and local alternatives, establishing consistency results. Multiplier bootstrap procedures are developed, and their conditional asymptotic validity is established under both simple and composite null hypotheses. The methodology is illustrated with simulation studies and real-data applications for data on the sphere, hyperboloid, and simplex.
\end{abstract}



\section{Introduction}

Metric-space data analysis provides a flexible framework that does not require linear, algebraic, or locally Euclidean structure. Data from this general framework are commonly referred to as random objects \citep{Marron2021}, encompassing functional, directional, and compositional data, as well as shapes, networks, distributions, and other complex data types. Our aim is to develop Goodness-Of-Fit (GOF) tests for distributions on a separable metric space $(\Omega,d)$, under both simple and composite null hypotheses.  %

GOF beyond the real line has largely developed one data geometry at a time. In $\R^d$, \cite{Baringhaus1988} and \cite{HenzeZirkler1990} developed tests of multivariate normality based on empirical characteristic functions. Moreover, \cite{Hallin2021} proposed tests based on the Wasserstein distance for simple and composite null hypotheses, applicable to multivariate group families, such as location–scale families and those arising from aﬃne transformations. %
For observations in a separable Hilbert space, random projections reduce GOF to tests on one-dimensional projections, for example for Gaussian models \citep{CuestaAlbertos2007}, whereas empirical characteristic functionals enable direct tests of Gaussianity \citep{Henze2021}. On the hypersphere, \cite{Boente2014} build on the density approach, comparing kernel density estimators against parametric density estimates, which requires the choice of a smoothing bandwidth. \cite{Ebner2024} use a characteristic-function approach, comparing the empirical characteristic functions of the observed sample and a sample generated under the fitted null model. For compositional data, the simplex can be mapped into a Euclidean space through the logratio transformations. This property was used by \cite{Aitchison1986} to assess logistic normality through multivariate Gaussianity of the transformed data. On compact Riemannian manifolds, coordinate-invariant Sobolev tests \citep{Jupp2005} provide a GOF framework. Despite its generality, it requires exploiting the specific geometric structure of each manifold. For simple nulls, kernel mean embedding methods apply on general domains but require a suitable kernel \citep{Balasubramanian2021}, while kernel Stein discrepancies require both a suitable kernel and a Stein operator adapted to the domain and null model \citep{Hagrass2026}. Thus, these procedures use linear, coordinate, volume, reproducing kernel, or differential structure in addition to the metric.

For real-valued observations, classical Kolmogorov--Smirnov (KS) and Cram\'er--von Mises (CM) procedures compare the empirical distribution function with the distribution specified under the null, whose definition relies on the order structure of the real line. This has no canonical analogue in a general metric space. An alternative is to consider instead real-valued random variables defined through the distances $d(\omega,X)$ from each center $\omega\in\Omega$. Following \citet{DubeyChenMuller2024}, we study the \emph{distance profile} at $\omega$, given by
\begin{align*}
F_\omega^\mu(t)\defin\Prob{d(\omega,X)\le t},\qquad t\ge0.
\end{align*}
Equivalently, $F_\omega^\mu(t)$ is the probability assigned by $\mu$, the law of $X$, to the closed ball with center $\omega$ and radius $t$. Under the forthcoming condition \ref{a1} or \ref{a2}, equality between distance profiles implies equality of the underlying probability measures \citep{DubeyChenMuller2024,Chen2025}.

Distance profiles have been applied to several problems in statistical inference. \citet{DubeyChenMuller2024} established a Donsker theorem, used optimal transports between distance profiles to define ranks and quantile sets, and developed a two-sample test. Subsequent uses include change-point detection \citep{Dubey2023}, conformal inference \citep{Zhou2024}, mutual independence testing through joint profiles \citep{Chen2025}, and quantification of association and independence testing \citep{Zhou2026}. A methodology closely related to distance profiles was proposed by \cite{Wang2024} under the name of metric distribution functions, with Glivenko--Cantelli and Donsker results, as well as CM-type procedures for homogeneity and independence. Ball-based discrepancies have also been used to test independence between a random object and categorical variables \citep{Pan2023}.

The work of \citet{Szabados1987} can be seen as a precursor of distance profiles for GOF testing for continuous distributions in a separable metric space $(\Omega, d)$. For each sample center $\omega\in\{X_1,\ldots,X_n\}$, he computed the Kolmogorov distance between the cumulative distribution function (cdf) of $d(\omega,X)$ under the null and the leave-one-out empirical cdf, and then maximized it over the sample centers. When the cdf of $d(\omega,X)$ is continuous for every $\omega \in \Omega$, this statistic has exponential bounds and is consistent against a broad class of fixed alternatives. However, only simple null hypotheses and fixed alternatives were considered, and the asymptotic distribution of the test statistic under the null was not derived.

To the best of our knowledge, a general distance-profile GOF theory that treats simple and composite null hypotheses within a common framework and that accounts for parameter estimation has not been developed. We address this gap by proposing distance-profile-based KS and CM test statistics. Our theoretical contributions are fourfold:
\begin{enumerate}[label=(\roman*)]
\item We derive the asymptotic distribution of the KS and CM statistics for simple and composite hypotheses. The CM integration measure is fully data-driven, and, under suitable conditions, the KS supremum is asymptotically equivalent to a maximum over sample centers and pairwise sample distances.

\item For composite null hypotheses, a Bahadur expansion for the estimators and a first-order Taylor expansion of the distance profiles show that the fitted-profile process is asymptotically an empirical process whose indexing function class is Donsker. The effect of parameter estimation on the weak limit has classical analogues in GOF for fitted distributions \citep{Durbin1973}. We show that exponential families fitted by maximum likelihood verify the differentiability conditions and the Bahadur representation, and derive a simplified covariance formula.

\item We establish consistency against fixed alternatives and derive weak limits under $n^{-1/2}$-local contamination alternatives. Under the local alternatives, for a simple null hypothesis, the limiting Gaussian process that arises under the null is shifted by the difference between the distance profiles under the contaminating distribution and under the null; for a composite null, parameter fitting contributes an additional deterministic shift. Rather than assuming asymptotic linearity of the estimator under the locally misspecified sequence of distributions, we derive a Bahadur-type expansion for the estimator under local regularity and moment conditions.

\item We propose a multiplier-bootstrap approach to compute the $p$-values, and derive the asymptotic null distributions of the bootstrap KS and CM statistics, conditionally on the data. Under a composite null, we consider two different procedures. The first re-estimates the parameter and recomputes the fitted profiles in every replication. The second replaces re-estimation in each bootstrap replication by the first-order effect of parameter estimation evaluated once at the fitted model, and applies the multipliers directly to the resulting empirical process. Both converge conditionally to the same null limits. The second avoids repeated estimation and profile evaluation, but requires calculating the derivatives of the distance profile with respect to the parameter and inverting a suitable Jacobian matrix.
\end{enumerate}

In some geometries, distance profiles can be seen as the cdfs of scalar projections (see Proposition \ref{prop:dp_examples}). This connects our GOF tests with projection-based GOF procedures \citep{Escanciano2006, CuestaAlbertos2007, Garcia-Portugues2020b, Garcia-Portugues2026}. Our tests are also connected to the GOF literature based on the analysis of regression residuals, since, for each fixed center $\omega$, $d(\omega,X)$ plays a role analogous to a residual. \citet{VanKeilegom2008} developed KS and CM tests based on residual cdfs and studied $n^{-1/2}$-local alternatives, obtaining a limiting structure similar to ours. Another connection is the first-order effect of parameter estimation on the empirical cdf process \citep{Loynes1980}.

The remainder of this paper is organized as follows. Section~\ref{sec:basic_concepts} introduces basic notation and results needed for defining the GOF tests. Section~\ref{sec:dp_examples} provides a characterization of distance profiles for a variety of distributions in different metric spaces. Section~\ref{sec:null_asymptotics} introduces the KS and CM test statistics and derives their asymptotic distributions under simple and composite null hypotheses. Section~\ref{sec:non_null_asymptotics} establishes consistency under fixed alternatives and derives the weak limits under local alternatives. Section~\ref{sec:bootstrap} introduces the multiplier-bootstrap procedures and derives the conditional asymptotic distributions of the bootstrap statistics. Sections~\ref{sec:numerical_experiments} and~\ref{sec:real_data} present the numerical experiments and applications, respectively, and Section~\ref{sec:discussion} concludes with a discussion. Proofs are deferred to the Supplementary Material.

\section{Basic concepts and notation}\label{sec:basic_concepts}
Consider a separable metric space $(\Omega, d)$. Let $(S, \mathcal{A}, \mathsf{P})$ be a probability space, where $S$ is a sample space, $\mathcal{A}$ is a $\sigma$-algebra of subsets of $S$, and $\mathsf{P}$ is a probability measure. An $\Omega$-valued random variable is called a random object, which is a measurable map $X: S \rightarrow \Omega$. Suppose that $X$ is distributed according to a Borel probability measure $\mu$. %
For any $t \geq 0$, we define the \emph{distance profile} at $\omega \in \Omega$ of $X$ as
$$
F_\omega^\mu(t) \defin \Prob{d(\omega, X) \leq t}.
$$
To estimate the distance profiles $F_\omega^\mu$ from a sample $X_1, \ldots, X_n$ of independent realizations of $X$, $n\ge1$, we use the empirical distance profile
$$
F_{n,\omega}^\mu(t)\defin\frac{1}{n} \sum_{i=1}^n 1_{\left\{d\left(\omega, X_i\right) \leq t\right\}}, \quad t \geq 0.
$$

Distance profiles characterize the underlying probability measures under suitable conditions.
Consider two Borel probability measures $\mu_1$ and $\mu_2$ on $(\Omega,d)$. \citet[][Proposition 1]{DubeyChenMuller2024} provide a sufficient condition for $\mu_1=\mu_2$ in terms of distance profiles: if there exists $\alpha>0$ such that $(\Omega,d^\alpha)$ is of strong negative type \citep{Lyons2013}, where $d^\alpha(\omega,\omega^\prime):=d(\omega,\omega^\prime)^\alpha$, then 
\begin{align*}
\mu_1=\mu_2
\quad\Longleftrightarrow\quad
F_{\omega}^{\mu_1}(t)=F_{\omega}^{\mu_2}(t)
\ \ \text{for all } \omega\in\Omega,\ t\ge 0.
\end{align*}

\citet{Chen2025} obtain an alternative characterization under a condition imposed on the probability measures.
Let $\mu$ be a Borel probability measure on $(\Omega,d)$. If there exists $L_\mu<\infty$ such that, for every $\omega\in \supp{\mu}$,
\begin{align*}
\limsup_{r\to 0^+}\frac{\mu\big(B(\omega,2r)\big)}{\mu\big(B(\omega,r)\big)}\le L_\mu,
\end{align*}
then $\mu$ is said to satisfy the doubling condition.
If $\mu_1,\mu_2$ satisfy this condition, then Theorem 1 in \cite{Chen2025} shows that
\begin{align*}
\mu_1=\mu_2
\quad\Longleftrightarrow\quad
\mu_1\big(\overline B(\omega,t)\big)=\mu_2\big(\overline B(\omega,t)\big)
\ \ \text{for all } \omega\in \supp{\mu_1},\ t>0,
\end{align*}
i.e., $\mu_1=\mu_2$ whenever $F_\omega^{\mu_1}(t)=F_\omega^{\mu_2}(t)$ for all $\omega\in\supp{\mu_1}$ and $t\ge 0$. The same result holds with $\supp{\mu_1}$ replaced by $\supp{\mu_2}$. 
Here, $\supp{\mu}$ denotes the support of $\mu$, which is the closure of the set of all points to which $\mu$ assigns positive mass in every open neighborhood, i.e.,
$\supp{\mu}=\operatorname{cl}(\{x \in \Omega: \mu(U)>0 \text { for all open sets } U \ni x\}).$

In this manuscript, we assume that at least one of the following two conditions holds:
\begin{enumerate}[label=(\textit{a.\arabic*}), ref=(\textit{a.\arabic*})]
    \item There exists $\alpha>0$ such that $(\Omega,d^\alpha)$ is of strong negative type. \label{a1} %

    \item The probability measures satisfy the doubling condition. \label{a2}
    
\end{enumerate}

Condition \ref{a1} is a condition on the metric space, whereas condition \ref{a2} is a condition on the probability measures.

Under certain conditions, \citet[][Theorem 5.1]{DubeyChenMuller2024} and \citet[][Theorem 3]{Chen2025} establish that distance profiles satisfy the Donsker condition.
Let
\begin{align}\label{eq:def_Y}
\cY \defin \{y_{\omega, t}:(\omega, t)\in\Omega\times\mathbb R_{\ge0}\}, \ \text{ where } y_{\omega, t}(x):=1_{\{d(\omega,x)\le t\}}.
\end{align}
Then, $\cY$ is $\mu$-Donsker.
The conditions under which this property is established for distance profiles differ between \cite{DubeyChenMuller2024} and \cite{Chen2025}. We endow $\Omega\times\mathbb R_{\ge0}$ with the distance
$
d_{\Omega\times\mathbb R_{\ge0}}
\big((\omega,t),(\omega',t')\big)
\defin
d(\omega,\omega')+|t-t'|,
$
which induces the product topology.

\cite{DubeyChenMuller2024} require the following two conditions:
\begin{enumerate}[label=(\textit{b.\arabic*}), ref=(\textit{b.\arabic*})]
\item\label{b1}
Let $N(\varepsilon,\Omega,d)$ be the covering number of $(\Omega,d)$, for $\varepsilon>0$. Then
$$
\varepsilon \log N(\varepsilon,\Omega,d)\ \to\ 0
\text{ as } \varepsilon\to 0^+.
$$
\item\label{b2}
For every $\omega\in\Omega$, $F_\omega^\mu$ is absolutely continuous with continuous density $f_\omega$. Defining
$
\underline\Delta_\omega \defin \inf_{t\in \supp{F_\omega^{\mu}}} f_\omega(t),$ $\bar\Delta_\omega \defin \sup_{t\in\mathbb R} f_\omega(t)$, it holds that $\underline\Delta_\omega>0$ for each $\omega\in\Omega$ and there exists $\bar\Delta<\infty$ such that $\sup_{\omega\in\Omega}\bar\Delta_\omega \le \bar\Delta$.
\end{enumerate}
\begin{remark}
    Condition~\ref{b1} implies that $(\Omega,d)$ is totally bounded. The assumption that the map $t\mapsto F_\omega^\mu(t)$ is continuous on $\R_{\ge0}$ for every $\omega\in\Omega$, which is implied by condition \ref{bp2}, dates back to \citet[p. 383]{Szabados1987}, who replaced the supremum over all $\omega \in \Omega$ in the Kolmogorov--Smirnov test statistic with maximization over a finite sample. To achieve this, \cite{Szabados1987} restricted the analysis to \emph{continuous measures}, i.e., measures $\mu$ such that $\mu(B(\omega,t))$ depends continuously on $t>0$ for any $\omega \in \Omega$.
\end{remark}

\citet[][Assumption 1]{Chen2025} assume the bracketing condition
\begin{align}\label{eq:bracketing_assumption_1}
\varepsilon \log N_{[]}\big(\varepsilon,\cY, L^1(\mu)\big)\ \to\ 0
\quad \text{as } \varepsilon\to 0^+,
\end{align}
where $N_{[]}(\varepsilon,\cY,L^1(\mu))$ denotes the $L^1(\mu)$ bracketing number of $\cY$. 
The condition given in \eqref{eq:bracketing_assumption_1} holds if both condition \ref{b1} and a modification of condition \ref{b2} are satisfied \citep[][Proposition~4]{Chen2025}.
The modification is:
\begin{enumerate}[label=(\textit{b.\arabic*'}), ref=(\textit{b.\arabic*'})]
\setcounter{enumi}{1}
\item\label{bp2}
There exists a constant $K_\mu>0$ such that
$$
\sup_{\omega\in \Omega}\big|F_\omega^\mu(r)-F_\omega^\mu(s)\big| \le K_\mu|r-s|,
$$
for every $r,s\in \R_{\ge0}$.
\end{enumerate}

It is assumed that conditions \ref{b1} and \ref{bp2} hold throughout this manuscript, in addition to either \ref{a1} or \ref{a2}.

\section{Distance profiles on particular metric spaces}\label{sec:dp_examples}

This section aims to provide closed formulas for the distance profile of a random variable $X$ following a specific distribution on a given metric space, with probability measure denoted by $\mu$.

\begin{example}[Multivariate normal distribution]\label{ex:dp_mvn}
Consider $\bX\sim \mathcal{N}_q(\bmu, \bSigma)$ taking values in $(\R^q, d_{\R^q})$, with $\bSigma$ positive definite and $d_{\R^q}$ the Euclidean distance in $\R^q$.
\end{example}

\begin{example}[\texorpdfstring{The Brownian motion in $L^2[0,1]$}{The Brownian motion in L2[0,1]}]\label{ex:dp_brownian}
Consider the Brownian motion $B(t)$, $t \in [0,1]$, as an element of $(L^2[0,1], d_{L^2})$, where $d_{L^2}(B, \omega) = \|B-\omega\|_{L^2}$, for $\omega\in L^2[0,1]$.
\end{example}

\begin{example}[Von Mises--Fisher distribution; vMF]\label{ex:dp_vmf}
Consider the unit sphere $\Sp^{q}\defin\{\bx \in \R^{q+1}: \, \| \bx \|_2 = 1\},$ $q\ge 1$.
Assume $\bX \sim \mathrm{vMF}(\bmu, \kappa)$, for some given $\bmu \in \Sp^{q}$, $\kappa>0$.
For $\bx \in \Sp^{q}$, the density of $\bX$ is given by $f_\mathrm{vMF}(\bx;\bmu,\kappa) = c_{q}^\mathrm{vMF}(\kappa) e^{\kappa\bmu^\top\bx}.$ 
We consider the chordal and geodesic distances, given respectively by $d_{\mathbb R^{q+1}}(\bX,\bomega)^2=2(1-\bomega^\top \bX)$ and $d_{\mathbb S^q}(\bX,\bomega)=\cos^{-1}(\bomega^\top \bX)$.
\end{example}

\begin{example}[Hyperboloid von Mises--Fisher distribution; HvMF]\label{ex:dp_hvmf}
Consider the Minkowski pseudo-inner product $(\bx, \by)\defin -x_{1} y_{1}+\sum_{i=2}^{q+1} x_i y_i$ on the unit hyperboloid $\mathbb{H}^q\defin\{\bx \in \R^{q+1}: \allowbreak (\bx, \bx)=-1,\, x_{1}>0\}$, for $q\ge 1$. In this surface, geodesic distances are calculated as $d_{\mathbb{H}^q}(\bx, \by)=\arccosh \left( -(\bx, \by)\right)$. The analogous distribution to the vMF on $\mathbb{H}^q$ has density $f_\mathrm{HvMF}(\by ; \bmu, \kappa) = c_q^{\mathrm{HvMF}}(\kappa) \exp \left\{\kappa(\by, \bmu)\right\}$, for $\by, \bmu \in \mathbb{H}^q$, and $\kappa>0$, with respect to the invariant measure on $\mathbb{H}^q$ \citep{Barndorff-Nielsen1978a,Jensen1981}.
\end{example}

\begin{example}[Logistic Gaussian distribution; LG]\label{ex:dp_logistic_gaussian}
Consider the simplex $\Delta^{q} \defin \{ \bx \in \R^{q+1} : x_i > 0, \ \sum_{i=1}^{q+1} x_i = 1 \}$, for $q\ge 1$, endowed with the Aitchison distance $d_{\mathrm{A}}(\bx, \bomega) \defin \|\mathrm{ilr}(\bx) - \mathrm{ilr}(\bomega)\|_2$, where $\mathrm{ilr}:\Delta^q\to\R^q$ is the isometric log-ratio transform \citep{Egozcue2003} defined by
$$
\mathrm{ilr}_j(\bx)
\defin
\frac{1}{\sqrt{j(j+1)}}
\bigg(
\sum_{k=1}^j \log x_k
-
j\log x_{j+1}
\bigg),
\qquad j=1,\ldots,q.
$$
Let $\mathbf{Y} \sim \mathcal{N}_{q}(\boldsymbol{\mu}, \boldsymbol{\Sigma})$, with $\boldsymbol{\Sigma}$ positive definite, and consider its inverse ilr transformation $\mathbf{X} \defin \mathrm{ilr}^{-1}(\mathbf{Y}).$ Then, $\mathbf{X}$ follows an ilr-variant of the logistic Gaussian distribution of \citet{Aitchison1980} on $\Delta^{q}$, which we denote by $\mathrm{LG}(\boldsymbol{\mu},\boldsymbol{\Sigma})$.
\end{example}

\begin{proposition}\label{prop:dp_examples}
The distance profiles for the distributions introduced in Examples~\ref{ex:dp_mvn}--\ref{ex:dp_logistic_gaussian} are given as follows:
\begin{enumerate}[label=(\textit{\roman*}), ref=(\textit{\roman*})]
\item\label{prop:dp_examples_normal}
Let $\bX\sim \mathcal{N}_q(\bmu,\bSigma)$. Then, $F_{\bomega}(t)=\Prob{\blambda^\top\bP \le t^2}$, where $P_i\sim \chi_1^2\big(\nu_i^2/\lambda_i\big)$, mutually independent, $i=1,\ldots,q$; $\bSigma=\bU \, \mathrm{diag}(\blambda) \, \bU^\top$ is the spectral decomposition of $\bSigma$; and $\boldsymbol{\nu}\defin \bU^\top(\bmu-\bomega)$. When $\bomega = \bmu$, $P_i \sim \chi_1^2$.

\item\label{prop:dp_examples_brownian}
Let $B$ be a Brownian motion in $L^2[0,1]$. Then, $F_{\omega}(t)=\Prob{\sum_{i=1}^\infty \lambda_i Q_i \le t^2}$, where $Q_i\sim \chi_1^2\big(c_i^2/\lambda_i\big)$, mutually independent; $c_i\defin\langle \omega,e_i\rangle$, where $\{e_i\}_{i=1}^\infty$ is an orthonormal eigenbasis of the covariance kernel of $B$; and $\lambda_i=(i-1/2)^{-2}\pi^{-2}$. When $\omega(t)=0$ for all $t\in[0,1]$, $Q_i\sim\chi_1^2$.

\item\label{prop:dp_examples_vmf}
Let $\bX \sim \mathrm{vMF}(\bmu, \kappa)$. Then, $F_{\bomega}(t)= 1 - F_T\big(1-t^2/2\big)$, with $0\le t\le 2$ (chordal distance), and $F_{\bomega}(t)= 1 - F_T\left(\cos(t)\right)$, with $0\le t\le \pi$ (geodesic distance). Here, $T \defin \bomega^\top \bX$ has density 
\begin{align}\label{eq:density_t_vmf}
    t \mapsto \frac{c_{q}^\mathrm{vMF}(\kappa) e^{\kappa t \bmu^\top\bomega } (1-t^2)^{\frac{q-2}{2}}}{ c_{q-1}^\mathrm{vMF}(\kappa\sqrt{(1-t^2)(1-(\bmu^\top\bomega)^2)})}, \quad  -1 < t < 1.
\end{align}

\item\label{prop:dp_examples_hvmf}
Let $\bX \sim \mathrm{HvMF}(\bmu, \kappa)$. Then, $F_{\bomega}(t)=F_Z(t)$, $t\ge0$. Here, $Z \defin d_{\mathbb{H}^q}(\bX,\bomega)$ has density
\begin{align}\label{eq:density_z_hvmf}
    z \mapsto \frac{c_{q}^\mathrm{HvMF}(\kappa)}{c_{q-1}^\mathrm{vMF}(\kappa\sinh(z)\sinh(\rho))} (\sinh(z))^{q-1} e^{\kappa(\bmu, \bomega) \cosh(z)}, \quad z \ge 0,
\end{align}
where $\rho\defin d_{\mathbb{H}^q}(\bomega,\bmu)$.

\item\label{prop:dp_examples_logistic_gaussian}
Let $\bX \sim \mathrm{LG}(\bmu, \bSigma)$. Then, $F_{\bomega}(t)=\Prob{\boldsymbol{\alpha}^\top \bQ \le t^2}$, where $Q_i\sim \chi_1^2(\delta_i)$, mutually independent, and $\delta_i\defin \alpha_i^{-1}\left(\bu_i^\top\big(\bmu-\mathrm{ilr}(\bomega)\big)\right)^2$, where $\bu_i$ denotes the $i$-th column of $\bU$, with $\bSigma=\bU \, \mathrm{diag}(\boldsymbol{\alpha}) \, \bU^\top$ the spectral decomposition of $\bSigma$.

\end{enumerate}
\end{proposition}

\section{Null asymptotics}\label{sec:null_asymptotics}
Consider a random object $X$ taking values in a separable metric space $(\Omega, d)$ and denote by $\mu$ the underlying probability measure of $X$. In this section, we study the asymptotic distribution of the distance-profile empirical processes under the null hypothesis $H_0$, for simple and composite null hypotheses, and use it to prove the asymptotic distribution of the Kolmogorov--Smirnov (KS) and Cramér--von Mises (CM) test statistics.

\subsection{Simple null hypothesis}\label{sec:null_asymptotics_simple}
The goal is to test $\mu = \mu_{0}$, where $\mu_{0}$ is a given probability measure on $(\Omega, d)$.
We assume condition \ref{a2}, so that distance profiles characterize distributions on $(\Omega,d)$. 
Hence, testing $\mu = \mu_{0}$ is equivalent to:
$$
H_0: F_\omega^\mu=F_{\omega}^{\mu_0}, \text{ for every } \omega \in \supp{\mu_{0}}\subset\Omega.
$$
The alternative hypothesis is thus
$$
H_1: \text{ there exists } \omega \in \supp{\mu_{0}} \text{ such that } F_\omega^\mu\neq F_{\omega}^{\mu_0}.
$$
Alternatively, one could assume condition \ref{a1} and replace $\supp{\mu_0}$ by $\Omega$ in the above hypotheses.
Unless otherwise specified, all the results hold under either condition.

Consider a sample $X_1, \ldots, X_n$ of iid random objects taking values in $(\Omega, d)$, and following a common distribution $\mu$. 
Denote by $F_{n,\omega}^{\mu}$ the empirical distance profile associated to this sample. 
Define the random field 
\begin{align*}
    \G_{n,\mu}^{\mu_0}(\omega, t) \defin  \sqrt{n}\left(F_{n,\omega}^{\mu}(t)-F_{\omega}^{\mu_0}(t)\right), \quad (\omega, t) \in \Omega \times \R_{\geq 0}.
\end{align*}
Observe that $\G_{n,\mu}^{\mu_0}(\omega, t)$ is the empirical process indexed by the class of functions $\cY$ given in \eqref{eq:def_Y}.  In other words, we can write
$\G_{n,\mu}^{\mu_0}(\omega, t) = \sqrt n(P_n y_{\omega, t} - \mu_0 y_{\omega, t})$
where $P_n$ is the empirical measure associated to the sample $X_1, \ldots, X_n$, i.e., $P_n = n^{-1} \sum_{i=1}^n \delta_{X_i}$. Thus, for notational convenience, we will also write $\G_{n,\mu}^{\mu_0}(y_{\omega,t})$ to denote $\G_{n,\mu}^{\mu_0}(\omega,t)$ in the proofs.
Define $\ell^{\infty}(\cY)$ as the set of all uniformly bounded real functions on $\cY$. Theorem 3 in \cite{Chen2025} proves that, under $H_0$, the class $\cY$ is $\mu_0$-Donsker, i.e.,
\begin{align}\label{eq:thm_3_chendubey}
\G_{n,\mu}^{\mu_0} \convl \G_{0} \text { in } \ell^{\infty}(\cY), 
\end{align}
as $n\to \infty$, where $\G_{0}$ is a zero mean Gaussian process with covariance given by
\begin{align}\label{eq:cov_simp}
\mathcal{C}_{(\omega_1,t_1), (\omega_2, t_2)} \defin \Prob{d(\omega_1, X) \le t_1, d(\omega_2, X) \le t_2} - F_{\omega_1}^{\mu_0}(t_1)F_{\omega_2}^{\mu_0}(t_2),
\end{align}
for every $(\omega_1, t_1), (\omega_2, t_2) \in \Omega \times \R_{\geq 0}$.

As an illustration, Figure~\ref{fig:s1_limit_process_visualization} gives a visual representation of one realization of the limiting Gaussian process $\G_0$ for the $\mathrm{vMF}((1,0)^\top,2)$ model on $\Sp^1$. %

\ifincludeimages
\begin{figure}[!htbp]
\centering
\begin{subfigure}[htbp]{0.67\textwidth}
    \centering
    \includegraphics[width=\textwidth]{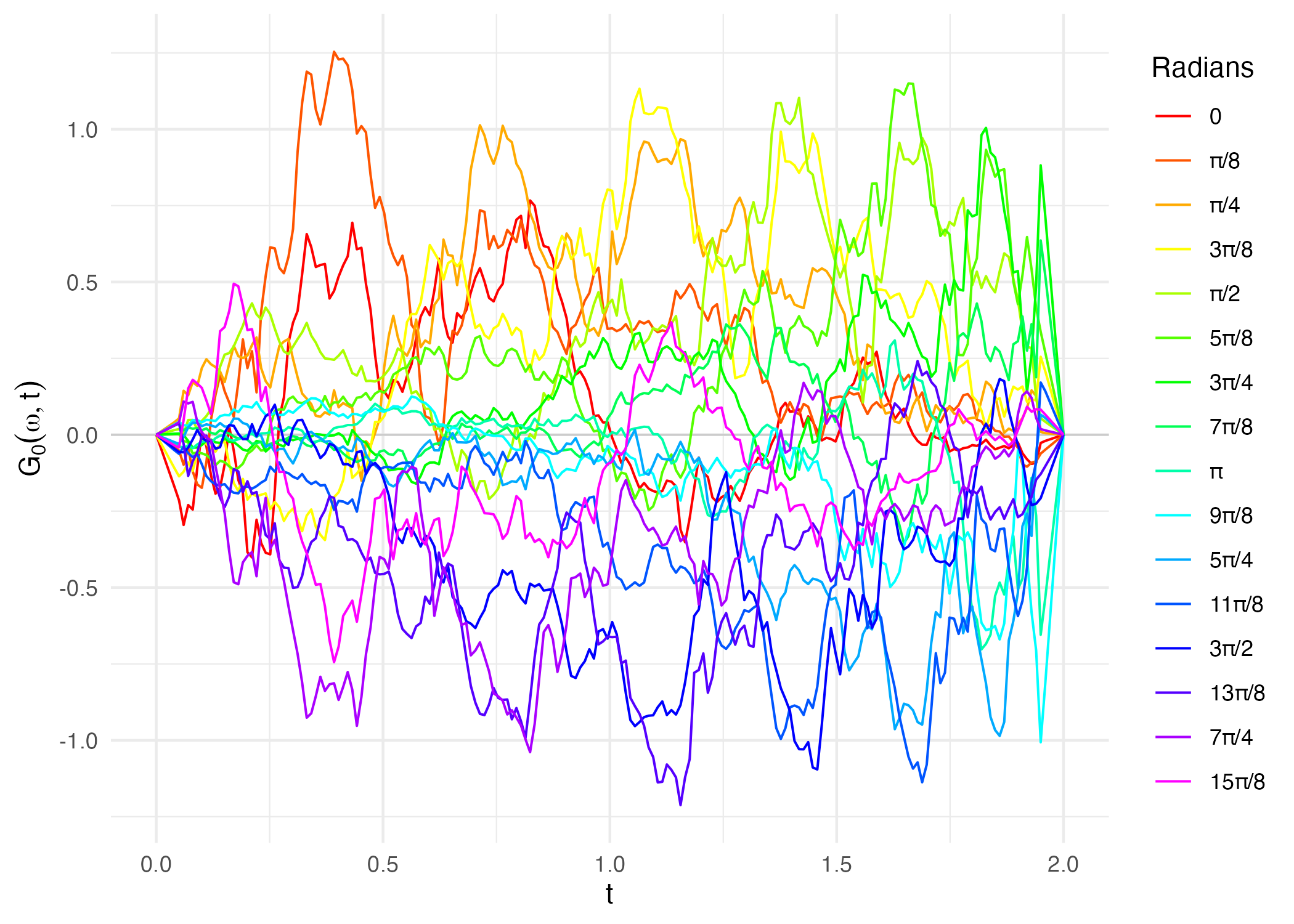}
\end{subfigure}
\hfill
\begin{subfigure}[htbp]{0.32\textwidth}
    \centering
    \includegraphics[width=\textwidth]{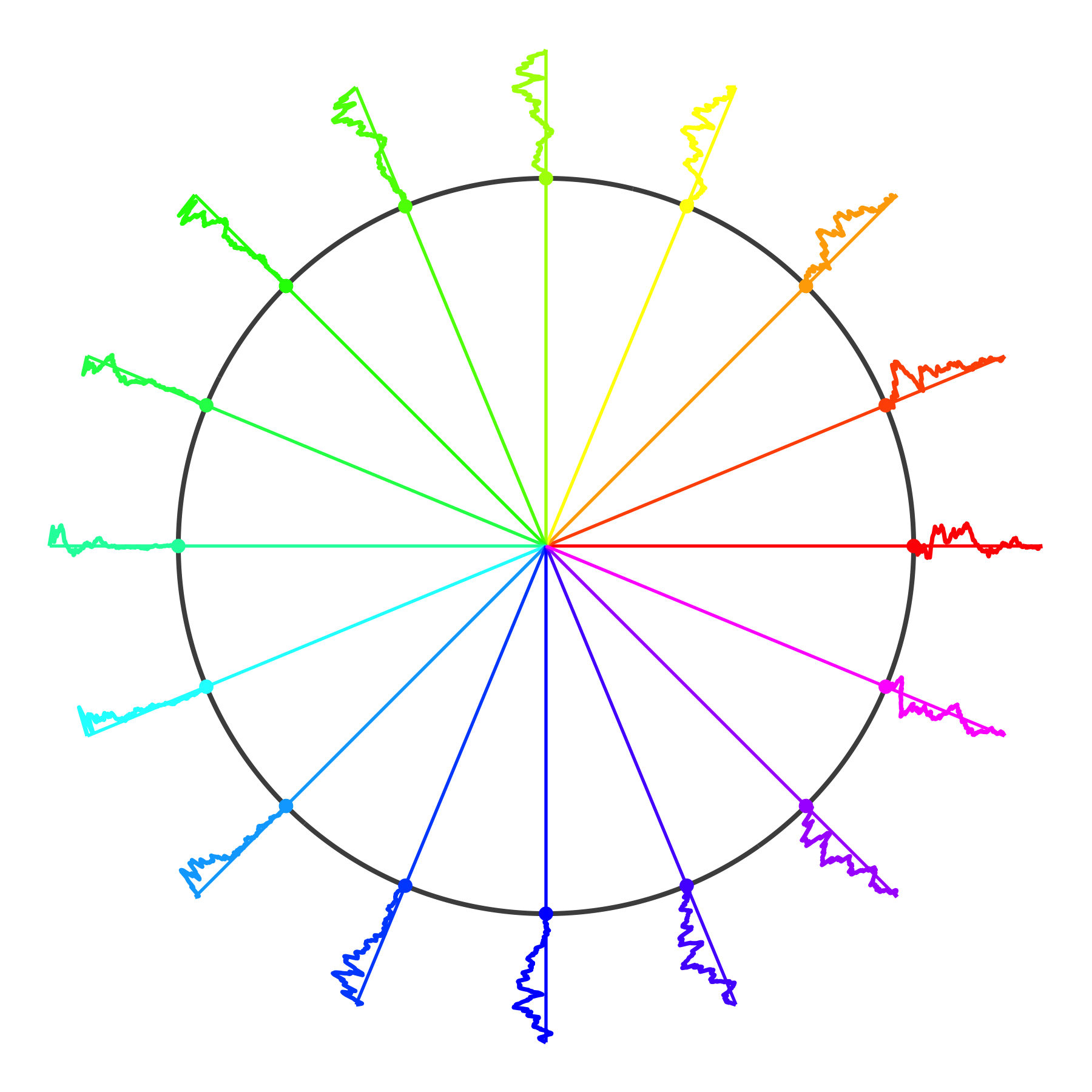}
\end{subfigure}
\caption{One realization of $\G_0$ for a $\mathrm{vMF}((1,0)^\top, 2)$ model on $\Sp^1$. The left panel displays the realization of $\G_0$ as a function of $t\in[0,2]$ along $16$ fixed directions $\bomega\in\Sp^1$. The 16 trajectories are obtained from the same realization of the Gaussian process. The right panel gives a geometric representation of the same trajectories along their corresponding directions.}
\label{fig:s1_limit_process_visualization}
\end{figure}
\fi

\subsection{Composite null hypothesis}
Assume that $\mu$ belongs to a parametric family of probability measures on $(\Omega, d)$, indexed by a parameter $\btheta$ in an open subset $\Theta$ of $\R^p$. 
One may be interested in testing $\mu \in \{\mu_{\btheta} : \btheta \in \Theta\}$. Denote by $\btheta_0\in \Theta$ the true parameter value, i.e., $\mu = \mu_{\btheta_0}$.
In this case, the (composite) null hypothesis can be expressed as
$$
H_0: F_\omega^\mu=F_{\omega}^{\mu_{\btheta_0}}, \text{ for every } \omega \in \mathrm{supp}(\mu_{\btheta_0})\subset\Omega, \ \btheta_0\in\Theta\subset\R^p \text{ unknown}.
$$

The construction of test statistics for the composite null hypothesis requires estimating the unknown parameter $\btheta_0$. \citet[Chapter 5, Section 5]{Shorack2009} provide an overview of GOF tests with estimated parameters. 

Let $\hat{\btheta}$ be a measurable estimator of $\btheta_0$ based on the sample $X_1, \ldots, X_n$. The quantity $\hat{\btheta} - \btheta_0$ will play a central role in deriving the asymptotic distributions. For estimators based on $n$ replications of an experiment, $\hat{\btheta}-\btheta_0$ is often of order $n^{-1/2}$ \citep[][p. 51]{vandervaart1998}, and its asymptotic distribution is usually normal. 
Under certain regularity conditions, an estimator $\hat{\btheta}$ of $\btheta_0$ admits the so-called Bahadur representation for asymptotic linearity:
\begin{align}\label{eq:bahadur}
\sqrt{n}\big(\hat{\btheta}-\btheta_0\big)= - \bV_{\btheta_0}^{-1} \frac{1}{\sqrt{n}}\sum_{i=1}^n \bpsi_{\btheta_0}(X_i)+o_{\mathsf{P}}(1).
\end{align}
In particular, the sequence $\sqrt{n}(\hat{\btheta}-\btheta_0)$ is asymptotically normal with mean zero and covariance matrix $\bV_{\btheta_0}^{-1} \Ebig{\bpsi_{\btheta_0}(X) \bpsi_{\btheta_0}(X)^\top}(\bV_{\btheta_0}^{-1})^\top$. See Theorem 5.21 in \cite{vandervaart1998} for this result.
A set of regularity conditions under which \eqref{eq:bahadur} holds is:
\begin{enumerate}[label=($c$\textit{.\arabic*}), ref=($c$\textit{.\arabic*})]
    \item For each $\btheta$ in an open subset of Euclidean space, the mapping $x \mapsto \bpsi_{\btheta}(x)$ is a vector-valued function such that, for every $\btheta_1$ and $\btheta_2$ in a neighborhood of $\btheta_0$ and a measurable function $L$ with $\mathsf{E}(L(X)^2) < \infty$, it holds that $\| \bpsi_{\btheta_1}(x) - \bpsi_{\btheta_2}(x)\| \le L(x) \|\btheta_1 - \btheta_2 \|.$ \label{cp1}
    \item $\Ebig{\|\bpsi_{\btheta_0}(X)\|^2}<\infty$. \label{cp2}
    \item The map $\btheta \mapsto \Ebig{\bpsi_{\btheta}(X)}$ is differentiable at a zero $\btheta_0$, with non-singular derivative matrix $\bV_{\btheta_0}$.\label{cp3}
    \item $\hat{\btheta} \convp \btheta_0$. \label{cp4}
    \item $n^{-1} \sum_{i=1}^n \bpsi_{\hat{\btheta}}(X_i) = o_{\mathsf{P}}(n^{-1/2})$. \label{cp5}
\end{enumerate}
These conditions replace the classical conditions, which are sometimes too stringent and require the existence of third derivatives \citep[][Section 5.6]{vandervaart1998}.

We next illustrate the Bahadur representation \eqref{eq:bahadur} for exponential families on $(\Omega, d)$ in their natural parametrization, whose densities with respect to a common dominating measure $\nu$ are $p_{\boldsymbol{\eta}}(x) = h(x)\exp\{\boldsymbol{\eta}^{\top}\bT(x) \allowbreak -A(\boldsymbol{\eta})\},$ $\boldsymbol{\eta} \in \mathcal{H}$, where $\mathcal{H}$ is the natural parameter space. It is assumed that $\mathcal{H}$ is open, $\bT:\Omega\to\mathbb{R}^p$ is measurable, $A:\mathcal{H}\to\mathbb{R}$, and the Fisher information matrix $\mathcal I_{\boldsymbol{\eta}}=\partial^2 A(\boldsymbol{\eta})/(\partial\boldsymbol{\eta}\,\partial\boldsymbol{\eta}^{\top})$ is non-singular. This class includes the multivariate normal and logistic Gaussian families, as well as the vMF and HvMF distributions under the canonical parametrization $\bxi = \kappa \bmu$ (see Section~\ref{sec:dp_examples}). %

\begin{proposition}[A Bahadur representation for exponential families]
\label{prop:bahadur_exponential_family}
Let $X_1,\ldots,X_n$ be iid from an exponential family with true parameter
$\boldsymbol{\eta}_0\in \mathcal{H}$, and let $\hat{\boldsymbol{\eta}}$ be the maximum likelihood estimator of $\boldsymbol{\eta}_0$.
Then,
\begin{align*}
\sqrt{n}\big(\hat{\boldsymbol{\eta}}-\boldsymbol{\eta}_0\big)
=
\mathcal I_{\boldsymbol{\eta}_0}^{-1}
\frac{1}{\sqrt{n}}
\sum_{i=1}^n
\Big(
\bT(X_i)- \frac{\partial}{\partial\boldsymbol{\eta}} A(\boldsymbol{\eta}_0)
\Big)
+
o_{\mathsf P}(1).
\end{align*}
\end{proposition}

To propose GOF tests for composite null hypotheses, consider $\mu_{\hat{\btheta}}$, the probability measure corresponding to $\hat{\btheta}$, and define $F_{\omega}^{\mu_{\hat{\btheta}}}$ as the distance profile of $\mu_{\hat{\btheta}}$ at $\omega$.
We can then define the random field
\begin{align*}
    \G_{n,\mu}^{\mu_{\hat{\btheta}}}(\omega, t) \defin \sqrt{n}\left(F_{n,\omega}^{\mu}(t)-F_{\omega}^{\mu_{\hat{\btheta}}}(t)\right), \quad (\omega, t) \in \Omega \times \R_{\geq 0}.
\end{align*}

The process $\G_{n,\mu}^{\mu_{\hat{\btheta}}}$ is a natural extension of $\G_{n,\mu}^{\mu_0}$ to the case of composite null hypotheses, plugging in the estimated parameter $\hat{\btheta}$ instead of the true parameter $\btheta_0$. The asymptotic distribution of $\G_{n,\mu}^{\mu_{\hat{\btheta}}}$ follows by proving that the class of functions 
\begin{align}\label{eq:def_F}
\cF \defin \{f_{\omega,t} : \omega \in \Omega, t \in \R_{\geq 0}\}, \quad f_{\omega,t}(x) \defin y_{\omega, t}(x) - F_{\omega}^{\mu_{\btheta_0}}(t) + \dot{F}_{\omega}^{\btheta_0}(t)^\top \bV_{\btheta_0}^{-1} \bpsi_{\btheta_0}(x),
\end{align}
is Donsker. However, $\G_{n,\mu}^{\mu_{\hat{\btheta}}}$ is not the empirical process indexed by $\cF$. In practice, this makes its computation inefficient for bootstrapping, since it requires the recomputation of the distance profile $F_{\omega}^{\mu_{\hat{\btheta}}}$. For this reason, we give an alternative process that is more efficient to compute, and that is asymptotically equivalent to $\G_{n,\mu}^{\mu_{\hat{\btheta}}}$ under some regularity conditions.

Let $\widehat{\bV}$ be a measurable estimator of $\bV_{\btheta_0}$ such that $\widehat{\bV}\convp\bV_{\btheta_0}$ and $\widehat{\bV}$ is invertible with probability tending to one. This leads to the following (stochastic) class of functions:
\begin{align*}
\widehat{\cF}
\defin
\big\{
\hat{f}_{\omega,t}:(\omega,t)\in\Omega\times\R_{\ge0}
\big\},
\quad
\hat{f}_{\omega,t}(x)
\defin
y_{\omega,t}(x)-F_\omega^{\mu_{\hat{\btheta}}}(t)
+
\dot F_\omega^{\hat{\btheta}}(t)^\top\widehat{\bV}^{-1}\bpsi_{\hat{\btheta}}(x),
\end{align*}
whenever the derivative exists.  In other words, the plug-in is done in the class of functions. From this, one can define the empirical process 
$$
\widehat{\G}_{n,\mu}(\omega,t) \defin \sqrt{n}\,P_n\hat{f}_{\omega,t}.
$$ 
If $\hat{\btheta}$ satisfies the equation $P_n\bpsi_{\hat{\btheta}}=0$, then %
$\G_{n,\mu}^{\mu_{\hat{\btheta}}}(\omega,t) = \widehat{\G}_{n,\mu}(\omega,t).$ 

Even though $\G_{n,\mu}^{\mu_{\hat{\btheta}}}$ is slower to compute, its asymptotic distribution is easier to derive, and the asymptotics of $\widehat{\G}_{n,\mu}$ can be deduced from those of $\G_{n,\mu}^{\mu_{\hat{\btheta}}}$. Additionally, it does not require the computation of the derivative $\dot F_\omega^{\hat{\btheta}}(t)$, nor the estimation of the matrix $\bV_{\btheta_0}$, which may not exist or be difficult to estimate, making this approach more robust in complex scenarios.

We study the asymptotic behavior of $\G_{n,\mu}^{\mu_{\hat{\btheta}}}$ and $\widehat{\G}_{n,\mu}$ under the null hypothesis $H_0$. We will use the notation $\dot{F}_{\omega}^{\btheta_0}(t) \defin \left.\partial F_{\omega}^{\mu_{\btheta}}(t)/\partial \btheta\right|_{\btheta=\btheta_0}$.
The following differentiability condition is required in the asymptotic analyses:
\begin{enumerate}[label=(\textit{d}), ref=(\textit{d})]
    \item \label{d} The map $\boldsymbol{\theta}\mapsto F^{\mu_{\boldsymbol{\theta}}}$ is differentiable at $\boldsymbol{\theta}_0$ as a map from $\Theta$ into $\ell^\infty(\Omega\times\mathbb{R}_{\ge0})$. Equivalently, it admits the following uniform first-order expansion:
    \begin{align*}
    \sup_{\omega\in\Omega, t\ge0}
    \left|
    F_{\omega}^{\mu_{\boldsymbol{\theta}}}(t)
    -
    F_{\omega}^{\mu_{\boldsymbol{\theta}_0}}(t)
    -
    \dot F_{\omega}^{\boldsymbol{\theta}_0}(t)^\top
    (\boldsymbol{\theta}-\boldsymbol{\theta}_0)
    \right|
    =
    o(\|\boldsymbol{\theta}-\boldsymbol{\theta}_0\|)
    \quad
    \text{as }
    \boldsymbol{\theta}\to\boldsymbol{\theta}_0.
    \end{align*}
\end{enumerate}
In particular, for every $(\omega,t)\in\Omega\times\mathbb R_{\ge0}$, the map $\boldsymbol{\theta}\mapsto F_{\omega}^{\mu_{\boldsymbol{\theta}}}(t)$ is differentiable at $\boldsymbol{\theta}_0$, and $\sup_{\omega\in\Omega, t\ge0}\|\dot F_{\omega}^{\boldsymbol{\theta}_0}(t)\|<\infty$. For the asymptotics of $\widehat{\G}_{n,\mu}$, we also require:
\begin{enumerate}[label=(\textit{e}), ref=(\textit{e})]
\item $ \sup_{\omega\in\Omega,\, t\ge0} \|\dot F_\omega^{\hat{\btheta}}(t)\|\,\|\widehat{\bV}^{-1}\| = O_{\mathsf P}(1).$ \label{e}
\end{enumerate}
\begin{remark}\label{rem:sufficient_condition_d}
A sufficient condition for condition~\ref{d} is that
$\dot F_{\omega}^{\boldsymbol{\theta}}(t)$ exists for
$\boldsymbol{\theta}$ in a neighborhood of $\boldsymbol{\theta}_0$ and
$
\sup_{\omega\in\Omega,\,t\ge0}
\big\|
\dot F_{\omega}^{\boldsymbol{\theta}}(t)
-
\dot F_{\omega}^{\boldsymbol{\theta}_0}(t)
\big\|
\to 0
$ as 
$
\boldsymbol{\theta}\to\boldsymbol{\theta}_0.
$
\end{remark}

\begin{proposition}
\label{prop:verification_assumptions_exponential_family}
Condition~\ref{d} holds for exponential families at every $\boldsymbol{\eta}_0\in\mathcal H$. Moreover,
$\dot F_{\omega}^{\boldsymbol{\eta}_0}(t)=\E{y_{\omega,t}(X)\bpsi_{\boldsymbol{\eta}_0}(X)},$ for each
$\omega\in\Omega, t\ge0.$
\end{proposition}
The following result characterizes the asymptotic behavior of $\G_{n,\mu}^{\mu_{\hat{\btheta}}}$.

\begin{theorem}[Asymptotic null distributions of $\G_{n,\mu}^{\mu_{\hat{\btheta}}}$ and $\widehat{\G}_{n,\mu}$]\label{thm:convergence_G_n_composite}
Consider an estimator $\hat{\btheta}$ satisfying conditions \ref{cp1}--\ref{cp5} and condition \ref{d}. 
Then, under $H_0$, the class of functions $\cF$ defined in \eqref{eq:def_F} is Donsker. Therefore,
$$
\G_{n,\mu}^{\mu_{\hat{\btheta}}} \convl \tilde{\G}_0 \text { in } \ell^{\infty}(\mathcal{F}),
$$
where $\tilde{\G}_0$ is a zero mean Gaussian process with covariance given by
\begin{equation}\label{eq:cov_comp}
\begin{aligned}
\tilde{\mathcal{C}}_{(\omega_1,t_1), (\omega_2, t_2)}
 & = \mathcal{C}_{(\omega_1,t_1), (\omega_2, t_2)} + \dot{F}_{\omega_2}^{\btheta_0}(t_2)^\top \bV_{\btheta_0}^{-1} \mathsf{E}\left(y_{\omega_1, t_1}(X) \bpsi_{\btheta_0}(X)\right)  
 \\
& \quad + \dot{F}_{\omega_1}^{\btheta_0}(t_1)^\top \bV_{\btheta_0}^{-1} \mathsf{E}\left(y_{\omega_2, t_2}(X) \bpsi_{\btheta_0}(X)\right)
\\
 & \quad + \dot{F}_{\omega_1}^{\btheta_0}(t_1)^\top  \bV_{\btheta_0}^{-1}\Ebig{\bpsi_{\btheta_0}(X) \bpsi_{\btheta_0}(X)^\top}  \big(\bV_{\btheta_0}^{-1}\big)^{\top}\dot{F}_{\omega_2}^{\btheta_0}(t_2) ,
\end{aligned}
\end{equation}
with $\mathcal{C}_{(\omega_1,t_1), (\omega_2, t_2)}$ the covariance process under the simple null, given in \eqref{eq:cov_simp}. 

If, in addition, condition \ref{e} holds, then  
\begin{align*}
\sup_{\omega\in\Omega,\, t\ge0}
\big|
\widehat{\G}_{n,\mu}(\omega,t)
-
\G_{n,\mu}^{\mu_{\hat{\btheta}}}(\omega,t)
\big|
=
o_{\mathsf P}(1).
\end{align*}
Consequently, $\widehat{\G}_{n,\mu}$ has the same weak limit $\tilde{\G}_0$ as $\G_{n,\mu}^{\mu_{\hat{\btheta}}}$ under the composite null hypothesis. 
\end{theorem}

\begin{proposition}[Covariance process for exponential families]
\label{prop:covariance_exponential_family}
For an exponential family with true parameter $\boldsymbol{\eta}_0\in\mathcal H$, when $\boldsymbol{\eta}_0$ is estimated by maximum likelihood, the covariance process in \eqref{eq:cov_comp} simplifies to
$
\widetilde{\mathcal C}_{(\omega_1,t_1),(\omega_2,t_2)}
=
\mathcal C_{(\omega_1,t_1),(\omega_2,t_2)}
-
\dot F_{\omega_1}^{\boldsymbol{\eta}_0}(t_1)^\top
\mathcal I_{\boldsymbol{\eta}_0}^{-1}
\dot F_{\omega_2}^{\boldsymbol{\eta}_0}(t_2),
$ for every $(\omega_1,t_1),(\omega_2,t_2)\in\Omega\times\mathbb R_{\ge0}$.
\end{proposition}

\subsection{Test statistics}

Given a sample $X_1, \dots, X_n$, define the empirical probability distribution $P_n$ as the discrete measure that assigns mass $1/n$ to each of the points $X_1,\dots,X_n$.
Consider the Kolmogorov--Smirnov and Cramér--von Mises test statistics:
\begin{alignat}{2}
T_n^{\mathrm{KS}}
&\defin \sup_{\omega \in \Omega,t\ge0}\big|\G_{n,\mu}^{\mu_0}(\omega, t)\big|,
&\qquad
T_n^{\mathrm{CM}}
&\defin \int_{\Omega\times\R_{\ge0}} \G_{n,\mu}^{\mu_0}(\omega, t)^2
\, \rd F_{n,\omega}^{\mu}(t) \, \rd P_n(\omega),
\label{eq:ks_test_statistic}\\
\tilde{T}_n^{\mathrm{KS}}
&\defin \sup_{\omega \in \Omega,t\ge0}\big|\G_{n,\mu}^{\mu_{\hat{\btheta}}}(\omega, t)\big|,
&\qquad
\tilde{T}_n^{\mathrm{CM}}
&\defin \int_{\Omega\times\R_{\ge0}} \G_{n,\mu}^{\mu_{\hat{\btheta}}}(\omega, t)^2
\, \rd F_{n,\omega}^{\mu}(t) \, \rd P_n(\omega),
\label{eq:test_statistic_CM_tilde}\\
\widehat T_n^{\mathrm{KS}}
&\defin \sup_{\omega \in \Omega,t\ge0}
\big|\widehat{\G}_{n,\mu}(\omega,t)\big|,
&\qquad
\widehat T_n^{\mathrm{CM}}
&\defin \int_{\Omega\times\R_{\ge0}}
\big(\widehat{\G}_{n,\mu}(\omega,t)\big)^2
\, \rd F_{n,\omega}^{\mu}(t)\,\rd P_n(\omega).
\label{eq:test_statistic_CM_hat}
\end{alignat}
$T_n^{\mathrm{KS}}$ and $T_n^{\mathrm{CM}}$ apply to the simple null hypothesis, whereas $\tilde{T}_n^{\mathrm{KS}}$, $\tilde{T}_n^{\mathrm{CM}}$, $\widehat T_n^{\mathrm{KS}}$, and $\widehat T_n^{\mathrm{CM}}$ apply to composite null hypotheses. The CM statistics use the empirical integration measure $\rd F_{n,\omega}^{\mu} \, \rd P_n$ because $\rd F_{\omega}^{\mu_0} \, \rd \mu_0$ depends on the simple null distribution, and under the composite null, $\rd F_{\omega}^{\mu_{\btheta_0}} \, \rd \mu_{\btheta_0}$ depends on the unknown parameter $\btheta_0$. In this way, the integration measure is model-free and fully data-driven.

Theorem \ref{thm:test_statistics_simp} proves the asymptotic distributions under the simple null hypothesis of the KS and CM test statistics given in \eqref{eq:ks_test_statistic}.
\begin{theorem}[Asymptotic null distributions of the test statistics, simple null]\label{thm:test_statistics_simp}
As $n$ diverges to infinity, and under the simple null hypothesis, the following statements hold:
\begin{align*}
T_n^{\mathrm{KS}} \convl \sup_{\omega \in \Omega, t\ge0}\big|\G_{0}(\omega, t)\big| \quad \text{and} \quad
T_n^{\mathrm{CM}}\convl \int_{\Omega\times\R_{\ge0}} \G_{0}(\omega, t)^2 \, \rd F_{\omega}^{\mu_0}(t) \, \rd \mu_0(\omega).
\end{align*}
\end{theorem}

\begin{remark}\label{rem:separability_G_0}
    We know that the class $\cY$ is Donsker, thus $\G_{0}(\omega, t)$ is tight in $\ell^\infty(\cY)$. Tightness is equivalent to there being a $\sigma$-compact set that has probability one (see, e.g., p. 15 in \citealp{VanderVaartWellner2023}, and p. 105 in \citealp{Kosorok2008}). Every $\sigma$-compact set is Lindelöf, and a metric space is Lindelöf if and only if it is separable \citep[see, e.g., Theorem 16.11 of][]{Willard2004}. Since there is a measurable and separable set with probability one, then $\G_{0}(\omega, t)$ is separable.  
\end{remark}
\begin{remark}
    The random variable $\sup_{\omega, t} \G_0(\omega, t)$ may fail to be measurable, since $\Omega \times \R_{\ge 0}$ is uncountable \citep[see, e.g., p. 15 in][]{Adler2007}. Separability solves the measurability issues, since it guarantees the existence of a countable subset $D \subset \Omega \times \R_{\ge 0}$ such that $\sup_{(\omega,t) \in \Omega \times \R_{\ge 0}} \G_0(\omega, t) = \sup_{(\omega,t) \in D} \G_0(\omega, t)$ almost surely, and the supremum of a countable set of measurable random variables is measurable.
    Thus, it is sensible to speak of weak convergence of $\sup_{\omega, t} \G_{n,\mu}^{\mu_0}(\omega, t)$ to $\sup_{\omega, t} \G_0(\omega, t)$, since the limiting process is required to be separable. Observe that this also implies that $\sup_{\omega, t} \G_{n,\mu}^{\mu_0}(\omega, t)$ is asymptotically measurable by Lemma 1.3.8 in \cite{VanderVaartWellner2023}.
\end{remark}

Theorem \ref{thm:test_statistics_comp} is the corresponding result for the composite null hypothesis.
\begin{theorem}[Asymptotic null distributions of the test statistics, composite null]\label{thm:test_statistics_comp}
Assume the conditions of Theorem~\ref{thm:convergence_G_n_composite}. The following statements hold under the composite null hypothesis:
\begin{enumerate}[label=(\textit{\roman*}), ref=(\textit{\roman*})]
\item As $n$ diverges to infinity, \label{thm:test_statistics_comp_i}
\begin{align*}
\tilde{T}_n^{\mathrm{KS}} \convl \sup_{\omega \in \Omega, t\ge0}\big|\tilde{\G}_{0}(\omega, t)\big|.
\end{align*}
The same result holds for $\widehat{T}_n^{\mathrm{KS}}$ under \ref{e}.

\item If the map $(\omega,t)\mapsto \dot F_{\omega}^{\btheta_0}(t)$ is uniformly continuous with respect to
$d_{\Omega\times\R_{\ge0}}$, then
$$
\tilde{T}_n^{\mathrm{CM}}\convl \int_{\Omega\times\R_{\ge0}} \tilde{\G}_0(\omega, t)^2 \, \rd F_{\omega}^{\mu_{\btheta_0}}(t) \, \rd \mu_{\btheta_0}(\omega)
$$
as $n$ diverges to infinity. The same result holds for $\widehat{T}_n^{\mathrm{CM}}$ under \ref{e}.
\label{thm:test_statistics_comp_ii}
\end{enumerate}
\end{theorem}

\begin{remark}\label{ap:rem_uniform_continuity_Fdot}
The uniform continuity assumption on $(\omega,t)\mapsto\dot F_{\omega}^{\boldsymbol{\theta}_0}(t)$ holds whenever
$\dot F_{\omega}^{\boldsymbol{\theta}_0}(t)=\E{y_{\omega,t}(X)\bpsi_{\boldsymbol{\theta}_0}(X)}$ for every $(\omega,t)\in\Omega\times\R_{\ge0}.$ This identity holds for exponential families by Proposition~\ref{prop:verification_assumptions_exponential_family}, and, more generally, for parametric families admitting densities with respect to a common dominating measure, provided that $\bpsi_{\boldsymbol{\theta}}$ is the score and differentiation under the integral sign is justified. If $\dot F_{\omega}^{\boldsymbol{\theta}_0}(t)=\E{y_{\omega,t}(X)\bpsi_{\boldsymbol{\theta}_0}(X)}$ holds, then for any $(\omega,t),(\omega^\prime,t^\prime)\in\Omega\times\R_{\ge0}$ and some constant $C_1>0$,
$$
\begin{aligned}
\big\|\dot F_{\omega}^{\boldsymbol{\theta}_0}(t)-\dot F_{\omega^\prime}^{\boldsymbol{\theta}_0}(t^\prime)\big\|
&\le \left(\E{\|\bpsi_{\boldsymbol{\theta}_0}(X)\|^2}\right)^{1/2}
\big(\Ebig{(y_{\omega,t}(X)-y_{\omega^\prime,t^\prime}(X))^2}\big)^{1/2}\\
&\le C_1\left(\E{\|\bpsi_{\boldsymbol{\theta}_0}(X)\|^2}\right)^{1/2}
d_{\Omega\times\R_{\ge0}}\big((\omega,t),(\omega^\prime,t^\prime)\big)^{1/2},
\end{aligned}
$$
by condition~\ref{cp2}, Cauchy--Schwarz applied componentwise, and Lemma~\ref{ap:lem_relation_L2_prod} in the Supplementary Material. Hence, $(\omega,t)\mapsto\dot F_{\omega}^{\boldsymbol{\theta}_0}(t)$ is uniformly continuous.
\end{remark}

Taking the supremum over all values of $\omega \in \Omega$ is impractical, especially in high dimensions or complex metric spaces. For this reason, we provide conditions under which the supremum over the whole metric space can be replaced with maximization over the sample elements.
                          
\begin{proposition}\label{prop:sup_max_ks}
    The following results hold, with $\mathsf{P}^*$ denoting outer probability:
    \begin{enumerate}[label=(\textit{\roman*}), ref=(\textit{\roman*})]
    \item Assume that $\Omega \subset \supp{\mu_0}$. Under the simple null hypothesis, it holds that \label{prop:sup_max_ks_i}
    \begin{align}\label{eq:sup_max_ks}
    \sup_{\omega\in \Omega, t\ge0} \left|\G_{n,\mu}^{\mu_0}(\omega, t)\right|
    =
    \max_{1\le i,j \le n}\left|\G_{n,\mu}^{\mu_0}(X_i,d(X_i,X_j))\right|
    + o_{\mathsf{P}^*}(1).
    \end{align}
\item Assume also that the conditions of Theorem~\ref{thm:convergence_G_n_composite} hold, $\Omega\subset\supp{\mu_{\btheta_0}}$, and that the map $(\omega,t)\mapsto  \dot{F}_{\omega}^{\btheta_0}(t)$ is uniformly continuous with respect to $d_{\Omega\times\R_{\ge0}}\big((\omega,t),(\omega^\prime,t^\prime)\big)\defin d(\omega,\omega^\prime)+|t-t^\prime|$. Then, under the composite null hypothesis, it holds that \label{prop:sup_max_ks_ii}
\begin{align}\label{eq:sup_max_ks_comp}
    \sup_{\omega\in \Omega, t\ge0} \left|\G_{n,\mu}^{\mu_{\hat{\btheta}}}(\omega, t)\right| = \max_{1\le i,j \le n} \left| \G_{n,\mu}^{\mu_{\hat{\btheta}}}(X_i, d(X_i,X_j)) \right| + o_{\mathsf{P}^*}(1).
\end{align}
\end{enumerate}
\end{proposition}

\section{Non-null asymptotics}\label{sec:non_null_asymptotics}

\subsection{Fixed alternatives}\label{sec:fixed_alternatives}

Suppose that $X_1, \ldots, X_n$ are generated from a distribution $\mu$ that is different from $\mu_0$, the distribution under the null hypothesis.
In other words, we are interested in the behavior of the test statistics under alternatives of the form
\begin{align*}
H_1: \mu \neq \mu_0 \text{ (simple null)}, \quad \text{and} \quad H_1: \mu \notin \{\mu_{\btheta}:\btheta\in\bTheta\} \text{ (composite null)}.
\end{align*}

\subsubsection{Simple null hypothesis}\label{sec:fixed_alternatives_simple}

The following theorem establishes the consistency of both the KS and CM tests against fixed alternatives. In terms of distance profiles, the alternative $H_1: \mu \neq \mu_0$ is equivalent to the existence of a pair $(\omega, t) \in \Omega \times \R_{\ge0}$ such that $F_{\omega}^{\mu}(t) \neq F_{\omega}^{\mu_0}(t)$. For the CM statistic, the separation at a single point $(\omega,t)$ is extended to a neighborhood where it is uniform. This neighborhood has positive probability under the alternative, so that the CM statistic can detect the separation between the distance profiles under the null and the alternative. A Glivenko--Cantelli argument then allows us to replace the empirical ingredients in $T_n^{\mathrm{CM}}$ by their population counterparts, and the law of large numbers for $U$-statistics yields the desired result.

\begin{theorem}[Consistency under fixed alternatives, simple null]\label{thm:fixed_alternative_divergence_simple}
Under $H_1$, $T_n^{\mathrm{KS}}$ diverges in outer probability, while $T_n^{\mathrm{CM}}$ also diverges in outer probability if condition \ref{a2} holds, or if condition \ref{a1} holds and $\supp{\mu_0}\subset\supp{\mu}$.
\end{theorem}

\begin{remark} 
Condition~\ref{a2} guarantees the existence of a point $(\omega,t)$ at which
$F_{\omega}^{\mu}(t)\neq F_{\omega}^{\mu_0}(t)$, and such that $\omega\in\operatorname{supp}(\mu)$. In contrast, condition~\ref{a1} alone only ensures
the existence of such a point with $\omega\in\supp{\mu}\cup\supp{\mu_0}$, which need not belong to
$\operatorname{supp}(\mu)$. The extra assumption $\supp{\mu_0} \subset \supp{\mu}$ guarantees that this point can be chosen with $\omega\in\operatorname{supp}(\mu)$. 
This ensures that the differences between $F_\omega^{\mu}(t)$ and $F_\omega^{\mu_0}(t)$ are detected by the CM statistic.
\end{remark}

\subsubsection{Composite null hypothesis}\label{sec:fixed_alternatives_composite}

In terms of distance profiles, the alternative hypothesis means that for every $\btheta\in\bTheta$ there exists at least one pair
$(\omega,t)\in\Omega\times\R_{\ge 0}$ such that $F^\mu_{\omega}(t)\neq F^{\mu_{\btheta}}_{\omega}(t)$. 
We work with a value $\btheta_1$, which is the limit to which $\hat{\btheta}$ converges in probability under $H_1$, and is $\btheta_0$ under $H_0$.

\begin{theorem}[Consistency under fixed alternatives, composite null]
\label{thm:fixed_alt_divergence_composite}
Consider an estimator $\hat{\btheta}$, and assume that there exists $\btheta_1\in\bTheta$ such that $\hat{\btheta}\convp\btheta_1$.
Assume also that $\btheta\mapsto F^{\mu_{\btheta}}$ is continuous at $\btheta_1$ as a map from $\Theta$ into $\ell^\infty(\Omega\times\mathbb R_{\ge0})$; that is,
\begin{align}\label{eq:theta_star_profile_continuity}
\sup_{\omega\in\Omega,t\ge0}
\left|
F_\omega^{\mu_{\btheta}}(t)
-
F_\omega^{\mu_{\btheta_1}}(t)
\right|
\to 0
\quad\text{as }\btheta\to\btheta_1.
\end{align}
Then, under $H_1$, $\tilde{T}_n^{\mathrm{KS}}$ diverges in outer probability, i.e.,
$\tilde{T}_n^{\mathrm{KS}}\convprob{\mathsf P^*}\infty$.
If, in addition, condition \ref{e} holds and $P_n\bpsi_{\hat{\btheta}}=o_{\mathsf P}(1)$, then
$\widehat{T}_n^{\mathrm{KS}}\convprob{\mathsf P^*}\infty$ under $H_1$.

If condition \ref{a2} holds, or if condition \ref{a1} holds and $\supp{\mu_{\btheta_1}}\subset\supp{\mu}$, $\tilde{T}_n^{\mathrm{CM}}$ also diverges in outer probability. If, in addition, condition \ref{e} holds and
$P_n\bpsi_{\hat{\btheta}}=o_{\mathsf P}(1)$, then $\widehat{T}_n^{\mathrm{CM}}\convprob{\mathsf P^*}\infty$ under $H_1$.
\end{theorem}
Condition \eqref{eq:theta_star_profile_continuity} is a weaker analogue of condition~\ref{d} at $\btheta_1$, requiring continuity rather than differentiability.

\begin{remark}[Sufficient conditions for \eqref{eq:theta_star_profile_continuity}]\label{rem:examples_local_derivative_bound}

A sufficient condition for \eqref{eq:theta_star_profile_continuity} is that there exist $\rho>0$ and $L<\infty$ such that
\begin{equation}\label{eq:local_derivative_bound}
\sup_{\btheta\in B(\btheta_1,\rho)} \sup_{\omega\in\Omega, t\ge 0}
\bigl\|\dot{F}^{\btheta}_{\omega}(t)\bigr\|
\le L,
\end{equation}
where the differentiability of $\btheta\mapsto F^{\mu_{\btheta}}_\omega(t)$ in $B(\btheta_1,\rho)$, for every $(\omega,t)\in\Omega\times\R_{\ge0}$, is implicitly assumed. Indeed, if \eqref{eq:local_derivative_bound} holds, then the mean value theorem gives, for every $\btheta\in B(\btheta_1,\rho)$, that
$
\sup_{\omega,t}
\bigl|F^{\mu_{\btheta}}_\omega(t)-F^{\mu_{\btheta_1}}_\omega(t)\bigr|
\le L\|\btheta-\btheta_1\|,
$
which implies \eqref{eq:theta_star_profile_continuity}.

For exponential families, \eqref{eq:local_derivative_bound} holds at every $\boldsymbol{\eta}_1\in\mathcal H$. Indeed, choose $\rho>0$ such that $\overline{B}(\boldsymbol{\eta}_1,\rho)\subset\mathcal H$. By Proposition~\ref{prop:verification_assumptions_exponential_family},
$$
\bigl\|\dot F_{\omega}^{\boldsymbol{\eta}}(t)\bigr\|
\le
\mathsf{E}_{\boldsymbol{\eta}}\bigl\|\bpsi_{\boldsymbol{\eta}}(X)\bigr\|
\le
\big(\mathsf{E}_{\boldsymbol{\eta}}\bigl\|\bpsi_{\boldsymbol{\eta}}(X)\bigr\|^2\big)^{1/2}
=
\operatorname{tr}(\mathcal I_{\boldsymbol{\eta}})^{1/2}.
$$
Therefore,
$$
\sup_{\boldsymbol{\eta}\in B(\boldsymbol{\eta}_1,\rho)}
\sup_{\omega\in\Omega,t\ge0}
\bigl\|\dot F_{\omega}^{\boldsymbol{\eta}}(t)\bigr\|
\le
\sup_{\boldsymbol{\eta}\in\overline{B}(\boldsymbol{\eta}_1,\rho)}
\operatorname{tr}(\mathcal I_{\boldsymbol{\eta}})^{1/2}
<\infty,
$$
where boundedness follows from the continuity of $\boldsymbol{\eta}\mapsto\mathcal I_{\boldsymbol{\eta}}$ and compactness of $\overline{B}(\boldsymbol{\eta}_1,\rho)$.
\end{remark}

\subsection{Local alternatives}\label{sec:local_alternatives}

We study the asymptotic behavior of the KS and CM statistics under local alternatives, for simple and composite null hypotheses.
Fix a probability measure $G$ on $(\Omega,\mathcal{B}(\Omega))$ that satisfies condition \ref{bp2}, with Lipschitz constant $K_G$.
Define the sequence of local alternatives converging to $H_0$ at rate $n^{-1/2}$ by
\begin{equation}\label{eq:local_alt_Qn_def}
Q_n \defin (1-n^{-1/2}) P + n^{-1/2}G,
\quad n\in\mathbb{N},
\end{equation}
where $P$ is the true distribution under the null ($P=\mu_0$ for a simple null and $P=\mu_{\btheta_0}$ for a composite null).
For $(\omega,t)\in\Omega\times \R_{\ge0}$, this implies
$F^{Q_n}_\omega(t)
\defin
F^{P}_\omega(t)+n^{-1/2}h(\omega,t),$
where 
$h(\omega,t)\defin F^G_\omega(t)-F^{P}_\omega(t).$
Hence,
\begin{equation}\label{eq:local_alt_process_decomp}
\sqrt{n}\big(F_{n,\omega}^{Q_n}(t)-F^{P}_\omega(t)\big)
=
\G_{n,Q_n}^{Q_n}(\omega,t) + h(\omega,t),
\end{equation}
where 
$\G_{n,Q_n}^{Q_n}(\omega,t)
\defin
\sqrt{n}\big(F_{n,\omega}^{Q_n}(t)-F^{Q_n}_\omega(t)\big)
=
\sqrt{n}(P_n-Q_n)y_{\omega,t}.$

Observe that $\G_{n,Q_n}^{Q_n}$ is a triangular-array empirical process, since the underlying law $Q_n$ varies with $n$. The following result shows that $\G_{n,Q_n}^{Q_n}$ 
admits a coupling with $\G_{n,P}^{P}$ such that their difference converges in probability to zero in $\ell^\infty(\cF)$.
\begin{proposition}\label{prop:bernoulli}
Let $\cF$ be a $P$- and $G$-Donsker class of measurable real-valued functions. Assume that
$\sup_{f\in\cF}\big|(G-P)f\big|<\infty$.
Then, under the local alternative \eqref{eq:local_alt_Qn_def}, $\G_{n,Q_n}^{Q_n}$ admits a coupling with $\G_{n,P}^{P}$ such that
\begin{equation}\label{eq:sup_bernoulli_prob}
\sup_{f \in \cF}
\Big|
\G_{n,Q_n}^{Q_n}(f)
-
\G_{n,P}^{P}(f)\Big|
\convp 0. 
\end{equation}
\end{proposition}
\begin{corollary}[Asymptotic distribution of $\G_{n,Q_n}^{Q_n}$ under the local alternative]\label{cor:bernoulli}
Under the assumptions of Proposition \ref{prop:bernoulli}, $\G_{n,Q_n}^{Q_n}$ has the same weak limit as $\G_{n,P}^{P}$  under the local alternative \eqref{eq:local_alt_Qn_def}. %
\end{corollary}
\subsubsection{Simple null hypothesis}\label{sec:local_alternatives_simple}

The following theorem gives the weak limit under the local alternative of the KS and CM statistics, $T_n^{\mathrm{KS}}$ and $T_n^{\mathrm{CM}}$.
\begin{theorem}[Asymptotic distribution of the test statistics under the local alternative, simple null]\label{thm:local_alt_simple}
Under the local alternative \eqref{eq:local_alt_Qn_def}, it holds that
\begin{equation*}
T_n^{\mathrm{KS}}
\convl
\sup_{\omega\in\Omega,\,t\ge0}
|\G_0(\omega,t)+h(\omega,t)|,
\quad
T_n^{\mathrm{CM}}
\convl
\int_{\Omega\times\R_{\ge0}}
(\G_0(\omega,t)+h(\omega,t))^2
\,\rd F_\omega^{\mu_0}(t)\,\rd\mu_0(\omega).
\end{equation*}
\end{theorem}

\subsubsection{Composite null hypothesis}\label{sec:local_alternatives_composite}

Under the local alternative, there are two relevant parameters: the estimator $\hat{\btheta}$ based on $P_n$, and the true population parameter under the local alternative, $\tilde{\btheta}_n$, defined as any solution to $\mathsf{E}_{Q_n}\big(\bpsi_{\btheta}(X)\big) = 0$. We first give sufficient local conditions under which $\hat{\btheta}$ and $\tilde{\btheta}_n$ both converge ($\hat{\btheta}$ in probability) to $\btheta_0$ under the local alternative \eqref{eq:local_alt_Qn_def}.

\begin{proposition}[Sufficient conditions for the convergence of $\hat{\btheta}$ and $\tilde{\btheta}_n$ to $\btheta_0$]\label{prop:local_conv_roots_Qn}
Assume conditions \ref{cp2} and \ref{cp3}, and $\mathsf E_G\|\bpsi_{\btheta_0}(X)\|^2<\infty$.
Consider a sufficiently small bounded neighborhood $U$ of $\btheta_0$. Assume that there exists a measurable function $L$ such that $\mathsf E_{\mu_{\btheta_0}}\big(L(X)^2\big)<\infty,$ $\mathsf E_G\big(L(X)^2\big)<\infty,$ and
\begin{equation}\label{eq:local_cp1_prop_local_conv}
\begin{aligned}
\|\bpsi_{\btheta_1}(x)-\bpsi_{\btheta_2}(x)\|
\le
L(x)\|\btheta_1-\btheta_2\|
\quad\text{for all }\btheta_1,\btheta_2\in U\text{ and all }x.
\end{aligned}
\end{equation}
The following results hold under the local alternative \eqref{eq:local_alt_Qn_def}:
\begin{enumerate}[label=(\textit{\roman*}), ref=(\textit{\roman*})]
\item Let $\tilde\btheta_n\in\Theta$ be such that $\mathsf{E}_{Q_n}\big(\bpsi_{\tilde\btheta_n}(X)\big)=0$ and $\tilde\btheta_n\in U$ for all $n$ sufficiently large. Then, $\tilde\btheta_n\to\btheta_0$ as $n\to\infty$. \label{prop:local_conv_roots_Qn_i}
\item Let $\hat\btheta\in\Theta$ be such that $\mathsf{E}_{P_n}\big(\bpsi_{\hat\btheta}(X)\big)=0$ and $\mathsf P(\hat\btheta\in U)\to1$. Then, $\hat\btheta\convp\btheta_0$ as $n\to\infty$. \label{prop:local_conv_roots_Qn_ii}
\end{enumerate}
\end{proposition}
Observe that condition \eqref{eq:local_cp1_prop_local_conv} is a local version of condition \ref{cp1}.

Obtaining the asymptotic distribution of the test statistics under the local alternative requires a Bahadur-type asymptotic linear expansion of $\hat{\btheta}$ and $\tilde{\btheta}_n$ as in \eqref{eq:bahadur}, under the local alternative. Since the parametric model is misspecified under $Q_n$, it is important to provide sufficient conditions under which such expansions hold. This is given in Theorem \ref{thm:local_alt_conv_process_correction}. To make the following result more general, we only require that $\tilde\btheta_n$ and $\hat\btheta$ converge to $\btheta_0$, instead of assuming the conditions of Proposition \ref{prop:local_conv_roots_Qn} directly.

\begin{theorem}[A Bahadur expansion of $\hat{\btheta}$ and $\tilde{\btheta}_n$ under the local alternative]\label{thm:local_alt_conv_process_correction}
Assume the conditions \ref{cp1}--\ref{cp3}, $\mathsf E_G\big(L(X)^2\big)<\infty$, and $\mathsf E_G\|\bpsi_{\btheta_0}(X)\|^2<\infty$.
Under the local alternative, the following results hold:
\begin{enumerate}[label=(\textit{\roman*}), ref=(\textit{\roman*})]
\item Let $\tilde\btheta_n\in\Theta$ be such that $\mathsf E_{Q_n}\big(\bpsi_{\tilde\btheta_n}(X)\big)=0$ 
and $\tilde\btheta_n\to\btheta_0$. Then, \label{thm:local_alt_conv_process_correction_i}
\begin{align}\label{eq:thm_local_alt_conv_process_correction_i}
\tilde{\btheta}_n - \btheta_0 = - \frac{1}{\sqrt{n}} \bV_{\btheta_0}^{-1} \mathsf E_G\big(\bpsi_{\btheta_0}(X)\big) + o(n^{-1/2}),
\end{align}
\item Let $\hat\btheta\in\Theta$ be such that $\mathsf E_{P_n}\big(\bpsi_{\hat\btheta}(X)\big)=o_{\mathsf{P}}(n^{-1/2})$ and $\hat\btheta\convp\btheta_0$. Then,\label{thm:local_alt_conv_process_correction_ii} 
\begin{align}\label{eq:thm_local_alt_conv_process_correction_ii}
\hat{\btheta} - \btheta_0 = - \frac{1}{n} \bV_{\btheta_0}^{-1} \sum_{i=1}^{n} \bpsi_{\tilde{\btheta}_n}(X_i) -  \frac{1}{\sqrt{n}} \bV_{\btheta_0}^{-1} \mathsf E_G\big(\bpsi_{\btheta_0}(X)\big) + o_{\mathsf{P}}(n^{-1/2}),
\end{align}
\end{enumerate}
\end{theorem}

The following result gives the asymptotic distribution of $\G_{n,Q_n}^{\mu_{\hat{\btheta}}} \defin \sqrt{n}\bigl(F_{n,\omega}^{Q_n}(t)-F^{\mu_{\hat{\btheta}}}_\omega(t)\bigr)$ under the local alternative \eqref{eq:local_alt_Qn_def} in the composite null case.

\begin{theorem}[Asymptotic distribution of $\G_{n,Q_n}^{\mu_{\hat{\btheta}}}$ under the local alternative]\label{thm:local_alt_Gn_composite}
Assume the conditions in Theorem \ref{thm:local_alt_conv_process_correction} and condition \ref{d}. Let $\tilde{\btheta}_n\in\Theta$ be such that $\mathsf E_{Q_n}(\bpsi_{\tilde{\btheta}_n}(X))=0$ and $\tilde{\btheta}_n\to\btheta_0$. Let also $\hat{\btheta}\in\Theta$ be such that $\mathsf E_{P_n}\big(\bpsi_{\hat{\btheta}}(X)\big)=o_{\mathsf{P}}(n^{-1/2})$ and $\hat{\btheta}\convp\btheta_0$ under the local alternative. Then, it holds that
\begin{equation}\label{eq:local_alt_composite_consistency}
\G_{n,Q_n}^{\mu_{\hat{\btheta}}}
\convl
\tilde{\G}_{0}(\omega,t)
+
h(\omega,t)
+
\dot F^{\btheta_0}_\omega(t)^\top \bV_{\btheta_0}^{-1}\mathsf E_G\big(\bpsi_{\btheta_0}(X)\big)
\quad\text{in }\ell^\infty(\cF)
\end{equation}
under the local alternative.
\end{theorem}

The next theorem uses this convergence result to derive the asymptotic behavior of the KS and CM statistics under local alternatives for the composite null hypothesis.

\begin{theorem}[Asymptotic distribution of the test statistics under the local alternative, composite null]\label{thm:consistency_local_composite}
Assuming the conditions of Theorem~\ref{thm:local_alt_Gn_composite}, it holds under the local alternative that
$$
\tilde{T}_n^{\mathrm{KS}}
=
\sup_{\omega\in\Omega, t \ge0}
\big|
\G_{n,Q_n}^{\mu_{\hat{\btheta}}}(\omega,t)
\big|
 \convl 
\sup_{\omega\in\Omega, t \ge0}
\big|
\tilde{\G}_{0}(\omega,t)
+
h(\omega,t)
+
\dot F^{\btheta_0}_\omega(t)^\top \bV_{\btheta_0}^{-1}\mathsf{E}_G(\bpsi_{ \btheta_0}(X))
\big|.
$$
Under condition \ref{e}, the same result holds for $\widehat{T}_n^{\mathrm{KS}}$.

Assume also that $(\omega,t)\mapsto \dot F^{\btheta_0}_\omega(t)$ is uniformly continuous with respect to $d_{\Omega\times\mathbb R_{\ge0}}$.
Then, under the local alternative, it holds that
$$
\tilde{T}_n^{\mathrm{CM}}
\convl
\int_{\Omega\times\mathbb R_{\ge0}}
\Big(
\tilde{\G}_{0}(\omega,t)
+
h(\omega,t)
+
\dot F^{\btheta_0}_\omega(t)^\top \bV_{\btheta_0}^{-1}\mathsf{E}_G(\bpsi_{\btheta_0}(X))
\Big)^2\,
\mathrm{d} F^{\mu_{\btheta_0}}_\omega(t)\, \mathrm{d}\mu_{\btheta_0}(\omega).
$$
Under condition \ref{e}, the same result holds for $\widehat{T}_n^{\mathrm{CM}}$.
\end{theorem}

\section{Multiplier bootstrap}\label{sec:bootstrap}

We propose a multiplier bootstrap approach and prove its consistency under some general conditions, in order to effectively emit test decisions in practice.

Some basic concepts and notation are required. Let $\mathcal H=\{h_{\omega,t}:\omega\in\Omega,\ t\ge0\}$ denote the indexing class, with $\mathcal H=\cY$ (simple null) or $\mathcal H=\cF$ (composite null), with $\cY$ and $\cF$ as defined in \eqref{eq:def_Y} and \eqref{eq:def_F}, respectively. Let $\bzeta = (\zeta_1, \ldots, \zeta_n)$ be nonnegative iid random variables, independent of $X_1,\ldots, X_n$, with mean $0<m<\infty$ and variance $0<\tau^2<\infty$. Assume $\|\zeta\|_{2,1} < \infty$ where $\|\zeta\|_{2,1} = \int_0^\infty \sqrt{\mathsf{P}(|\zeta| > x)} \, \rd x$. Let $\mathsf{E}_{*}$ and $\mathsf{P}_{*}$ denote expectation and probability with respect to the multipliers $\bzeta$, conditionally on $X_1,\ldots,X_n$. %

For a sequence of random elements $(\mathbb{Z}_n)_{n\ge1}$ and a random element $\mathbb{Z}$ in $\ell^\infty(\mathcal H)$, $\mathbb{Z}_n \convl \mathbb{Z}$ if and only if
\begin{align}\label{eq:BL_convergence}
\sup_{\varphi\in \mathrm{BL}_1(\ell^\infty(\mathcal H))}
\left| \mathsf{E}^*\varphi(\mathbb{Z}_n) - \mathsf{E}\varphi(\mathbb{Z}) \right| \to 0,
\end{align}
where $\mathsf{E}^*$ denotes outer expectation \citep[Equation 2.8]{Kosorok2008}. Here, $\mathrm{BL}_1(\ell^\infty(\mathcal H))$ is the class of bounded Lipschitz functions $\varphi: \ell^\infty(\mathcal H) \to \R$ satisfying $\|\varphi\|_{\infty} \le 1$ and $|\varphi(z)-\varphi(z^\prime)| \le \|z-z^\prime\|_{\mathcal H}$ for all $z,z^\prime\in\ell^\infty(\mathcal H)$. This result is used to define the convergence of the bootstrap distributions.

For a tight random element $\mathbb{Z}$ in $\ell^\infty(\mathcal H)$, the convergence $\mathbb{Z}_n^{*}\convpboot{}\mathbb{Z}$ \citep[p. 20]{Kosorok2008} means that, conditionally on $X_1,\ldots,X_n$,
\begin{align}\label{eq:BL_convergence_bootstrap}
\sup_{\varphi\in \mathrm{BL}_1(\ell^\infty(\mathcal H))}
\left| \mathsf{E}_{*} \varphi(\mathbb{Z}_n^{*}) - \E{\varphi(\mathbb{Z})} \right| \convp 0,
\end{align}
and for all $\varphi\in \mathrm{BL}_1(\ell^\infty(\mathcal H))$, $\mathsf{E}_{*} \overline{\varphi(\mathbb{Z}_n^{*})} - \mathsf{E}_{*} \underline{\varphi(\mathbb{Z}_n^{*})} \convp 0$. The quantities $\overline{\varphi(\mathbb{Z}_n^{*})}$ and $\underline{\varphi(\mathbb{Z}_n^{*})}$ are measurable majorants and minorants, respectively, with respect to the
joint randomness in $(X_1,\ldots,X_n,\zeta_1,\ldots,\zeta_n)$. Contrary to \eqref{eq:BL_convergence}, the convergence in \eqref{eq:BL_convergence_bootstrap} is in probability because the conditional expectations depend on the sample.

Let $P_n^{*}$ denote the (multiplier) bootstrapped empirical measure, given by
$$
P_n^{*} f = \frac{1}{n} \sum_{i=1}^n \frac{\zeta_i}{\bar{\zeta}_n} f(X_i),
$$
where $\bar{\zeta}_n = n^{-1} \sum_{i=1}^n \zeta_i$ and $P_n^{*} \equiv 0$ if $\bar{\zeta}_n=0$. Define
$F_{n,\omega}^{*}(t)\defin P_n^{*}y_{\omega,t}
=n^{-1}\sum_{i=1}^n(\zeta_i/\bar\zeta_n)\allowbreak y_{\omega,t}(X_i)$,
and set $F_{n,\omega}^{*}(t)=0$ if $\bar\zeta_n=0$.

We consider the following definition of the multiplier bootstrap processes $\G_n^{*}$ and $\tilde{\G}_n^{*}$, which are the bootstrap analogues of $\G_{n,\mu}^{\mu_0}$ and $\G_{n,\mu}^{\mu_{\hat{\btheta}}}$, respectively.
Under a simple null hypothesis, set
$$
\G_n^{*}(\omega,t)
\defin
\frac{m}{\tau}\sqrt{n}
\left[
\left\{F_{n,\omega}^{*}(t)-F_{\omega}^{\mu_0}(t)\right\}
-
\left\{F_{n,\omega}^{\mu}(t)-F_{\omega}^{\mu_0}(t)\right\}
\right]
=
\frac{m}{\tau}\sqrt{n}\left(F_{n,\omega}^{*}(t)-F_{n,\omega}^{\mu}(t)\right).
$$
This definition connects with the bootstrapped empirical process, since $\G_n^{*}(\omega,t)=(m/\tau)\sqrt{n}\allowbreak(P_n^{*}-P_n)y_{\omega,t}.$

For a composite null hypothesis, let $\hat{\btheta}^{*}$ be the (multiplier) bootstrap version of $\hat{\btheta}$, obtained by replacing $P_n$ with $P_n^{*}$ in the estimating procedure. We assume that $\hat{\btheta}^{*}$ is conditionally measurable given the sample.
Define
\begin{align}\label{eq:slow_bootstrap_process}
\tilde{\G}_n^{*}(\omega,t)
\defin
\frac{m}{\tau}\sqrt{n}
\left[
\left\{F_{n,\omega}^{*}(t)-F_{\omega}^{\mu_{\hat{\btheta}^{*}}}(t)\right\}
-
\left\{F_{n,\omega}^{\mu}(t)-F_{\omega}^{\mu_{\hat{\btheta}}}(t)\right\}
\right].
\end{align}
This is the bootstrap version of $\G_{n,\mu}^{\mu_{\hat{\btheta}}}$. Each bootstrap replication requires recomputing $\hat{\btheta}^{*}$ and evaluating $F_{\omega}^{\mu_{\hat{\btheta}^{*}}}(t)$. To avoid this, define
\begin{align}\label{eq:fast_bootstrap_process}
\widehat{\G}_{n,\mu}^{*}(\omega,t)
\defin
\frac{m}{\tau}\sqrt n(P_n^{*}-P_n)\hat{f}_{\omega,t}
=
\frac{m}{\tau}\sqrt n(P_n^{*}-P_n)
\left(
y_{\omega,t}
+
\dot F_\omega^{\hat{\btheta}}(t)^\top
\widehat{\bV}^{-1}
\bpsi_{\hat{\btheta}}
\right),
\end{align}
which is the bootstrap version of $\widehat{\G}_{n,\mu}$. Calculating $\widehat{\G}_{n,\mu}^{*}$ requires computing $
y_{\omega,t}(X_i)
+
\dot F_\omega^{\hat{\btheta}}(t)^\top
\widehat{\bV}^{-1}
\bpsi_{\hat{\btheta}}(X_i)$
only once, for each $i=1,\ldots,n$, and then multiplying by the bootstrap weights $\zeta_i/\bar\zeta_n$.
For the composite case, the connection between $\tilde{\G}_n^{*}$ and the bootstrapped empirical process is given next. %

\begin{theorem}[Asymptotic distributions of the multiplier bootstrap processes]\label{thm:bootstrap_processes}
The following statements hold:
\begin{enumerate}[label=(\textit{\roman*}), ref=(\textit{\roman*})]
\item Under the simple null hypothesis, \label{thm:bootstrap_processes_i}
$\G_n^{*}\convpboot{}\G_0$ in $\ell^\infty(\cY).$

\item Assume the conditions of Theorem~\ref{thm:convergence_G_n_composite} and 
the bootstrap analogue of the Bahadur representation \eqref{eq:bahadur}:
\begin{align}\label{eq:weighted_bahadur_bootstrap}
\frac{m}{\tau}\sqrt{n}\big(\hat{\btheta}^{*}-\hat{\btheta}\big)
=
-\bV_{\btheta_0}^{-1}
\frac{m}{\tau}\frac{1}{\sqrt n}\sum_{i=1}^n
\left(\frac{\zeta_i}{\bar\zeta_n}-1\right)\bpsi_{\btheta_0}(X_i)
+o_{P_{*}}(1).
\end{align}
Then, under the composite null hypothesis, \label{thm:bootstrap_processes_ii}
$\tilde{\G}_n^{*}\convpboot{}\tilde{\G}_0$ in $\ell^\infty(\cF).$

If, in addition, the derivative $\dot F_\omega^{\btheta}(t)$ exists for $\btheta$ in a neighborhood of $\btheta_0$, and
$
\sup_{\omega, t}
\|
\dot F_\omega^{\hat{\btheta}}(t)-\dot F_\omega^{\btheta_0}(t)
\|
\convp
0,
$
then 
$\widehat{\G}_{n,\mu}^{*}\convpboot{}\tilde{\G}_0$ in $\ell^\infty(\cF).$
\end{enumerate}
\end{theorem}

\begin{remark}
The sufficient condition for condition~\ref{d} given in Remark~\ref{rem:sufficient_condition_d}, together with condition~\ref{cp4}, implies
$
\sup_{\omega\in\Omega,\,t\ge0}
\|
\dot F_\omega^{\hat{\btheta}}(t)
-
\dot F_\omega^{\btheta_0}(t)
\|
\convp 0.
$
\end{remark}

The bootstrap versions of the test statistics are then
\begin{alignat*}{2}
T_n^{*\mathrm{KS}}
&\defin \sup_{\omega \in \Omega, t \ge 0}\big|\G_n^{*}(\omega, t)\big|,
&\qquad
T_n^{*\mathrm{CM}}
&\defin \int_{\Omega\times\R_{\ge0}} \G_n^{*}(\omega, t)^2
\,\rd F_{n,\omega}^{\mu}(t)\,\rd P_n(\omega),\\
\tilde{T}_n^{*\mathrm{KS}}
&\defin \sup_{\omega \in \Omega, t \ge 0}\big|\tilde{\G}_n^{*}(\omega, t)\big|,
&\qquad
\tilde{T}_n^{*\mathrm{CM}}
&\defin \int_{\Omega\times\R_{\ge0}} \tilde{\G}_n^{*}(\omega, t)^2
\,\rd F_{n,\omega}^{\mu}(t)\,\rd P_n(\omega),\\
\widehat{T}_n^{*\mathrm{KS}}
&\defin \sup_{\omega \in \Omega, t \ge 0}\big|\widehat{\G}_{n,\mu}^{*}(\omega, t)\big|,
&\qquad
\widehat{T}_n^{*\mathrm{CM}}
&\defin \int_{\Omega\times\R_{\ge0}} \widehat{\G}_{n,\mu}^{*}(\omega, t)^2
\,\rd F_{n,\omega}^{\mu}(t)\,\rd P_n(\omega).
\end{alignat*}
Theorem \ref{thm:bootstrap_test_statistics} gives the asymptotic distribution of these statistics under the null hypothesis.

\begin{theorem}[Asymptotic distributions of the bootstrap test statistics]\label{thm:bootstrap_test_statistics}
The following statements hold:
\begin{enumerate}[label=(\textit{\roman*}), ref=(\textit{\roman*})]
\item Under the simple null hypothesis, \label{thm:bootstrap_test_statistics_i}
$$
T_n^{*\mathrm{KS}}\convpboot{}\sup_{\omega\in\Omega, t\ge0}|\G_0(\omega,t)|,
\quad
T_n^{*\mathrm{CM}}\convpboot{}
\int_{\Omega\times\R_{\ge0}}\G_0(\omega,t)^2\,\rd F_{\omega}^{\mu_0}(t)\,\rd\mu_0(\omega).
$$

\item Under the composite null hypothesis, assuming the conditions of Theorem~\ref{thm:bootstrap_processes}, \label{thm:bootstrap_test_statistics_ii}
$$
\tilde{T}_n^{*\mathrm{KS}}\convpboot{}\sup_{\omega\in\Omega, t\ge0}|\tilde{\G}_0(\omega,t)|, \quad \widehat{T}_n^{*\mathrm{KS}}\convpboot{}\sup_{\omega\in\Omega, t\ge0}|\tilde{\G}_0(\omega,t)|.
$$
If, in addition, the map $(\omega,t)\mapsto\dot F_{\omega}^{\btheta_0}(t)$ is uniformly continuous with respect to $d_{\Omega\times\R_{\ge0}}$ and, for every $(\omega,t)\in\Omega\times\R_{\ge0}$, the map $\btheta\mapsto F_\omega^{\mu_{\btheta}}(t)$ is Borel measurable, then
$$
\tilde{T}_n^{*\mathrm{CM}}\convpboot{}
\int_{\Omega\times\R_{\ge0}}\tilde{\G}_0(\omega,t)^2\,\rd F_{\omega}^{\mu_{\btheta_0}}(t)\,\rd\mu_{\btheta_0}(\omega), 
\quad
\widehat{T}_n^{*\mathrm{CM}}\convpboot{}
\int_{\Omega\times\R_{\ge0}}\tilde{\G}_0(\omega,t)^2\,\rd F_{\omega}^{\mu_{\btheta_0}}(t)\,\rd\mu_{\btheta_0}(\omega).
$$
\end{enumerate}
\end{theorem}

\section{Numerical experiments}\label{sec:numerical_experiments}

We assessed the calibration under the null and the power of the KS and CM tests in finite samples through simulations. In each experiment, $M=1,000$ independent samples of size $n$ were generated. Every $p$-value is computed using the multiplier bootstrap procedure with $B=1,000$ replicates and independent standard exponential weights. All simulation scenarios considered in this section involve composite null hypotheses. The numerical results are obtained using the test statistics constructed with $\widehat{\G}_{n,\mu}^{*}$, unless stated otherwise. The empirical rejection rates are computed as the proportion of test rejections at level $\alpha=5\%$. For the bootstrap KS statistics, we used the same finite set of points as in Proposition~\ref{prop:sup_max_ks}, based on sample centers and pairwise sample distances.

Denote by $\mu_{\beta}$ the distribution of the data generating process, with $\beta \in \{0,0.5,1\}$ controlling the departure from the null hypothesis, corresponding to $\beta=0$. For a dimension $d$ and a sample size $n$, the simulation scenarios considered are collected in Table~\ref{tab:simulation_scenarios}. We describe the additional notation employed as follows:

\begin{itemize}
\item For the Euclidean space, the null model is
$\mathcal P_d^{\mathrm{sp}} \defin 
\{
\mathcal N_d(
\btheta,
\bI_d+\lambda
\btheta\btheta^\top / \|\btheta\|^2
):
\btheta\neq\boldsymbol 0,\ \lambda>0
\}.
$
We set
$\btheta_d=\one_d/\sqrt d$,
$\bq_d=(\be_1-\be_2)/\sqrt 2$, and
$\bSigma_d=\bI_d+2\btheta_d\btheta_d^\top$. $T_d$ denotes the $d$-variate standardized Student distribution with $d+1$ degrees of freedom. 

\item For the simplex, write $D_d\sim\mathrm{Dirichlet}(1.5,\ldots,1.5)$ and $[\bR_d(\rho)]_{jk}=\rho^{|j-k|}$. $S_d$ is the distribution induced on $\Delta^d$ by applying the inverse ilr transform to a $d$-variate standardized $t_4$ distribution with covariance $\bR_d(0.5)$.

\item For the sphere, $Q_d$ denotes the distribution of $\bZ/\|\bZ\|$, where $\bZ\sim\mathcal N_{d+1}(2\sqrt d\,\be_1,\bI_{d+1})$. 

\item For the hyperboloid, define $\bmu_0=(\sqrt{2},\be_1^\top)^\top$ and $\bmu_1=(\cosh(\operatorname{asinh}(1)+\sqrt{2/d}), \sinh(\operatorname{asinh}(1)+\sqrt{2/d})\be_1^\top)^\top$. $R$ is defined by modifying the conditional distribution of the angular component of a HvMF distribution. If $X\sim\mathrm{HvMF}(\bmu_0,2\sqrt{d})$, writing $X=(\cosh U,\sinh U\,\bV^\top)^\top$, then $\bV\mid U=u$ follows a $\mathrm{vMF}(\be_1,2\sqrt d\sinh(u))$ distribution on $\Sp^{d-1}$. For $R$, we set
\[
\begin{aligned}
\bV\mid U=u
&\sim
\frac12\mathrm{vMF}\big(
\cos(\pi/5)\be_1-\sin(\pi/5)\be_2,
2\sqrt d\sinh(u)
\big)
\\
&\quad+
\frac12\mathrm{vMF}\big(
\cos(\pi/5)\be_1+\sin(\pi/5)\be_2,
2\sqrt d\sinh(u)
\big).
\end{aligned}
\]
\end{itemize}
\begin{table}[!htbp]
\centering
\setlength{\tabcolsep}{3.5pt}
\caption{Simulation scenarios considered in the numerical experiments.} \label{tab:simulation_scenarios}
\begin{tabular}{lllll} 
$\Omega$ & $H_0$ & Scen. & $\mu_\beta$ & Description \\
\hline
\noalign{\vskip 2pt}

$\mathbb{R}^d$
& $\mathcal P_d^{\mathrm{sp}}$
& 1
& $(1-\beta)\mathcal N_d(\btheta_d,\bSigma_d)
+\beta\mathcal N_d(\btheta_d,\bSigma_d+2\bq_d\bq_d^\top)$
& Orthogonal covariance spike
\\
 & $\mathcal N_d(\bmu,\sigma^2\bI_d)$
& 2
& $(1-\beta)\mathcal N_d(\boldsymbol 0,\bI_d)+\beta T_d$
& Student alternative
\\[0.1cm]

$\Delta^d$
& $\mathrm{LG}(\bmu,\bI_d)$
& 3
& $(1-\beta)\mathrm{LG}(\boldsymbol 0,\bI_d)+\beta D_d$
& Dirichlet alternative
\\
 & $\mathrm{LG}(\bmu,\bR_d(\rho))$
& 4
& $(1-\beta)\mathrm{LG}(\boldsymbol 0,\bR_d(0.5))+\beta S_d$
& Student alternative
\\[0.1cm]

$\Sp^d$
& $\mathrm{vMF}(\bmu,2)$
& 5
& $(1-\beta/2)\mathrm{vMF}(\be_1,2)
+(\beta/2)\mathrm{vMF}(\be_2,2)$
& Orthogonal vMF mixture
\\
 & $\mathrm{vMF}(\bmu,\kappa)$
& 6
& $(1-\beta/2)\mathrm{vMF}(\be_1,2d)+(\beta/2)Q_d$
& Projected-normal alternative
\\[0.1cm]

$\mathbb H^d$
& $\mathrm{HvMF}(\bmu,\kappa)$
& 7
& $(1-\beta/2)\mathrm{HvMF}(\bmu_0,d)
+(\beta/2)\mathrm{HvMF}(\bmu_1,d)$
& Radial-location mixture
\\
 & $\mathrm{HvMF}(\bmu,\kappa)$
& 8
& $(1-\beta)\mathrm{HvMF}(\bmu_0,2\sqrt d)+\beta R$
& Conditional angular mixture
\\
\hline
\end{tabular}
\end{table}

Table~\ref{tab:empirical_null_calibration} reports the rejection percentages for the simulation scenarios. The power increases with the sample size and the departure from the null. Under the null, most empirical rejection rates lie within the $95\%$ prediction interval for the nominal $5\%$ level. Departures from this interval occur both above and below its bounds, with no uniform direction and with a slight deterioration for $d=5$. Under the alternatives, the CM test tends to have higher empirical power than the KS test, although neither test dominates uniformly across scenarios.

\begin{table}[!htbp]
\centering
\caption{Rejection percentages at level $5\%$. Underlined values for $\beta=0$ indicate empirical sizes outside the equal-tail $95\%$ prediction interval for $\mathrm{Bin}(M,0.05)\times100/M$.}
\label{tab:empirical_null_calibration}
\scriptsize
\renewcommand{\arraystretch}{0.88}
\resizebox{\textwidth}{!}{%
\begin{tabular}{lcc*{10}{r}}
\toprule
& & & \multicolumn{2}{c}{$n=50$} & \multicolumn{2}{c}{$n=100$} & \multicolumn{2}{c}{$n=200$} & \multicolumn{2}{c}{$n=400$} & \multicolumn{2}{c}{$n=800$} \\
\cmidrule(lr){4-5}\cmidrule(lr){6-7}\cmidrule(lr){8-9}\cmidrule(lr){10-11}\cmidrule(lr){12-13}
Scen. & $d$ & $\beta$ & {\color{black}KS} & {\color{black}CM} & {\color{black}KS} & {\color{black}CM} & {\color{black}KS} & {\color{black}CM} & {\color{black}KS} & {\color{black}CM} & KS & CM \\
\midrule
1 & 2 & 0.0 & 5.8 & 5.8 & 5.8 & 5.6 & \underline{6.9} & 5.4 & 5.1 & 4.6 & 5.3 & 5.2 \\
& & 0.5 & 19.2 & 21.8 & 45.1 & 58.8 & 81.6 & 93.7 & 99.3 & 99.9 & 100.0 & 100.0 \\
& & 1.0 & 66.8 & 79.1 & 97.1 & 99.6 & 100.0 & 100.0 & 100.0 & 100.0 & 100.0 & 100.0 \\
\cmidrule(lr){3-13}
& 5 & 0.0 & 5.8 & 6.3 & 5.5 & 4.7 & 5.7 & 5.4 & 5.5 & 5.2 & 6.1 & 5.3 \\
& & 0.5 & 13.0 & 12.3 & 27.0 & 36.7 & 60.0 & 81.5 & 92.5 & 99.6 & 100.0 & 100.0 \\
& & 1.0 & 48.4 & 60.3 & 85.9 & 96.2 & 99.3 & 100.0 & 100.0 & 100.0 & 100.0 & 100.0 \\
\cmidrule(lr){3-13}
2 & 2 & 0.0 & 5.1 & 4.8 & 5.7 & 5.0 & 5.5 & 4.1 & 5.5 & 4.2 & 4.2 & 4.7 \\
& & 0.5 & 28.4 & 32.0 & 49.6 & 53.1 & 81.1 & 82.0 & 94.2 & 94.1 & 96.6 & 96.2 \\
& & 1.0 & 69.0 & 74.0 & 90.3 & 91.6 & 97.8 & 98.4 & 99.5 & 99.5 & 99.8 & 99.9 \\
\cmidrule(lr){3-13}
& 5 & 0.0 & 4.7 & \underline{1.6} & 4.3 & \underline{1.7} & 3.7 & \underline{1.3} & 4.5 & \underline{3.0} & \underline{3.5} & 4.0 \\
& & 0.5 & 17.1 & 20.6 & 36.6 & 48.1 & 70.0 & 82.9 & 96.8 & 98.5 & 99.9 & 99.7 \\
& & 1.0 & 51.9 & 65.2 & 87.1 & 93.7 & 99.5 & 99.4 & 99.9 & 99.9 & 99.9 & 99.9 \\
\midrule
3 & 2 & 0.0 & 6.1 & \underline{7.1} & 4.0 & 4.5 & 5.5 & 6.0 & \underline{6.8} & \underline{6.6} & 5.1 & 5.4 \\
& & 0.5 & 10.5 & 15.3 & 12.1 & 16.3 & 17.6 & 20.3 & 31.1 & 31.2 & 52.6 & 54.1 \\
& & 1.0 & 20.7 & 28.5 & 34.4 & 38.5 & 53.9 & 56.0 & 80.1 & 81.5 & 97.6 & 98.0 \\
\cmidrule(lr){3-13}
& 5 & 0.0 & \underline{7.1} & \underline{7.8} & 5.5 & 6.0 & 5.5 & 4.8 & 5.7 & 5.3 & \underline{6.8} & 5.8 \\
& & 0.5 & 16.9 & 20.7 & 21.9 & 24.8 & 31.9 & 32.3 & 52.4 & 51.0 & 80.3 & 77.2 \\
& & 1.0 & 34.0 & 37.8 & 50.9 & 53.2 & 80.9 & 78.8 & 97.1 & 96.2 & 100.0 & 100.0 \\
\cmidrule(lr){3-13}
4 & 2 & 0.0 & 6.4 & \underline{7.3} & 5.1 & \underline{6.7} & 5.5 & 6.1 & 4.3 & 4.2 & 6.0 & 4.9 \\
& & 0.5 & 24.9 & 29.2 & 33.4 & 36.7 & 62.2 & 63.8 & 87.1 & 86.1 & 99.1 & 99.0 \\
& & 1.0 & 61.5 & 67.9 & 84.0 & 86.0 & 99.0 & 99.2 & 99.9 & 100.0 & 99.9 & 100.0 \\
\cmidrule(lr){3-13}
& 5 & 0.0 & 5.6 & \underline{7.6} & \underline{3.2} & 5.1 & 5.1 & 5.4 & \underline{6.8} & 5.6 & 4.2 & \underline{3.6} \\
& & 0.5 & 42.5 & 46.4 & 66.6 & 65.7 & 89.7 & 88.7 & 99.7 & 99.6 & 100.0 & 100.0 \\
& & 1.0 & 91.0 & 91.1 & 99.4 & 99.3 & 100.0 & 100.0 & 99.9 & 100.0 & 100.0 & 100.0 \\
\midrule
5 & 2 & 0.0 & \underline{6.5} & \underline{8.2} & 5.0 & 5.5 & 4.5 & 5.7 & 4.9 & 5.6 & 6.0 & 5.2 \\
& & 0.5 & 25.5 & 25.4 & 52.5 & 58.6 & 87.4 & 92.9 & 99.7 & 100.0 & 100.0 & 100.0 \\
& & 1.0 & 51.4 & 56.2 & 89.5 & 92.9 & 100.0 & 100.0 & 100.0 & 100.0 & 100.0 & 100.0 \\
\cmidrule(lr){3-13}
& 5 & 0.0 & \underline{6.8} & \underline{7.7} & \underline{6.5} & \underline{6.5} & 5.7 & 6.4 & 4.9 & 5.4 & \underline{6.5} & 6.4 \\
& & 0.5 & 6.7 & 7.0 & 19.6 & 21.1 & 46.4 & 54.9 & 84.5 & 88.5 & 99.4 & 100.0 \\
& & 1.0 & 11.3 & 11.3 & 35.8 & 42.6 & 81.3 & 87.9 & 99.7 & 99.9 & 100.0 & 100.0 \\
\cmidrule(lr){3-13}
6 & 2 & 0.0 & 6.1 & 5.2 & 5.6 & 4.6 & 5.2 & 5.1 & 3.9 & 4.6 & 5.6 & 4.9 \\
& & 0.5 & 8.6 & 8.2 & 9.1 & 9.0 & 13.7 & 15.2 & 22.0 & 24.5 & 40.8 & 45.9 \\
& & 1.0 & 11.3 & 12.2 & 14.7 & 17.1 & 28.7 & 32.3 & 52.1 & 59.4 & 79.2 & 86.6 \\
\cmidrule(lr){3-13}
& 5 & 0.0 & 5.3 & \underline{1.4} & 5.8 & \underline{1.5} & 4.7 & \underline{1.8} & 3.7 & \underline{1.8} & 5.1 & 4.6 \\
& & 0.5 & 8.2 & 6.4 & 15.2 & 15.6 & 30.1 & 36.3 & 48.3 & 67.3 & 84.8 & 95.8 \\
& & 1.0 & 16.8 & 15.0 & 28.6 & 36.3 & 57.5 & 74.9 & 90.5 & 96.9 & 99.9 & 100.0 \\
\midrule
7 & 2 & 0.0 & 6.2 & 5.5 & 5.9 & 4.4 & 3.9 & 4.2 & 5.8 & 4.7 & 4.1 & 4.1 \\
& & 0.5 & 14.2 & 12.6 & 21.4 & 23.1 & 43.5 & 45.2 & 73.7 & 80.8 & 98.7 & 99.6 \\
& & 1.0 & 16.2 & 12.8 & 24.0 & 22.4 & 44.9 & 52.7 & 84.3 & 91.8 & 99.7 & 100.0 \\
\cmidrule(lr){3-13}
& 5 & 0.0 & 5.5 & \underline{1.8} & 6.1 & \underline{2.6} & 5.2 & \underline{2.9} & 3.9 & \underline{2.5} & 4.4 & 4.6 \\
& & 0.5 & 11.6 & 8.6 & 17.6 & 11.6 & 39.4 & 33.5 & 73.3 & 71.6 & 98.2 & 98.3 \\
& & 1.0 & 14.4 & 9.4 & 27.6 & 19.9 & 59.1 & 51.7 & 95.2 & 94.2 & 100.0 & 100.0 \\
\cmidrule(lr){3-13}
8 & 2 & 0.0 & 5.3 & 5.0 & 5.8 & 5.0 & 4.4 & 4.0 & 5.0 & 4.3 & 4.7 & 4.6 \\
& & 0.5 & 9.4 & 7.8 & 12.6 & 11.2 & 21.9 & 22.9 & 44.1 & 50.3 & 83.6 & 91.1 \\
& & 1.0 & 25.3 & 28.0 & 51.9 & 54.7 & 89.0 & 93.9 & 100.0 & 100.0 & 100.0 & 100.0 \\
\cmidrule(lr){3-13}
& 5 & 0.0 & 5.8 & \underline{1.9} & \underline{6.6} & \underline{2.2} & 5.4 & \underline{3.1} & \underline{3.5} & \underline{3.1} & 4.5 & 4.0 \\
& & 0.5 & 23.6 & 18.2 & 48.2 & 44.2 & 87.8 & 85.3 & 99.9 & 99.8 & 100.0 & 100.0 \\
& & 1.0 & 87.6 & 82.9 & 99.8 & 99.7 & 100.0 & 100.0 & 100.0 & 100.0 & 100.0 & 100.0 \\
\bottomrule
\end{tabular}}
\end{table}

Table \ref{tab:bootstrap_procedure_comparison} compares the two multiplier-bootstrap procedures defined in \eqref{eq:slow_bootstrap_process} and \eqref{eq:fast_bootstrap_process}, for every scenario in Table~\ref{tab:simulation_scenarios}, with $n=100$ and $M=1,000$. The methodology given in \eqref{eq:fast_bootstrap_process}, which does not require re-estimation, achieved an average speed-up of $438\times$. The largest absolute difference in rejection rates was $1.1\%$. However, the differences were mostly positive, suggesting a slightly more conservative finite-sample behavior when \eqref{eq:fast_bootstrap_process} is employed. Accordingly, we use the multiplier-bootstrap procedure in \eqref{eq:fast_bootstrap_process} when $\widehat{\bV}$ is invertible and its estimated condition number is smaller than $10^{12}$; otherwise, we use the reestimated multiplier bootstrap in \eqref{eq:slow_bootstrap_process}, which avoids the explicit inversion of $\widehat{\bV}$. Timings were obtained on the UC3M C3 cluster using AMD EPYC 7713 processors with 16 CPUs per job. Both procedures were applied to the same simulated datasets.

\begin{table}[!htbp]
\centering
\caption{Comparison of the multiplier-bootstrap procedures given in \eqref{eq:slow_bootstrap_process} and \eqref{eq:fast_bootstrap_process}, for $n=100$. Here, $\Delta$ is the empirical rejection rate obtained with \eqref{eq:slow_bootstrap_process} minus that obtained with \eqref{eq:fast_bootstrap_process}, in percentage points, and speed-up is the ratio of the median elapsed time of \eqref{eq:slow_bootstrap_process} to that of \eqref{eq:fast_bootstrap_process}. For each method, elapsed time is reported as the median elapsed time per replication, in seconds, over \(M=1,000\) samples.}
\label{tab:bootstrap_procedure_comparison}
\begin{tabular}{lc*{8}{r}}
\toprule
& $d$ & \multicolumn{8}{c}{Scenario} \\
& & $1$ & $2$ & $3$ & $4$ & $5$ & $6$ & $7$ & $8$ \\
\midrule
$\Delta_{\mathrm{KS}}$
& $2$ & $0.0$  & $0.5$ & $0.3$ & $0.4$ & $0.0$  & $0.5$ & $-0.1$  & $0.6$ \\
& $5$ & $1.1$ & $0.3$ & $0.6$ & $0.4$ & $-1.0$  & $0.3$ & $0.7$ & $0.9$ \\
\addlinespace
$\Delta_{\mathrm{CM}}$
& $2$ & $0.4$ & $0.8$ & $0.2$ & $0.4$ & $0.3$ & $0.7$ & $0.4$ & $0.6$ \\
& $5$ & $0.7$ & $0.0$  & $0.4$ & $0.3$ & $-0.6$  & $0.4$ & $0.4$ & $0.3$ \\
\addlinespace
Elapsed time \eqref{eq:fast_bootstrap_process}
& $2$ & 1.4 & 0.8 & 0.7 & 1.4 & 1.1 & 1.2 & 1.7 & 1.7 \\
& $5$ & 1.6 & 1.0 & 1.0 & 1.6 & 3.1 & 1.6 & 2.5 & 2.5 \\
\addlinespace

Elapsed time \eqref{eq:slow_bootstrap_process}
& $2$ & 846.4 & 198.1 & 198.5 & 793.4 & 544.4 & 553.9 & 575.0 & 583.7 \\
& $5$ & 858.0 & 220.9 & 219.7 & 810.5 & 2490.1 & 733.7 & 1149.4 & 1154.3 \\
\addlinespace
Speed-up
& $2$ & 613 & 261 & 278 & 576 & 478 & 466 & 345 & 353 \\
& $5$ & 532 & 216 & 226 & 500 & 801 & 466 & 466 & 468 \\
\bottomrule
\end{tabular}
\end{table}

\section{Real data applications}\label{sec:real_data}

\subsection{Sunspot occurrences on \texorpdfstring{$\mathbb{S}^2 \times \mathbb{R}$}{S² x R}}\label{sec:sunspots}

Sunspots are dark, cooler regions of the solar photosphere. %
During the approximately 11-year solar cycle, sunspot groups tend to appear in two activity belts, one in each hemisphere, whose central latitudes progressively move towards the equator \citep[Sp\"orer's law;][]{Babcock1961}.

A time-varying model is proposed for the positions of sunspot groups on the Sun, motivated by Spörer’s law. Specifically, we analyzed the birth positions of sunspot groups from solar cycle 23, using the \texttt{sunspots\_\allowbreak births} dataset in the \texttt{rotasym} package \citep{Garcia-Portugues2020e}. %
The sample contains $5373$ sunspot groups with birth dates from August 6, 1996 to November 27, 2008.
Since sunspot births were recorded only once per day in the dataset, many groups artificially share the same birth time. We added to each recorded time an independent uniform perturbation between $-1/2$ and $1/2$ days and rescaled the resulting times to $(0,1)$. 

Let $U$ denote the birth time in $(0,1)$. We modelled $U$ through a two-component beta mixture with density $g_{\bbeta}$, where $\bbeta=(w,\alpha_1,\beta_1,\alpha_2,\beta_2)^\top$, $w\in(0,1)$ is the mixture weight, and $\alpha_j,\beta_j>0$ are the shape parameters, for $j=1,2$. Conditional on $U=u$, the position $\bX\in\mathbb S^2$ is modelled through a mixture of two small-circle distributions \citep{Bingham1978}. We write $\mathrm{SC}(\bmu,\kappa,\nu)$ for the small-circle distribution with axis $\bmu\in\Sp^2$, concentration parameter $\kappa\ge0$, and $\nu\in(-1,1)$ controlling the location of the circle on the sphere. For $\nu(u)=a+bu$ and $\mathbf e_3 =(0,0,1)^\top$,
\begin{align*}
 \bX \mid U=u
 &\sim
 \frac{1}{2}\,\mathrm{SC}\bigl(\mathbf e_3,\kappa,\nu(u)\bigr)
 +
 \frac{1}{2}\,\mathrm{SC}\bigl(\mathbf e_3,\kappa,-\nu(u)\bigr).
\end{align*}
The parameters satisfy $b<0$, $\kappa>0$, and
$\nu(u)\in(0,1)$ for $u\in[0,1]$. Thus, the model incorporates Spörer’s law by imposing that the sunspots on each hemisphere follow the same latitudinal path towards the equator. Write $\boldsymbol{\gamma}=(a,b,\kappa)^\top$ and
$\btheta=(\bbeta^\top,\boldsymbol{\gamma}^\top)^\top$. %
Maximum Likelihood Estimation (MLE) was done in two stages, first on the time and then on the conditional model.

The space $\mathbb S^2\times(0,1)$ was endowed with the metric
$
d((\bx,u),(\by,v))
=
\arccos(\bx^\top \by)/(2\pi)
+
|u-v|/2.
$
For $s\in(0,1)$, define %
$
F_{\bomega,s}^{\boldsymbol{\gamma}}(q)
\defin
\Prob{
d_{\mathbb S^2}(\bomega,\bX)\leq q
\mid U=s
}.
$
The distance profile of $(\bX,U)$ at $(\bomega,u)\in\mathbb S^2\times(0,1)$ can be written as
\begin{align*}
F_{(\bomega,u)}(r)
=
\int_0^1
F_{\bomega,s}^{\boldsymbol{\gamma}}
\left(
\pi
\min\left\{
1,
\left(2r-|u-s|\right)_+
\right\}
\right)
g_{\bbeta}(s)\,\rd s,
\end{align*}
where $x_+\defin\max\{x,0\}$.

Table~\ref{tab:sunspots_cycle23_temporal_models} reports the estimated parameters and the $p$-values of the KS and CM tests for the proposed model on $\Sp^2 \times (0,1)$. Neither test rejected the posited model at the $5\%$ level. Figure~\ref{fig:sunspots_cycle23_spatial_windows} illustrates the evolution of the birth distribution across solar cycle 23, showing a good empirical fit between the data and the contours of the model. %
\begin{table}[htbp]
\centering
\caption{MLE estimates and GOF test results for the sunspot-birth model. In bold, $p$-values larger than $5\%$.}
\label{tab:sunspots_cycle23_temporal_models}
\begin{tabular}{ccccc ccc cc}
\toprule
\multicolumn{5}{c}{Temporal model}
&
\multicolumn{3}{c}{Conditional model}
&
\multicolumn{2}{c}{GOF} \\
\cmidrule(lr){1-5}
\cmidrule(lr){6-8}
\cmidrule(lr){9-10}
$\widehat w$
& $\widehat\alpha_1$
& $\widehat\beta_1$
& $\widehat\alpha_2$
& $\widehat\beta_2$
& $\widehat a$
& $\widehat b$
& $\widehat\kappa$
& KS
& CM \\
\midrule
0.62
& 3.24
& 5.38
& 1.38
& 1.35
& 0.40
& $-0.31$
& 29.90
& $\mathbf{0.0989}$
& $\mathbf{0.1618}$ \\
\bottomrule
\end{tabular}
\end{table}

\ifincludeimages
\begin{figure}[htbp]
\centering
\begin{subfigure}[htbp]{0.19\textwidth}
\centering
\includegraphics[width=\textwidth]{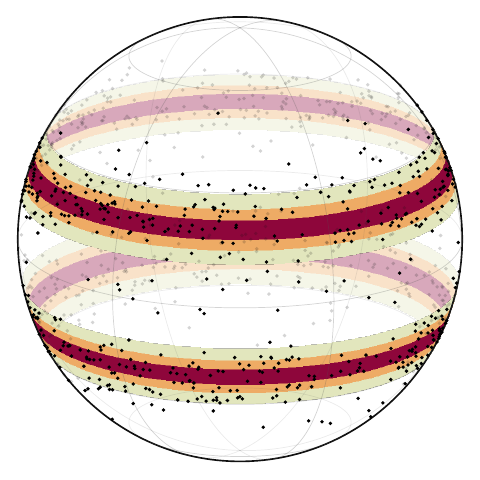}
\caption{\tiny $0$--$20\%$}
\end{subfigure}\hfill
\begin{subfigure}[htbp]{0.19\textwidth}
\centering
\includegraphics[width=\textwidth]{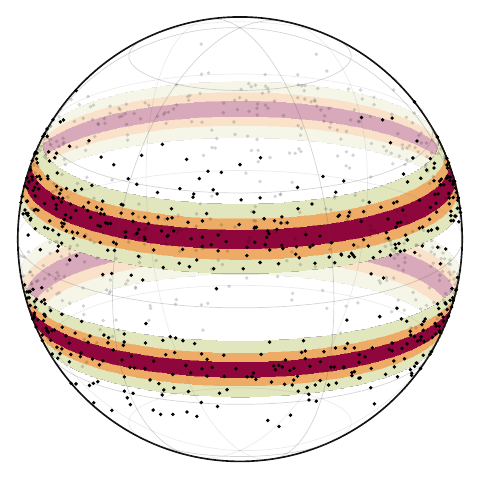}
\caption{\tiny $20$--$40\%$}
\end{subfigure}\hfill
\begin{subfigure}[htbp]{0.19\textwidth}
\centering
\includegraphics[width=\textwidth]{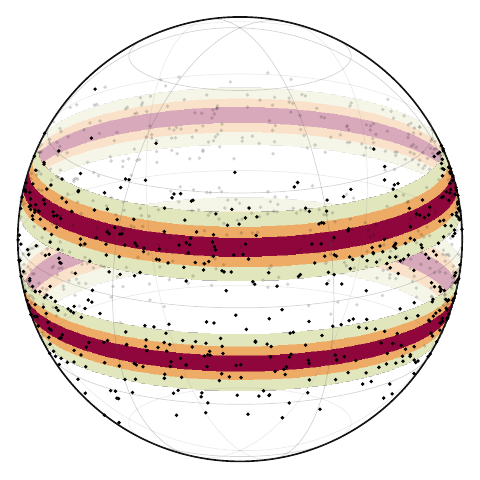}
\caption{\tiny $40$--$60\%$}
\end{subfigure}\hfill
\begin{subfigure}[htbp]{0.19\textwidth}
\centering
\includegraphics[width=\textwidth]{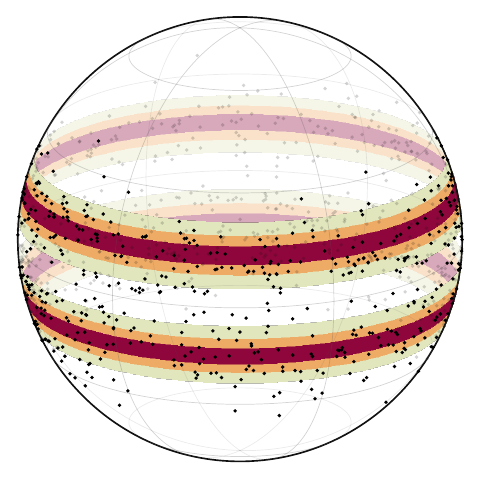}
\caption{\tiny $60$--$80\%$}
\end{subfigure}\hfill
\begin{subfigure}[htbp]{0.19\textwidth}
\centering
\includegraphics[width=\textwidth]{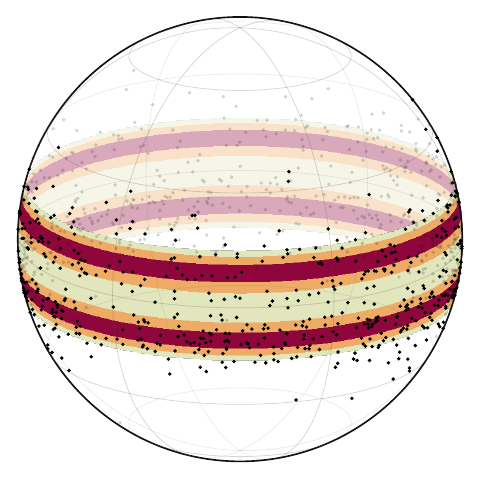}
\caption{\tiny $80$--$100\%$}
\end{subfigure}
\caption{Fitted small-circle densities across five temporal windows of solar cycle 23. The $25\%$, $50\%$, and $75\%$ highest-density regions are shown in colour, and the points represent the observed birth positions within each window. The model parameters were estimated in the full cycle, and the fitted conditional density was calculated at the midpoint of each window.}
\label{fig:sunspots_cycle23_spatial_windows}
\end{figure}
\fi

\subsection{Comets on \texorpdfstring{$\Sp^2$}{S²}}

We consider the orbital-pole data obtained from the \texttt{comets} dataset available in the R package \texttt{sphunif} \citep{Garcia-Portugues2020c}. Each comet was mapped to the unit normal vector of its orbital plane, resulting in data on $\Sp^2$. Fragmented comets were removed, as well as those with hyperbolic or parabolic orbits. The sample was split into long-period comets (hypothesized to originate in the Oort cloud), with orbital period over $200$ years ($n=610$), and short-period comets, with orbital period under $200$ ($n=784$).

Several models were fitted to the data. First, the spherical cardioid model $\mathrm{C}_k(\bmu,\rho)$, where $\bmu\in\Sp^2$ is the symmetry axis, $\rho\in[-1,1]$ controls the concentration, and $k$ is the model order \citep{Garcia-Portugues2026}. A GOF test for the $\mathrm{C}_2$ model was applied to the long-period comets in ibid., with no rejection of the null hypothesis at the $5\%$ level. Second, the $\mathrm{SC}(\bmu,\kappa,\nu)$ distribution was considered. The case $\nu=0$ corresponds to the Watson (W) distribution \citep{Dimroth1962,Watson1965}, which is locally approximated by $\mathrm{C}_2$ at low concentration \citep{Garcia-Portugues2026}. Additionally, we considered a tailored model with density
$$
f_{\bX}(\bx)
=
\frac{1}{4\pi}
\left[
w+(1-w)\frac{y(\bx)^{\alpha-1}(1-y(\bx))^{\beta-1}}{\mathrm{B}(\alpha,\beta)}
\right],
\qquad
y(\bx)=\frac{1+\bmu^\top\bx}{2},
$$
where $\mathrm{B}(\alpha,\beta)$ is the beta function, $\bmu\in\Sp^2$, $w\in(0,1)$, and $\alpha,\beta>0$. This model is a mixture of the uniform distribution on $\Sp^2$ and a rotationally symmetric component for which $y(\bX)$ has a $\mathrm{Beta}(\alpha,\beta)$ distribution. We refer to this model as the uniform--beta mixture, denoted by $\mathrm{UB}(\bmu,w,\alpha,\beta)$. For $\alpha=\beta=2$, the UB model reduces to $\mathrm{C}_2$ with $\rho=w-1$, and hence contains its $-1<\rho<0$ subfamily.

For each parametric family $\mathcal M\in \{\mathrm{C}_2,\mathrm{SC},\mathrm{UB}\}$, we tested the composite null hypothesis $H_0:\mu\in\mathcal M$. Table \ref{tab:comets-gof} reports the $p$-values of the hypothesis tests. The $\mathrm{C}_2$ model was not rejected for the long-period comets, with parameters $\hat{\bmu}=(0.0804,-0.0067,0.9967)^\top$ and $\hat\rho=0.4727$, analogously to \cite{Garcia-Portugues2026}. The SC model was also not rejected for the long-period sample. Due to its flexibility, the UB model was not rejected in either sample, unlike the $\mathrm{C}_2$ and small-circle models, rejected for the short-period comets. The estimated parameters were $\hat{\bmu}=(0.0411,0.0027,0.9991)^\top$, $\hat w=0.1204$, $\hat\alpha=39.2691$, and $\hat\beta=0.7683$ (short-period), and $\hat{\bmu}=(0.0127,0.1849,0.9827)^\top$, $\hat w=0.3217$, $\hat\alpha=0.6850$, and $\hat\beta=0.7365$ (long-period).
In both samples, $\hat \bV$ was ill-conditioned. Thus, we resorted to the reestimated multiplier bootstrap \eqref{eq:slow_bootstrap_process}, which avoids the direct use of $\hat{\bV}^{-1}$.
\begin{table}[htbp]
\centering
\caption{AIC, BIC, and GOF test results for the long- and short-period comet samples. In bold, $p$-values larger than $5\%$.}
\label{tab:comets-gof}
\begin{tabular}{l c c c c c c c c}
\toprule
Model & \multicolumn{4}{c}{Long period} & \multicolumn{4}{c}{Short period} \\
\cmidrule(lr){2-5}\cmidrule(lr){6-9}
 & AIC & BIC & KS & CM & AIC & BIC & KS & CM \\
\midrule
$\mathrm{C}_2$ & 3065.7 & 3078.9 & \textbf{0.0819} & \textbf{0.1648} & 3098.8 & 3112.8 & 0.0010 & 0.0010 \\
SC & 3073.1 & 3090.8 & \textbf{0.0679} & \textbf{0.0969} & 2976.2 & 2994.8 & 0.0010 & 0.0010 \\
UB & 3051.2 & 3073.3 & \textbf{0.0929} & \textbf{0.1928} & 347.7 & 371.0 & \textbf{0.5445} & \textbf{0.5085} \\
\bottomrule
\end{tabular}
\end{table}

\subsection{Wind speed and direction on \texorpdfstring{$\mathbb{H}^2$}{H²}}

We considered an application to wind data, which consists of observations of wind speeds and directions from the Ris{\o} meteorological mast. \cite{Jensen1981} proposed fitting an HvMF distribution to this data after transforming each speed--direction pair $(s,\omega)$ into a point on $\mathbb{H}^2$ by first rescaling the speed by the mean speed, $\bar s$, that is,
$\bx=(\cosh(u),\sinh(u)\cos(\omega),\sinh(u)\sin(\omega))^\top,$ where $u = s/\bar s.$

We used modern open data also from Ris{\o}, distributed through the DTU dataset \texttt{Risoe\_m\_all.nc} \citep{Hansen2021}. The archive contains 10-minute averaged meteorological measurements, including wind speeds and directions at $77\mathrm{m}$ above ground level. We retained the cleaned records from 1996--2004, after availability and quality filters, and recorded the timestamp closest to $12{:}00$ on days $4,8,12,16,20,24,$ and $28$ of each month.

\cite{Jensen1981} considered November--December observations sampled every four days, and claimed that the HvMF model provided a good fit to the data, through graphical inspection. We analyzed the period November--December--January, and additionally considered the May--June--July observations. As shown in Table~\ref{tab:risoe-wind-gof}, the tests did not reject the HvMF model for November--December--January data, consistently with \citet{Jensen1981}, while both tests rejected the HvMF model for May--June--July.

\begin{table}[htbp]
\centering
\caption{MLE estimates and GOF test results for the HvMF distribution fitted to the Ris{\o} wind data. In bold, $p$-values larger than $5\%$.}
\label{tab:risoe-wind-gof}
\begin{tabular}{lcccccc}
\toprule
Months & $n$ & $\widehat{\bmu}$ & $\widehat{\kappa}$ & KS & CM \\
\midrule
Nov--Dec--Jan & 169 & $(1.04,-0.22,-0.16)^\top$ & 1.5 & \textbf{0.4605} & \textbf{0.4975} \\
May--Jun--Jul & 188 & $(1.04, -0.04, -0.28)^\top$ & 1.5 & 0.0010 & 0.0010 \\
\bottomrule
\end{tabular}
\end{table}

Figure~\ref{fig:risoe-wind-densities} compares the fitted parametric densities with circular--linear kernel density estimates %
\citep{Garcia-Portugues2013b}. %
The fitted HvMF density was multiplied by the hyperbolic Jacobian $\sinh(u)$ to compare it with the nonparametric estimate. %
The nonparametric estimate for May--June--July exhibits two distinct modes, which is incompatible with the unimodal HvMF model. For November--December--January, the reported parametric and nonparametric estimates are similar. This is consistent with the decision of the GOF tests in Table~\ref{tab:risoe-wind-gof}.

\ifincludeimages
\begin{figure}[htbp]
\centering
\begin{subfigure}[t]{0.2025\textwidth}
\centering
\includegraphics[width=\textwidth]{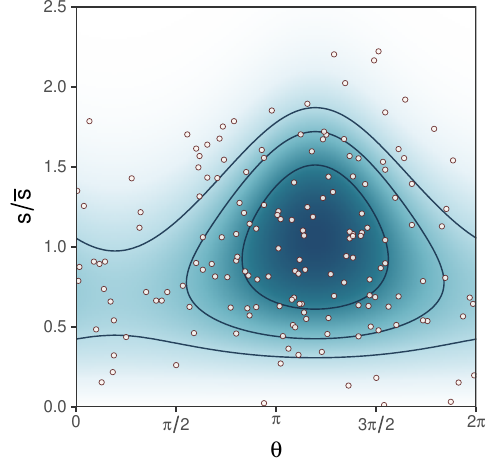}
\caption{Nov--Dec--Jan, \\ parametric fit.}
\end{subfigure}
\hfill
\begin{subfigure}[t]{0.2775\textwidth}
\centering
\includegraphics[width=\textwidth]{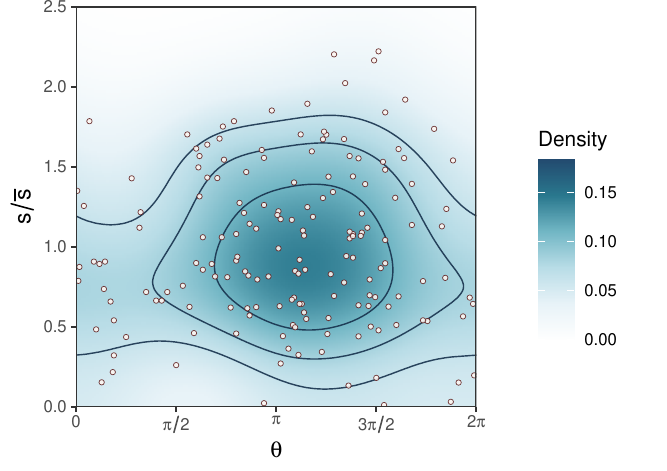}
\caption{Nov--Dec--Jan, \\ nonparametric fit.}
\end{subfigure}
\hfill
\begin{subfigure}[t]{0.2025\textwidth}
\centering
\includegraphics[width=\textwidth]{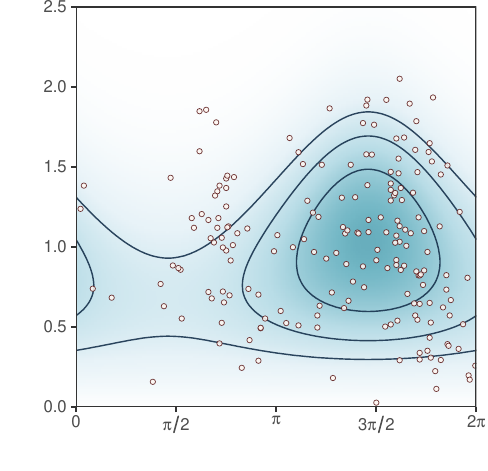}
\caption{May--Jun--Jul, \\ parametric fit.}
\end{subfigure}
\hfill
\begin{subfigure}[t]{0.2775\textwidth}
\centering
\includegraphics[width=\textwidth]{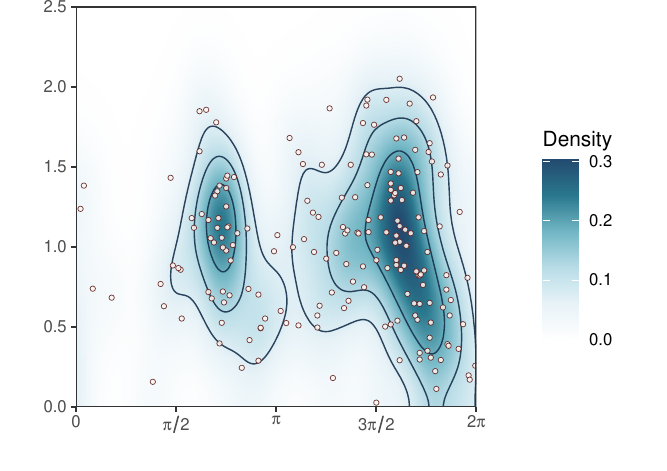}
\caption{May--Jun--Jul, \\ nonparametric fit.}
\end{subfigure}
\caption{Fitted HvMF and nonparametric density estimates for the wind data. The curves delimit highest-density regions with probability contents $25\%,50\%,$ and $75\%$.}
\label{fig:risoe-wind-densities}
\end{figure}
\fi

\subsection{Compositional datasets on \texorpdfstring{$\Delta^d$}{Δᵈ}} \label{subsec:real_data_compositional}

We considered compositional data on the simplex endowed with the Aitchison distance, and tested for logistic normality.
Specifically, we considered several datasets in the R package \texttt{compositions} \citep{vandenBoogaart2025}, as shown in Table \ref{tab:logistic-gaussian-real-data}. %

The datasets \texttt{ArcticLake} and \texttt{Sediments} collect compositions of sand, silt and clay. \citet{Aitchison1982} discussed the goodness-of-fit for logistic normality of \texttt{ArcticLake}, and concluded that the raw sediment compositions show significant departure from logistic normality. This is consistent with the results of our KS and CM tests. \citet{Aitchison1980} fitted logistic normal distributions to the two sediment types in \texttt{Sediments}, for predictive purposes. Our tests showed no significant evidence against the assumption of logistic normality in \texttt{Sediments}. Figure \ref{fig:logistic-gaussian-simplex-d3} shows contour density plots of the fitted logistic Gaussian distribution for all the datasets with $d=2$.

Table \ref{tab:logistic-gaussian-real-data} reports the results of the KS and CM tests, together with the Baringhaus--Henze--Epps--Pulley (BHEP) test of multivariate normality \citep{HenzeWagner1997}, applied after ilr-transforming each composition. Its tuning parameter was chosen by the Henze--Zirkler rule \citep{HenzeZirkler1990}. Overall, the conclusions of our tests are similar to those of BHEP: at the $5\%$ level, all three tests agree for seven of the eight datasets.

\begin{table}[htbp]
\centering
\caption{GOF test results for logistic Gaussianity in the compositional datasets. In bold, $p$-values larger than $5\%$.}
\label{tab:logistic-gaussian-real-data}
\begingroup
\begin{tabular}{lccccc}
\toprule
Dataset & $n$ & $d$ & KS & CM & BHEP \\
\midrule
\texttt{Activity31} & 20 & 5 & \textbf{0.4326} & \textbf{0.3067} & \textbf{0.2817} \\
\texttt{ArcticLake} & 39 & 2 & 0.0020 &0.0010 & 0.0001 \\
\texttt{ClamEast} & 20 & 5 & \textbf{0.4585} & \textbf{0.5165} & \textbf{0.3283} \\
\texttt{HouseholdExp} & 40 & 3 & \textbf{0.1189} & \textbf{0.3337} & \textbf{0.2173} \\
\texttt{PogoJump} & 28 & 2 & \textbf{0.6923} & \textbf{0.7203} & \textbf{0.3453} \\
\texttt{Sediments} & 21 & 2 & \textbf{0.5265} & \textbf{0.4326} & \textbf{0.6516} \\
\texttt{WhiteCells\_microscopic} & 30 & 2 & \textbf{0.0589} & 0.0490 & 0.0015 \\
\texttt{Yatquat\_panel} & 40 & 2 & \textbf{0.5564} & \textbf{0.5514} & \textbf{0.7943} \\
\bottomrule
\end{tabular}
\endgroup
\end{table}

\ifincludeimages
\begin{figure}[htbp]
\centering
\begin{subfigure}[htbp]{0.19\textwidth}
\centering
\includegraphics[width=\textwidth]{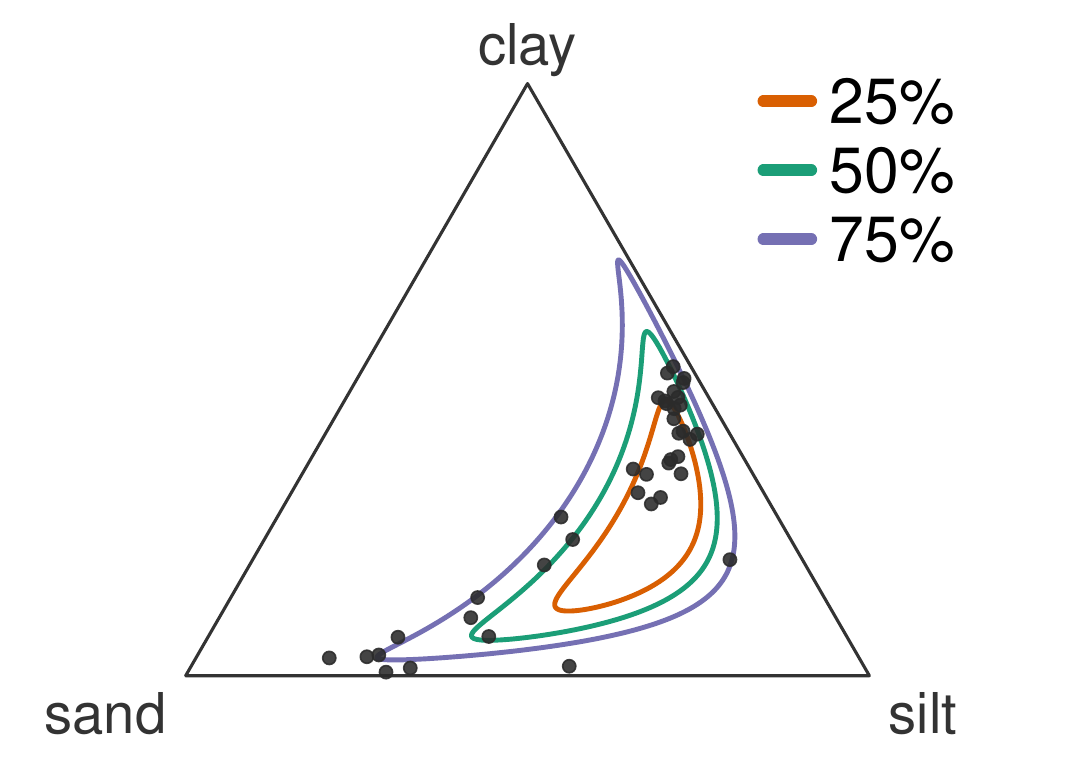}
\caption{\tiny\texttt{ArcticLake}}
\end{subfigure}\hfill
\begin{subfigure}[htbp]{0.19\textwidth}
\centering
\includegraphics[width=\textwidth]{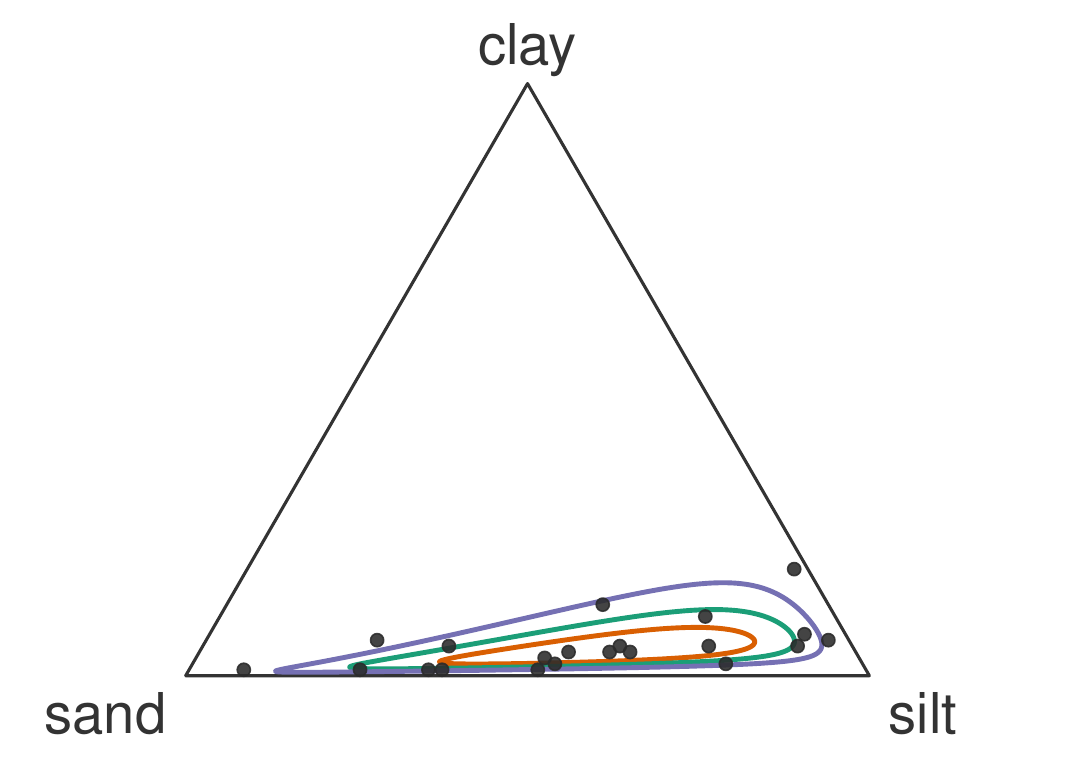}
\caption{\tiny\texttt{Sediments}}
\end{subfigure}\hfill
\begin{subfigure}[htbp]{0.19\textwidth}
\centering
\includegraphics[width=\textwidth]{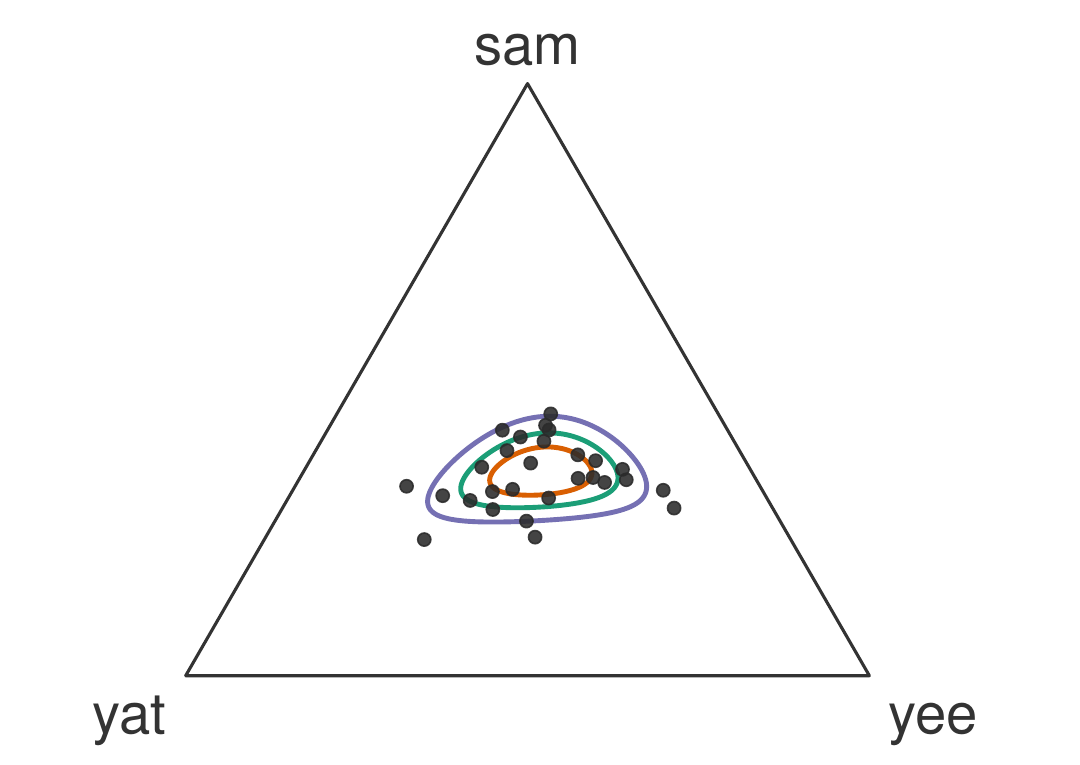}
\caption{\tiny\texttt{PogoJump}}
\end{subfigure}\hfill
\begin{subfigure}[htbp]{0.19\textwidth}
\centering
\includegraphics[width=\textwidth]{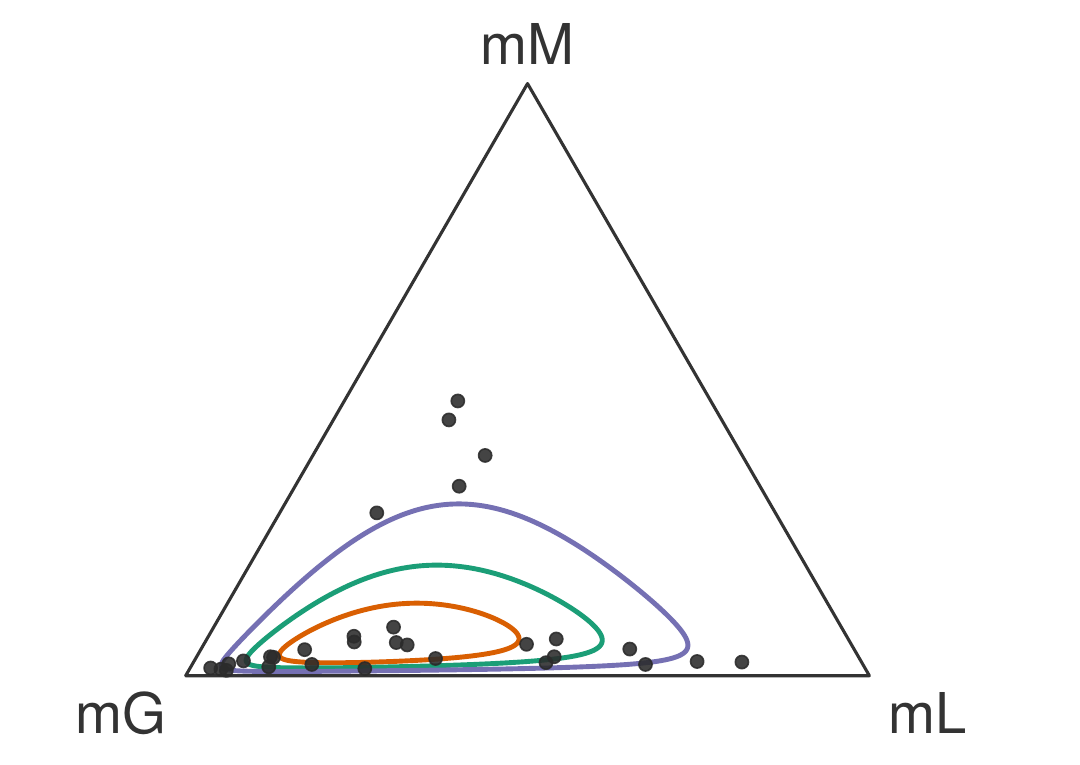}
\caption{\tiny\texttt{WhiteCells\_micro.}}
\end{subfigure}\hfill
\begin{subfigure}[htbp]{0.19\textwidth}
\centering
\includegraphics[width=\textwidth]{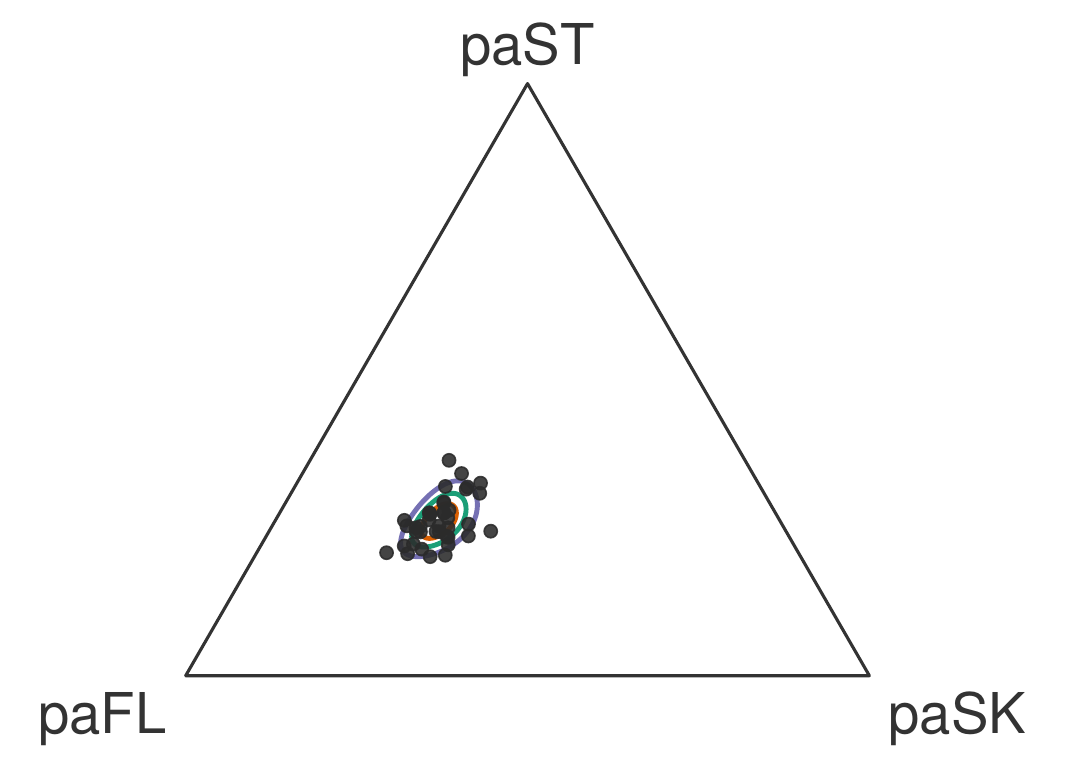}
\caption{\tiny\texttt{Yatquat\_panel}}
\end{subfigure}
\caption{Density contour plots on $\Delta^2$ for the datasets of Table \ref{tab:logistic-gaussian-real-data}.}
\label{fig:logistic-gaussian-simplex-d3}
\end{figure}
\fi

\section{Discussion}\label{sec:discussion}

We have proposed KS and CM tests for simple and composite null hypotheses on separable metric spaces, based on distance-profile empirical processes. The asymptotic theory accounts for parameter estimation and describes the behavior of the tests under the null and under fixed and local alternatives. The multiplier bootstrap provides practical calibration, with an alternative implementation under composite null hypotheses that avoids repeated estimation and evaluation of the fitted profiles. The numerical experiments and applications illustrate the use of the testing framework across different metric spaces.

The work opens several future research directions. First, it would be interesting to investigate whether conditions \ref{b1} and \ref{bp2} can be replaced by the bracketing condition \eqref{eq:bracketing_assumption_1} of \citet{Chen2025}. Such a replacement would shift the entropy assumption from a covering-number condition on the metric space to an $L^1(P)$ bracketing condition on the class of ball indicators, allowing for sub-Gaussian and sub-exponential models with unbounded support. Future work could extend our null and non-null asymptotic results, together with the bootstrap theory, under this condition.

Second, the CM test was conservative in some of the higher-dimensional settings considered in Table~\ref{tab:empirical_null_calibration}. Additional exploratory experiments indicate that its finite-sample behavior can depend on the distribution of the multipliers. It would be of interest to compare the performance of our tests across different dimensions when the multiplier weights follow distributions other than exponential, such as Poisson and two-point distributions.

Third, Proposition~\ref{prop:sup_max_ks} gives conditions under which the KS supremum can be replaced by a maximum over sample centers and pairwise sample distances under the null hypothesis. It remains of interest to determine conditions under which this replacement also preserves consistency against fixed alternatives and the weak limits under $n^{-1/2}$-local alternatives. 

Finally, a broader direction for future work is to extend our GOF methodology to regression models with responses taking values in metric spaces. The residual-based GOF tests for parametric regression models of \citet{Stute1998} and the projection-based approach of \citet{Escanciano2006} provide natural starting points for such an extension.

\section*{Acknowledgments}
The authors gratefully acknowledge the C3 computer resources provided by ``Centro para el Análisis y Modelado de Sistemas Complejos en Ingeniería y Biomedicina'', a center funded by Programa Estatal para Impulsar la Investigación Científico-Técnica y su Transferencia (ref. EQC2021-007184-P) and by the European Union NextGenerationEU/PRTR program.

\bibliographystyle{apalike-custom}
\bibliography{bibliography_main}

\newpage
\section*{Supplement to ``Goodness-of-fit for distributions on metric spaces''.}
The supplementary material contains the proofs of the theoretical results; computational details with model-specific theoretical results used in the code; and numerical experiments illustrating how the KS statistics approach their null limits in simple and composite settings, when the data follow a vMF distribution on $\Sp^2$.

\appendix

\section{\texorpdfstring{Proofs of Section \ref{sec:dp_examples}}{Proofs of Section \ref{sec:dp_examples}}}\label{ap:proof_prop_dp_examples}

\begin{proof}[Proof of Proposition~\ref{prop:dp_examples}]
\emph{Proof of \ref{prop:dp_examples_normal}.} The distance profile at $\bomega \in \R^q$ is
$
F_{\bomega}(t) = \Prob{d_{\R^q}(\bomega, \bX) \le t} = \Prob{\bZ^\top \bZ \le t^2},
$
where $\bZ = \bX - \bomega$ and $\|\cdot\|_2$ is the Euclidean norm in $\R^q$. Thus, $\bZ \sim \mathcal{N}\left(\bmu - \bomega, \bSigma \right)$. We are interested in the distribution of $\bZ^\top \bZ$. Consider $\boldsymbol{\eta} = \bU^\top \bZ$ and define $P_i\defin\eta_i^2/\lambda_i$, $i=1,\ldots,q$. Thus, $\boldsymbol{\eta} \sim \mathcal{N}_q(\boldsymbol{\nu},\mathrm{diag}(\blambda))$. We conclude that
\begin{align*}
\bZ^\top \bZ =
\boldsymbol{\eta}^\top \boldsymbol{\eta} = \sum_{i=1}^q \eta_i^2 =\blambda^\top\bP,
\end{align*}
with $\eta_i$ independent, $\eta_i \sim \mathcal{N}(\nu_i, \lambda_i)$, and thus $P_i$ are also independent and follow a noncentral chi-squared distribution, $P_i \sim \chi_1^2(\nu_i^2/\lambda_i)$. In the special case when $\bomega = \bmu$, we have that $\bZ^\top \bZ = \sum_{i=1}^q \lambda_i P_i,$
where $P_i$ are independent, $P_i \sim \chi_1^2$. If additionally, $\bSigma=\bI_q$, where $\bI_q$ is the $q\times q$ identity matrix, then $\bZ^\top \bZ$ follows a $\chi_q^2$ distribution.

\emph{Proof of \ref{prop:dp_examples_brownian}.} The distance profile of $B$ at a point $\omega \in L^2[0,1]$ is given by $F_{\omega}(t) = \Prob{ \|B - \omega\|_{L^2} \le t}.$
In order to analyze the distribution of $ \|B - \omega\|_{L^2}$, express $B$ and $\omega$ in the Karhunen-Loève expansion. Let $\{e_k\}_{k=1}^\infty$ denote the orthonormal eigenfunctions of the covariance kernel of $B$, given by $K_B(s,t) = \Cov{B(s)}{B(t)} = \min\{s, t\}$. Then, $B$ and $\omega$ can be expressed as
$$
B(t) = \sum_{i=1}^\infty Z_i e_i(t) \text{ and } \omega(t) = \sum_{i=1}^\infty c_i e_i(t).
$$
Here, $Z_i$ are independent random variables with $Z_i \sim \mathcal{N}(0, \lambda_i)$, with $\lambda_i=\left(i-1/2\right)^{-2} \pi^{-2}$, and the coefficients $c_i$ are given by $c_i = \langle \omega, e_i \rangle = \int_{0}^1\omega(t)e_i(t) \, \rd t$. Hence,
$
B(t) - \omega(t) = \sum_{i=1}^\infty ( Z_i- c_i )e_i(t).
$
Define $Q_i\defin(Z_i-c_i)^2/\lambda_i$. Then, $Q_i$ are mutually independent with $Q_i\sim\chi_1^2(c_i^2/\lambda_i)$.
Thus,
$$
\|B - \omega\|_{L^2}^2 = \sum_{i=1}^\infty ( Z_i- c_i )^2 = \sum_{i=1}^\infty\lambda_i Q_i.
$$
Hence, $F_\omega(t)=\Prob{\sum_{i=1}^\infty\lambda_i Q_i\le t^2}$ for $t\ge0$. As a special case, if $\omega(t)=0$ for all $t \in [0,1]$, then $\|B - \omega\|_{L^2}^2 =  \sum_{i=1}^\infty \lambda_i Q_i$, where $Q_i$ are iid and $Q_i \sim \chi_1^2$. Observe the similarity to the distance profile for the multivariate normal when $\bomega = \bmu$.

\emph{Proof of \ref{prop:dp_examples_vmf}.}
Assume first that $|\bmu^\top\bomega|<1$. To characterize the distance profile, the distribution of $T = \bomega^\top \bX$ plays a central role. To simplify notation, let $\bmu_\perp \defin \bmu - (\bmu^\top \bomega)\bomega$, where $\| \bmu_\perp\| = \sqrt{1 - (\bmu^\top\bomega)^2}$. Define also $\Tilde{\bmu} = \bmu_\perp/\| \bmu_\perp \|$.
The tangent-normal decomposition allows expressing any $\bx \in \Sp^{q}$ as $\bx = t\bomega + \sqrt{1-t^2}\bB_{\bomega}\boldsymbol{\xi}$, where $\bB_{\bomega}\in\R^{(q+1)\times q}$ is semi-orthogonal, $\bB_{\bomega}^\top\bB_{\bomega}=\bI_q$, $\bB_{\bomega}\bB_{\bomega}^\top=\bI_{q+1}-\bomega\bomega^\top$, and $\boldsymbol{\xi}\in\Sp^{q-1}$. We have that
\begin{align*}
\bmu^\top \bx = t\bmu^\top \bomega + \sqrt{1-t^2} \bmu^\top \bB_{\bomega}\boldsymbol{\xi} &= t \bmu^\top\bomega +  \sqrt{1-t^2} \bmu_\perp^\top \bB_{\bomega}\boldsymbol{\xi}
\\
&= t \bmu^\top\bomega + \sqrt{(1-t^2)(1-(\bmu^\top\bomega)^2)}\Tilde{\bmu}^\top \bB_{\bomega}\boldsymbol{\xi}.
\end{align*}
Hence,
$$
f_\mathrm{vMF}(\bx;\bmu,\kappa) = c_{q}^\mathrm{vMF}(\kappa)e^{ \kappa\bmu^\top\bomega t} e^{\kappa \sqrt{(1-t^2)(1-(\bmu^\top\bomega)^2)} \Tilde{\bmu}^\top \bB_{\bomega}\boldsymbol{\xi}}.
$$
We know that $\rd\sigma_{q}(\bx) = (1-t^2)^{(q-2)/2} \rd t\,\rd\sigma_{q-1}(\boldsymbol{\xi})$, where $\rd\sigma_{q}$ is the Lebesgue measure on the $q$-dimensional unit sphere. So
\begin{align}\label{integral_form_density}
  f_{T}(t) = c_{q}^\mathrm{vMF}(\kappa)e^{ \kappa\bmu^\top\bomega t} (1-t^2)^{(q-2)/2}
\int_{\Sp^{q-1}} e^{\kappa \sqrt{(1-t^2)(1-(\bmu^\top\bomega)^2)} \Tilde{\bmu}^\top \bB_{\bomega}\boldsymbol{\xi}} \, \rd \sigma_{q-1}(\boldsymbol{\xi}).
\end{align}
Since $\bB_{\bomega}^\top\Tilde{\bmu}\in\Sp^{q-1}$, the integrand in \eqref{integral_form_density} corresponds to the density of a $\mathrm{vMF}\big(\bB_{\bomega}^\top\Tilde{\bmu}, \allowbreak \kappa \sqrt{(1-t^2)(1-(\bmu^\top\bomega)^2)}\big)$ (up to the normalization constant). For $q=1$, we endow $\Sp^0$ with counting measure and set $c_0^\mathrm{vMF}(\kappa)\defin(2\cosh(\kappa))^{-1}$. We conclude that \eqref{eq:density_t_vmf} holds. If $\bomega=\pm\bmu$, then the tangent-normal decomposition gives $\bmu^\top \bx=t\,\bmu^\top\bomega$. Integrating over $\Sp^{q-1}$ yields \eqref{eq:density_t_vmf}. Hence, \eqref{eq:density_t_vmf} holds for all $\bomega\in\Sp^q$.
Obtaining $F_T$, the cdf of $T$, from \eqref{eq:density_t_vmf} is not direct.

For the chordal distance, if $0\le t\le 2$, then $F_{\bomega}(t) = \Probbig{2(1-\bX^\top \bomega)\le t^2} = 1 - F_T\big(1-t^2/2\big), $
whereas $F_{\bomega}(t)=1$ for $t>2$. Under the geodesic distance, if $0\le t\le \pi$, then
$F_{\bomega}(t) = \Probbig{\cos^{-1}(\bX^\top \bomega)\le t} = 1 - F_T(\cos(t)),$
whereas $F_{\bomega}(t)=1$ for $t>\pi$.

\emph{Proof of \ref{prop:dp_examples_hvmf}.} For $\bx\in\mathbb H^q$, define the hyperbolic-spherical coordinates by $\bx=(\cosh(x_r), \allowbreak \sinh(x_r) \bx_s^\top)^\top$, with $x_r \ge 0$ and $\bx_s\in\mathbb S^{q-1}$.
Denoting by $\mu_q$ the invariant measure on $\mathbb{H}^q$ from \citetSM{Jensen1981SM}, and $\sigma_{q-1}$ the Lebesgue measure on $\Sp^{q-1}$ (with $\sigma_0$ understood as counting measure), it holds that $\mu_q(\rd \bx)= \sinh^{q-1}(x_r)\,\rd x_r \,\sigma_{q-1}(\rd \bx_s)$.

Take a hyperbolic transformation $\bA$ such that $\bA\bomega=\be_1=(1,0,\ldots,0)^\top$ \citepSM[see Lemma 2(ii) of][]{Jensen1981SM}. Since hyperbolic transformations preserve $d_{\mathbb H^q}$, the distance profile for $\bX\sim\mathrm{HvMF}(\bmu,\kappa)$ can be written as
\begin{align*}
F_{\bomega}(t)
= \Prob{d_{\mathbb{H}^q}(\bX, \bomega) \le t}
= \Prob{d_{\mathbb{H}^q}(\bA\bX, \be_1) \le t} 
= \Prob{(\bA\bX)_1 \le \cosh(t)},
\end{align*}
for $t\ge0$. Writing $\bA\bX=(\cosh(Z),\sinh(Z)\boldsymbol{Z}_s^\top)^\top$, with
$Z\ge0$ and $\boldsymbol{Z}_s\in\mathbb S^{q-1}$, this gives
$
F_{\bomega}(t)=\Prob{Z\le t}.
$

The case $\rho=0$ is immediate, so assume $\rho>0$.
Since
$d_{\mathbb H^q}(\be_1,\bA\bmu)
=
d_{\mathbb H^q}(\bA\bomega,\bA\bmu)
=
d_{\mathbb H^q}(\bomega,\bmu)
=
\rho$,
writing $\bA\bmu = (\cosh(\rho), \sinh(\rho) \bxi)^\top$, with $(\rho, \bxi) \in \R_{\ge0} \times \Sp^{q-1}$, and since $\bA\bX\sim\mathrm{HvMF}(\bA\bmu,\kappa)$, integrating out the angular component gives
\begin{align*}
f_{Z}(z)
=
\frac{c_{q}^\mathrm{HvMF}(\kappa)}
{c_{q-1}^\mathrm{vMF}(\kappa\sinh(z)\sinh(\rho))}
(\sinh(z))^{q-1}
e^{-\kappa\cosh(\rho)\cosh(z)},
\quad z \ge 0.
\end{align*}
See the proof of Proposition D.1 in \citetSM{Serrano2026SM}. Finally, since
$
(\bomega,\bmu)
=
(\be_1,\bA\bmu)
=
-\cosh(\rho),
$
one has
\begin{align*}
f_{Z}(z)
=
\frac{c_{q}^\mathrm{HvMF}(\kappa)}
{c_{q-1}^\mathrm{vMF}(\kappa\sinh(z)\sinh(\rho))}
(\sinh(z))^{q-1}
e^{\kappa(\bmu,\bomega)\cosh(z)},
\quad z \ge 0.
\end{align*}

\emph{Proof of \ref{prop:dp_examples_logistic_gaussian}.} With the notation of Example~\ref{ex:dp_logistic_gaussian}, let
$$
\bZ \defin \mathrm{ilr}(\bX) - \mathrm{ilr}(\bomega), \qquad
\bW \defin \bU^\top \bZ, \qquad
\bm \defin \bmu-\mathrm{ilr}(\bomega), \qquad
\bv \defin \bU^\top \bm.
$$
One has that $\bZ\sim \mathcal{N}_q(\bm,\bSigma)$ and $\bW\sim \mathcal{N}_q(\bv,\mathrm{diag}(\boldsymbol{\alpha}))$. Since $\bU$ is orthogonal,
$$
d_{\mathrm{A}}(\bX, \bomega)^2 = \|\mathrm{ilr}(\bX)-\mathrm{ilr}(\bomega)\|_2^2 = \bZ^\top \bZ = \bW^\top \bW = \sum_{i=1}^{q} W_i^2,
$$
where
$W_i \sim \mathcal{N}(v_i,\alpha_i),$
for $i=1,\dots,q$. Hence, we have that $W_i = v_i + \sqrt{\alpha_i}\,\varepsilon_i$ where $\varepsilon_i$ are independent and identically distributed, $\varepsilon_i \sim \mathcal{N}(0,1)$. Hence,
$W_i^2 = \alpha_i\left(\varepsilon_i + v_i/\sqrt{\alpha_i}\right)^2.$
By definition, $\left(\varepsilon_i + v_i/\sqrt{\alpha_i}\right)^2$ is a noncentral $\chi^2_1$ random variable with one degree of freedom and noncentrality parameter
$\delta_i = v_i^2/\alpha_i = (\bu_i^\top\bm)^2/\alpha_i$. Therefore, setting $Q_i \defin W_i^2/\alpha_i$, it holds that $Q_i \sim \chi_1^2(\delta_i)$, mutually independent, and
\begin{align*}
d_{\mathrm{A}}(\bX,\bomega)^2
&=\sum_{i=1}^q W_i^2
=\sum_{i=1}^q\alpha_i Q_i
=\boldsymbol{\alpha}^\top\bQ.
\end{align*}
Consequently, $F_{\bomega}(t)=\Prob{\boldsymbol{\alpha}^\top\bQ\le t^2}$ for $t\ge0$.
\end{proof}

\section{Proofs of Section \ref{sec:null_asymptotics}}\label{ap:sec_proofs_null_asymptotics}

\subsection{Proof of Theorem \ref{thm:convergence_G_n_composite}}\label{ap:sec_proofs_null_asymptotics_thm_convergence_G_n_composite}
\begin{proof}[Proof of Theorem \ref{thm:convergence_G_n_composite}]
First, observe that
\begin{align}\label{eq:split_G_n_mu^mu_hat_theta}
\G_{n,\mu}^{\mu_{\hat{\btheta}}}(\omega, t) =
\G_{n,\mu}^{\mu_{\btheta_0}}(\omega, t) 
- \sqrt{n} \left( 
F_{\omega}^{\mu_{\hat{\btheta}}}(t) - F_{\omega}^{\mu_{\btheta_0}}(t)
\right).
\end{align}
By Theorem 3 of \citeSM{Chen2025SM}, $\G_{n,\mu}^{\mu_{\btheta_0}}(\omega, t)$ converges weakly to a Gaussian process:
\begin{align*}
\G_{n,\mu}^{\mu_{\btheta_0}}(\omega, t)  \convl \G_{\btheta_0}(\omega, t) \text { in } \ell^{\infty}(\mathcal{Y}),
\end{align*}
where $\G_{\btheta_0}(\omega, t)$ is a zero-mean Gaussian process with covariance given by 
\begin{align}\label{eq:cov_G_theta_0}
\Cov{y_{\omega_1, t_1}}{y_{\omega_2, t_2}} = \Prob{d(\omega_1, X) \le t_1, d(\omega_2, X) \le t_2} - F_{\omega_1}^{\mu_{\btheta_0}}(t_1)F_{\omega_2}^{\mu_{\btheta_0}}(t_2),
\end{align}
for $\omega_1, \omega_2 \in \Omega$ and $t_1, t_2 \in \R_{\ge0}$.

Now, focus on $F_{\omega}^{\mu_{\hat{\btheta}}}(t) - F_{\omega}^{\mu_{\btheta_0}}(t)$. By differentiability of $\btheta \mapsto F_{\omega}^{\mu_{\btheta}}(t)$ at $\btheta_0$, it holds that
\begin{align}\label{eq:taylor_F_omega_t}
F_{\omega}^{\mu_{\hat{\btheta}}}(t) - F_{\omega}^{\mu_{\btheta_0}}(t) = 
\dot{F}_{\omega}^{\btheta_0}(t)^\top \big(\hat{\btheta} - \btheta_0\big) + o_{\mathsf{P}}\big(\|\hat{\btheta} - \btheta_0 \| \big).
\end{align}
From \eqref{eq:taylor_F_omega_t} and the Bahadur representation \eqref{eq:bahadur}, one obtains that
\begin{align}\label{eq:taylor_bahadur}
\sqrt{n} \left(F_{\omega}^{\mu_{\hat{\btheta}}}(t) - F_{\omega}^{\mu_{\btheta_0}}(t)\right) =
-\dot{F}_{\omega}^{\btheta_0}(t)^\top \bV_{\btheta_0}^{-1} \frac{1}{\sqrt{n}} \sum_{i=1}^n \bpsi_{\btheta_0}(X_i) \ + o_{\mathsf{P}}(1) + \sqrt{n} o_{\mathsf{P}}\left(\|\hat{\btheta} - \btheta_0 \| \right).
\end{align}
Observe that $\dot{F}_{\omega}^{\btheta_0}(t)^\top o_{\mathsf{P}}(1) = o_{\mathsf{P}}(1)$ by condition \ref{d}.
By the Bahadur representation \eqref{eq:bahadur}, $\sqrt{n}\|\hat{\btheta} - \btheta_0\| = O_{\mathsf{P}}(1)$, so $\sqrt{n} o_{\mathsf{P}}\big(\|\hat{\btheta} - \btheta_0 \|\big) = o_{\mathsf{P}}(1)$. 
Hence, by condition \ref{d}, the remainders in \eqref{eq:taylor_bahadur} are $o_{\mathsf{P}}(1)$ uniformly over $(\omega,t)\in\Omega\times\R_{\ge0}$.
Therefore, \eqref{eq:split_G_n_mu^mu_hat_theta} and \eqref{eq:taylor_bahadur} yield that
\begin{align}\label{eq:G_n_Thm_4.1}
\G_{n,\mu}^{\mu_{\hat{\btheta}}}(\omega, t)
& =
\frac{1}{\sqrt{n}}\sum_{i=1}^n
\big( y_{\omega,t}(X_i) - F_{\omega}^{\mu_{\btheta_0}}(t) + g_{\omega,t}(X_i)\big)
+ o_{\mathsf{P}}(1) \notag 
\\
& =
\frac{1}{\sqrt{n}} \sum_{i=1}^n f_{\omega, t}(X_i) + o_{\mathsf{P}}(1),
\end{align}
where 
\begin{align*}
g_{\omega,t}(x) \defin \dot{F}_{\omega}^{\btheta_0}(t)^\top \bV_{\btheta_0}^{-1} \bpsi_{\btheta_0}(x), \quad \text{and } \quad f_{\omega,t}(x) \defin y_{\omega, t}(x) - F_{\omega}^{\mu_{\btheta_0}}(t) + g_{\omega, t}(x).
\end{align*}
Note $\mathsf{E}(f_{\omega, t}(X)) = 0$, so $\G_{n,\mu}^{\mu_{\hat{\btheta}}}$
is, up to an $o_{\mathsf{P}}(1)$ remainder term, the empirical process indexed by the class 
\begin{align*}
\cF \defin \{f_{\omega, t}: \ \omega \in \Omega, \ t \ge0\}.
\end{align*}
Thus, proving that $\cF$ is Donsker concludes the proof.

To prove the Donsker property, write $\cF \subset \cY + \mathcal{G}$, where $\cY \defin \{y_{\omega, t} - F_{\omega}^{\mu_{\btheta_0}}(t): \ \omega \in \Omega, \ t \ge 0\}$ and $\mathcal{G} \defin \{g_{\omega, t}: \ \omega\in \Omega, \ t \ge 0\}$. 
If $\cY$ and $\mathcal{G}$ are Donsker and $\sup_{h\in\cY \cup \mathcal{G}} {|\mathsf{E}(h(X))|} < \infty$, then  $\cY + \mathcal{G}$ is also Donsker \citepSM[see, e.g., Example 2.10.9 in][]{VanderVaartWellner2023SM}, and since $\cF$ is a subset of a Donsker class, it is Donsker.

The class $\cY$ is Donsker by \eqref{eq:thm_3_chendubey}. Corollary 16.1 in \citeSM{DasGupta2011SM} shows that a sufficient condition for $\mathcal{G}$ to be Donsker is that $\mathcal{G}$ is finite-dimensional, and
its envelope function is in $L_2(\mathsf{P})$. The fact that $\mathcal{G}$ is finite-dimensional can be seen by writing
$$
g_{\omega,t}(x) = \dot{F}_{\omega}^{\btheta_0}(t)^\top \bV_{\btheta_0}^{-1}\bpsi_{\btheta_0}(x) = \sum_{i=1}^p \alpha_i(\omega, t) \psi_{\btheta_0, i}(x),
$$
where $\boldsymbol{\alpha}(\omega, t)^\top \defin (\dot{F}_{\omega}^{\btheta_0}(t))^{\top}\bV_{\btheta_0}^{-1} \in \R^p$ (coefficients), and $\psi_{\btheta_0,i}$ represents the $i$-th component of $\bpsi_{\btheta_0}$, $i=1, \ldots,p$. Therefore, one has $\mathcal{G} \subset \mathrm{span}\{\psi_{\btheta_0, 1}, \ldots, \psi_{\btheta_0, p}\}$, a linear subspace of dimension at most $p$. Condition \ref{d} implies that $\sup_{\omega, t} \|\dot{F}_{\omega}^{\btheta_0}(t) \| < \infty $, hence condition \ref{cp2} guarantees that the envelope function of $\mathcal{G}$ is in $L_2(\mathsf{P})$.  Therefore, $\mathcal{G}$ is Donsker.

It is only left to show that $\sup_{h\in\cY \cup \mathcal{G}} {\mathsf{E}|h(X)|} < \infty$. 
This is trivial for class $\cY$. For $\mathcal G$, the $L_2(P)$ envelope is also in $L_1(P)$, since $P$ is a probability measure. Therefore, $\cF$ is Donsker. 
This implies that $\G_{n,\mu}^{\mu_{\hat{\btheta}}}$ converges weakly to a Gaussian process $\tilde{\G}_0$, with mean zero and covariance $\tilde{\mathcal{C}}_{(\omega_1,t_1), (\omega_2, t_2)} \defin %
\mathsf{E}(f_{\omega_1, t_1}(X) f_{\omega_2, t_2}(X))$, which is given in \eqref{eq:cov_comp}.

It remains to derive the asymptotics of $\widehat{\G}_{n,\mu}$. By the definition of
$\hat{f}_{\omega,t}$,
\begin{align*}
\widehat{\G}_{n,\mu}(\omega,t)
-
\G_{n,\mu}^{\mu_{\hat{\btheta}}}(\omega,t)
=
\sqrt n\,P_n\hat{f}_{\omega,t}
-
\sqrt n\left(P_n y_{\omega,t}-F_\omega^{\mu_{\hat{\btheta}}}(t)\right)
=
\sqrt n\,
\dot F_\omega^{\hat{\btheta}}(t)^\top
\widehat{\bV}^{-1}
P_n\bpsi_{\hat{\btheta}}.
\end{align*}
Hence, by conditions~\ref{e} and~\ref{cp5},
\begin{align*}
\sup_{\omega\in\Omega,\,t\ge0}
\left|
\widehat{\G}_{n,\mu}(\omega,t)
-
\G_{n,\mu}^{\mu_{\hat{\btheta}}}(\omega,t)
\right|
\le
\sup_{\omega\in\Omega,\,t\ge0}
\big\|
\dot F_\omega^{\hat{\btheta}}(t)
\big\|
\big\|
\widehat{\bV}^{-1}
\big\|
\sqrt n
\big\|
P_n\bpsi_{\hat{\btheta}}
\big\|
=
o_{\mathsf P}(1).
\end{align*}
This completes the proof.
\end{proof}

\subsection{Proof of Theorem \ref{thm:test_statistics_simp}}\label{ap:sec_proofs_null_asymptotics_thm_test_statistics_simp}

We first state the auxiliary lemmas used in the proof of Theorem \ref{thm:test_statistics_simp}. The proofs of all lemmas in the supplementary material are collected in Section~\ref{ap:sec_proofs_lemmas}. We use the product distance on $(\Omega, d) \times (\R_{\ge 0}, |\cdot|)$, that is, $d_{\Omega \times \R_{\ge 0}}\big((\omega,t), (\omega^\prime, t^\prime)\big) = d(\omega, \omega^\prime) + |t - t^\prime|. $

Lemma \ref{ap:lem_bound_int_g_epsilon} is an auxiliary result in the proof of Lemma \ref{ap:lem_sup_r_n_r}. 

\begin{lemma}\label{ap:lem_bound_int_g_epsilon}
    Fix $\varepsilon > 0$. Let $g: \Omega \times \R_{\ge0} \to \R$ be a bounded, uniformly continuous function on $\Omega \times \R_{\ge0}$. For each $\omega \in \Omega$, define the following regularization of $g$:
\begin{equation}\label{ap:eq_g_epsilon}
        g_\varepsilon(\omega, t) \defin \inf_{s \in [0, M]} \left\{ g(\omega, s) + \frac{2\|g\|_\infty}{\delta} |t - s| \right\}.
\end{equation}    
    Here, $M \defin \mathrm{diam}(\Omega) = \sup_{\omega, \omega^\prime \in \Omega} d(\omega, \omega^\prime)$, and $\delta(\varepsilon) > 0$ is such that for every $(\omega, t), (\omega^\prime, t^\prime) \in \Omega \times [0, M]$ with $d_{\Omega \times \R_{\ge 0}}\big((\omega, t), (\omega^\prime, t^\prime)\big) < \delta$, it holds that $|g(\omega, t) - g(\omega^\prime, t^\prime)| < \varepsilon$. 
    For every fixed $\omega\in\Omega$, the map $t\mapsto g_\varepsilon(\omega,t)$ is Lipschitz on $[0,M]$, with Lipschitz constant $\mathrm{Lip}(g_{\varepsilon}) = 2\|g\|_\infty/\delta$, and for every $\omega \in \Omega$ and $t \in [0, M]$, it holds a.s. that
    $$
    \left| \int_{0}^{M} g_\varepsilon(\omega, t) \, \rd (F_{n,\omega}^{\mu_0} - F_{\omega}^{\mu_0})(t) \right|
    \le
    \frac{2 M \|g\|_\infty}{\delta} \sup_{t \in [0, M]} |F_{n,\omega}^{\mu_0}(t) - F_{\omega}^{\mu_0}(t)|.
    $$
\end{lemma}

Lemma \ref{ap:lem_sup_r_n_r} is used to prove convergence to zero of a term in Lemma \ref{ap:lem_DCT_P_hat}.
\begin{lemma}\label{ap:lem_sup_r_n_r}
    Let $g: \Omega \times \R_{\ge0} \to \R$ be a bounded, uniformly continuous function on $\Omega \times \R_{\ge0}$.
    For each $\omega \in \Omega$, define
    \begin{align}\label{eq:def_r_r_n}
    r(\omega) \defin \int_{\R_{\ge0}} g(\omega, t) \, \rd F_{\omega}^{\mu_0}(t), \quad r_n(\omega) \defin  \int_{\R_{\ge0}} g(\omega, t) \, \rd F_{n,\omega}^{\mu_0}(t).
    \end{align}
    Under $H_0$, it holds that $\sup_{\omega \in \Omega} |r_n(\omega) - r(\omega)| \convas 0$ as $n \to \infty$.
\end{lemma}
    
Lemma \ref{ap:lem_continuity_r} is used in the proof of Lemma \ref{ap:lem_DCT_P_hat}.
\begin{lemma}\label{ap:lem_continuity_r}
    The function $r$ defined in \eqref{eq:def_r_r_n} is continuous on $\Omega$.
\end{lemma}

Lemma \ref{ap:lem_DCT_P_hat} is used to prove convergence of a term in the proof of Theorem \ref{thm:test_statistics_simp}.
\begin{lemma}\label{ap:lem_DCT_P_hat}
    Let $g: \Omega \times \R_{\ge0} \to \R$ be a bounded, uniformly continuous function on $\Omega \times \R_{\ge0}$. It holds under $H_0$ that
    $$
    \int_{\Omega \times \R_{\ge0}} g(\omega, t) \, \rd F_{n,\omega}^{\mu_0}(t) \, \rd P_n(\omega) \convas \int_{\Omega \times \R_{\ge0}} g(\omega, t) \, \rd F_{\omega}^{\mu_0}(t) \, \rd \mu_0(\omega) \text{ as } n \to \infty.
    $$
\end{lemma}

Lemma \ref{ap:lem_relation_L2_prod} connects $d_{\Omega \times \R_{\ge 0}}$ with the $L^2(\mu)$ pseudometric induced by $\G_0$, which we denote by $\rho_{2,\G_0}$. This allows obtaining uniform continuity with respect to $d_{\Omega \times \R_{\ge 0}}$ from uniform continuity with respect to $\rho_{2,\G_0}$, which is required in Lemma \ref{ap:lem_continuity_sample_paths_G_rho2}.
\begin{lemma}\label{ap:lem_relation_L2_prod}
Let $\rho_{2,y}$ and $\rho_{2,\G_0}$ be the $L^2(\mu)$ pseudometrics induced by the class of functions $\cY$ and by the Gaussian process $\G_0$, respectively. 
That is, for $p\ge1$ and $(\omega,t), (\omega^\prime, t^\prime) \in \Omega \times \R_{\ge 0}$, we define
\begin{align}
\rho_{p,y}\big((\omega,t), (\omega^\prime, t^\prime)\big)
&\defin
\Big(\mathsf{E}\big|y_{\omega,t}(X) - y_{\omega^\prime, t^\prime}(X)\big|^p\Big)^{1/p}, \notag
\\
\rho_{p,\G_0}\big((\omega,t),(\omega^\prime,t^\prime)\big)
&\defin
\Big(
\mathsf E\big|
\G_0(\omega,t)-\G_0(\omega^\prime,t^\prime)
\big|^p
\Big)^{1/p}.\label{eq:ap_distances}
\end{align}
It holds that 
\begin{align}\label{eq:rho2_product_bound_simple}
\rho_{2,\G_0}\big((\omega,t),(\omega^\prime,t^\prime)\big)
&\le \rho_{2,y}\big((\omega,t),(\omega^\prime,t^\prime)\big)
\le \sqrt{2K_\mu
d_{\Omega\times\R_{\ge0}}
\big((\omega,t),(\omega^\prime,t^\prime)\big)
},
\end{align}
where $K_\mu$ is the constant in condition \ref{bp2}. In particular, any function that is uniformly continuous with respect to $\rho_{2,\G_0}$ is also uniformly continuous with respect to $d_{\Omega\times\R_{\ge0}}$.
\end{lemma}

Lemma \ref{ap:lem_continuity_sample_paths_G_rho2} shows that the Gaussian process $\G_{0}(\omega, t)$ has almost surely uniformly continuous sample paths on $(\Omega \times \R_{\ge 0}, d_{\Omega \times \R_{\ge 0}})$. This result is used, combined with Lemma \ref{ap:lem_DCT_P_hat}, in the proof of Theorem \ref{thm:test_statistics_simp}.

\begin{lemma}\label{ap:lem_continuity_sample_paths_G_rho2}
The process $\G_0(\omega,t)$ has almost surely uniformly continuous sample paths on the metric space $(\Omega\times\R_{\ge0}, d_{\Omega\times\R_{\ge0}})$, considering the sample paths as functions from $(\Omega\times\R_{\ge0}, d_{\Omega\times\R_{\ge0}})$ to $(\R,|\cdot|)$.
\end{lemma}

Using Lemmas \ref{ap:lem_bound_int_g_epsilon}--\ref{ap:lem_continuity_sample_paths_G_rho2}, Theorem \ref{thm:test_statistics_simp} is proved.

\begin{proof}[Proof of Theorem \ref{thm:test_statistics_simp}]
First, consider the KS statistic. The continuous mapping theorem gives the asymptotics of $T_n^{\mathrm{KS}}$. 
Since we can identify $\G_{n,\mu}^{\mu_0}(\omega,t) = \G_{n,\mu}^{\mu_0}(y_{\omega,t})$ (see the definition of $\G_{n,\mu}^{\mu_0}$ in Section \ref{sec:null_asymptotics_simple}), it holds that
$$
\sup_{(\omega,t)\in\Omega\times\R_{\ge0}} |\G_{n,\mu}^{\mu_0}(\omega,t)|=\sup_{y\in\cY}|\G_{n,\mu}^{\mu_0}(y)|.
$$
The mapping $\sup_{\cY}: \ell^\infty(\cY) \to \R$ defined by $f \mapsto \sup_{\cY} |f(y)|$ for every $f \in \ell^\infty(\cY)$ is in fact $1$-Lipschitz, hence it is continuous (with respect to the sup-norm).

Focus now on the CM statistic. By Theorems 1 and 3 of \citeSM{Varadarajan1958SM}, $P_n$ converges weakly to $\mu_0$ almost surely. Slutsky's theorem gives $(\G_{n,\mu}^{\mu_0},P_n)\convl(\G_0,\mu_0)$. We know that $\G_0(\omega, t)$ is separable by Remark \ref{rem:separability_G_0}. Therefore, Skorohod's construction \citepSM[see Theorem 1.10.3 in][p. 58]{VanderVaartWellner2023SM} applies to this joint convergence, yielding almost uniform convergence of the coupled pair. Egorov's theorem \citepSM[Lemma 1.9.2 in][]{VanderVaartWellner2023SM} states the equivalence of almost uniform and outer almost sure convergence for sequences of Borel measurable random variables. Therefore, there exist random elements $((\G_{n,\mu}^{\mu_0})^\star,P_n^\star)$ and $\G_0^\star$ defined on the same probability space such that, for every $n$, $ ((\G_{n,\mu}^{\mu_0})^\star,P_n^\star)\overset{d}{=}(\G_{n,\mu}^{\mu_0},P_n)$, $\G_0^\star\overset{d}{=}\G_0,$
and
\begin{equation}\label{eq:aux_CM_sup_convas}
 \sup_{(\omega,t)\in \Omega \times \R_{\ge 0}} \left| (\G_{n,\mu}^{\mu_0})^\star(\omega, t) - \G_0^\star(\omega, t) \right| \convas 0,
 \qquad
 P_n^\star \convas \mu_0 \text{ weakly}.
\end{equation}
The first almost sure convergence in \eqref{eq:aux_CM_sup_convas} follows from the definition of outer almost sure convergence, using the supremum distance. See Definition 1.9.1 in \citetSM{VanderVaartWellner2023SM} for almost uniform and outer almost sure convergence.

For every $\omega\in\Omega$ and $t\ge0$, let $(F_{n,\omega}^{\mu_0})^\star(t) \defin P_n^\star y_{\omega,t}.$ Now, observe that under $H_0$,
\begin{equation}\label{eq:CM_split}
\begin{aligned}
\Bigg| \int_{\Omega \times \R_{\geq 0}} &(\G_{n,\mu}^{\mu_0})^{\star}(\omega, t)^2 \, \rd(F_{n,\omega}^{\mu_0})^\star(t)\,\rd P_n^\star(\omega) - \int_{\Omega \times \R_{\geq 0}} \G_{0}^{\star}(\omega, t)^2 \, \rd F_{\omega}^{\mu_0}(t) \, \rd \mu_0(\omega) \Bigg| \\
&\le
\bigg| \int_{\Omega \times \R_{\ge0}} \left((\G_{n,\mu}^{\mu_0})^{\star}(\omega, t)^2 - \G_{0}^{\star}(\omega, t)^2  \right) \, \rd (F_{n,\omega}^{\mu_0})^\star(t)\,\rd P_n^\star(\omega) \bigg|
\\ &+
\left| \int_{\Omega \times \R_{\ge0}}  \G_{0}^{\star}(\omega, t)^2 \, \rd \left((F_{n,\omega}^{\mu_0})^\star(t)\, P_n^\star(\omega) - F_{\omega}^{\mu_0}(t) \, \mu_0 (\omega)\right) \right|.
\end{aligned}
\end{equation}
For the first term of \eqref{eq:CM_split}, one has
$$
\begin{aligned}
\Bigl| \int_{\Omega \times \R_{\ge0}} \big(&(\G_{n,\mu}^{\mu_0})^{\star}(\omega, t)^2 - \G_{0}^{\star}(\omega, t)^2 \big) \,\rd (F_{n,\omega}^{\mu_0})^\star(t)\,\rd P_n^\star(\omega) \Bigr| \\
&\le 
\int_{\Omega \times \R_{\ge0}} \left| (\G_{n,\mu}^{\mu_0})^{\star}(\omega, t)^2 - \G_{0}^{\star}(\omega, t)^2 \right| \,\rd (F_{n,\omega}^{\mu_0})^\star(t)\,\rd P_n^\star(\omega)\\
&\le 
\sup_{(\omega, t) \in \Omega\times \R_{\ge0}} \left| (\G_{n,\mu}^{\mu_0})^{\star}(\omega, t)^2 - \G_{0}^{\star}(\omega, t)^2 \right| \int_{\Omega \times \R_{\ge0}} \,\rd (F_{n,\omega}^{\mu_0})^\star(t)\,\rd P_n^\star(\omega)\\
&=
\sup_{(\omega, t) \in \Omega\times \R_{\ge0}} \left| (\G_{n,\mu}^{\mu_0})^{\star}(\omega, t)^2 - \G_{0}^{\star}(\omega, t)^2 \right|.
\end{aligned}
$$
Now, observe that
\begin{align*}
\left\|
\big((\G_{n,\mu}^{\mu_0})^{\star}\big)^2-\big(\G_0^{\star}\big)^2
\right\|_\infty
&\le
\left\|
(\G_{n,\mu}^{\mu_0})^{\star}-\G_0^{\star}
\right\|_\infty^2
+
2\left\|\G_0^{\star}\right\|_\infty
\left\|
(\G_{n,\mu}^{\mu_0})^{\star}-\G_0^{\star}
\right\|_\infty
\convas 0,
\end{align*}
by \eqref{eq:aux_CM_sup_convas} and the fact that $\G_{0}^{\star}$ has almost surely bounded sample paths.
Boundedness holds because $\G_0^{\star}$ is a map into $\ell^\infty(\cY)$.

Now, focus on the second term on the RHS of \eqref{eq:CM_split}. The sample paths of the Gaussian process $\G_{0}^{\star}(\omega, t)$ are bounded and uniformly continuous almost surely. Uniform continuity holds by Lemma \ref{ap:lem_continuity_sample_paths_G_rho2}.
The argument of the proof of Lemma \ref{ap:lem_DCT_P_hat} applies to the coupled versions, since \eqref{eq:aux_CM_sup_convas} yields the almost sure weak convergence of $P_n^\star$ and uniform convergence of $(F_{n,\omega}^{\mu_0})^\star$. This gives
\begin{align*}
\int_{\Omega \times \R_{\ge0}} \G_{0}^{\star}(\omega, t)^2 \,
\rd \left((F_{n,\omega}^{\mu_0})^\star(t)\, P_n^\star(\omega)
- F_{\omega}^{\mu_0}(t)\, \mu_0(\omega)\right)
\convas 0 \text{ as } n\to \infty,
\end{align*}
which is the second term on the RHS of \eqref{eq:CM_split}. It follows that
$$
T_n^{\mathrm{CM}} \convl \int_{\Omega \times \R_{\geq 0}} \G_{0}(\omega, t)^2 \,\rd F_{\omega}^{\mu_0}(t) \, \rd \mu_0(\omega) \text{ as } n \to \infty.
$$
This concludes the proof.
\end{proof}

\subsection{Proof of Theorem \ref{thm:test_statistics_comp}}\label{ap:sec_proofs_null_asymptotics_thm_test_statistics_comp}

This section contains the proof of Theorem \ref{thm:test_statistics_comp}, together with the lemmas needed in the proof. 
Observe that Lemmas \ref{ap:lem_bound_int_g_epsilon}--\ref{ap:lem_DCT_P_hat} also hold when $\mu_{0}$ is replaced by $\mu_{\btheta_0}$.

Lemma \ref{ap:lem_relation_L2_prod_composite} extends Lemma \ref{ap:lem_relation_L2_prod} for the composite null hypothesis.
\begin{lemma}\label{ap:lem_relation_L2_prod_composite}
Let $\rho_{2,f}$ and $\rho_{2,\tilde{\G}_0}$ be the $L^2(P)$ pseudometrics induced by the class $\cF$ and by the Gaussian process $\tilde{\G}_0$, respectively. That is, for $p\ge1$ and $(\omega,t),(\omega^\prime,t^\prime)\in\Omega\times\R_{\ge0}$, define
\begin{align*}
\rho_{p,f}\big((\omega,t),(\omega^\prime,t^\prime)\big)
&\defin
\Big(
\Ebig{\big|f_{\omega,t}(X)-f_{\omega^\prime,t^\prime}(X)\big|^p}
\Big)^{1/p}
\\
\rho_{p,\tilde{\G}_0}\big((\omega,t),(\omega^\prime,t^\prime)\big)
&\defin
\Big(
\Ebig{\big|\tilde{\G}_0(\omega,t)-\tilde{\G}_0(\omega^\prime,t^\prime)\big|^p}
\Big)^{1/p}.
\end{align*}
Assume that the map $(\omega,t)\mapsto \dot F_{\omega}^{\btheta_0}(t)$ is uniformly continuous with respect to $d_{\Omega\times\R_{\ge0}}$. Assume also conditions~\ref{cp2}, \ref{cp3}, and \ref{d}. Then, there exists a nondecreasing function $\phi:[0,\infty)\to[0,\infty)$ with $\phi(r)\to0$ as $r\to0$ such that, for all $(\omega,t),(\omega^\prime,t^\prime)\in\Omega\times\R_{\ge0}$,
\begin{align}\label{eq:rho_2_tilde_G_0_rho_2_f_composite}
\rho_{2,\tilde{\G}_0}\big((\omega,t),(\omega^\prime,t^\prime)\big)
= 
\rho_{2,f}\big((\omega,t),(\omega^\prime,t^\prime)\big)
\le
\phi\!\left(
d_{\Omega\times\R_{\ge0}}
\big((\omega,t),(\omega^\prime,t^\prime)\big)
\right).
\end{align}
In particular, any function that is uniformly continuous with respect to $\rho_{2,\tilde{\G}_0}$ is also uniformly continuous with respect to $d_{\Omega\times\R_{\ge0}}$.
\end{lemma}

Lemma \ref{ap:lem_continuity_sample_paths_G_rho2_comp} is the analogue of Lemma \ref{ap:lem_continuity_sample_paths_G_rho2} for the composite null hypothesis.

\begin{lemma}\label{ap:lem_continuity_sample_paths_G_rho2_comp}
    Under the conditions of Theorem~\ref{thm:convergence_G_n_composite} and Lemma~\ref{ap:lem_relation_L2_prod_composite}, the process $\tilde{\G}_0(\omega,t)$ has almost surely uniformly continuous sample paths on the metric space $(\Omega\times\R_{\ge0}, d_{\Omega\times\R_{\ge0}})$, considering the sample paths as functions from $(\Omega\times\R_{\ge0}, d_{\Omega\times\R_{\ge0}})$ to $(\R,|\cdot|)$.
\end{lemma}

Using Lemmas \ref{ap:lem_relation_L2_prod_composite} and \ref{ap:lem_continuity_sample_paths_G_rho2_comp}, Theorem \ref{thm:test_statistics_comp} can be proved.

\begin{proof}[Proof of Theorem \ref{thm:test_statistics_comp}]
    
We divide the proof into its two statements.

\emph{Proof of \ref{thm:test_statistics_comp_i}.} As explained in the proof of Theorem \ref{thm:convergence_G_n_composite}, $\G_{n,\mu}^{\mu_{\hat{\boldsymbol\theta}}}$ is, up to an $o_{\mathsf{P}}(1)$ remainder (that converges uniformly in $\omega$ and $t$), the empirical process indexed by $\mathcal F$. Therefore, the continuous mapping theorem gives
$$
\tilde T_n^{\mathrm{KS}} \convl \sup_{\omega\in \Omega, t\ge 0}\big|\tilde{\G}_0(\omega,t)\big|.
$$
Using condition \ref{e}, one has
\begin{align*}
\left|
\widehat T_n^{\mathrm{KS}}
-
\tilde T_n^{\mathrm{KS}}
\right|
&\le
\sup_{\omega\in\Omega,\, t\ge0}
\left|
\widehat{\G}_{n,\mu}(\omega,t)
-
\G_{n,\mu}^{\mu_{\hat{\btheta}}}(\omega,t)
\right|
=
o_{\mathsf P}(1),
\end{align*}
so $\widehat T_n^{\mathrm{KS}}$ has the same null limit as $\tilde T_n^{\mathrm{KS}}$.

\emph{Proof of \ref{thm:test_statistics_comp_ii}.} The proof for $\tilde{T}_n^{\mathrm{CM}}$ is completely analogous to that of $T_n^{\mathrm{CM}}$ in Theorem~\ref{thm:test_statistics_simp}. 
It suffices to replace the class $\cY$ by $\cF$, the Gaussian process $\G_0$ by $\tilde{\G}_0$, and Lemmas~\ref{ap:lem_relation_L2_prod} and \ref{ap:lem_continuity_sample_paths_G_rho2} by Lemmas~\ref{ap:lem_relation_L2_prod_composite} and \ref{ap:lem_continuity_sample_paths_G_rho2_comp}, respectively. 

For $\widehat T_n^{\mathrm{CM}}$, one has 
\begin{align*}
\left|
\widehat T_n^{\mathrm{CM}}
-
\tilde T_n^{\mathrm{CM}}
\right|
&\le
\sup_{\omega, t}
\big|
\widehat{\G}_{n,\mu}(\omega,t)
-
\G_{n,\mu}^{\mu_{\hat{\btheta}}}(\omega,t)
\big|
\Big(
2\tilde T_n^{\mathrm{KS}}
+
\sup_{\omega, t}
\big|
\widehat{\G}_{n,\mu}(\omega,t)
-
\G_{n,\mu}^{\mu_{\hat{\btheta}}}(\omega,t)
\big|
\Big) = o_{\mathsf P}(1).
\end{align*}
This concludes the proof.
\end{proof}

\subsection{Proof of Proposition \ref{prop:sup_max_ks}}\label{ap:sec_proofs_null_asymptotics_prop_sup_max_ks}
This section contains the proof of Proposition \ref{prop:sup_max_ks}, as well as technical lemmas required to prove this result.

\begin{lemma}[The sample forms an asymptotic $\delta$-net]
\label{lem:sample-delta-net}
Let $\mathcal{L}_n = \{X_1,X_2, \allowbreak \dots, X_n\}$ be iid draws from a   distribution $\mu$ such that $\Omega \subset \supp{\mu}$.
Then, the sample $\mathcal{L}_n$ is a.s. dense in $\Omega$ as $n$ diverges to infinity, i.e.,
$$
\sup_{\omega\in\Omega}\min_{1\le i\le n} d(\omega,X_i) \to 0 \ \text{ a.s. as } n\to\infty. 
$$
In particular, for every $\varepsilon>0$ the random set $\mathcal{L}_n$ is a.s. a $2\varepsilon$-net of $\Omega$, for $n$ sufficiently large.
\end{lemma}

Using Lemmas \ref{ap:lem_relation_L2_prod} and \ref{lem:sample-delta-net}, Proposition \ref{prop:sup_max_ks} can be proved.
\begin{proof}[Proof of Proposition \ref{prop:sup_max_ks}]
We divide the proof into its two statements.

\emph{Proof of \ref{prop:sup_max_ks_i}.} Consider the class $\cY = \{y_{\omega,t} : \omega \in \Omega,\ t\ge0\}$. Under $H_0$, the class $\cY$ is Donsker.
Throughout the proof, we use the $L^2(\mu)$ pseudometric $\rho_{2,y}$ defined in Lemma~\ref{ap:lem_relation_L2_prod}, identifying $\cY$ with its index set $\Omega\times\R_{\ge0}$.

Since $\cY$ is Donsker, the empirical process is asymptotically equicontinuous with respect to $\rho_{2,y}$ \citepSM[p. 139 in][]%
{VanderVaartWellner2023SM}: for every $\eta >0$,
$$
\lim_{\delta \to 0^+} \limsup_{n\to \infty} \mathsf{P}^*\Bigg\{ \sup_{\substack{f,g\in\cY\\ \rho_{2,y}(f,g) < \delta}}  \left|\G_{n,\mu}^{\mu_0}(f-g)\right| > \eta \Bigg\} = 0.
$$

Fix $\tau>0$ and $\gamma>0$. By asymptotic equicontinuity, there exists $\delta = \delta(\tau,\gamma) > 0$ such that
$$
\limsup_{n\to\infty} \mathsf{P}^*\Bigg\{\sup_{\substack{f,g\in\cY\\ \rho_{2,y}(f,g)<\delta}} \left| \G_{n,\mu}^{\mu_0}(f) - \G_{n,\mu}^{\mu_0}(g)\right| > \tau \Bigg\} < \gamma.
$$
Hence, for $n$ sufficiently large, 
\begin{align}\label{eq:eta_equicont}
\mathsf{P}^*\Bigg\{\sup_{\substack{f,g\in\cY\\ \rho_{2,y}(f,g) < \delta}} \left| \G_{n,\mu}^{\mu_0}(f) - \G_{n,\mu}^{\mu_0}(g)\right| > \tau \Bigg\} < \gamma.
\end{align}

For any function $y_{\omega,t} \in \cY$, Lemma~\ref{ap:lem_relation_L2_prod} shows that the map $\omega \mapsto y_{\omega,t}$ is uniformly continuous with respect to $\rho_{2,y}$, uniformly in $t$. Indeed,
\begin{align*}
\rho_{2,y}(y_{\omega,t}, y_{\omega^\prime,t})
&\le \sqrt{2K_{\mu_0}\,d(\omega,\omega^\prime)},
\end{align*}
for all $\omega, \omega^\prime \in \Omega$ and all $t\ge0$. Therefore, choosing $r > 0$ such that $r < \delta^2/(2K_{\mu_0}),$
it follows that for all $\omega, \omega^\prime \in \Omega$ and all $t\ge0$, $d(\omega, \omega^\prime) < r$ implies $\rho_{2,y}(y_{\omega,t}, y_{\omega^\prime,t}) < \delta$.

By Lemma~\ref{lem:sample-delta-net}, the sample $\mathcal{L}_n$ is dense in $\Omega$, so for $n$ sufficiently large:
$$
\Prob{\sup_{\omega \in \Omega} \min_{1 \le i \le n} d(\omega, X_i) < r} \geq 1 - \gamma.
$$
Thus, with probability at least $1 - \gamma$, for every $\omega \in \Omega$ there exists $X_i \in \mathcal{L}_n$ such that $d(\omega, X_i) < r$, and hence by uniform continuity $\rho_{2,y}(y_{\omega,t}, y_{X_i,t}) < \delta$ for all $t\ge0$. 
This means that $\cY_n \defin \{y_{X_i,t} : i=1,\ldots,n,\ t\ge0\}$ forms a $\delta$-net for $\cY$ with probability at least $1-\gamma$ when $n$ is sufficiently large, i.e., for every $y_{\omega,t} \in \cY$ there exists $y_{X_i,t} \in \cY_n$ such that $\rho_{2,y}(y_{\omega,t}, y_{X_i,t}) < \delta$.

By the triangle inequality, one has
$$
\left| \G_{n,\mu}^{\mu_0}(y_{\omega,t}) \right| \le \left|\G_{n,\mu}^{\mu_0}(y_{X_i,t})\right| + \left| \G_{n,\mu}^{\mu_0}(y_{\omega,t}) - \G_{n,\mu}^{\mu_0}(y_{X_i,t}) \right|.
$$
Therefore, if $\cY_n$ is a $\delta$-net of $\cY$, then
\begin{align*}
    \sup_{y\in \cY} \left| \G_{n,\mu}^{\mu_0}(y) \right| \le \max_{1 \le i \le n}\sup_{t\ge0} \left|\G_{n,\mu}^{\mu_0}(y_{X_i,t})\right| + \sup_{\substack{f,g\in\cY\\ \rho_{2,y}(f,g) < \delta}} \left| \G_{n,\mu}^{\mu_0}(f) - \G_{n,\mu}^{\mu_0}(g) \right|.
\end{align*}
The last term on the right is less than or equal to $\tau$ with (outer) probability at least $1-\gamma$ by \eqref{eq:eta_equicont}. Trivially, $\max_{1 \le i \le n}\sup_{t\ge0} \left|\G_{n,\mu}^{\mu_0}(y_{X_i,t})\right| \le \sup_{y\in \cY} \left| \G_{n,\mu}^{\mu_0}(y) \right|$. 
Hence, for every $\tau>0$ and $\gamma>0$, it holds that
$$
\mathsf{P}^*\Bigg\{\max_{1 \le i \le n}\sup_{t\ge0} \left|\G_{n,\mu}^{\mu_0}(y_{X_i,t})\right| \le \sup_{y\in \cY} \left| \G_{n,\mu}^{\mu_0}(y) \right| \le \max_{1 \le i \le n}\sup_{t\ge0} \left|\G_{n,\mu}^{\mu_0}(y_{X_i,t})\right| + \tau \Bigg\} \ge 1 - 2\gamma.
$$
This means that 
$$
\mathsf{P}^*\Bigg\{\Big|\sup_{y\in\cY} \left|\G_{n,\mu}^{\mu_0}(y)\right| - \max_{1\le i \le n}\sup_{t\ge0} \left| \G_{n,\mu}^{\mu_0}(y_{X_i,t}) \right| \Big| > \tau \Bigg\} < 2\gamma
$$
for $n$ sufficiently large. Therefore, 
$$
\sup_{y\in\cY} \left| \G_{n,\mu}^{\mu_0}(y)\right| - \max_{1\le i \le n}\sup_{t\ge0} \left| \G_{n,\mu}^{\mu_0}(y_{X_i,t}) \right| \convprob{\mathsf{P}^*} 0 \text{ as } n \to \infty,
$$
that is,
$$
\sup_{\omega\in \Omega,\ t\ge0} \left|\G_{n,\mu}^{\mu_0}(\omega,t)\right| = \max_{1\le i \le n}\sup_{t\ge0} \left|\G_{n,\mu}^{\mu_0}(X_i,t)\right| + o_{\mathsf{P}^*}(1).
$$
Condition~\ref{bp2} implies that, conditionally on each $X_i$, the positive distances $d(X_i,X_j)$, $j\ne i$, are almost surely distinct. Since $F_{X_i}^{\mu_0}$ is continuous on $(0,\infty)$, the standard argument for the KS statistic discretizes the supremum over $t$ into a maximum over the observed distances.

\emph{Proof of \ref{prop:sup_max_ks_ii}.}
By \eqref{eq:G_n_Thm_4.1}, it holds that
$$
\sup_{\omega\in\Omega,\,t\ge0}
\left|\G_{n,\mu}^{\mu_{\hat{\btheta}}}(\omega,t)\right|
=
\sup_{\omega\in\Omega,\,t\ge0}
\left|
\frac1{\sqrt n}\sum_{k=1}^n f_{\omega,t}(X_k)
\right|
+o_{\mathsf P^*}(1).
$$
The reduction of the supremum of the empirical process indexed by $\cF$ over $\omega$ to a maximum over the sample centers is analogous to that of \eqref{eq:sup_max_ks}. One replaces $\cY$ by $\cF$, and Lemma~\ref{ap:lem_relation_L2_prod} by Lemma~\ref{ap:lem_relation_L2_prod_composite}. Since the remainder in \eqref{eq:G_n_Thm_4.1} is uniform, it also follows that
$$
\max_{1\le i\le n}\sup_{t\ge0}
\left|\G_{n,\mu}^{\mu_{\hat{\btheta}}}(X_i,t)\right|
=
\max_{1\le i\le n}\sup_{t\ge0}
\left|
\frac1{\sqrt n}\sum_{k=1}^n f_{X_i,t}(X_k)
\right|
+o_{\mathsf P^*}(1).
$$
Therefore,
$$
\sup_{\omega\in\Omega,\,t\ge0}
\left|\G_{n,\mu}^{\mu_{\hat{\btheta}}}(\omega,t)\right|
=
\max_{1\le i\le n}\sup_{t\ge0}
\left|\G_{n,\mu}^{\mu_{\hat{\btheta}}}(X_i,t)\right|
+o_{\mathsf P^*}(1).
$$
It remains to discretize the supremum over $t$ at each sample center. Unlike in the simple case, condition \ref{bp2} does not directly imply continuity of the fitted profile $F_\omega^{\mu_{\hat\btheta}}$. We show instead that any jumps of $F_\omega^{\mu_{\hat\btheta}}$ are asymptotically negligible. Condition~\ref{d} and the Bahadur representation give
\begin{align*}
\sqrt{n}\sup_{\omega\in\Omega,\,t\ge0}
\left|
F_\omega^{\mu_{\hat{\btheta}}}(t)
-F_\omega^{\mu_{\btheta_0}}(t)
-\dot F_\omega^{\btheta_0}(t)^\top(\hat{\btheta}-\btheta_0)
\right|
=o_{\mathsf P}(1).
\end{align*}
The first-order approximation to $F_\omega^{\mu_{\hat\btheta}}$ is continuous in $t$ by condition~\ref{bp2} and the assumed uniform continuity of $\dot F^{\btheta_0}$. This yields $\sqrt{n}\sup_{\omega\in\Omega, t>0} \left\{ F_\omega^{\mu_{\hat\btheta}}(t) - F_\omega^{\mu_{\hat\btheta}}(t-) \right\} = o_{\mathsf P}(1)$, which, together with the standard KS argument, gives $\max_{1\le i\le n}\sup_{t\ge0} \left|\G_{n,\mu}^{\mu_{\hat\btheta}}(X_i,t)\right| = \max_{1\le i,j\le n} \left| \G_{n,\mu}^{\mu_{\hat\btheta}} (X_i,d(X_i,X_j)) \right| + o_{\mathsf P}(1)$. This concludes the proof.
\end{proof}

\section{Proofs of Section \ref{sec:non_null_asymptotics}} \label{ap:sec_proofs_non_null_asymptotics}

\subsection{Fixed alternatives} \label{ap:sec_proofs_fixed_alternatives}

\subsubsection{Simple null hypothesis} \label{ap:sec_proofs_fixed_alternatives_simple}

Theorem \ref{thm:fixed_alternative_divergence_simple} states the consistency of the KS and CM test statistics under fixed alternatives. We now turn to the CM test statistic. To prove the result for the CM test statistic, Lemmas \ref{lem:delta_regularity}, \ref{lem:robust_separation} and \ref{lem:positive_mass_region} are needed.

\begin{lemma}\label{lem:delta_regularity}
Let $M\defin\mathrm{diam}(\Omega)$ and define $\Delta(\omega,t)\defin F^\mu_\omega(t)-F^{\mu_0}_\omega(t)$ for $(\omega,t)\in\Omega\times[0,M]$.
Condition~\ref{bp2} gives, for any
$\omega,\omega^\prime\in\Omega$ and each $t\in[0,M]$,
$$
|\Delta(\omega,t)-\Delta(\omega^\prime,t)|
\le
(K_\mu+K_{\mu_0})\,d(\omega,\omega^\prime).
$$
That is, for each $t\in[0,M]$ the map $\omega\mapsto \Delta(\omega,t)$ is Lipschitz on $\Omega$ with Lipschitz constant bounded uniformly over $t\in[0,M]$.
\end{lemma}

\begin{lemma}\label{lem:robust_separation}
Assume that condition \ref{bp2} holds.
If there exist $\omega_0\in\Omega$ and $t_0\in[0,M]$ such that $\Delta(\omega_0,t_0)\neq 0$, then there exist $\varepsilon_\omega>0$, $\varepsilon_t>0$, and $c_0>0$ such that
$$
|\Delta(\omega,t)|\ge c_0
\qquad\forall\,\omega\in B(\omega_0,\varepsilon_\omega),\ \forall\,t\in (t_0-\varepsilon_t,t_0+\varepsilon_t)\cap[0,M].
$$
\end{lemma}
\begin{lemma}\label{lem:F_H_differ_at_t}
If two continuous cdfs $F$ and $H$ on $[0,M]$ differ, then they differ at some $t\in\supp F$. 
\end{lemma} 
\begin{lemma}\label{lem:positive_mass_region}
Let $\omega_0\in\supp{\mu}$ and $t_0\in\supp{F_{\omega_0}^\mu}$. For $\varepsilon_\omega,\varepsilon_t>0$, set
$$
B_0\defin B(\omega_0,\varepsilon_\omega),
\qquad
I_0\defin(t_0-\varepsilon_t,t_0+\varepsilon_t)\cap[0,M].
$$
If $X,Y$ are iid with law $\mu$, then
$$
\Probbig{X\in B_0,\ d(X,Y)\in I_0}>0.
$$
\end{lemma}

\begin{proof}[Proof of Theorem \ref{thm:fixed_alternative_divergence_simple}]
Begin with the KS test statistic. For $(\omega,t)\in\Omega\times\R_{\ge 0}$, write
$$
F_{n,\omega}^\mu(t)-F^{\mu_0}_\omega(t)
=
\bigl(F_{n,\omega}^\mu(t)-F^\mu_\omega(t)\bigr)
+
\bigl(F^\mu_\omega(t)-F^{\mu_0}_\omega(t)\bigr),
$$
and define the empirical process under $\mu$ ($\mu \neq \mu_0$) by
\begin{align}\label{eq:G_n_nu}
\G_{n,\mu}^{\mu}(\omega,t)
\defin
\sqrt{n}\,\bigl(F_{n,\omega}^\mu(t)-F^\mu_\omega(t)\bigr).    
\end{align}
Next, define $\Delta(\omega,t)\defin F^\mu_\omega(t)-F^{\mu_0}_\omega(t)$.
Since $\mu\neq\mu_{0}$, there exists at least one pair $(\omega_0,t_0)\in\Omega\times\R_{\ge 0}$ such that $\Delta(\omega_0,t_0)\neq 0$. Therefore,
$$
\delta
\defin
\sup_{\omega\in\Omega, t\ge 0}
|\Delta(\omega,t)|
\ \ge\
|\Delta(\omega_0,t_0)|
\ >\ 0.
$$

Since $\cY$ is $\mu$-Donsker, it is also $\mu$-Glivenko--Cantelli. Therefore,
\begin{align}\label{eq:e_n}
e_n \defin \sup_{\omega\in\Omega, t\ge0}\bigl|F_{n,\omega}^\mu(t)-F_\omega^\mu(t)\bigr|
=
o_{\mathsf{P}^*}(1).
\end{align}

Now, write
\begin{align*}
\frac{T_n^{\mathrm{KS}}}{\sqrt n}
=
\sup_{\omega\in\Omega, t\ge0}\left|\Delta(\omega,t)+\bigl(F_{n,\omega}^\mu(t)-F_\omega^\mu(t)\bigr)\right|.
\end{align*}
Using the inequality $\bigl|\sup |a+b|-\sup |a|\bigr|\le \sup |b|$, we obtain
$$
\left|\frac{T_n^{\mathrm{KS}}}{\sqrt n}-\delta\right|
\le
e_n
=
\frac{\|\G_{n,\mu}^{\mu}\|_\infty}{\sqrt n}
=
o_{\mathsf{P}^*}(1).
$$
Therefore, $T_n^{\mathrm{KS}}/{\sqrt n}\convprob{\mathsf{P}^*}\delta,$ and in particular, $T_n^{\mathrm{KS}}\convprob{\mathsf{P}^*}\infty$.

Focus now on the CM test statistic. Since either condition \ref{a2} holds, or condition \ref{a1} holds and $\supp{\mu_{0}} \subset \supp{\mu}$, under $H_1$, by Lemmas \ref{lem:F_H_differ_at_t} and \ref{lem:positive_mass_region} there exist $\omega_0\in\supp{\mu}$ and $t_0\in\supp{F_{\omega_0}^\mu}$ such that
$F^\mu_{\omega_0}(t_0)\neq F^{\mu_{0}}_{\omega_0}(t_0)$, i.e.  $\Delta(\omega_0,t_0)\neq 0$.
Apply Lemma~\ref{lem:robust_separation} to obtain $\varepsilon_\omega>0$, $\varepsilon_t>0$, and $c_0>0$ such that
$$
|\Delta(\omega,t)|\ge c_0
\quad\forall\,(\omega,t)\in B_0\times I_0,
$$
where $B_0=B(\omega_0,\varepsilon_\omega)$ and $I_0=(t_0-\varepsilon_t,t_0+\varepsilon_t)\cap[0,M]$.
Since $\omega_0\in\supp{\mu}$, Lemma~\ref{lem:positive_mass_region} yields
\begin{equation}\label{eq:positive_mass_event}
p_0 \defin \Probbig{X\in B_0,\ d(X,Y)\in I_0}>0,
\end{equation}
for iid $X,Y\sim\mu$.

Define
$$
q(x,y)
\defin
1_{\{x\in B_0,\ d(x,y)\in I_0\}}
\quad 
\text{and}
\quad
A_n
\defin
\frac{1}{n^2}
\sum_{i=1}^n\sum_{j=1}^n
q(X_i,X_j).
$$
We first show that $A_n\convas p_0$.
Define the symmetrized kernel
$$
\widetilde q(x,y)
\defin
\frac12\{q(x,y)+q(y,x)\},
$$
and the corresponding $U$-statistic
$$
U_n
\defin
\frac{1}{\binom n2}
\sum_{1\le i<j\le n}
\widetilde q(X_i,X_j).
$$
By the strong law of large numbers for $U$-statistics of order two
applied to $\widetilde q$
\citepSM[see, e.g., Theorem 5.4A of][]{Serfling1980SM},
$
U_n
\xrightarrow{\mathrm{a.s.}}
\Ebig{\widetilde q(X,Y)}.
$
Since $X,Y$ are iid,
$$
\Ebig{\widetilde q(X,Y)}
=
\Ebig{q(X,Y)}
=
p_0.
$$
Moreover,
$$
A_n
=
\left(1-\frac1n\right)U_n
+
\frac{1}{n^2}\sum_{i=1}^n q(X_i,X_i),
$$
and the second term is bounded by $1/n$. Hence, $A_n \convas p_0.$

If $e_n<c_0/2$, for every pair $(i,j)$ such that
$q(X_i,X_j)=1$, Lemma~\ref{lem:robust_separation} gives
$$
\begin{aligned}
&
\big|
F_{n,X_i}^{\mu}(d(X_i,X_j))
-
F_{X_i}^{\mu_0}(d(X_i,X_j))
\big|
\ge
c_0-e_n
>
\frac{c_0}{2},
\end{aligned}
$$
with $c_0>0$ as defined in Lemma~\ref{lem:robust_separation}.

Observe that one can write
\begin{align*}%
\frac{1}{n}T_n^{\mathrm{CM}}
=
\frac{1}{n^2}
\sum_{i=1}^n\sum_{j=1}^n
\left(
F_{n,X_i}^{\mu}(d(X_i,X_j))
-
F_{X_i}^{\mu_0}(d(X_i,X_j))
\right)^2.
\end{align*}
Since all the terms %
are nonnegative, if $e_n<c_0/2$, it follows that,
$
T_n^{\mathrm{CM}}
\ge
c_0^2 n A_n /4.
$
Since $A_n\convas p_0>0$,
$
\Prob{A_n<p_0/2}\to0,
$
and therefore,
$$
\mathsf P^*\left(
\frac{T_n^{\mathrm{CM}}}{n}
\ge
\frac{c_0^2p_0}{8}
\right)\to1.
$$
Hence,
$T_n^{\mathrm{CM}}\convprob{\mathsf P^*}\infty.$

\end{proof}

\subsubsection{Composite null hypothesis} \label{ap:sec_proofs_fixed_alternatives_composite}

Theorem \ref{thm:fixed_alt_divergence_composite} states the consistency of the KS and CM test statistics under fixed alternatives, when the null hypothesis is composite and a parameter is estimated.

\begin{proof}[Proof of Theorem~\ref{thm:fixed_alt_divergence_composite}]
Begin with the KS test statistic. Define
\begin{align*}
\Delta^\star(\omega,t)\defin F^\mu_\omega(t)-F^{\mu_{\btheta_1}}_\omega(t),
\qquad
\delta^\star\defin
\sup_{(\omega,t)\in\Omega\times\R_{\ge 0}}
\bigl|\Delta^\star(\omega,t)\bigr|.
\end{align*}
Since $\mu\notin\{\mu_{\btheta}:\btheta\in\bTheta\}$, in particular
$\mu\neq\mu_{\btheta_1}$, and hence there exists
$(\omega_0,t_0)\in\Omega\times\R_{\ge 0}$ such that
$\Delta^\star(\omega_0,t_0)\neq 0$. Therefore,
$\delta^\star\ge|\Delta^\star(\omega_0,t_0)|>0$.

Define
\begin{align*}
r_n \defin \sup_{\omega\in\Omega, t\ge 0}
\bigl|F^{\mu_{\hat{\btheta}}}_\omega(t)-F^{\mu_{\btheta_1}}_\omega(t)\bigr|.
\end{align*}
By \eqref{eq:theta_star_profile_continuity}, for every $\varepsilon>0$,
there exists $\eta>0$ such that if $
\|\btheta-\btheta_1\|<\eta$, then
$
\sup_{\omega \in\Omega, t\ge 0}
\bigl|F^{\mu_{\btheta}}_\omega(t)-F^{\mu_{\btheta_1}}_\omega(t)\bigr|
<\varepsilon.$
Since $\hat{\btheta}\convp\btheta_1$, it follows that
$$
\mathsf{P}^*(r_n>\varepsilon)
\le
\Prob{\|\hat{\btheta}-\btheta_1\|\ge\eta}
\to0.
$$
Thus, $r_n=o_{\mathsf P^*}(1)$.

It holds that
$e_n = o_{\mathsf P^*}(1)$ by \eqref{eq:e_n}. Now, write
$$
\frac{\tilde T_n^{\mathrm{KS}}}{\sqrt n}
=
\sup_{\omega\in\Omega,t\ge0}
\left|
\Delta^\star(\omega,t)
+
\bigl(F_{n,\omega}^\mu(t)-F_\omega^\mu(t)\bigr)
-
\bigl(F_\omega^{\mu_{\hat{\btheta}}}(t)
      -F_\omega^{\mu_{\btheta_1}}(t)\bigr)
\right|.
$$
Using the inequality
$\bigl|\sup |a+b|-\sup |a|\bigr|\le\sup|b|$, we obtain
$$
\bigg|
\frac{\tilde T_n^{\mathrm{KS}}}{\sqrt n}
-\delta^\star
\bigg|
\le
e_n+r_n
=
o_{\mathsf P^*}(1).
$$
Therefore,
$
\tilde T_n^{\mathrm{KS}}/{\sqrt n}
\convprob{\mathsf P^*}
\delta^\star,
$
and in particular,
$\tilde T_n^{\mathrm{KS}}\convprob{\mathsf P^*}\infty$.

Finally, we show that $\widehat T_n^{\mathrm{KS}}\convprob{\mathsf P^*}\infty$.
Using condition \ref{e} and $P_n\bpsi_{\hat{\btheta}}=o_{\mathsf P}(1)$, one has
\begin{align}\label{eq:sup_G_n_G_o_P_n12}
\sup_{\omega\in\Omega, t\ge0}
\big|
\widehat{\G}_{n,\mu}(\omega,t)
-
\G_{n,\mu}^{\mu_{\hat{\btheta}}}(\omega,t)
\big|
&\le
\sqrt{n}
\sup_{\omega\in\Omega, t\ge0}
\|\dot F_\omega^{\hat{\btheta}}(t)\|\,\|\widehat{\bV}^{-1}\|\|P_n\bpsi_{\hat{\btheta}}\|
=
o_{\mathsf P}(\sqrt n).
\end{align}

Therefore,
$$
\big|
\widehat T_n^{\mathrm{KS}}
-
\tilde T_n^{\mathrm{KS}}
\big|
\le
\sup_{\omega\in\Omega,\,t\ge0}
\Big|
\widehat{\G}_{n,\mu}(\omega,t)
-
\G_{n,\mu}^{\mu_{\hat{\btheta}}}(\omega,t)
\Big|
=
o_{\mathsf P}(\sqrt{n}).
$$
This concludes the result for the KS test statistic.

Focus now on the CM test statistic. By Lemma \ref{lem:F_H_differ_at_t}, there exist $\omega_1\in\supp{\mu}$ and $t_1\in\supp{F_{\omega_1}^\mu}$ such that $\Delta^\star(\omega_1,t_1)\neq0$. Applying Lemma~\ref{lem:robust_separation}, with $\mu_0$ replaced by $\mu_{\btheta_1}$, we obtain $\varepsilon_\omega>0$, $\varepsilon_t>0$, and $c_1>0$ such that $|\Delta^\star(\omega,t)|\ge c_1$ for every $(\omega,t)\in B_1\times I_1,$ where $B_1=B(\omega_1,\varepsilon_\omega)$ and $I_1=(t_1-\varepsilon_t,t_1+\varepsilon_t)\cap[0,M]$. Since $\omega_1\in\supp{\mu}$, Lemma~\ref{lem:positive_mass_region} yields
$$
p_1
\defin
\Probbig{X\in B_1,\ d(X,Y)\in I_1}>0,
$$
for iid $X,Y\sim\mu$. Define
$$
q(x,y)
\defin
1_{\{x\in B_1,\ d(x,y)\in I_1\}},
\qquad
A_n
\defin
\frac{1}{n^2}
\sum_{i=1}^n\sum_{j=1}^n q(X_i,X_j).
$$
By the same $U$-statistic argument as in the proof of
Theorem~\ref{thm:fixed_alternative_divergence_simple}, $A_n\convas p_1.$

Recall that $e_n=o_{\mathsf P^*}(1)$ and $r_n=o_{\mathsf P^*}(1)$.
If $e_n+r_n<c_1/2$, for every pair $(i,j)$ such that
$q(X_i,X_j)=1$,
$$
\begin{aligned}
\big|
F_{n,X_i}^{\mu}(d(X_i,X_j))
-
F_{X_i}^{\mu_{\hat{\btheta}}}(d(X_i,X_j))
\big|
\ge
\big|\Delta^\star(X_i,d(X_i,X_j))\big|
-e_n-r_n
>
\frac{c_1}{2}.
\end{aligned}
$$
Since the CM statistic can be expressed as
$$
\frac{1}{n}\tilde T_n^{\mathrm{CM}}
=
\frac{1}{n^2}
\sum_{i=1}^n\sum_{j=1}^n
\left(
F_{n,X_i}^{\mu}(d(X_i,X_j))
-
F_{X_i}^{\mu_{\hat{\btheta}}}(d(X_i,X_j))
\right)^2,
$$
one has that $\tilde T_n^{\mathrm{CM}}\ge n c_1^2A_n/4.$
Hence,
$
\mathsf P^*\big(
\tilde T_n^{\mathrm{CM}}
\ge
n c_1^2p_1/8
\big)\to1,
$
and therefore
$\tilde T_n^{\mathrm{CM}}\convprob{\mathsf P^*}\infty$.

Finally, it holds that
$
\sup_{\omega\in\Omega,\,t\ge0}
\big|
\G_{n,\mu}^{\mu_{\hat{\btheta}}}(\omega,t)
\big|
\le \sqrt{n}
$
which together with \eqref{eq:sup_G_n_G_o_P_n12}, gives
$
\sup_{\omega\in\Omega,\,t\ge0}
\big|
\widehat{\G}_{n,\mu}(\omega,t)
\big|
=
O_{\mathsf P}(\sqrt{n}).
$
Therefore,
$$
\begin{aligned}
\frac{1}{n}
\big|
\widehat T_n^{\mathrm{CM}}
-
\tilde T_n^{\mathrm{CM}}
\big|
&\le
\frac{1}{\sqrt n}
\sup_{\omega,t}
\big|
\widehat{\G}_{n,\mu}(\omega,t)
-
\G_{n,\mu}^{\mu_{\hat{\btheta}}}(\omega,t)
\big|
\\
&\quad\times
\bigg(
\frac{1}{\sqrt n}\sup_{\omega,t}
|\widehat{\G}_{n,\mu}(\omega,t)|
+
\frac{1}{\sqrt n}\sup_{\omega,t}
|\G_{n,\mu}^{\mu_{\hat{\btheta}}}(\omega,t)|
\bigg)
\\
&=
o_{\mathsf P}(1).
\end{aligned}
$$
This yields
$\mathsf P^*\big(
\widehat T_n^\mathrm{CM}
\ge
n c_1^2p_1/16
\big)\to1,$
and hence
$\widehat T_n^{\mathrm{CM}}\convprob{\mathsf P^*}\infty$.
\end{proof}

\subsection{Local alternatives}\label{ap:sec_local_alternatives}

\begin{proof}[Proof of Proposition~\ref{prop:bernoulli}]
Define the empirical process under $Q_n$ by $\G_{n,Q_n}^{Q_n}(f) \defin \sqrt{n}(P_n-Q_n)f$, $f\in \cF$. Consider two mutually independent iid sequences $U_i$ and $V_i$, $i\ge1$, where $U_i\sim P$ and $V_i\sim G$. For each $n\in\mathbb{N}$, consider an iid sequence $B_{n,i}$, $i\ge1$, independent of $(U_i)_{i\ge1}$ and $(V_i)_{i\ge1}$, where $B_{n,i}\sim\mathrm{Bernoulli}(n^{-1/2})$. Define $X_{n,i} \defin U_i$ if $B_{n,i}=0$ and $X_{n,i} \defin V_i$ if $B_{n,i}=1$. 
Then $X_{n,1},\ldots,X_{n,n}$ are iid with common distribution $Q_n=\big(1-1/\sqrt{n}\big)P+(1/\sqrt{n}) G$. Observe that for any $f \in \cF$, it holds that $f(X_{n,i})=(1-B_{n,i})f(U_i)+B_{n,i} f(V_i)$. Write $N_{n,1}\defin\sum_{i=1}^n B_{n,i}$ and $N_{n,0}\defin n-N_{n,1}$. One has that
\begin{align}
\G_{n,Q_n}^{Q_n}(f)
&=\frac{1}{\sqrt{n}}\sum_{i=1}^n\Bigl((1-B_{n,i})f(U_i)+B_{n,i} f(V_i)-Q_n f\Bigr)\nonumber\\
&=\frac{1}{\sqrt{n}}\sum_{i=1}^n(1-B_{n,i})\Bigl(f(U_i)-Pf\Bigr)\nonumber
+\frac{1}{\sqrt{n}}\sum_{i=1}^n B_{n,i}\Bigl(f(V_i)-Gf\Bigr)
\\
& \quad +\frac{1}{\sqrt{n}}\Bigl(N_{n,0}\, Pf+N_{n,1}\,Gf-nQ_n f\Bigr)\nonumber\\
&=\frac{1}{\sqrt{n}}\sum_{i=1}^n(1-B_{n,i})\Bigl(f(U_i)-Pf\Bigr) \label{eq:local_alt_split_1}
+\frac{1}{\sqrt{n}}\sum_{i=1}^n B_{n,i}\Bigl(f(V_i)-Gf\Bigr)\\ 
&\quad
+\sqrt{n}\left(\frac{N_{n,0}}{n}-\Big(1-\frac{1}{\sqrt{n}}\Big)\right)Pf\label{eq:local_alt_split_2}
+\sqrt{n}\left(\frac{N_{n,1}}{n}-\frac{1}{\sqrt{n}}\right)Gf,
\end{align}
Since $N_{n,0}/n-\big(1-1/\sqrt{n}\big)=-\big(N_{n,1}/n-1/\sqrt{n}\big)$, \eqref{eq:local_alt_split_2} combines into
$$
R_n(f)
\defin
\frac{N_{n,1} - \sqrt{n}}{\sqrt{n}}\big(G-P\big)f.
$$
Since $N_{n,1}\sim\mathrm{Bin}(n,n^{-1/2})$, it follows that 
$\E{(N_{n,1}-\sqrt{n})^2}/n = \mathrm{Var}(N_{n,1})/n=(1/\sqrt{n})\allowbreak \left(1-1/\sqrt{n}\right)\to 0$. Thus, since $\sup_{f\in\cF}\big|(G-P)f\big|<\infty$ by assumption, it holds that
\begin{equation}\label{eq:local_alt_Rn_vanish}
\sup_{f\in\cF}|R_n(f)|\convp0.
\end{equation}

Next, consider the second sum in \eqref{eq:local_alt_split_1}. Define, for $n\ge0$, the empirical process based on $V_i\sim G$ by
$$
\G_{n,G}^{G}(f)\defin
\begin{cases}
\frac{1}{\sqrt{n}}\sum_{i=1}^{n}\Big(f(V_i)-Gf\Big), & n\ge1,\\
0, & n=0.
\end{cases}
$$
Let $I_{n,1}\defin\{i\in\{1,\ldots,n\}:B_{n,i}=1\}$ and, on the event $\{N_{n,1}\ge 1\}$, write $I_{n,1}=\{j_{n,1}<\cdots<j_{n,N_{n,1}}\}$. Then,
$$
\frac{1}{\sqrt{n}}\sum_{i=1}^n B_{n,i}\Big(f(V_i)-Gf\Big)
=
\frac{1}{\sqrt{n}}\sum_{\ell=1}^{N_{n,1}}\Big(f(V_{j_{n,\ell}})-Gf\Big).
$$
Since $(V_i)_{i\ge1}$ are iid and independent of $(B_{n,i})_{i=1}^n$ for every $n$, then conditional on $B_{n,1},\ldots,B_{n,n}$, if $N_{n,1}=m$, then $(V_{j_{n,\ell}})_{\ell=1}^{m}$ has the same law as $(V_\ell)_{\ell=1}^{m}$. Hence, the following equality in distribution also holds unconditionally:
$$
\frac{1}{\sqrt{n}}\sum_{i=1}^n B_{n,i}\Big(f(V_i)-Gf\Big)
\stackrel{d}{=}
\sqrt{\frac{N_{n,1}}{n}}\ \G_{N_{n,1},G}^{G}(f),
$$
as random elements in $\ell^\infty(\cF)$.
By the Donsker property, $\G_{n,G}^{G}\convl \G_{G}$ in $\ell^\infty(\cF)$ for some tight limit $\G_{G}$. Since $\cF$ is Donsker under $G$, the empirical process $\G_{n,G}^{G}$ is asymptotically tight in
$\ell^\infty(\cF)$; \citepSM[see, e.g., p. 139 in][]{VanderVaartWellner2023SM}.
By the continuous mapping theorem for asymptotic tightness
\citepSM[see, e.g., Problem~7(i) in Section~1.3 of][]{VanderVaartWellner2023SM}, the sequence $\sup_{f \in \cF}|\G_{n,G}^{G}(f)|$ is
asymptotically tight in $\R$. In particular, for every $\varepsilon>0$ there exists $M_\varepsilon<\infty$ such that
\begin{equation}\label{eq:local_alt_asympt_tight_G}
\limsup_{n\to\infty}\mathsf{P}\bigg(\sup_{f \in \cF}|\G_{n,G}^{G}(f)|>M_\varepsilon\bigg)\le \frac{\varepsilon}{2}.
\end{equation}

For every fixed $m>0$, one has
\begin{equation}\label{eq:local_alt_random_index_bound}
\mathsf{P}\bigg(\sup_{f \in \cF}|\G_{N_{n,1},G}^{G}(f)|>M_\varepsilon\bigg)
\le
\mathsf{P}(N_{n,1}\le m)+
\sup_{k\ge m}\mathsf{P}\bigg(\sup_{f \in \cF}|\G_{k,G}^{G}(f)|>M_\varepsilon\bigg).
\end{equation}
By \eqref{eq:local_alt_asympt_tight_G}, there exists $m_0\in\mathbb{N}$ such that
$$
\sup_{k\ge m_0}\mathsf{P}\bigg(\sup_{f \in \cF}|\G_{k,G}^{G}(f)|>M_\varepsilon\bigg)\le \varepsilon.
$$
Applying \eqref{eq:local_alt_random_index_bound} with $m=m_0$, and since $N_{n,1}\convprob{\mathsf{P}}\infty$, we obtain
$$
\limsup_{n\to\infty}\mathsf{P}\bigg(\sup_{f \in \cF}|\G_{N_{n,1},G}^{G}(f)|>M_\varepsilon\bigg)
\le
\limsup_{n\to\infty}\mathsf{P}(N_{n,1}\le m_0)+\varepsilon
=
\varepsilon.
$$
Therefore,
$$
\sup_{f \in \cF}|\G_{N_{n,1},G}^{G}(f)|=O_{\mathsf{P}}(1).
$$
Moreover, $N_{n,1}/n \convp 0$ because $\E{N_{n,1}/n} = n^{-1/2}$. Consequently,
\begin{align}\label{eq:local_alt_G_part_vanish}
\sup_{f \in \cF}
\bigg|
\frac{1}{\sqrt{n}}\sum_{i=1}^n B_{n,i}\big(f(V_i)-Gf\big)
\bigg|
\stackrel{d}{=}
\sqrt{\frac{N_{n,1}}{n}} \sup_{f \in \cF}|\G_{N_{n,1},G}^{G}(f)|
\convp 0.
\end{align}

It remains to handle the first term in \eqref{eq:local_alt_split_1}. Define, for $n\ge0$, the empirical process based on $U_i\sim P$ by
$$
\G_{n,P}^{P}(f)\defin
\begin{cases}
\frac{1}{\sqrt{n}}\sum_{i=1}^{n}\big(f(U_i)-Pf\big), & n\ge1,\\
0, & n=0.
\end{cases}
$$
Then, one has
$$
\frac{1}{\sqrt{n}}\sum_{i=1}^n (1-B_{n,i})\Big(f(U_i)-Pf\Big)
=
\G_{n,P}^{P}(f)
-
\frac{1}{\sqrt{n}}\sum_{i=1}^n B_{n,i}\Big(f(U_i)-Pf\Big).
$$
For the second term on the RHS of the equality, by the definition of $j_{n,\ell}$, one has
$$
\frac{1}{\sqrt{n}}\sum_{i=1}^n B_{n,i}\Big(f(U_i)-Pf\Big)
=
\frac{1}{\sqrt{n}}\sum_{\ell=1}^{N_{n,1}}\Big(f(U_{j_{n,\ell}})-Pf\Big).
$$
Since $(U_i)_{i\ge1}$ are iid and independent of $(B_{n,i})_{i=1}^n$, then conditional on $B_{n,1},\ldots,B_{n,n}$, $(U_{j_{n,\ell}})_{\ell=1}^{m}$ has the same law as $(U_\ell)_{\ell=1}^{m}$, if $N_{n,1}=m$. Hence, the following equality in distribution also holds unconditionally:
$$
\frac{1}{\sqrt{n}}\sum_{i=1}^n B_{n,i}\Big(f(U_i)-Pf\Big)
\stackrel{d}{=}
\sqrt{\frac{N_{n,1}}{n}}\ \G_{N_{n,1}, P}^{P}(f)
$$
as random elements in $\ell^\infty(\cF)$.
Since $\cF$ is $P$-Donsker, the same argument as above yields
$\sup_{f \in \cF}|\G_{N_{n,1},P}^{P}(f)|=O_{\mathsf{P}}(1)$, and hence
\begin{equation}\label{eq:local_alt_P_part_correction_vanish}
\sup_{f \in \cF}
\Big|
\frac{1}{\sqrt{n}}\sum_{i=1}^n B_{n,i}\Big(f(U_i)-Pf\Big)
\Big|
\stackrel{d}{=}
\sqrt{\frac{N_{n,1}}{n}} \sup_{f \in \cF}|\G_{N_{n,1},P}^{P}(f)|
\convp 0.
\end{equation}

Finally, observe that applying the triangle inequality and taking suprema, it follows from \eqref{eq:local_alt_split_1}--\eqref{eq:local_alt_split_2}  that 
\begin{align*}
\sup_{f \in \cF}\big|\G_{n,Q_n}^{Q_n}(f)-\G_{n,P}^{P}(f)\big|
&\le
\sup_{f \in \cF}
\Big|\frac{1}{\sqrt n}\sum_{i=1}^n B_{n,i}\Big(f(U_i)-Pf\Big)\Big|
\\
&\quad +
\sup_{f \in \cF}
\Big|\frac{1}{\sqrt n}\sum_{i=1}^n B_{n,i}\Big(f(V_i)-Gf\Big)\Big|
\\
&\quad + \sup_{f \in \cF}|R_n(f)|.
\end{align*}
Combining \eqref{eq:local_alt_Rn_vanish}, \eqref{eq:local_alt_G_part_vanish}, and \eqref{eq:local_alt_P_part_correction_vanish} gives \eqref{eq:sup_bernoulli_prob}.
\end{proof}

\begin{proof}[Proof of Corollary \ref{cor:bernoulli}]
Using the coupled versions provided by Proposition~\ref{prop:bernoulli}, one has $\sup_{f\in\cF}\big|\G_{n,Q_n}^{Q_n}(f)-\G_{n,P}^{P}(f)\big|\convp0$. Since $\cF$ is $P$-Donsker, $\G_{n,P}^{P}$ converges weakly in
$\ell^\infty(\cF)$. Hence, by Slutsky's theorem, $\G_{n,Q_n}^{Q_n}$ has the same weak limit as $\G_{n,P}^{P}$.
\end{proof}

Lemma \ref{ap:lem_DCT_P_hat_local} is the analogue of Lemma~\ref{ap:lem_DCT_P_hat} for the local alternative.
Let $M=\operatorname{diam}(\Omega)$. For $0<B<\infty$ and a modulus of continuity $w:[0,\infty)\to[0,\infty)$, let $\mathcal G_{B,w}$ denote the class of functions on $\Omega\times[0,M]$ that are uniformly bounded by $B$ and have modulus of continuity $w$, that is,
$
\mathcal G_{B,w}
\defin
\{
g:\Omega\times[0,M]\to\R:
\|g\|_\infty\le B,\ 
|g(z)-g(z')|\le w(d_{\Omega\times\mathbb R_{\ge0}}(z,z'))
\text{ for all }  z,z'\in \Omega\times[0,M]
\}.
$
For this reason, Lemma~\ref{ap:lem_DCT_P_hat_local} gives a uniform version, which is the extra ingredient needed in the proof of Theorem~\ref{thm:local_alt_simple}.
\begin{lemma}\label{ap:lem_DCT_P_hat_local}
Under the local alternative \eqref{eq:local_alt_Qn_def}, it holds that
$$
\sup_{g\in\mathcal G_{B,w}}
\bigg|
\int_{\Omega\times[0,M]} g(\omega,t)\,\rd F_{n,\omega}^{Q_n}(t)\,\rd P_n(\omega)
-
\int_{\Omega\times[0,M]} g(\omega,t)\,\rd F_{\omega}^{\mu_0}(t)\,\rd \mu_0(\omega)
\bigg|
=o_{\mathsf{P}}(1).
$$
\end{lemma}

\subsubsection{Simple null hypothesis}\label{ap:sec_local_alternatives_simple}

This section contains the proof of Theorem \ref{thm:local_alt_simple}, which establishes the weak limit of the KS and CM statistics under local alternatives, when the null hypothesis is simple.
\begin{proof}[Proof of Theorem \ref{thm:local_alt_simple}]
Begin with the KS statistic. Observe that since $0\le y_{\omega,t}\le 1$ for all $(\omega,t)\in\Omega\times\R_{\ge0}$, it holds that $\sup_{\omega,t}\big|(G-\mu_{0})y_{\omega,t}\big|\le 2$, so the conditions of Proposition \ref{prop:bernoulli} are satisfied for $\cY$. Under the local alternative defined in \eqref{eq:local_alt_Qn_def}, an application of Corollary \ref{cor:bernoulli} to the class of functions $\cY$ yields $\G_{n,Q_n}^{Q_n} \convl  \G_{0}$ in $\ell^\infty(\cY)$.
Combined with \eqref{eq:local_alt_process_decomp}, this gives
\begin{equation}\label{eq:local_alt_weak_limit_process}
\sqrt{n}\big(F_{n,\omega}^{Q_n}(t)-F^{\mu_0}_\omega(t)\big)
\convl
\G_{0}(\omega,t)+h(\omega,t)
\qquad\text{in }\ell^\infty(\cY).
\end{equation}
Finally, by an application of the continuous mapping theorem, we obtain
\begin{equation*}
T_n^{\mathrm{KS}}
\convl
\sup_{\omega\in\Omega, t \ge 0}
\left|\G_{0}(\omega,t)+h(\omega,t)\right|.
\end{equation*}

Focus now on the CM statistic.  The proof is analog to that of Theorem~\ref{thm:test_statistics_simp}. From \eqref{eq:local_alt_process_decomp}, we have
$$
T_n^{\mathrm{CM}}
=
\int_{\Omega\times\mathbb R_{\ge0}}
\big(\G_{n,Q_n}^{Q_n}(\omega,t)+h(\omega,t)\big)^2\,
\rd F_{n,\omega}^{Q_n}(t)\,\rd P_n(\omega).
$$
By Corollary \ref{cor:bernoulli},
$$
\G_{n,Q_n}^{Q_n}+h \convl  \G_0+h
\quad\text{in }\ell^\infty(\cY).
$$
In the same way as in the proof of Theorem \ref{thm:test_statistics_simp}, it can be shown that $\G_0+h$ is separable, and therefore the Skorohod representation theorem yields a probability space
on which it holds that
\begin{align}\label{eq:aux_CM_sup_convas_2}
\sup_{\omega\in\Omega, t \ge 0}
\big|\G_{n,Q_n}^{Q_n}(\omega,t)-\G_0(\omega,t)\big|
\convas 0.
\end{align}
Now, write 
\begin{align}
&\left| \notag
\int \big(\G_{n,Q_n}^{Q_n}+h\big)^2\,\rd F_{n,\omega}^{Q_n}\,\rd P_n
-
\int \big(\G_0+h\big)^2\,\rd F^{\mu_0}_\omega\,\rd\mu_{0}
\right|
\\[0.15cm]
&\le \label{eq:local_CM_aux_ineq_1}
\left|
\int \big(\G_{n,Q_n}^{Q_n}+h\big)^2-\big(\G_0+h\big)^2 \,
\rd F_{n,\omega}^{Q_n}(t)\,\rd P_n(\omega)
\right| 
\\
&\quad+ \label{eq:local_CM_aux_ineq_2}
\bigg|
\int \big(\G_0(\omega,t)+h(\omega,t)\big)^2\,
\rd F_{n,\omega}^{Q_n}(t)\,\rd P_n(\omega)
\\
&\qquad\quad
-
\int \big(\G_0(\omega,t)+h(\omega,t)\big)^2\,
\rd F^{\mu_0}_\omega(t)\,\rd\mu_{0}(\omega)
\bigg|.
\notag
\end{align}
For \eqref{eq:local_CM_aux_ineq_1}, one has:
\begin{align*}
\bigg|
\int
\Big(
\big(\G_{n,Q_n}^{Q_n}&+h\big)^2-\big(\G_0+h\big)^2
\Big)
\,
\rd F_{n,\omega}^{Q_n}(t)\,\rd P_n(\omega)
\bigg|
\\
&\le
\sup_{\omega,t}
\bigg|
\big(\G_{n,Q_n}^{Q_n}(\omega,t)+h(\omega,t)\big)^2 -
\big(\G_0(\omega,t)+h(\omega,t)\big)^2
\bigg| \convas 0.
\end{align*}
The convergence to $0$ follows from the identity $(a+h)^2-(b+h)^2=(a-b)(a+b+2h)$, the almost sure convergence
\eqref{eq:aux_CM_sup_convas_2}, and the facts that $|h|\le 1$ and
$\sup_{\omega,t}|\G_0(\omega,t)|<\infty$ almost surely. %

To control \eqref{eq:local_CM_aux_ineq_2}, observe first that $\G_0$ has bounded and uniformly continuous sample paths almost surely on $\Omega\times[0,M]$, by Lemma~\ref{ap:lem_continuity_sample_paths_G_rho2}. Moreover, $h$ is deterministic and Lipschitz on $\Omega\times[0,M]$.

Fix $\eta>0$. Since $\|\G_0\|_\infty<\infty$ almost surely, there exists $L<\infty$ such that
$$
\mathsf{P}\big(\|\G_0\|_\infty\le L\big)\ge 1-\eta/2.
$$
For $\delta>0$, define
$$
W(\delta)
:=
\sup_{\substack{
d_{\Omega\times\R_{\ge0}}((\omega,t),(\omega^\prime,t^\prime))\le\delta}}
|\G_0(\omega,t)-\G_0(\omega^\prime,t^\prime)|.
$$
Since $\G_0$ has uniformly continuous sample paths almost surely, $W(\delta)\to0$ in probability as $\delta\to 0^+$. Hence, one may choose a decreasing sequence $(\delta_m)_{m\ge1}$ such that
$$
\mathsf{P}\big(W(\delta_m)>2^{-m}\big)\le \eta\,2^{-m-1},
\qquad m\ge1.
$$
Define $w_0:[0,\infty)\to[0,\infty)$ by
$$
w_0(u):=
\begin{cases}
0, & u=0,\\
2^{-m}, & \delta_{m+1}<u\le\delta_m,\\
\max\{2L,2^{-1}\}, & u>\delta_1.
\end{cases}
$$
Then $w_0$ is nondecreasing and $w_0(u)\to0$ as $u\to 0^+$. Let
$$
A_\eta
:=
\left\{\|\G_0\|_\infty\le L\right\}
\cap
\bigcap_{m\ge1}\left\{W(\delta_m)\le 2^{-m}\right\},
$$
which satisfies $\mathsf{P}(A_\eta)\ge 1-\eta$. On $A_\eta$, one has
$$
|\G_0(\omega,t)-\G_0(\omega^\prime,t^\prime)|
\le
w_0\!\left(
d_{\Omega\times\mathbb R_{\ge0}}((\omega,t),(\omega^\prime,t^\prime))
\right)
$$
for all $(\omega,t),(\omega^\prime,t^\prime)\in\Omega\times[0,M]$.

By \eqref{eq:rho2y_prod_bound} applied with the measure $\nu\in\{\mu_0,G\}$, there exists $C_\nu<\infty$ such that
$$
\left|F^\nu_\omega(t)-F^\nu_{\omega^\prime}(t^\prime)\right|
=
\left|\nu\big(y_{\omega,t}-y_{\omega^\prime,t^\prime}\big)\right|
\le
\nu\big(\big(y_{\omega,t}-y_{\omega^\prime,t^\prime}\big)^2\big)
\le
C_\nu\, d_{\Omega\times\mathbb R_{\ge0}}\big((\omega,t),(\omega^\prime,t^\prime)\big).
$$
Hence, $(\omega,t)\mapsto F^\nu_\omega(t)$ is Lipschitz for each $\nu\in\{\mu_0,G\}$, and therefore $h(\omega,t)=F_{\omega}^G(t)-F_{\omega}^{\mu_0}(t)$ is Lipschitz. Let $L_h$ be a Lipschitz constant of $h$. On $A_\eta$, $\|(\G_0+h)^2\|_\infty\le (L+\|h\|_\infty)^2,$
and, using $|a^2-b^2|\le(|a|+|b|)|a-b|$, it holds that
\begin{align*}
|(\G_0+h)^2(\omega,t)-(\G_0+h)^2(\omega^\prime,t^\prime)|
&\le
2(L+\|h\|_\infty)\Big(
w_0\!\left(d_{\Omega\times\mathbb R_{\ge0}}((\omega,t),(\omega^\prime,t^\prime))\right)
\nonumber
\\
&\quad +
L_h\,d_{\Omega\times\mathbb R_{\ge0}}((\omega,t),(\omega^\prime,t^\prime))
\Big)
\end{align*}
for every $(\omega,t), (\omega',t')\in\Omega\times[0,M].$
Therefore, defining
$$
B:=(L+\|h\|_\infty)^2,
\qquad
w(u):=2(L+\|h\|_\infty)\big(w_0(u)+L_hu\big),
$$
then
$$
\mathsf{P}\big((\G_0+h)^2\in\mathcal G_{B,w}\big)\ge1-\eta.
$$
Therefore, for every $\varepsilon>0$,
\begin{align*}
&\mathsf{P}\Bigg(
\left|
\int (\G_0+h)^2\,\rd F_{n,\omega}^{Q_n}(t)\,\rd P_n(\omega)
-
\int (\G_0+h)^2\,\rd F_\omega^{\mu_0}(t)\,\rd\mu_0(\omega)
\right|
>\varepsilon
\Bigg)
\\
&\;\le
\mathsf{P}\big((\G_0+h)^2\notin\mathcal G_{B,w}\big)
\\
&\qquad+
\mathsf{P}\Bigg(
\sup_{g\in\mathcal G_{B,w}}
\left|
\int g(\omega,t)\,\rd F_{n,\omega}^{Q_n}(t)\,\rd P_n(\omega)
-
\int g(\omega,t)\,\rd F_\omega^{\mu_0}(t)\,\rd\mu_0(\omega)
\right|
>\varepsilon
\Bigg).
\end{align*}
The first term is bounded by $\eta$, and the second converges to $0$ by Lemma~\ref{ap:lem_DCT_P_hat_local}. Since $\eta>0$ is arbitrary, the quantity in \eqref{eq:local_CM_aux_ineq_2} is $o_{\mathsf{P}}(1)$. This concludes the proof.
\end{proof}

\subsubsection{Composite null hypothesis}\label{ap:sec_local_alternatives_composite}

This section contains the proof of Proposition \ref{prop:local_conv_roots_Qn}, which establishes the local convergence of $\hat{\btheta}_n$ and $\tilde{\btheta}_n$ to $\btheta_0$. Remark \ref{rem:local_conv_roots_Qn} shows a useful fact that will be applied in the proof.
\begin{remark}\label{rem:local_conv_roots_Qn}
By condition \ref{cp3}, as $\btheta\to\btheta_0$, one has
\begin{equation}\label{eq:U1_taylor_mu}
\mu_{\btheta_0}\bpsi_{\btheta} = \bV_{\btheta_0}(\btheta-\btheta_0)+o(\|\btheta-\btheta_0\|).
\end{equation}
Let $c\defin \|\bV_{\btheta_0}^{-1}\|_2^{-1}/2$, with $\bV_{\btheta_0}$ as defined in condition \ref{cp3} and $\|\cdot\|_2$ the spectral norm. By condition \ref{cp3}, $\bV_{\btheta_0}$ is nonsingular, hence $c>0$. It holds that $\|\bV_{\btheta_0}\bv\|\ge 2c\|\bv\|$ for all $\bv\in\R^p$.

In addition, by definition of $o(\|\btheta-\btheta_0\|)$, taking $U$ sufficiently small, it holds that
\begin{equation*}
\big\|\mu_{\btheta_0}\bpsi_{\btheta} - \bV_{\btheta_0}(\btheta-\btheta_0)\big\|\le c\,\|\btheta-\btheta_0\|.
\end{equation*}
\end{remark}
\begin{proof}[Proof of Proposition~\ref{prop:local_conv_roots_Qn}]
We divide the proof into its two statements.

\emph{Proof of \ref{prop:local_conv_roots_Qn_i}.} By the reverse triangle inequality, one has
\begin{equation}\label{eq:mu_separation}
\|\mu_{\btheta_0}\bpsi_{\btheta}\|
\ge
\|\bV_{\btheta_0}(\btheta-\btheta_0)\|-\big\|\mu_{\btheta_0}\bpsi_{\btheta} - \bV_{\btheta_0}(\btheta-\btheta_0)\big\|
\ge
2c\|\btheta-\btheta_0\|-c\|\btheta-\btheta_0\|
=
c\|\btheta-\btheta_0\|
\end{equation}
in a neighborhood of $\btheta_0$.

By \eqref{eq:local_cp1_prop_local_conv}, for every $\btheta\in U$,
\begin{align*}
\|G\bpsi_{\btheta}\|
\le
\|G\bpsi_{\btheta_0}\| + \|G(\bpsi_{\btheta}-\bpsi_{\btheta_0})\| 
\le
\|G\bpsi_{\btheta_0}\| + G\|\bpsi_{\btheta}-\bpsi_{\btheta_0}\| 
\le
\|G\bpsi_{\btheta_0}\| + GL\|\btheta-\btheta_0\|.
\end{align*}
Since $GL<\infty$ by Cauchy--Schwarz and $U$ is bounded, it follows that
\begin{equation*}
M_G \defin \sup_{\btheta\in U}\|G\bpsi_{\btheta}\| < \infty.
\end{equation*}

Since $Q_n\bpsi_{\tilde\btheta_n}=0$,
by definition of $Q_n$ one has
$$
\left(1-\frac{1}{\sqrt{n}}\right)\,\|\mu_{\btheta_0}\bpsi_{\tilde\btheta_n}\|
=
\frac{\|G\bpsi_{\tilde\btheta_n}\|}{\sqrt{n}}
\le
\frac{M_G}{\sqrt{n}}.
$$
Hence, for $n$ sufficiently large,
$$
\|\mu_{\btheta_0}\bpsi_{\tilde\btheta_n}\|\le \frac{M_G}{\sqrt{n}\left(1-1/\sqrt{n}\right)} \le \frac{2M_G}{\sqrt{n}},
$$
and combining with \eqref{eq:mu_separation} at $\btheta=\tilde\btheta_n$ gives $
\|\tilde\btheta_n-\btheta_0\| \to 0 $ as $n\to \infty$. 

\emph{Proof of \ref{prop:local_conv_roots_Qn_ii}.} For each $j=1,\ldots,p$, let $\cF_j\defin\{\psi_{\btheta,j} : \btheta \in U\}$. We show that $\cF_j$ satisfies the conditions of Corollary~\ref{cor:bernoulli}. Each class $\cF_j$ is $\mu_{\btheta_0}$- and $G$-Donsker by \eqref{eq:local_cp1_prop_local_conv} and condition \ref{cp2}, because $\mu_{\btheta_0}L^2<\infty$, $\mu_{\btheta_0}\|\bpsi_{\btheta_0}\|^2<\infty$, $GL^2<\infty$, and $G\|\bpsi_{\btheta_0}\|^2<\infty$ (see Example 19.7 in \citealpSM{vandervaart1998SM}). Moreover, for every $\btheta\in U$,
\begin{align*}
\big|(G-\mu_{\btheta_0})\psi_{\btheta,j}\big|
\le
\big\|(G-\mu_{\btheta_0})\bpsi_{\btheta}\big\|
\le
\|G\bpsi_{\btheta}\| + \mu_{\btheta_0}\|\bpsi_{\btheta}-\bpsi_{\btheta_0}\|
\le
M_G + \mu_{\btheta_0}L\|\btheta-\btheta_0\|,
\end{align*}
so $\sup_{\btheta\in U}\big|(G-\mu_{\btheta_0})\psi_{\btheta,j}\big|<\infty$ because $U$ is bounded.

An application of Corollary~\ref{cor:bernoulli} to each $\cF_j$ and the continuous mapping theorem gives $\sup_{\btheta\in U}|(P_n-Q_n)\psi_{\btheta,j}|=O_{Q_n}(n^{-1/2})$. Since $p<\infty$,
\begin{equation}\label{eq:application_bernoulli}
\sup_{\btheta\in U}\|(P_n-Q_n)\bpsi_{\btheta}\|
\le
\sqrt{p}\,\max_{1\le j\le p}\sup_{\btheta\in U}|(P_n-Q_n)\psi_{\btheta,j}|
=O_{Q_n}(n^{-1/2}).
\end{equation}
Since $P_n\bpsi_{\hat\btheta}=0$, \eqref{eq:application_bernoulli} implies that, on the event $\{\hat\btheta\in U\}$,
\begin{equation}\label{eq:Qn_psi_hat_bound}
\|Q_n\bpsi_{\hat\btheta}\|
=
\|(P_n-Q_n)\bpsi_{\hat\btheta}\|
\le
\sup_{\btheta\in U}\|(P_n-Q_n)\bpsi_{\btheta}\|
=
O_{Q_n}(n^{-1/2}),
\end{equation}
and in particular $\|Q_n\bpsi_{\hat\btheta}\|=o_{Q_n}(1)$ on $\{\hat\btheta\in U\}$. 

By \eqref{eq:mu_separation},
\begin{equation}\label{eq:sep_ring_mu}
\inf_{\btheta \in U, \;\|\btheta-\btheta_0\|\ge\varepsilon}\|\mu_{\btheta_0}\bpsi_{\btheta}\|
\ge
c\varepsilon.
\end{equation}
Moreover,
$$
\|(Q_n-\mu_{\btheta_0})\bpsi_{\btheta}\|
=
\frac{1}{\sqrt{n}}\|(G-\mu_{\btheta_0})\bpsi_{\btheta}\|
\le
\frac{1}{\sqrt{n}}\big(\|G\bpsi_{\btheta}\|+\|\mu_{\btheta_0}\bpsi_{\btheta}\|\big)
\le
\frac{1}{\sqrt{n}}\Big(M_G+\sup_{\btheta\in U}\|\mu_{\btheta_0}\bpsi_{\btheta}\|\Big),
$$
so $\sup_{\btheta\in U}\|(Q_n-\mu_{\btheta_0})\bpsi_{\btheta}\|=O(n^{-1/2})=o(1)$. The fact that $\sup_{\btheta\in U}\|\mu_{\btheta_0}\bpsi_{\btheta}\|<\infty$ follows from \eqref{eq:local_cp1_prop_local_conv}.
Therefore, for $n$ sufficiently large,
\begin{equation}\label{eq:drift_small_ring}
\sup_{\btheta\in U}\|(Q_n-\mu_{\btheta_0})\bpsi_{\btheta}\| < \frac{c\varepsilon}{2}.
\end{equation}

Combining \eqref{eq:sep_ring_mu} and \eqref{eq:drift_small_ring}, for $n$ sufficiently large and all $\btheta\in U$ with $\|\btheta-\btheta_0\|\ge\varepsilon$,
\begin{equation*}
\|Q_n\bpsi_{\btheta}\|
\ge
\|\mu_{\btheta_0}\bpsi_{\btheta}\|-\|(Q_n-\mu_{\btheta_0})\bpsi_{\btheta}\|
>
c\varepsilon-\frac{c\varepsilon}{2}
=
\frac{c\varepsilon}{2}.
\end{equation*}
The preceding bound applied at $\btheta=\hat\btheta$ on the event $\{\hat\btheta\in U\}$ gives
$$
\mathsf P\big(\|\hat\btheta-\btheta_0\|\ge\varepsilon\big)
\le
\mathsf P(\hat\btheta\notin U)
+
\mathsf P\Big(\|Q_n\bpsi_{\hat\btheta}\|\ge \frac{c\varepsilon}{2},\ \hat\btheta\in U\Big)
\to0,
$$
because $\mathsf P(\hat\btheta\notin U)\to0$ and $\|Q_n\bpsi_{\hat\btheta}\|=O_{Q_n}(n^{-1/2})$ by \eqref{eq:Qn_psi_hat_bound} if $\hat\btheta\in U$. This implies $\mathsf P(\|\hat\btheta-\btheta_0\|\ge\varepsilon)\to0$ for every $\varepsilon>0$, i.e. $\hat\btheta \convp \btheta_0$.
\end{proof}

We introduce Lemmas \ref{lem:aux_P_n_Q_n_tilde} and \ref{lem:aux_P_n_Q_n_tilde_hat}, which are auxiliary results used in the proof of Theorem \ref{thm:local_alt_conv_process_correction}.

\begin{lemma}\label{lem:aux_P_n_Q_n_tilde_hat}
Assume conditions \ref{cp1} and \ref{cp2}, $GL^2<\infty$, and $G\|\bpsi_{\btheta_0}\|^2<\infty$.
Let $\tilde\btheta_n\in\Theta$ be such that $Q_n\bpsi_{\tilde\btheta_n}=0$ and $\tilde\btheta_n\to\btheta_0$.  
Let also $\hat\btheta\in\Theta$ be such that $P_n\bpsi_{\hat\btheta}=o_{\mathsf{P}}(n^{-1/2})$ and $\hat\btheta\convp\btheta_0$.
Under the local alternative defined in \eqref{eq:local_alt_Qn_def}, it holds that 
$$
\left(P_n - Q_n\right)\big(\bpsi_{\hat{\btheta}} - \bpsi_{\tilde\btheta_n}\big) = o_{\mathsf{P}}(n^{-1/2})
$$
\end{lemma}

\begin{lemma}\label{lem:aux_P_n_Q_n_tilde}
Assume conditions \ref{cp1} and \ref{cp2}, $GL^2<\infty$, and $G\|\bpsi_{\btheta_0}\|^2<\infty$.
Let $\tilde\btheta_n\in\Theta$ be such that $Q_n\bpsi_{\tilde\btheta_n}=0$ and $\tilde\btheta_n\to\btheta_0$.  
Under the local alternative defined in \eqref{eq:local_alt_Qn_def}, it holds that 
\begin{align*}
\left(P_n - Q_n\right)\bpsi_{\tilde{\btheta}_n} = O_{\mathsf{P}}(n^{-1/2}).
\end{align*}
\end{lemma}

Using Lemmas \ref{lem:aux_P_n_Q_n_tilde} and \ref{lem:aux_P_n_Q_n_tilde_hat}, we can now prove Theorem \ref{thm:local_alt_conv_process_correction}.

\begin{proof}[Proof of Theorem \ref{thm:local_alt_conv_process_correction}]
We divide the proof into its two statements.

    \emph{Proof of \ref{thm:local_alt_conv_process_correction_i}.} Since $\btheta\mapsto \mu_{\btheta_0}\bpsi_{\btheta}$ is differentiable at $\btheta_0$ by condition \ref{cp3}, it holds that
\begin{align*}
    \mu_{\btheta_0}\bpsi_{\btheta} = \mu_{\btheta_0}\bpsi_{\btheta_0} + \bV_{\btheta_0}(\btheta-\btheta_0) + o(\|\btheta-\btheta_0\|).
\end{align*}
Since $\tilde\btheta_n\to\btheta_0$, evaluating this expansion at $\btheta=\tilde\btheta_n$ gives
$$
\mu_{\btheta_0}\bpsi_{\tilde{\btheta}_n} = \mu_{\btheta_0}\bpsi_{\btheta_0} + \bV_{\btheta_0}(\tilde{\btheta}_n-\btheta_0) + o(\|\tilde{\btheta}_n-\btheta_0\|).
$$
Together with $Q_n\bpsi_{\tilde{\btheta}_n}=0$, this gives
\begin{align*}
    0= \Big(1-\frac{1}{\sqrt{n}}\Big)\left(\bV_{\btheta_0}(\tilde{\btheta}_n-\btheta_0) + o(\|\tilde{\btheta}_n-\btheta_0\|)\right) + \frac{1}{\sqrt{n}} G \bpsi_{\tilde{\btheta}_n},
\end{align*}
and thus
\begin{align}\label{eq:tilde_theta_n_theta_0}
    \bV_{\btheta_0}(\tilde{\btheta}_n - \btheta_0) = - \frac{1}{\sqrt{n}\left(1-1/\sqrt{n}\right)} G \bpsi_{\tilde{\btheta}_n} + o(\|\tilde{\btheta}_n-\btheta_0\|).
\end{align}

Now, by condition \ref{cp1} and the assumption that $\tilde{\btheta}_n\to\btheta_0$, it holds that 
$$
\|G\bpsi_{\tilde{\btheta}_n}\|\le \|G\bpsi_{\btheta_0}\| + GL \|\tilde{\btheta}_n-\btheta_0\|
$$ 
for $n$ sufficiently large, where $GL<\infty$ because $GL^2<\infty$.
Plugging into \eqref{eq:tilde_theta_n_theta_0} gives
\begin{align*}
    \Big(1-\frac{1}{\sqrt{n}}\Big)\|\bV_{\btheta_0}(\tilde{\btheta}_n-\btheta_0)\| \le \frac{1}{\sqrt{n}}\|G\bpsi_{\btheta_0}\| + \frac{GL}{\sqrt{n}}\|\tilde{\btheta}_n-\btheta_0\| + o(\|\tilde{\btheta}_n-\btheta_0\|).
\end{align*}
By Remark~\ref{rem:local_conv_roots_Qn}, $\|\bV_{\btheta_0}(\tilde{\btheta}_n-\btheta_0)\|\ge 2c\|\tilde{\btheta}_n-\btheta_0\|$. Therefore, for every $\varepsilon>0$, there exists $n(\varepsilon)>0$ such that, for every $n>n(\varepsilon)$, one has
\begin{align*}
    2c\left(1-\frac{1}{\sqrt{n}}\right)\|\tilde{\btheta}_n-\btheta_0\|  \le \frac{1}{\sqrt{n}}\|G\bpsi_{\btheta_0}\| + \frac{GL}{\sqrt{n}}\|\tilde{\btheta}_n-\btheta_0\| + \varepsilon\|\tilde{\btheta}_n-\btheta_0\|.
\end{align*}
Thus,
\begin{align*}
    \left(2c\Big(1-\frac{1}{\sqrt{n}}\Big) - \frac{GL}{\sqrt{n}} - \varepsilon \right)\|\tilde{\btheta}_n-\btheta_0\| \le \frac{1}{\sqrt{n}}\|G\bpsi_{\btheta_0}\|.
\end{align*} 
Take $\varepsilon=c/2$ to obtain
\begin{align}\label{eq:tilde_theta_theta_0_On-12}
    \|\tilde{\btheta}_n-\btheta_0\| \le \frac{2}{c\sqrt{n}}\|G\bpsi_{\btheta_0}\| = O\big(n^{-1/2}\big),
\end{align}
for $n$ sufficiently large. The bound $\|G\bpsi_{\btheta_0}\|<\infty$ follows from $G\|\bpsi_{\btheta_0}\|^2<\infty$.

Now, rewrite \eqref{eq:tilde_theta_n_theta_0} as
\begin{align}\label{eq:tilde_theta_n_theta_0_rewrite}
   \bV_{\btheta_0} (\tilde{\btheta}_n - \btheta_0) = - \frac{1}{\sqrt{n}(1-1/\sqrt{n})} G \bpsi_{\btheta_0} - \frac{1}{\sqrt{n}(1-1/\sqrt{n})} G \bigl(\bpsi_{\tilde{\btheta}_n} - \bpsi_{\btheta_0}\bigr) + o(\|\tilde{\btheta}_n-\btheta_0\|).
\end{align}
By condition \ref{cp1}, one has
\begin{align*}
   \frac{1}{\sqrt{n}(1-1/\sqrt{n})} \|G(\bpsi_{\tilde{\btheta}_n} - \bpsi_{\btheta_0})\| \le \frac{1}{\sqrt{n}(1-1/\sqrt{n})} GL \|\tilde{\btheta}_n-\btheta_0\| = O(n^{-1})
\end{align*}
by \eqref{eq:tilde_theta_theta_0_On-12}. Thus, \eqref{eq:tilde_theta_n_theta_0_rewrite} yields
\begin{align*}
    \tilde{\btheta}_n - \btheta_0 = - \frac{1}{\sqrt{n}(1-1/\sqrt{n})} \bV_{\btheta_0}^{-1} G \bpsi_{\btheta_0} + o\big(n^{-1/2}\big),
\end{align*}
which gives \eqref{eq:thm_local_alt_conv_process_correction_i}, applying the assumption $\|G \bpsi_{\btheta_0}\|<\infty$.

\emph{Proof of \ref{thm:local_alt_conv_process_correction_ii}.} Now, we turn to \eqref{eq:thm_local_alt_conv_process_correction_ii}. Since $P_n\bpsi_{\hat\btheta}=o_{\mathsf{P}}(n^{-1/2})$ and $Q_n\bpsi_{\tilde\btheta_n}=0$, one has
\begin{align}\label{eq:hat_theta_split}
    0 = \left(P_n - Q_n\right)\bpsi_{\hat{\btheta}} + Q_n\big(\bpsi_{\hat{\btheta}} - \bpsi_{\tilde\btheta_n}\big) + o_{\mathsf{P}}(n^{-1/2}).
\end{align}
To control the first summand in \eqref{eq:hat_theta_split}, write 
\begin{align*}
    \left(P_n - Q_n\right)\bpsi_{\hat{\btheta}} = \left(P_n - Q_n\right)\bpsi_{\tilde{\btheta}_n} + \left(P_n - Q_n\right)\big(\bpsi_{\hat{\btheta}} - \bpsi_{\tilde{\btheta}_n}\big).
\end{align*}
By Lemma \ref{lem:aux_P_n_Q_n_tilde_hat}, it holds that $(P_n - Q_n)(\bpsi_{\hat{\btheta}} - \bpsi_{\tilde{\btheta}_n}) = o_{\mathsf{P}}(n^{-1/2})$. 
Therefore, 
\begin{align}\label{eq:first_RHS_hat_theta_split}
    \left(P_n - Q_n\right)\bpsi_{\hat{\btheta}} = \left(P_n - Q_n\right)\bpsi_{\tilde{\btheta}_n} + o_{\mathsf{P}}(n^{-1/2}).
\end{align}

Focus now on the second summand of \eqref{eq:hat_theta_split}, i.e, $Q_n(\bpsi_{\hat{\btheta}} - \bpsi_{\tilde\btheta_n})$. By definition of $Q_n$, one has
$$
Q_n\big(\bpsi_{\hat{\btheta}} - \bpsi_{\tilde\btheta_n}\big) = \left(1-\frac{1}{\sqrt{n}}\right)\mu_{\btheta_0} \left(\bpsi_{\hat{\btheta}} - \bpsi_{\tilde\btheta_n}\right) + \frac{1}{\sqrt{n}} G \left(\bpsi_{\hat{\btheta}} - \bpsi_{\tilde\btheta_n}\right).
$$
By \eqref{eq:U1_taylor_mu} applied at $\btheta=\hat{\btheta}$ and $\btheta=\tilde{\btheta}_n$, condition \ref{cp1}, and $GL<\infty$, one has
\begin{equation}\label{eq:second_RHS_hat_theta_split}
\begin{aligned}
   Q_n\left(\bpsi_{\hat{\btheta}} - \bpsi_{\tilde\btheta_n}\right) & = \left(1-\frac{1}{\sqrt{n}}\right)\left(\bV_{\btheta_0}(\hat{\btheta}-\tilde\btheta_n) + o_{\mathsf{P}}(\|\hat{\btheta}-\btheta_0\|) + o(\|\tilde{\btheta}_n-\btheta_0\|)\right) + O_{\mathsf{P}}\left(\frac{1}{\sqrt{n}}\|\hat{\btheta}-\tilde\btheta_n\|\right) \\
    &= \bV_{\btheta_0}(\hat{\btheta}-\tilde\btheta_n) +  O_{\mathsf{P}}\left(\frac{1}{\sqrt{n}} \|\hat{\btheta}-\tilde\btheta_n\|\right) + o_{\mathsf{P}}(\|\hat{\btheta}-\btheta_0\|) + o(\|\tilde{\btheta}_n-\btheta_0\|) \\
    &= \bV_{\btheta_0}(\hat{\btheta}-\tilde\btheta_n) + O_{\mathsf{P}}\left(\frac{1}{\sqrt{n}} \|\hat{\btheta}-\tilde\btheta_n\|\right) + o_{\mathsf{P}}(\|\hat{\btheta}-\btheta_0\|) + o_{\mathsf{P}}(n^{-1/2}),
\end{aligned}
\end{equation}
where the last equality follows from \eqref{eq:tilde_theta_theta_0_On-12}.

Combining \eqref{eq:first_RHS_hat_theta_split} and \eqref{eq:second_RHS_hat_theta_split} into \eqref{eq:hat_theta_split} gives
\begin{equation}\label{eq:pre_hat_tilde}
\begin{aligned}
    0 =& (P_n-Q_n) \bpsi_{\tilde{\btheta}_n} + \bV_{\btheta_0}(\hat{\btheta}-\tilde\btheta_n) +  O_{\mathsf{P}}\left(\frac{1}{\sqrt{n}} \|\hat{\btheta}-\tilde\btheta_n\|\right) 
    \\
    & \; + o_{\mathsf{P}}(\|\hat{\btheta}-\btheta_0\|) + o_{\mathsf{P}}(n^{-1/2})
\end{aligned}
\end{equation}
We show that $\|\hat{\btheta}-\tilde\btheta_n\| = O_{\mathsf{P}}(n^{-1/2})$. Using Lemma \ref{lem:aux_P_n_Q_n_tilde} in \eqref{eq:pre_hat_tilde}, we obtain
\begin{align*} 
    \left\|\bV_{\btheta_0}(\hat{\btheta}-\tilde\btheta_n)\right\| &\le O_{\mathsf{P}}(n^{-1/2}) + O_{\mathsf{P}}\left(\frac{1}{\sqrt{n}} \|\hat{\btheta}-\tilde\btheta_n\|\right) + o_{\mathsf{P}}(\|\hat{\btheta}-\btheta_0\|) + o_{\mathsf{P}}(n^{-1/2}) \\
    &= O_{\mathsf{P}}(n^{-1/2}) + o_{\mathsf{P}}(\|\hat{\btheta}-\btheta_0\|). 
\end{align*} 
The last equality uses $\|\hat{\btheta}-\tilde\btheta_n\|=o_{\mathsf{P}}(1)$. By Remark~\ref{rem:local_conv_roots_Qn}, $\|\bV_{\btheta_0}(\hat{\btheta}-\tilde{\btheta}_n)\|\ge 2c\|\hat{\btheta}-\tilde{\btheta}_n\|$. Then, applying the triangle inequality to $\|\hat{\btheta}-\btheta_0\|$ and using $\|\tilde{\btheta}_n-\btheta_0\|=O(n^{-1/2})$ from \eqref{eq:tilde_theta_theta_0_On-12} gives 
$$
2c\|\hat{\btheta}-\tilde{\btheta}_n\| \le O_{\mathsf{P}}(n^{-1/2}) + o_{\mathsf{P}}(\|\hat{\btheta}-\tilde{\btheta}_n\|).
$$
By Lemma \ref{lem:absorbtion_OP}, this implies that $\|\hat{\btheta}-\tilde\btheta_n\| = O_{\mathsf{P}}(n^{-1/2})$. Observe also that since $\|\hat{\btheta}-\tilde\btheta_n\|=O_{\mathsf{P}}(n^{-1/2})$ and $\|\tilde{\btheta}_n-\btheta_0\|=O(n^{-1/2})$, then the triangle inequality gives $\|\hat{\btheta}-\btheta_0\|=O_{\mathsf{P}}(n^{-1/2})$. 
Plugging back into \eqref{eq:pre_hat_tilde} and using $Q_n \bpsi_{\tilde{\btheta}_n}=0$ gives 
\begin{align*} 
    \hat{\btheta} - \tilde\btheta_n = -\bV_{\btheta_0}^{-1} P_n \bpsi_{\tilde{\btheta}_n} + o_{\mathsf{P}}(n^{-1/2}),
\end{align*} 
which combined with \eqref{eq:thm_local_alt_conv_process_correction_i} gives \eqref{eq:thm_local_alt_conv_process_correction_ii}.  
\end{proof}

\begin{proof}[Proof of Theorem \ref{thm:local_alt_Gn_composite}]
First, observe that under the local alternative \eqref{eq:local_alt_Qn_def} we can write, for every
$(\omega,t)\in\Omega\times\R_{\ge0}$,
\begin{align}
F_{n,\omega}^{Q_n}(t)-F^{\mu_{\hat\btheta}}_\omega(t)
&=
\Big(F_{n,\omega}^{Q_n}(t)-F^{\mu_{\btheta_0}}_\omega(t)\Big)
-
\Big(F^{\mu_{\hat\btheta}}_\omega(t)-F^{\mu_{\btheta_0}}_\omega(t)\Big).
\label{eq:proof_E5_start}
\end{align}
The asymptotic behavior of the first term on the RHS of \eqref{eq:proof_E5_start} was given in \eqref{eq:local_alt_weak_limit_process} with $\mu_0$ replaced by $\mu_{\btheta_0}$. 

It remains to analyze $F^{\mu_{\hat\btheta}}_\omega(t)-F^{\mu_{\btheta_0}}_\omega(t)$. By condition \ref{d}, it holds that
\begin{equation}
\sqrt{n}\Big(F^{\mu_{\hat\btheta}}_\omega(t)-F^{\mu_{\btheta_0}}_\omega(t)\Big)
=
\sqrt{n} \dot F^{\btheta_0}_\omega(t)^\top(\hat\btheta-\btheta_0)
+
o_{\mathsf{P}}\big(\sqrt{n}\|\hat\btheta-\btheta_0\|\big).
\label{eq:proof_E5_taylor_scaled}
\end{equation}
where the remainder is uniform in $(\omega,t)\in\Omega\times\R_{\ge0}$.

Recall that by Theorem \ref{thm:local_alt_conv_process_correction}, it holds that
\begin{equation}
\sqrt{n}(\hat\btheta-\btheta_0)
=
-\bV_{\btheta_0}^{-1}\sqrt{n}\,P_n\bpsi_{\tilde\btheta_n}
-\bV_{\btheta_0}^{-1}G\bpsi_{\btheta_0}
+o_{\mathsf{P}}(1),
\label{eq:proof_E5_use_95}
\end{equation}
Since $Q_n\bpsi_{\tilde\btheta_n}=0$, we have $\sqrt{n}\,P_n\bpsi_{\tilde\btheta_n}=\sqrt{n}(P_n-Q_n)\bpsi_{\tilde\btheta_n}=\G_{n,Q_n}^{Q_n}(\bpsi_{\tilde\btheta_n}).$
Moreover,
\begin{align*}
\mathsf E\big\|\G_{n,Q_n}^{Q_n}(\bpsi_{\tilde\btheta_n}-\bpsi_{\btheta_0})\big\|^2
\le
Q_n\big\|\bpsi_{\tilde\btheta_n}-\bpsi_{\btheta_0}\big\|^2 
\le
Q_n L^2\,\|\tilde\btheta_n-\btheta_0\|^2
\to 0,
\end{align*}
because $Q_nL^2$ is bounded and $\tilde\btheta_n\to\btheta_0$. Hence, by Markov's inequality, $\G_{n,Q_n}^{Q_n}(\bpsi_{\tilde\btheta_n}-\bpsi_{\btheta_0})=o_{\mathsf{P}}(1)$. Therefore,
\begin{equation}
\sqrt{n}P_n\bpsi_{\tilde\btheta_n}
=
\G_{n,Q_n}^{Q_n}(\bpsi_{\btheta_0})+o_{\mathsf{P}}(1).
\label{eq:proof_E5_sum_to_theta0}
\end{equation}
Plugging \eqref{eq:proof_E5_sum_to_theta0} into \eqref{eq:proof_E5_use_95} gives
\begin{equation}
\sqrt{n}(\hat\btheta-\btheta_0)
=
-\bV_{\btheta_0}^{-1}\,\G_{n,Q_n}^{Q_n}(\bpsi_{\btheta_0})
-\bV_{\btheta_0}^{-1}G\bpsi_{\btheta_0}
+o_{\mathsf{P}}(1).
\label{eq:proof_E5_theta_expansion_final}
\end{equation}
Since $Q_n\|\bpsi_{\btheta_0}\|^2$ is bounded, $\G_{n,Q_n}^{Q_n}(\bpsi_{\btheta_0})=O_{\mathsf{P}}(1)$; hence \eqref{eq:proof_E5_theta_expansion_final} gives $\sqrt{n}\|\hat\btheta-\btheta_0\|=O_{\mathsf{P}}(1)$, and the remainder in \eqref{eq:proof_E5_taylor_scaled} is $o_{\mathsf{P}}(1)$ uniformly in $(\omega,t)\in\Omega\times\R_{\ge0}$.
Substituting \eqref{eq:proof_E5_theta_expansion_final} into \eqref{eq:proof_E5_taylor_scaled}, we obtain
\begin{equation}\label{eq:proof_E5_Fhat_expansion_like_3p2}
\begin{aligned}
\sqrt{n}\Big(F^{\mu_{\hat\btheta}}_\omega(t)-F^{\mu_{\btheta_0}}_\omega(t)\Big)
=&
-\dot F^{\btheta_0}_\omega(t)^\top \bV_{\btheta_0}^{-1}\,\G_{n,Q_n}^{Q_n}(\bpsi_{\btheta_0})
\\
&\quad
-\dot F^{\btheta_0}_\omega(t)^\top \bV_{\btheta_0}^{-1}\,G\bpsi_{\btheta_0}
+o_{\mathsf{P}}(1).
\end{aligned}
\end{equation}

Combining \eqref{eq:proof_E5_start} with \eqref{eq:proof_E5_Fhat_expansion_like_3p2}, the rest of the proof follows analogously to that of Theorem~\ref{thm:convergence_G_n_composite}. Define
\begin{equation*}
g_{\omega,t}(x):=-\dot F^{\btheta_0}_\omega(t)^\top \bV_{\btheta_0}^{-1}\bpsi_{\btheta_0}(x),
\qquad
f_{\omega,t}(x):=y_{\omega,t}(x)-F^{\mu_{\btheta_0}}_\omega(t)-g_{\omega,t}(x),
\end{equation*}
so that $\mu_{\btheta_0}f_{\omega,t}=0$.

Combining \eqref{eq:proof_E5_start}, \eqref{eq:local_alt_Qn_def}, and \eqref{eq:proof_E5_Fhat_expansion_like_3p2},
and using that $\G_{n,Q_n}^{Q_n}(g_{\omega,t})=
-\dot F^{\btheta_0}_\omega(t)^\top \bV_{\btheta_0}^{-1}\G_{n,Q_n}^{Q_n}(\bpsi_{\btheta_0})$, one obtains
\begin{align}
\sqrt{n}\Big(F_{n,\omega}^{Q_n}(t)-F^{\mu_{\hat\btheta}}_\omega(t)\Big)
&=
\G_{n,Q_n}^{Q_n}(y_{\omega,t})
-\G_{n,Q_n}^{Q_n}(g_{\omega,t})
+h(\omega,t)
+\dot F^{\btheta_0}_\omega(t)^\top \bV_{\btheta_0}^{-1}G\bpsi_{\btheta_0}
+o_{\mathsf{P}}(1)
\nonumber\\
&=
\G_{n,Q_n}^{Q_n}(f_{\omega,t})
+h(\omega,t)
+\dot F^{\btheta_0}_\omega(t)^\top \bV_{\btheta_0}^{-1}G\bpsi_{\btheta_0}
+o_{\mathsf{P}}(1),
\label{eq:proof_E5_representation}
\end{align}
where we used that $\G_{n,Q_n}^{Q_n}(F^{\mu_{\btheta_0}}_\omega(t))=0$.

It remains to prove the weak convergence of $\G_{n,Q_n}^{Q_n}$ to $\tilde{\G}_{0}$ in $\ell^\infty(\cF)$, 
where $\cF:=\{f_{\omega,t}:(\omega,t)\in\Omega\times\R_{\ge0}\}$.
As in the proof of Theorem~\ref{thm:convergence_G_n_composite}, write $\cF\subset \cY-\mathcal{G}$, where
$$
\cY:=\{y_{\omega,t}-F^{\mu_{\btheta_0}}_\omega(t):(\omega,t)\in\Omega\times\R_{\ge0}\},
\qquad
\mathcal{G}:=\{g_{\omega,t}:(\omega,t)\in\Omega\times\R_{\ge0}\}.
$$
The class $\cY$ is $\mu_{\btheta_0}$- and $G$-Donsker by Theorem~3 of \citeSM{Chen2025SM}. For the class $\mathcal{G}$, note that
$$
g_{\omega,t}(x)=\sum_{j=1}^p \alpha_j(\omega,t)\,\psi_{\btheta_0,j}(x),
\qquad
\boldsymbol{\alpha}(\omega,t)^\top:=-\dot F^{\btheta_0}_\omega(t)^\top \bV_{\btheta_0}^{-1},
$$
so $\mathcal{G}\subset\mathrm{span}\{\psi_{\btheta_0,1},\dots,\psi_{\btheta_0,p}\}$ is finite-dimensional. Moreover, by conditions \ref{cp2}, \ref{cp3}, and \ref{d}, and by $G\|\bpsi_{\btheta_0}\|^2<\infty$, the envelope of $\mathcal{G}$ is in $L^2(\mu_{\btheta_0})$ and $L^2(G)$.
Thus, $\mathcal{G}$ is $\mu_{\btheta_0}$-Donsker and $G$-Donsker. Moreover, $\sup_{h\in \cY\cup \mathcal{G}}\mu_{\btheta_0}h^2<\infty$ and $\sup_{h\in \cY\cup \mathcal{G}}Gh^2<\infty$, hence $\cY-\mathcal{G}$ is $\mu_{\btheta_0}$-Donsker and $G$-Donsker, and therefore so is $\cF$.
If $\sup_{f\in\cF}|(G-\mu_{\btheta_0})f|<\infty$ holds, then Corollary~\ref{cor:bernoulli} with $\mu_0$ replaced by $\mu_{\btheta_0}$ applies and yields
\begin{equation}
\G_{n,Q_n}^{Q_n} \convl \tilde{\G}_{0}
\quad\text{in }\ell^\infty(\cF).
\label{eq:proof_E5_donsker_apply}
\end{equation}
Combining \eqref{eq:proof_E5_representation} and \eqref{eq:proof_E5_donsker_apply} and using Slutsky's theorem gives \eqref{eq:local_alt_composite_consistency}.

It is only left to verify the condition $\sup_{f\in\cF}|(G-\mu_{\btheta_0})f|<\infty$, in order to apply Corollary~\ref{cor:bernoulli}.
Since $0\le y_{\omega,t}\le 1$, one has $\sup_{\omega,t}|(G-\mu_{\btheta_0})y_{\omega,t}|\le 2$.
Moreover, $(G-\mu_{\btheta_0})F^{\mu_{\btheta_0}}_\omega(t)=0$ and $\mu_{\btheta_0}\bpsi_{\btheta_0}=0$, hence
$$
|(G-\mu_{\btheta_0})g_{\omega,t}|
=
\big|\dot F^{\btheta_0}_\omega(t)^\top \bV_{\btheta_0}^{-1}G\bpsi_{\btheta_0}\big|
\le
\sup_{\omega,t}\|\dot F^{\btheta_0}_\omega(t)\|\,\|\bV_{\btheta_0}^{-1}\|\,\|G\bpsi_{\btheta_0}\|.
$$
by conditions \ref{d} and \ref{cp3}, and by $G\|\bpsi_{\btheta_0}\|^2<\infty$, it holds that $\sup_{f\in\cF}|(G-\mu_{\btheta_0})f|<\infty$.
\end{proof}

\begin{proof}[Proof of Theorem \ref{thm:consistency_local_composite}]
The convergence of the KS statistic follows from Theorem \ref{thm:local_alt_Gn_composite} and the continuous mapping theorem.

For the CM statistic, the result follows by the same argument as in Theorem~\ref{thm:local_alt_simple}, replacing $\mu_0$ by $\mu_{\btheta_0}$ and $\G_0+h$ by $\tilde{\G}_0(\omega,t)+h(\omega,t)+\dot F^{\btheta_0}_\omega(t)^\top \bV_{\btheta_0}^{-1}G\bpsi_{\btheta_0}$.
As in Theorem \ref{thm:local_alt_Gn_composite}, the only caveat is to verify boundedness and uniform continuity of the deterministic drift, 
which in this case is given by $h(\omega,t)+\dot F^{\btheta_0}_\omega(t)^\top \bV_{\btheta_0}^{-1}G\bpsi_{\btheta_0}.$

The uniform continuity and boundedness of $h$ were established in the proof of Theorem~\ref{thm:local_alt_simple}. Uniform continuity of $(\omega,t)\mapsto\dot F^{\btheta_0}_\omega(t)^\top \bV_{\btheta_0}^{-1}G\bpsi_{\btheta_0}$ with respect to $d_{\Omega\times\mathbb R_{\ge0}}$
follows from the assumed uniform continuity of $(\omega,t)\mapsto \dot F^{\btheta_0}_\omega(t)$. Boundedness follows from condition \ref{d}, the invertibility of $\bV_{\btheta_0}$ by condition \ref{cp3}, and $G\|\bpsi_{\btheta_0}\|^2<\infty$.

Using the conditions of Theorem~\ref{thm:local_alt_Gn_composite} and condition \ref{e}, it holds that
\begin{align*}
\sup_{\omega\in\Omega,\, t\ge0}
\left|
\widehat{\G}_{n,\mu}(\omega,t)
-
\G_{n,Q_n}^{\mu_{\hat{\btheta}}}(\omega,t)
\right|
=
o_{\mathsf P}(1),
\end{align*}
under the local alternative \eqref{eq:local_alt_Qn_def}.
Consequently, $\widehat{\G}_{n,\mu}$ and $\G_{n,Q_n}^{\mu_{\hat{\btheta}}}$ have the same weak limit under the local alternative. Then, the asymptotics of $\widehat T_n^{\mathrm{KS}}$ and $\widehat T_n^{\mathrm{CM}}$ follow in the same way as for $\tilde T_n^{\mathrm{KS}}$ and $\tilde T_n^{\mathrm{CM}}$.
\end{proof}

\begin{lemma}\label{lem:absorbtion_OP}
For $n\ge1$, let $A_n$, $B_n$, and $C_n$ be nonnegative random variables such that $A_n \le B_n + C_n$ for all $n$. Assume that $C_n = o_{\mathsf{P}}(A_n)$, i.e., for every $\epsilon>0$, $\mathsf{P}(C_n>\epsilon A_n)\to0$. Then, $A_n = O_{\mathsf{P}}(B_n)$.
\end{lemma}

\section{Proofs of Section~\ref{sec:bootstrap}}\label{ap:sec_proofs_multiplier_bootstrap}

\begin{proof}[Proof of Theorem~\ref{thm:bootstrap_processes}]
We divide the proof into its two statements.

\emph{Proof of \ref{thm:bootstrap_processes_i}.} It was shown in Section~\ref{sec:bootstrap} that
$$
\G_n^{*}(\omega,t)
=
\frac{m}{\tau}\sqrt n(P_n^{*}-P_n)y_{\omega,t}.
$$
Since $\cY$ is $\mu_0$-Donsker, Theorem~2.6 in \citeSM{Kosorok2008SM} gives
$$
\G_n^{*}\convpboot{}\G_0
\quad\text{in }\ell^\infty(\cY).
$$
This concludes the result.

\emph{Proof of \ref{thm:bootstrap_processes_ii}.} First observe that condition \ref{cp2} and Theorem~2.6 in \citeSM{Kosorok2008SM}, applied to the finite-dimensional class generated by the components of $\bpsi_{\btheta_0}$, give
$$
\frac{m}{\tau}\frac{1}{\sqrt n}\sum_{i=1}^n
\left(\frac{\zeta_i}{\bar\zeta_n}-1\right)\bpsi_{\btheta_0}(X_i)
=O_{P_{*}}(1).
$$
Hence, by the bootstrap Bahadur representation \eqref{eq:weighted_bahadur_bootstrap}, it holds that $\sqrt n(\hat{\btheta}^{*}-\hat{\btheta})=O_{P_{*}}(1)$. Moreover, \eqref{eq:bahadur} gives $\sqrt n(\hat{\btheta}-\btheta_0)=O_{\mathsf{P}}(1)$, and hence $\sqrt n(\hat{\btheta}^{*}-\btheta_0)=O_{P_{*}}(1)$. Therefore, by condition \ref{d},
\begin{align*}
\sqrt n\left(F_\omega^{\mu_{\hat{\btheta}}}(t)-F_\omega^{\mu_{\btheta_0}}(t)\right)
&=
\dot F_\omega^{\btheta_0}(t)^\top\sqrt n(\hat{\btheta}-\btheta_0)
+o_{\mathsf{P}}(1),
\\
\sqrt n\left(F_\omega^{\mu_{\hat{\btheta}^{*}}}(t)-F_\omega^{\mu_{\btheta_0}}(t)\right)
&=
\dot F_\omega^{\btheta_0}(t)^\top\sqrt n(\hat{\btheta}^{*}-\btheta_0)
+o_{P_{*}}(1).
\end{align*}
The remainders $o_{\mathsf{P}}(1)$ and $o_{P_{*}}(1)$ are both uniform in $(\omega,t)\in\Omega\times\R_{\ge0}$.
Subtracting these two expansions and using \eqref{eq:weighted_bahadur_bootstrap} in the definition of $\tilde{\G}_n^{*}$ gives
\begin{equation}\label{eq:bootstrap_comp_process_remainder}
\sup_{\omega\in\Omega, t\ge0}
\left|
\tilde{\G}_n^{*}(\omega,t)
-
\frac{m}{\tau}\sqrt n(P_n^{*}-P_n)f_{\omega,t}
\right|
=o_{P_{*}}(1).
\end{equation}
Since $\cF$ is $\mu_{\btheta_0}$-Donsker, %
Theorem~2.6 in \citeSM{Kosorok2008SM} implies that
$$
\frac{m}{\tau}\sqrt n(P_n^{*}-P_n)
\convpboot{}\tilde{\G}_0
\quad\text{in }\ell^\infty(\cF).
$$

Let $\varphi\in \mathrm{BL}_1(\ell^\infty(\cF))$. Define
$$
D_n^{*}
\defin
\sup_{\omega\in\Omega, t\ge0}
\left|
\tilde{\G}_n^{*}(\omega,t)
-
\frac{m}{\tau}\sqrt n(P_n^{*}-P_n)f_{\omega,t}
\right|.
$$
By \eqref{eq:bootstrap_comp_process_remainder}, $D_n^{*}=o_{P_{*}}(1)$. Moreover,
for every $\varepsilon>0$,
\begin{align}\label{eq:bootstrap_comp_BL_comparison}
\left|
\mathsf{E}_{*}\varphi(\tilde{\G}_n^{*})
-
\mathsf{E}_{*}\varphi\left(\frac{m}{\tau}\sqrt n(P_n^{*}-P_n)\right)
\right|
&\le
\mathsf{E}_{*}\left|
\varphi(\tilde{\G}_n^{*})
-
\varphi\left(\frac{m}{\tau}\sqrt n(P_n^{*}-P_n)\right)
\right|
\\
&\le
\varepsilon\,\mathsf{P}_{*}(D_n^{*}\le\varepsilon)
+
2\,\mathsf{P}_{*}(D_n^{*}>\varepsilon)
\notag\\
&\le
\varepsilon+
2\,\mathsf{P}_{*}(D_n^{*}>\varepsilon).
\notag
\end{align}
Here, we have used that, if $D_n^{*}\le\varepsilon$, the Lipschitz property of $\varphi$ gives
$$
\left|
\varphi(\tilde{\G}_n^{*})
-
\varphi\left(\frac{m}{\tau}\sqrt n(P_n^{*}-P_n)\right)
\right|
\le D_n^{*}\le\varepsilon.
$$
Also, if $D_n^{*}>\varepsilon$, then using $\|\varphi\|_\infty\le1$, this absolute difference can be bounded by $2$. Since $\mathsf{P}_{*}(D_n^{*}>\varepsilon)\convp0$ by \eqref{eq:bootstrap_comp_process_remainder}, and $\varepsilon$ is arbitrary, the LHS of \eqref{eq:bootstrap_comp_BL_comparison} converges to zero in probability.

It remains to verify the asymptotic measurability condition. Since
$$
\frac{m}{\tau}\sqrt n(P_n^{*}-P_n)
\convpboot{}\tilde{\G}_0
\quad\text{in }\ell^\infty(\cF),
$$
Theorem~2.6 in \citeSM{Kosorok2008SM} gives, for every $\varphi\in\mathrm{BL}_1(\ell^\infty(\cF))$,
$$
\mathsf E_{*}\overline{
\varphi\left(\frac{m}{\tau}\sqrt n(P_n^{*}-P_n)\right)}
-
\mathsf E_{*}\underline{
\varphi\left(\frac{m}{\tau}\sqrt n(P_n^{*}-P_n)\right)}
\convp0.
$$
Since $\varphi$ is bounded by one and Lipschitz with constant at most one,
$$
\left|
\varphi(\tilde{\G}_n^{*})
-
\varphi\left(\frac{m}{\tau}\sqrt n(P_n^{*}-P_n)\right)
\right|
\le
\min\{D_n^{*},2\}.
$$
Taking majorants and minorants gives
\begin{align*}
\mathsf E_{*}\overline{\varphi(\tilde{\G}_n^{*})}
-
\mathsf E_{*}\underline{\varphi(\tilde{\G}_n^{*})}
\le\;&
\mathsf E_{*}\overline{
\varphi\left(\frac{m}{\tau}\sqrt n(P_n^{*}-P_n)\right)}
-
\mathsf E_{*}\underline{
\varphi\left(\frac{m}{\tau}\sqrt n(P_n^{*}-P_n)\right)}
\\
&+
2\mathsf E_{*}\overline{\min\{D_n^{*},2\}}.
\end{align*}
Moreover, for every $\varepsilon>0$,
$$
\mathsf E_{*}\overline{\min\{D_n^{*},2\}}
\le
\varepsilon+2\mathsf P_{*}(D_n^{*}>\varepsilon)
\convp \varepsilon
$$
by \eqref{eq:bootstrap_comp_process_remainder}. Since $\varepsilon$ is arbitrary,
$$
\mathsf E_{*}\overline{\varphi(\tilde{\G}_n^{*})}
-
\mathsf E_{*}\underline{\varphi(\tilde{\G}_n^{*})}
\convp0.
$$

We now study the asymptotics of $\widehat{\G}_{n,\mu}^{*}$. It holds that
\begin{align*}
\sup_{\omega\in\Omega,\, t\ge0}
\left|
\widehat{\G}_{n,\mu}^{*}(\omega,t)
-
\frac{m}{\tau}\sqrt n(P_n^{*}-P_n)f_{\omega,t}
\right|
&=
\sup_{\omega\in\Omega,\, t\ge0}
\left|
\frac{m}{\tau}\sqrt n(P_n^{*}-P_n)
\bigl(\hat{f}_{\omega,t}-f_{\omega,t}\bigr)
\right|
\\
&\le
A_n B_n + C_n D_n,
\end{align*}
where
\begin{align*}
A_n
&\defin
\sup_{\omega\in\Omega,\, t\ge0}
\|\dot F_\omega^{\hat{\btheta}}(t)\|\,\|\widehat{\bV}^{-1}\|,
\qquad
B_n
\defin
\left\|
\frac{m}{\tau}\sqrt n(P_n^{*}-P_n)
\bigl(\bpsi_{\hat{\btheta}}-\bpsi_{\btheta_0}\bigr)
\right\|,
\\
C_n
&\defin
\sup_{\omega\in\Omega,\, t\ge0}
\left\|
\dot F_\omega^{\hat{\btheta}}(t)^\top\widehat{\bV}^{-1}
-
\dot F_\omega^{\btheta_0}(t)^\top\bV_{\btheta_0}^{-1}
\right\|,
\qquad
D_n
\defin
\left\|
\frac{m}{\tau}\sqrt n(P_n^{*}-P_n)\bpsi_{\btheta_0}
\right\|.
\end{align*}
By condition~\ref{d}, $\sup_{\omega\in\Omega,\, t\ge0}\|\dot F_\omega^{\btheta_0}(t)\|<\infty$. Hence
\begin{align*}
\sup_{\omega\in\Omega,\, t\ge0}\|\dot F_\omega^{\hat{\btheta}}(t)\|
&\le
\sup_{\omega\in\Omega,\, t\ge0}\|\dot F_\omega^{\hat{\btheta}}(t)-\dot F_\omega^{\btheta_0}(t)\|
+
\sup_{\omega\in\Omega,\, t\ge0}\|\dot F_\omega^{\btheta_0}(t)\|
=
O_{\mathsf P}(1).
\end{align*}
Since $\widehat{\bV}$ is invertible with probability tending to one and $\widehat{\bV}\convp\bV_{\btheta_0}$, it follows that $A_n=O_{\mathsf P}(1)$. Moreover,
\begin{align*}
C_n
&\le
\sup_{\omega\in\Omega,\, t\ge0}
\|\dot F_\omega^{\hat{\btheta}}(t)-\dot F_\omega^{\btheta_0}(t)\|\,\|\widehat{\bV}^{-1}\|
+
\sup_{\omega\in\Omega,\, t\ge0}\|\dot F_\omega^{\btheta_0}(t)\|
\big\|
\widehat{\bV}^{-1}-\bV_{\btheta_0}^{-1}
\big\|
=
o_{\mathsf P}(1),
\end{align*}
so $C_n=o_{\mathsf P}(1)$.
Also, since the class $\{\psi_{\btheta_0,1},\dots,\psi_{\btheta_0,p}\}$ is finite, it is Donsker, it follows from Theorem~2.6 in \citeSM{Kosorok2008SM} that $D_n=O_{\mathsf P_*}(1)$. 

By conditions \ref{cp1} and \ref{cp4}, and the law of large numbers,
\begin{align*}
P_n\left\|\bpsi_{\hat{\btheta}}-\bpsi_{\btheta_0}\right\|^2
\le
\|\hat{\btheta}-\btheta_0\|^2 P_n L^2
=
o_{\mathsf P}(1)O_{\mathsf P}(1)
=
o_{\mathsf P}(1).
\end{align*}
Conditioning on the data, define
$$
h(x)
\defin
\bpsi_{\hat{\btheta}}(x)-\bpsi_{\btheta_0}(x),
\qquad
\bar h_n
\defin
P_n h.
$$
Let $E_n\defin\{\bar\zeta_n\ge m/2\}$. Since $\mathsf E_*(\bar\zeta_n)=m$ and $\mathsf{Var}_*(\bar\zeta_n)=\tau^2/n$, Chebyshev's inequality gives $\mathsf P_*(E_n^c) \le 4\tau^2/(m^2n).$
Moreover,
$$
\sum_{i=1}^n(\zeta_i-\bar\zeta_n)h(X_i)
=
\sum_{i=1}^n(\zeta_i-m)\bigl(h(X_i)-\bar h_n\bigr).
$$
Hence, conditional on the data,
\begin{align*}
\mathsf E_*\bigl(1_{E_n} B_n^2\bigr)
\le
\frac{4}{\tau^2}\,
\mathsf E_*\left\|
\frac{1}{\sqrt n}\sum_{i=1}^n
(\zeta_i-\bar\zeta_n)h(X_i)
\right\|^2
=
4P_n\|h-\bar h_n\|^2
\le
4P_n\|h\|^2.
\end{align*}
Therefore, for every $\varepsilon>0$,
\begin{align*}
\mathsf P_*(B_n>\varepsilon)
\le
\mathsf P_*(E_n^c)
+
\varepsilon^{-2}\mathsf E_*\bigl(1_{E_n} B_n^2\bigr)
\le
\frac{4\tau^2}{m^2n}
+
4\varepsilon^{-2}
P_n\|\bpsi_{\hat{\btheta}}-\bpsi_{\btheta_0}\|^2
=
o_{\mathsf P}(1).
\end{align*}

Since $A_n$ and $C_n$ depend only on the data, %
then for every $\varepsilon,M>0$,
\begin{align*}
\mathsf P_*(|A_nB_n|>\varepsilon)
&\le 1_{\{|A_n|>M\}}+\mathsf P_*(|B_n|>\varepsilon/M),
\\
\mathsf P_*(|C_nD_n|>\varepsilon)
&\le 1_{\{|C_n|>\varepsilon/M\}}+\mathsf P_*(|D_n|>M).
\end{align*}
Since $A_n=O_{\mathsf{P}}(1)$, $B_n=o_{P_*}(1)$, $C_n=o_{\mathsf{P}}(1)$, and $D_n=O_{P_*}(1)$, both right-hand sides converge to zero in probability, for $M$ sufficiently large.
Consequently,
\begin{align*}
\sup_{\omega\in\Omega,\, t\ge0}
\left|
\widehat{\G}_{n,\mu}^{*}(\omega,t)
-
\frac{m}{\tau}\sqrt n(P_n^{*}-P_n)f_{\omega,t}
\right|
=
o_{\mathsf P_*}(1),
\end{align*}
and therefore also $\widehat{\G}_{n,\mu}^{*}\convpboot{}\tilde{\G}_0$.
This concludes the proof.
\end{proof}

\begin{proof}[Proof of Theorem~\ref{thm:bootstrap_test_statistics}]
The proof is divided into \ref{thm:bootstrap_test_statistics_i} and \ref{thm:bootstrap_test_statistics_ii}.

\emph{Proof of \ref{thm:bootstrap_test_statistics_i}.} The asymptotic distribution of the KS statistic follows from Theorem~\ref{thm:bootstrap_processes}\ref{thm:bootstrap_processes_i} by the same argument used in the proof of Theorem~\ref{thm:test_statistics_simp}.

For the CM statistic under the simple null hypothesis, the difference with the proof of Theorem~\ref{thm:test_statistics_simp} is that the convergence $\G_{n,\mu}^{\mu_0}\convl\G_0$ is replaced by
$\G_n^{*}\convpboot{}\G_0$ in $\ell^\infty(\cY).$
The only point that is not verbatim is the application of Skorohod's construction. To justify it, consider the sequence $(\G_n^{*})_{n\ge1}$ as random elements in $\ell^\infty(\cY)$, and let $(n_j)_{j\ge1}$ be an arbitrary subsequence of positive integers. There exists a further subsequence $(n_{j_k})_{k\ge1}$ such that, for almost every realization of the data,
$$
\sup_{\varphi\in\mathrm{BL}_1(\ell^\infty(\cY))}
\left|
\mathsf E_{*}\varphi(\G_{n_{j_k}}^{*})
-
\mathsf E\varphi(\G_0)
\right|
\to0.
$$
Fix a realization of the data for which this convergence holds and for which Lemma~\ref{ap:lem_DCT_P_hat} applies along the same subsequence. Conditionally on this data sequence, the law of $\G_{n_{j_k}}^{*}$ converges weakly to the law of $\G_0$ in $\ell^\infty(\cY)$. Therefore, the same Skorohod construction used in the proof of Theorem~\ref{thm:test_statistics_simp} can be applied. From this point, the same argument as in the proof of Theorem~\ref{thm:test_statistics_simp} gives, conditionally on the fixed data realization, $T_{n_{j_k}}^{*\mathrm{CM}}\convl T_0$, where
$$
T_0
\defin
\int_{\Omega\times\R_{\ge0}}\G_0(\omega,t)^2\,\rd F_{\omega}^{\mu_0}(t)\,\rd\mu_0(\omega).
$$
Equivalently,
\begin{align*}
\sup_{\varphi\in\mathrm{BL}_1(\R)}
\left|
\mathsf E_{*}\varphi(T_{n_{j_k}}^{*\mathrm{CM}})
-
\mathsf E\varphi(T_0)
\right|
\to0.
\end{align*}
This result holds almost surely with respect to the data. Since the original subsequence $(n_{j})_{j\ge1}$ was arbitrary, every subsequence of the bounded Lipschitz distance has a further subsequence converging to zero almost surely. Therefore, 
\begin{align*}
\sup_{\varphi\in\mathrm{BL}_1(\R)}
\left|
\mathsf E_{*}\varphi(T_{n}^{*\mathrm{CM}})
-
\mathsf E\varphi(T_0)
\right|
\convp0.
\end{align*}
Additionally, observe that conditionally on the data, $T_{n}^{*\mathrm{CM}}$ is a finite sum of measurable functions of the multipliers. Hence, for every $\varphi\in\mathrm{BL}_1(\R)$, the asymptotic measurability condition 
$$
\mathsf E_{*}\overline{\varphi(T_{n}^{*\mathrm{CM}})}
-
\mathsf E_{*}\underline{\varphi(T_{n}^{*\mathrm{CM}})}
=0
$$
is automatically satisfied. Hence,
$$
T_n^{*\mathrm{CM}}
\convpboot{}
\int_{\Omega\times\R_{\ge0}}\G_0(\omega,t)^2\,\rd F_{\omega}^{\mu_0}(t)\,\rd\mu_0(\omega).
$$

\emph{Proof of \ref{thm:bootstrap_test_statistics_ii}.} For the KS statistic, the result follows from Theorem~\ref{thm:bootstrap_processes}\ref{thm:bootstrap_processes_ii} using the same argument as in the proof of Theorem~\ref{thm:test_statistics_comp}.

The asymptotic distribution of the CM statistic follows by the same argument as in the proof of \ref{thm:bootstrap_test_statistics_i}, replacing Theorem~\ref{thm:test_statistics_simp} by Theorem~\ref{thm:test_statistics_comp}, and using $\tilde{\G}_n^{*}\convpboot{}\tilde{\G}_0$ in $\ell^\infty(\cF),$ which follows from Theorem~\ref{thm:bootstrap_processes}\ref{thm:bootstrap_processes_ii}. Conditionally on the data, $\hat{\btheta}^{*}$ is measurable as a function of the multipliers and, for every $(\omega,t)$, the map $\btheta\mapsto F_\omega^{\mu_{\btheta}}(t)$ is Borel measurable. Hence, $\tilde T_n^{*\mathrm{CM}}$ is a finite sum of measurable functions of the multipliers. Therefore, the asymptotic measurability condition
$$
\mathsf E_{*}\overline{\varphi(\tilde T_n^{*\mathrm{CM}})} - \mathsf E_{*}\underline{\varphi(\tilde T_n^{*\mathrm{CM}})} =0
$$
for every $\varphi\in\mathrm{BL}_1(\R)$ holds.

The same argument as above gives the convergence of $\widehat T_n^{*\mathrm{KS}}$ and $\widehat T_n^{*\mathrm{CM}}$ from $\widehat{\G}_{n,\mu}^{*}\convpboot{}\tilde{\G}_0$.
\end{proof}

\section{Proofs for specific parametric models}\label{ap:sec_specific_parametric_models}

\subsection{Proof of Proposition \ref{prop:bahadur_exponential_family}}
\label{ap:sec_bahadur_exponential_family}

\begin{proof}[Proof of Proposition \ref{prop:bahadur_exponential_family}]
Since $\mathcal H$ is open, Theorem~5.8 and equations~(5.14)--(5.15) of  \citetSM{Lehmann1998SM} prove that $A(\boldsymbol{\eta})$ has derivatives of all orders with respect to $\boldsymbol{\eta}$, and 
$$
\frac{\partial A(\boldsymbol{\eta})}{\partial\boldsymbol{\eta}}
=
\E{\bT(X)},
\qquad
\frac{\partial^2 A(\boldsymbol{\eta})}
{\partial\boldsymbol{\eta}\,\partial\boldsymbol{\eta}^{\top}}
=
\Var{\bT(X)}
=
\mathcal I_{\boldsymbol{\eta}},
\qquad
\boldsymbol{\eta}\in\mathcal H,
$$
for $X\sim p_{\boldsymbol\eta}$. The score function is therefore $\bpsi_{\boldsymbol{\eta}}(x) = \bT(x)-\partial A(\boldsymbol{\eta})/\partial\boldsymbol{\eta}$.

We verify conditions \ref{cp1}--\ref{cp5}. Let $K\subset \mathcal H$ be a closed ball centered at
$\boldsymbol{\eta}_0$. Since $\boldsymbol{\eta}\mapsto\mathcal I_{\boldsymbol{\eta}}$ is continuous, then $C_K \defin
\sup_{\boldsymbol{\eta}\in K}
\|\mathcal I_{\boldsymbol{\eta}}\|
<
\infty.$
Thus, by the mean value theorem, for every
$\boldsymbol{\eta}_1,\boldsymbol{\eta}_2\in K$,
$$
\big\|
\bpsi_{\boldsymbol{\eta}_1}(x)
-
\bpsi_{\boldsymbol{\eta}_2}(x)
\big\|
=
\Bigg\|
\bigg.
\frac{\partial A(\boldsymbol\eta)}
{\partial\boldsymbol\eta}
\bigg|_{\boldsymbol\eta=\boldsymbol\eta_1}
-
\bigg.
\frac{\partial A(\boldsymbol\eta)}
{\partial\boldsymbol\eta}
\bigg|_{\boldsymbol\eta=\boldsymbol\eta_2}
\Bigg\|
\le
C_K\|\boldsymbol{\eta}_1-\boldsymbol{\eta}_2\|,
$$
so condition \ref{cp1} holds with $L(x)=C_K$. 

Moreover, if $X\sim p_{\boldsymbol{\eta}_0}$, then
$
\Ebig{\|\bpsi_{\boldsymbol{\eta}_0}(X)\|^2}
=
\sum_{j=1}^p
\Var{T_j(X)}
<\infty,
$
which verifies condition \ref{cp2}. Also,
$$
\E{
\bpsi_{\boldsymbol{\eta}}(X)
}
=
\left.
\frac{\partial A(\boldsymbol{\eta})}{\partial\boldsymbol{\eta}}
\right|_{\boldsymbol{\eta}=\boldsymbol{\eta}_0}
-
\frac{\partial A(\boldsymbol{\eta})}{\partial\boldsymbol{\eta}},
$$
which is zero at $\boldsymbol{\eta}_0$ and has derivative
$-\mathcal I_{\boldsymbol{\eta}_0}$ at $\boldsymbol{\eta}_0$. Hence,
Condition \ref{cp3} follows because
$\mathcal I_{\boldsymbol{\eta}_0}$ is invertible.

It remains to verify conditions \ref{cp4} and \ref{cp5}. First, show that the components of $\bT(X)$ are affinely independent a.s. If $\boldsymbol{a}^{\top}\bT(X)$ were constant a.s. for some $\boldsymbol{a}\neq\boldsymbol{0}$, then $\boldsymbol{a}^{\top}\Var{\bT(X)}\boldsymbol{a}%
=0,$ but $\mathcal I_{\boldsymbol{\eta_0}}=\Var{\bT(X)}$ is non-singular. Hence, Example~6.3 of \citetSM{Lehmann1998SM} applies, and so the conditions of Theorem~5.1 in ibid are satisfied. Therefore, with probability tending to one as $n \to \infty$, the likelihood equations admit a consistent solution. Moreover, Example~6.3 shows that this solution is unique and coincides with the MLE, since the log-likelihood is strictly concave in $\boldsymbol{\eta}$. Hence, conditions \ref{cp4} and \ref{cp5} hold.
\end{proof}

\subsection{Proof of Proposition \ref{prop:verification_assumptions_exponential_family}}

\begin{proof}[Proof of Proposition \ref{prop:verification_assumptions_exponential_family}]
Since $0\le y_{\omega,t}\le1$, for every $\boldsymbol{\eta}\in\mathcal H$, %
$$
\int
y_{\omega,t}(x)h(x)\exp\{\boldsymbol{\eta}^{\top}\bT(x)\}
\,\rd\nu(x)
\le
\int
h(x)\exp\{\boldsymbol{\eta}^{\top}\bT(x)\}
\,\rd\nu(x)
=
\exp\{A(\boldsymbol{\eta})\}
<\infty.
$$
Hence, Theorem~5.8 of \citetSM{Lehmann1998SM} allows differentiation under the integral sign. This gives
$$
\dot F_{\omega}^{\boldsymbol{\eta}}(t)
=
\E{
y_{\omega,t}(X)\bpsi_{\boldsymbol{\eta}}(X)
}, \quad
\ddot F_{\omega}^{\boldsymbol{\eta}}(t)
=
\E{
y_{\omega,t}(X)
\bpsi_{\boldsymbol{\eta}}(X)
\bpsi_{\boldsymbol{\eta}}(X)^\top
}
-
F_{\omega}^{\boldsymbol{\eta}}(t)\mathcal I_{\boldsymbol{\eta}}.
$$
Let $K\subset\mathcal H$ be a closed ball centered at $\boldsymbol{\eta}_0$. It holds that
$$
\sup_{\boldsymbol{\eta}\in K}
\sup_{\omega\in\Omega,\,t\ge0}
\big\|
\ddot F_{\omega}^{\boldsymbol{\eta}}(t)
\big\|
\le
\sup_{\boldsymbol{\eta}\in K}
\{
\operatorname{tr}(\mathcal I_{\boldsymbol{\eta}})
+
\|\mathcal I_{\boldsymbol{\eta}}\|
\}
<\infty,
$$
where the last inequality follows from the continuity of $\boldsymbol{\eta}\mapsto\mathcal I_{\boldsymbol{\eta}}$ and the compactness of $K$. Therefore, Taylor's expansion gives
$$
\sup_{\omega\in\Omega,\,t\ge0}
\left|
F_{\omega}^{\boldsymbol{\eta}}(t)
-
F_{\omega}^{\boldsymbol{\eta}_0}(t)
-
\dot F_{\omega}^{\boldsymbol{\eta}_0}(t)^\top
(\boldsymbol{\eta}-\boldsymbol{\eta}_0)
\right|
=
O\left(\|\boldsymbol{\eta}-\boldsymbol{\eta}_0\|^2\right),
$$
as $\boldsymbol{\eta}\to\boldsymbol{\eta}_0$, which proves condition~\ref{d}.
\end{proof}

\subsection{Proof of Proposition \ref{prop:covariance_exponential_family}}\label{ap:sec_covariance_exponential_family}

\begin{proof}[Proof of Proposition \ref{prop:covariance_exponential_family}]
The proof of Proposition~\ref{prop:bahadur_exponential_family} gives $\bV_{\boldsymbol{\eta}_0}
=
-\mathcal I_{\boldsymbol{\eta}_0}.
$
Additionally, since
$
\boldsymbol{\psi}_{\boldsymbol{\eta}_0}(X)
=
\boldsymbol{T}(X)
-
\mathbb E_{\boldsymbol{\eta}_0}
\left[\boldsymbol{T}(X)\right]
$
and
$
\operatorname{Var}_{\boldsymbol{\eta}_0}
\left[\boldsymbol{T}(X)\right]
=
\mathcal I_{\boldsymbol{\eta}_0},
$
it follows that
$
\Ebig{
\bpsi_{\boldsymbol{\eta}_0}(X)
\bpsi_{\boldsymbol{\eta}_0}(X)^\top
}
=
\mathcal I_{\boldsymbol{\eta}_0}.
$
Moreover, Proposition~\ref{prop:verification_assumptions_exponential_family} gives
$
\E{
y_{\omega,t}(X)\bpsi_{\boldsymbol{\eta}_0}(X)
}
=
\dot F_{\omega}^{\boldsymbol{\eta}_0}(t).
$
Substituting these identities into \eqref{eq:cov_comp}, and using that
$\mathcal I_{\boldsymbol{\eta}_0}$ is symmetric, gives
$$
\begin{aligned}
\widetilde{\mathcal C}_{(\omega_1,t_1),(\omega_2,t_2)}
&=
\mathcal C_{(\omega_1,t_1),(\omega_2,t_2)}
-
2\dot F_{\omega_1}^{\boldsymbol{\eta}_0}(t_1)^\top
\mathcal I_{\boldsymbol{\eta}_0}^{-1}
\dot F_{\omega_2}^{\boldsymbol{\eta}_0}(t_2)
+
\dot F_{\omega_1}^{\boldsymbol{\eta}_0}(t_1)^\top
\mathcal I_{\boldsymbol{\eta}_0}^{-1}
\dot F_{\omega_2}^{\boldsymbol{\eta}_0}(t_2),
\end{aligned}
$$
which yields the result.
\end{proof}

\section{Proofs of Lemmas}\label{ap:sec_proofs_lemmas}

\begin{proof}[Proof of Lemma \ref{ap:lem_bound_int_g_epsilon}]
    Observe that $M < \infty$, since $\Omega$ is totally bounded, which is implied by condition \ref{b1}.
    Thus, $F_{\omega}^{\mu_0}([0, M]) = F_{n,\omega}^{\mu_0}([0, M])=1$ and for any measurable function $\phi$,
    $\int_{\R_{\ge0}} \phi(t) \, \rd F_{\omega}^{\mu_0}(t) = \int_{[0, M]} \phi(t) \, \rd F_{\omega}^{\mu_0}(t)$ (similarly for $F_{n,\omega}^{\mu_0}$).
    Since $g$ is uniformly continuous on $\Omega \times \R_{\ge0}$, it is also uniformly continuous on $\Omega\times [0, M]$.

    We first show that $g_\varepsilon(\omega, \cdot)$ is Lipschitz with Lipschitz constant
    $\mathrm{Lip}(g_{\varepsilon}) = 2\|g\|_\infty/\delta$. Indeed, let $t,t^\prime\in[0,M]$ with $t<t^\prime$. For every $s\in[0,M]$, one has
    $$
    g(\omega,s)+\frac{2\|g\|_\infty}{\delta}|t-s|
    \le
    g(\omega,s)+\frac{2\|g\|_\infty}{\delta}|t^\prime-s|
    +\frac{2\|g\|_\infty}{\delta}|t-t^\prime|,
    $$
    and taking the infimum over $s$ yields
    $
    g_\varepsilon(\omega,t)
    \le
    g_\varepsilon(\omega,t^\prime)
    +2\|g\|_\infty/\delta|t-t^\prime|.
    $
    By symmetry, one also has $g_\varepsilon(\omega, t^\prime) \le g_\varepsilon(\omega, t) + 2\|g\|_\infty/\delta |t - t^\prime|$, so that $|g_\varepsilon(\omega, t) - g_\varepsilon(\omega, t^\prime)| \le 2\|g\|_\infty/\delta |t - t^\prime|$. 

    Fix $\omega\in\Omega$. Since $g_\varepsilon(\omega,\cdot)$ is Lipschitz on $[0,M]$, it is absolutely continuous and differentiable almost everywhere, with $|g_\varepsilon'(\omega,t)|\le\mathrm{Lip}(g_\varepsilon)$. An integration by parts argument gives
    \begin{align}\label{ap:eq_bound_d_BL_3}
    \left| \int_{0}^{M} g_\varepsilon(\omega,t)\, \rd(F_{n,\omega}^{\mu_0}-F_{\omega}^{\mu_0})(t) \right| \notag
    &=
    \left| -\int_{0}^{M} g_\varepsilon'(\omega,t)\,
    \big(F_{n,\omega}^{\mu_0}(t)-F_{\omega}^{\mu_0}(t)\big)\, \rd t \right| \\ \notag
    &\le
    \int_{0}^{M} |g_\varepsilon'(\omega,t)|\,
    |F_{n,\omega}^{\mu_0}(t)-F_{\omega}^{\mu_0}(t)|\, \rd t \\
    &\le
     \frac{2 M \|g\|_\infty}{\delta}
    \sup_{t\in[0,M]}|F_{n,\omega}^{\mu_0}(t)-F_{\omega}^{\mu_0}(t)|.
    \end{align}
    This holds almost surely and concludes the proof.
\end{proof}

\begin{proof}[Proof of Lemma \ref{ap:lem_sup_r_n_r}]
    Since $M=\operatorname{diam}(\Omega)$, observe that
    $$
    r(\omega) = \int_{0}^{M} g(\omega, t) \, \rd F_{\omega}^{\mu_0}(t), \quad r_n(\omega) = \int_{0}^{M} g(\omega, t) \, \rd F_{n,\omega}^{\mu_0}(t).
    $$
Also, one has
\begin{equation}\label{ap:eq_h_n}
    |r_n(\omega) - r(\omega)| \le \left| \int_{0}^{M} (g(\omega, t) - g_\varepsilon(\omega, t)) \, \rd (F_{n,\omega}^{\mu_0} - F_{\omega}^{\mu_0})(t) \right| + \left| \int_{0}^{M} g_\varepsilon(\omega, t) \, \rd (F_{n,\omega}^{\mu_0} - F_{\omega}^{\mu_0})(t) \right|.
\end{equation}
Begin with the first term on the RHS of \eqref{ap:eq_h_n}.    
    Fix $\varepsilon > 0$, and define $g_\varepsilon$ as in \eqref{ap:eq_g_epsilon}. 
    If one takes $s = t$ in the definition of $g_\varepsilon(\omega, t)$, one has $g_\varepsilon(\omega, t) \le g(\omega, t)$ for every $(\omega, t) \in \Omega \times [0, M]$. 
    To obtain a lower bound for $g_\varepsilon(\omega, t)$, we consider two cases. If $|t-s|<\delta$, by uniform continuity of $g$, one has
    $g(\omega,s)\ge g(\omega,t)-\varepsilon$, so,
    $$
    g(\omega,s)+\frac{2\|g\|_\infty}{\delta}|t-s|\ge g(\omega,t)-\varepsilon.
    $$
    If
    $|t-s|\ge\delta$, then
    $$
    g(\omega,s)+\frac{2\|g\|_\infty}{\delta}|t-s|
    \ge
    -\|g\|_\infty+2\|g\|_\infty
    =
    \|g\|_\infty
    \ge
    g(\omega,t)
    \ge
    g(\omega,t)-\varepsilon. 
    $$
Taking infima over $s$ yields $g_\varepsilon(\omega, t) \ge g(\omega, t) - \varepsilon$ for every $(\omega, t) \in \Omega \times [0, M]$.    
    Therefore, one has $|g_\varepsilon(\omega, t) - g(\omega, t)| \le \varepsilon \text{ for every } (\omega, t) \in \Omega \times [0, M]$. This provides an upper bound for the first term on the RHS of \eqref{ap:eq_h_n}:
    $$
    \left| \int_{0}^{M} (g(\omega, t) - g_{\varepsilon}(\omega, t)) \, \rd (F_{n,\omega}^{\mu_0} - F_{\omega}^{\mu_0})(t) \right| \le \varepsilon \left|\mu_{\omega,n} - \mu_{\omega,0}\right|([0,M])
    \le \varepsilon \left|\gamma\right|(\R) \le 2\varepsilon
    $$
    where $|\gamma|$ denotes the total variation of a finite signed measure $\gamma = \mu_{\omega,n} - \mu_{\omega,0}$. Here, $\mu_{\omega,n}$ and $\mu_{\omega,0}$ are the measures induced by $F_{n,\omega}$ and $F_{\omega}^{\mu_0}$, respectively.
     The total variation of a measurable subset $E$ is, in particular, 
    $$
    |\gamma|(E)=\sup_{\text{Partitions $E_1,\ldots,E_m$ of $E$}}\lrb{\sum_{i=1}^m|\gamma(E_i)|}
    $$
    \citepSM[Exercise 8 in][]{Folland1999SM}. That $|\gamma|(\R) \le 2$ follows easily from the calculation
\begin{align*}
    |\gamma|(\R)&=\sup \lrb{\sum_{i=1}^m|\gamma(E_i)|}=\sup \lrb{\sum_{i=1}^m|\mu_{\omega,n}(E_i)-\mu_{\omega,0}(E_i)|}\\
    &\leq \sup \lrb{\sum_{i=1}^m|\mu_{\omega,n}(E_i)+\mu_{\omega,0}(E_i)|}= \sup \lrb{\sum_{i=1}^m\mu_{\omega,n}(E_i)+\sum_{i=1}^m\mu_{\omega,0}(E_i)}\\
    &= \sup \lrb{1+1}=2.
\end{align*}
    Now, focus on the second term on the RHS of \eqref{ap:eq_h_n}. Since the class $\cY = \{ y_{\omega,t} : (\omega,t)\in\Omega\times\mathbb R_{\ge0} \}$ is Donsker and $\G_0$ is separable, %
we obtain the Glivenko--Cantelli property:
$$
\sup_{(\omega,t)\in\Omega\times\R_{\ge0}} \big|F_{n,\omega}^{\mu_0}(t)-F_\omega^{\mu_0}(t)\big|
=
\sup_{y\in\cY} \big|(P_n - \mu_0)y\big|
\convas 0,
$$
see p. 130 in \citeSM{VanderVaartWellner2023SM}.
Thus, applying Lemma \ref{ap:lem_bound_int_g_epsilon} and taking the supremum over $\omega \in \Omega$ in \eqref{ap:eq_bound_d_BL_3} shows that the second term on the RHS of \eqref{ap:eq_h_n} converges to zero as $n \to \infty$. Therefore, we can conclude that
    $$
    \sup_{\omega \in \Omega} |r_n(\omega) - r(\omega)| \le 3\varepsilon,
    $$ 
    for $n$ sufficiently large such that $2 M \|g\|_\infty/\delta \sup_{\omega\in\Omega}\sup_{t\in[0,M]} |F_{n,\omega}^{\mu_0}(t) - F_{\omega}^{\mu_0}(t)| \le \varepsilon$.  Since $\varepsilon > 0$ is arbitrary, this shows that $\sup_{\omega \in \Omega} |r_n(\omega) - r(\omega)| \to 0$ a.s. as $n \to \infty$.
\end{proof}

\begin{proof}[Proof of Lemma \ref{ap:lem_continuity_r}]
    By the definition of expectation, %
    one has 
    $$
    r(\omega) = \int_{0}^{M} g(\omega,t) \, \rd F_{\omega}^{\mu_0}(t) = \E{g(\omega, d(\omega, X))}.
    $$

    Consider a sequence $(\omega_k)_{k\ge 1}$ with $d(\omega_k, \omega_0) \to 0$ as $k \to \infty$. We want to prove $\lim_{k \to \infty} \E{g(\omega_k, d(\omega_k, X))} =  \E{g(\omega_0, d(\omega_0, X))}$. Since $d(\omega_k, x) \to d(\omega_0, x)$ for every $x \in \Omega$ and $g$ is continuous on $\Omega \times [0, M]$, it holds that $g(\omega_k, d(\omega_k, X)) \convas g(\omega_0, d(\omega_0, X))$ as $k\to \infty$. Now, since $g$ is bounded, an application of the dominated convergence theorem gives
    $$
    \lim_{k \to \infty} r(\omega_k) = \lim_{k \to \infty} \E{g(\omega_k, d(\omega_k, X))} = \E{g(\omega_0, d(\omega_0, X))} = r(\omega_0).
    $$
    Thus, $r$ is continuous on $\Omega$. 
\end{proof}

\begin{proof}[Proof of Lemma \ref{ap:lem_DCT_P_hat}]
    One has
    \begin{align*}
    \bigg|
    \int_\Omega r_n(\omega)\,\rd P_n(\omega)
    &-
    \int_\Omega r(\omega)\,\rd\mu_0(\omega)
    \bigg|
    \\
    &\le
    \bigg|
    \int_\Omega \big(r_n(\omega)-r(\omega)\big)\,\rd P_n(\omega)
    \bigg|
    +
    \bigg|
    \int_\Omega r(\omega)\,\rd\big(P_n(\omega)-\mu_0(\omega)\big)
    \bigg|
    \\
    &\le
    \sup_{\omega\in\Omega}|r_n(\omega)-r(\omega)|
    +
    \bigg|
    \int_\Omega r(\omega)\,\rd\big(P_n(\omega)-\mu_0(\omega)\big)
    \bigg|
    \to 0
    \end{align*}
    as $n \to \infty$. The convergence follows by Lemma \ref{ap:lem_sup_r_n_r} for the first term. For the second term, convergence follows by the weak convergence of $P_n$ to $\mu_0$ and the continuity of $r$ by Lemma \ref{ap:lem_continuity_r}. For the weak convergence of $P_n$ to $\mu_0$, see Theorems 1 and 3 of \citeSM{Varadarajan1958SM}.
\end{proof}

\begin{proof}[Proof of Lemma~\ref{ap:lem_relation_L2_prod}]

To prove the first inequality in \eqref{eq:rho2_product_bound_simple}, recall the covariance structure \eqref{eq:cov_G_theta_0}:
$$
\Cov{\G_0(\omega,t)}{\G_0(\omega^\prime,t^\prime)} 
=
\mu(y_{\omega,t}y_{\omega^\prime,t^\prime})-\mu y_{\omega,t}\,\mu y_{\omega^\prime,t^\prime}.
$$
One has
\begin{align*}
\rho_{2,\G_0}\big((\omega,t),(\omega^\prime,t^\prime)\big)^2
&=
\Ebig{
\big(\G_0(\omega,t)-\G_0(\omega^\prime,t^\prime)\big)^2
}
\\
&=
\mu\Big(
\big(y_{\omega,t}-y_{\omega^\prime,t^\prime}\big)^2
\Big)
-
\Big(
\mu\big(y_{\omega,t}-y_{\omega^\prime,t^\prime}\big)
\Big)^2
\\
&\le
\mu\Big(
\big(y_{\omega,t}-y_{\omega^\prime,t^\prime}\big)^2
\Big)
=
\rho_{2,y}\big((\omega,t),(\omega^\prime,t^\prime)\big)^2,
\end{align*}
which yields the desired inequality.

Now, we prove the second inequality in \eqref{eq:rho2_product_bound_simple}. 
Observe that since $y_{\omega, t}$ and $y_{\omega^\prime, t^\prime}$ are indicator functions, one has $(y_{\omega, t} - y_{\omega^\prime, t^\prime})^2 = 1_{B(\omega,t)\triangle B(\omega^\prime,t^\prime)}$, where $A\triangle B \defin (A\setminus B)\cup(B\setminus A)$ denotes the symmetric difference between two sets $A$ and $B$. 
Therefore, it is verified that
$$
\rho_{2,y}\big((\omega,t), (\omega^\prime,t^\prime)\big)^2 
= \E{(y_{\omega, t} - y_{\omega^\prime, t^\prime})^2} 
= \Probbig{ B(\omega,t) \triangle B(\omega^\prime,t^\prime) }.
$$

Define $r \defin d_{\Omega \times \R_{\ge 0}}((\omega,t),(\omega^\prime,t^\prime))$. Then,
The triangle inequality gives
$$
\begin{aligned}
B(\omega,t)\setminus B(\omega^\prime,t^\prime)
&\subset
B(\omega^\prime,t^\prime+r)\setminus B(\omega^\prime,t^\prime),
\\
B(\omega^\prime,t^\prime)\setminus B(\omega,t)
&\subset
B(\omega,t+r)\setminus B(\omega,t).
\end{aligned}
$$
Therefore, condition~\ref{bp2} yields
\begin{align*}
\Probbig{B(\omega,t)\triangle B(\omega^\prime,t^\prime)}
&\le
F_{\omega^\prime}^\mu(t^\prime+r)-F_{\omega^\prime}^\mu(t^\prime)
+F_\omega^\mu(t+r)-F_\omega^\mu(t) \le 2K_\mu r.
\end{align*}
Hence,
\begin{equation}\label{eq:rho2y_prod_bound}
\rho_{2,y}\big((\omega,t),(\omega^\prime,t^\prime)\big)^2
\le
2K_\mu\,d_{\Omega\times\R_{\ge0}}\big((\omega,t),(\omega^\prime,t^\prime)\big),
\end{equation}
which gives the second inequality in \eqref{eq:rho2_product_bound_simple}.
\end{proof}

\begin{proof}[Proof of Lemma~\ref{ap:lem_continuity_sample_paths_G_rho2}]
Under the null hypothesis, the class $\cY$ is $\mu$-Donsker. Hence, the limiting process $\G_0$ is a tight centered Gaussian process in $\ell^\infty(\cY)$.

By Lemma 1.5.9 in \citetSM{VanderVaartWellner2023SM}, there exists a semimetric $\rho$ on $\cY$ such that $\cY$ is totally bounded with respect to $\rho$, and almost all sample paths of $\G_0$ are uniformly $\rho$-continuous. 
If, in addition, the map $y\in \cY \mapsto \G_0(y)$ is uniformly $\rho$-continuous in $L_2$, then Lemma 1.5.9 in \citetSM{VanderVaartWellner2023SM} (equivalence between (i) and (iii)) can be applied (with $p=2$). Therefore, $\cY$ is totally bounded with respect to $\rho_{2,\G_0}$, and almost all sample paths of $\G_0$ are uniformly $\rho_{2,\G_0}$-continuous. Here, $\rho_{2,\G_0}(y,y^\prime) = (\mathsf E|\G_0(y)-\G_0(y^\prime)|^2)^{1/2}$, which is the same as \eqref{eq:ap_distances}, but as a function of $\cY\times\cY$, rather than $(\Omega\times\R_{\ge0})\times(\Omega\times\R_{\ge0})$.
However, for every $y,y^\prime \in \cY$, there exist $(\omega,t), (\omega^\prime,t^\prime) \in \Omega\times\R_{\ge0}$ such that $y=y_{\omega,t}$ and $y^\prime = y_{\omega^\prime,t^\prime}$. Hence, almost all sample paths of $\G_0(\omega,t)$ are uniformly $\rho_{2,\G_0}$-continuous as functions from $(\Omega\times\R_{\ge0}, \rho_{2,\G_0})$ to $(\R,|\cdot|)$. It follows from Lemma~\ref{ap:lem_relation_L2_prod} that any function that is uniformly continuous with respect to $\rho_{2,\G_0}$ is also uniformly continuous with respect to $d_{\Omega\times\R_{\ge0}}$. Therefore, $\G_0$ has almost surely uniformly continuous sample paths as a function from $(\Omega\times\R_{\ge0}, d_{\Omega\times\R_{\ge0}})$ to $(\R,|\cdot|)$.

It remains to show that the map $y \in \cY \mapsto \G_0(y)$ is uniformly $\rho$-continuous in second mean. Let $(y_n)_{n\ge0}$ and $(y_n^\prime)_{n\ge0}$ be sequences in $\cY$ such that $\rho(y_n,y_n^\prime)\to0$, as $n\to\infty$.
Since almost all sample paths of $\G_0$ are uniformly $\rho$-continuous, it follows that $\G_0(y_n)-\G_0(y_n^\prime) \convas 0$, and so
$\G_0(y_n)-\G_0(y_n^\prime) \convl 0$.

For every $n$, define
$$
Z_n:=\G_0(y_n)-\G_0(y_n^\prime).
$$  
Since $\G_0$ is a centered Gaussian process, one has that $Z_n\sim \mathcal{N}(0,\sigma_n^2)$, with $\sigma_n^2=\mathsf E(Z_n^2).$
Since $Z_n \convl 0$, it holds that $\exp(-\sigma_n^2u^2/2)\to1$ for every $u\in\R,$
which implies that $\sigma_n^2\to0$. Therefore,
$$
\mathsf E\big|
\G_0(y_n)-\G_0(y_n^\prime)
\big|^2
\to0.
$$
This proves that $y\mapsto \G_0(y)$ is uniformly $\rho$-continuous in second mean.
\end{proof}

\begin{proof}[Proof of Lemma~\ref{ap:lem_relation_L2_prod_composite}]
We first prove the equality
$$
\rho_{2,\tilde{\G}_0}\big((\omega,t),(\omega^\prime,t^\prime)\big)
=
\rho_{2,f}\big((\omega,t),(\omega^\prime,t^\prime)\big).
$$
Since $\tilde{\G}_0$ is the centered Gaussian process indexed by the class $\cF$, and $\Ebig{f_{\omega,t}(X)}=0$ for every $(\omega,t)\in\Omega\times\R_{\ge0}$, its covariance is given by
$$
\Cov{\tilde{\G}_0(\omega,t)}{\tilde{\G}_0(\omega^\prime,t^\prime)}
=
\Ebig{f_{\omega,t}(X)f_{\omega^\prime,t^\prime}(X)}.
$$
Hence,
\begin{align*}
\rho_{2,\tilde{\G}_0}\big((\omega,t),(\omega^\prime,t^\prime)\big)^2
&=
\Ebig{\big(\tilde{\G}_0(\omega,t)-\tilde{\G}_0(\omega^\prime,t^\prime)\big)^2}
\nonumber
\\
&=
\Ebig{\big(f_{\omega,t}(X)-f_{\omega^\prime,t^\prime}(X)\big)^2}
\nonumber
\\
&=
\rho_{2,f}\big((\omega,t),(\omega^\prime,t^\prime)\big)^2.
\end{align*}

Now, we prove the inequality in \eqref{eq:rho_2_tilde_G_0_rho_2_f_composite}. Recall that
$$
f_{\omega,t}(x)
=
y_{\omega,t}(x)-F_{\omega}^{\mu_{\btheta_0}}(t)+g_{\omega,t}(x),
\qquad
g_{\omega,t}(x)
=
\dot F_{\omega}^{\btheta_0}(t)^\top \bV_{\btheta_0}^{-1}\bpsi_{\btheta_0}(x).
$$
Fix $(\omega,t),(\omega^\prime,t^\prime)\in\Omega\times\R_{\ge0}$, and set
$$
r\defin d_{\Omega\times\R_{\ge0}}\big((\omega,t),(\omega^\prime,t^\prime)\big).
$$
Since
\[
F_{\omega}^{\mu_{\btheta_0}}(t)-F_{\omega^\prime}^{\mu_{\btheta_0}}(t^\prime)
=
\Ebig{y_{\omega,t}(X)-y_{\omega^\prime,t^\prime}(X)},
\]
by the triangle inequality,
\begin{align*}
\rho_{2,f}\big((\omega,t),(\omega^\prime,t^\prime)\big)
&=
\Big(
\Ebig{\big(f_{\omega,t}(X)-f_{\omega^\prime,t^\prime}(X)\big)^2}
\Big)^{1/2}
\\
&\le
\big(
\Var{y_{\omega,t}(X)-y_{\omega^\prime,t^\prime}(X)}
\big)^{1/2}
 +
\Big(
\Ebig{\big(g_{\omega,t}(X)-g_{\omega^\prime,t^\prime}(X)\big)^2}
\Big)^{1/2}
\\
&\le
\rho_{2,y}\big((\omega,t),(\omega^\prime,t^\prime)\big)
+
\Big(
\Ebig{\big(g_{\omega,t}(X)-g_{\omega^\prime,t^\prime}(X)\big)^2}
\Big)^{1/2}.
\end{align*}
For the first term, Lemma~\ref{ap:lem_relation_L2_prod} gives $
\rho_{2,y}\big((\omega,t),(\omega^\prime,t^\prime)\big)
\le \sqrt{2K_{\mu_{\btheta_0}}r}.$
Finally, for the second term, observe that
$$
g_{\omega,t}(x)-g_{\omega^\prime,t^\prime}(x)
=
\big(
\dot F_{\omega}^{\btheta_0}(t)-\dot F_{\omega^\prime}^{\btheta_0}(t^\prime)
\big)^\top
\bV_{\btheta_0}^{-1}\bpsi_{\btheta_0}(x).
$$
Hence, by Cauchy--Schwarz,
$$
\Big(
\Ebig{\big(g_{\omega,t}(X)-g_{\omega^\prime,t^\prime}(X)\big)^2}
\Big)^{1/2}
\le
\big\|
\dot F_{\omega}^{\btheta_0}(t)-\dot F_{\omega^\prime}^{\btheta_0}(t^\prime)
\big\|
\,
\|\bV_{\btheta_0}^{-1}\|
\,
\Big(\Ebig{\|\bpsi_{\btheta_0}(X)\|^2}\Big)^{1/2}.
$$

Recall that by condition~\ref{cp2}, $\mathsf{E}\|\bpsi_{\btheta_0}(X)\|^2<\infty$ and by condition~\ref{cp3} the matrix $\bV_{\btheta_0}$ is non-singular. 
Define the following modulus of continuity of the map $(\omega,t)\mapsto \dot F_{\omega}^{\btheta_0}(t)$:
$$
\eta_{\dot F}(r)
\defin
\sup\Big\{
\big\|\dot F_{\omega}^{\btheta_0}(t)-\dot F_{\omega^\prime}^{\btheta_0}(t^\prime)\big\| :
(\omega,t),(\omega^\prime,t^\prime) \in \Omega \times \R_{\ge0}, \, 
d_{\Omega\times\R_{\ge0}}\big((\omega,t),(\omega^\prime,t^\prime)\big)\le r
\Big\}.
$$
Thus,
$$
\Big(
\Ebig{\big(g_{\omega,t}(X)-g_{\omega^\prime,t^\prime}(X)\big)^2}
\Big)^{1/2}
\le
\eta_{\dot F}(r)\big\|\bV_{\btheta_0}^{-1}\big\|\Big(\mathsf{E}\big\|\bpsi_{\btheta_0}(X)\big\|^2\Big)^{1/2}
$$
and by the assumed uniform continuity of $(\omega,t)\mapsto \dot F_{\omega}^{\btheta_0}(t)$,
$\eta_{\dot F}(r)\to0$ as $r\to0$.

Combining the three bounds, we obtain
$$
\rho_{2,f}\big((\omega,t),(\omega^\prime,t^\prime)\big)
\le
\sqrt{2K_{\mu_{\btheta_0}}r}
+
\|\bV_{\btheta_0}^{-1}\|
\Big(\Ebig{\|\bpsi_{\btheta_0}(X)\|^2}\Big)^{1/2}
\eta_{\dot F}(r).
$$
Thus, the result holds with
$$
\phi(r)
\defin
\sqrt{2K_{\mu_{\btheta_0}}r}
+
\|\bV_{\btheta_0}^{-1}\|
\Big(\Ebig{\|\bpsi_{\btheta_0}(X)\|^2}\Big)^{1/2}
\eta_{\dot F}(r).
$$
It holds that $\phi$ is nondecreasing and $\phi(r)\to0$ as $r\to0$, which concludes the proof.
\end{proof}

\begin{proof}[Proof of Lemma~\ref{ap:lem_continuity_sample_paths_G_rho2_comp}]
The proof is completely analogous to that of Lemma~\ref{ap:lem_continuity_sample_paths_G_rho2}, replacing the class $\cY$ by $\cF$, the Gaussian process $\G_0$ by $\tilde{\G}_0$, and Lemma~\ref{ap:lem_relation_L2_prod} by Lemma~\ref{ap:lem_relation_L2_prod_composite}.
\end{proof}

\begin{proof}[Proof of Lemma~\ref{lem:sample-delta-net}]
Fix $\varepsilon>0$. By total boundedness of $\Omega$ (implied by condition \ref{b1}) there exists a finite $\varepsilon$-net
$$
\Omega \subset \bigcup_{j=1}^m B(\omega_j,\varepsilon)
$$
for some points $\omega_1,\dots,\omega_m\in\Omega$.

Define, for each $j=1,\dots,m$, the event
$$
A_{n,j} \defin \{B(\omega_j,\varepsilon)\cap \mathcal{L}_n = \emptyset\}.
$$
For each fixed $j$ one has $\mathsf{P}\big(A_{n,j}\big) = \big(1 - \mathsf{P}\left(B(\omega_j,\varepsilon)\right)\big)^n$.
Set $c_\varepsilon \defin \min_{1\le j\le m} \mathsf{P}\big(B(\omega_j,\varepsilon)\big)$. Since $\omega_j\in\Omega\subset \mathrm{supp}(\mu)$, one has $c_\varepsilon>0$.

Consider the event 
$$
A_n \defin \Big\{\exists j\in\{1,\dots,m\}\text{ such that }B(\omega_j,\varepsilon)\cap \mathcal{L}_n=\emptyset\Big\}
= \bigcup_{j=1}^m A_{n,j}.
$$
It holds that
$$
\sum_{n=1}^\infty \mathsf{P}(A_n) \le \sum_{n=1}^\infty \sum_{j=1}^m \mathsf{P}(A_{n,j})
= \sum_{n=1}^\infty \sum_{j=1}^m \big(1 - \mathsf{P}\big(B(\omega_j,\varepsilon))\big)^n
\le \sum_{n=1}^\infty m \big(1 - c_\varepsilon\big)^n < \infty,
$$
because $1-c_\varepsilon<1$.
By the Borel--Cantelli lemma, with probability one, only finitely many of the events $A_n$ occur; equivalently, with probability one there exists a (random) integer $N = N(\varepsilon)$ such that for all $n\ge N(\varepsilon)$ no $A_{n,j}$ occurs, i.e. every ball $B(\omega_j,\varepsilon)$ contains at least one sample point. Since $\{\omega_1, \ldots, \omega_m\}$ forms an $\varepsilon$-net of $\Omega$, for every $\omega\in\Omega$ we can pick $1\le j\le m$ with $d(\omega,\omega_j)<\varepsilon$, and hence for any sample element $X_i\in B(\omega_j,\varepsilon)$, the triangle inequality gives $d(\omega,X_i) \le 2\varepsilon$ for all $\omega \in \Omega$ with probability one. 
Therefore,
$$
\Prob{\exists N(\varepsilon):\ \forall n\ge N(\varepsilon),\ 
\sup_{\omega\in\Omega}\min_{1\le i\le n} d(\omega,X_i)\le 2\varepsilon}=1.
$$
Equivalently,
$$
\Prob{\liminf_{n\to\infty}\Big\{\sup_{\omega\in\Omega}\min_{1\le i\le n} d(\omega,X_i)\le 2\varepsilon\Big\}}=1.
$$
Observe that $\sup_{\omega\in\Omega}\min_{1\le i\le n} d(\omega,X_i)$ is nonnegative and nonincreasing as a function of $n$. 
Therefore,
\begin{align}\label{eq:as_lem_dense}
\Prob{\lim_{n\to\infty}\sup_{\omega\in\Omega}\min_{1\le i\le n} d(\omega,X_i)\le 2\varepsilon}=1.
\end{align}

Finally, let $\varepsilon_k = 1/k$ for $k\in\mathbb N$. The above argument can be repeated to obtain \eqref{eq:as_lem_dense} for every $\varepsilon_k$. Making $k\to \infty$, 
this implies
that $\sup_{\omega\in\Omega}\min_{1\le i\le n} d(\omega,X_i) \to 0$ a.s. as $n\to\infty$.
\end{proof}

\begin{proof}[Proof of Lemma \ref{lem:delta_regularity}]
Fix $t\in[0,M]$ and $\omega,\omega^\prime\in\Omega$, and define $r\defin d(\omega,\omega^\prime)$. For every $x\in\Omega$, the triangle inequality yields
$$
d(\omega^\prime,x) \le d(\omega^\prime,\omega) + d(\omega,x) = r + d(\omega,x).
$$
Hence, $\{d(\omega,x)\le t\}\subset\{d(\omega^\prime,x)\le t+r\}$, and so
$F^\mu_\omega(t)\le F^\mu_{\omega^\prime}(t+r)$. Similarly, $F^\mu_{\omega^\prime}(t)\le F^\mu_{\omega}(t+r)$, so
$$
|F^\mu_\omega(t)-F^\mu_{\omega^\prime}(t)|
\le \max\{F^\mu_{\omega^\prime}(t+r)-F^\mu_{\omega^\prime}(t),\,F^\mu_{\omega}(t+r)-F^\mu_{\omega}(t)\}.
$$
By condition~\ref{bp2},
$$
F^\mu_{\omega}(t+r)-F^\mu_{\omega}(t)
\le K_\mu r,
$$
and similarly for $\omega^\prime$. Therefore,
$$
|F^\mu_\omega(t)-F^\mu_{\omega^\prime}(t)|
\le K_\mu\,d(\omega,\omega^\prime).
$$
The same argument applied to $\mu_0$ yields
$|F^{\mu_0}_\omega(t)-F^{\mu_0}_{\omega^\prime}(t)|\le K_{\mu_0}\,d(\omega,\omega^\prime)$.
Finally,
\begin{align*}
|\Delta(\omega,t)-\Delta(\omega^\prime,t)|
&\le |F^\mu_\omega(t)-F^\mu_{\omega^\prime}(t)| + |F^{\mu_0}_\omega(t)-F^{\mu_0}_{\omega^\prime}(t)|\\
&\le (K_\mu+K_{\mu_0})\, d(\omega,\omega^\prime),
\end{align*}
which gives the result.
\end{proof}

\begin{proof}[Proof of Lemma \ref{lem:robust_separation}]
Let $\delta_0\defin|\Delta(\omega_0,t_0)|>0$. 
By condition~\ref{bp2},
$$
|\Delta(\omega_0,t)-\Delta(\omega_0,t_0)|
\le (K_\mu+K_{\mu_0})|t-t_0|.
$$
Thus, taking $\varepsilon_t\defin(\delta_0/2)(K_\mu+K_{\mu_0})^{-1}$ gives
$$
|\Delta(\omega_0,t)-\Delta(\omega_0,t_0)|<\delta_0/2
\quad\text{for all $t\in[0,M]$ such that } |t-t_0|<\varepsilon_t,
$$
which implies $|\Delta(\omega_0,t)|\ge \delta_0/2$. Next, by Lemma~\ref{lem:delta_regularity}, for all $t\in[0,M]$ one has
$$
|\Delta(\omega,t)-\Delta(\omega_0,t)|
\le (K_\mu+K_{\mu_0}) d(\omega,\omega_0).
$$
Observe that the right-hand side does not depend on $t$.

Choose $\varepsilon_\omega\defin (\delta_0/4)(K_\mu+K_{\mu_0})^{-1}$ and set $c_0\defin \delta_0/4$.
Then, for all $\omega\in B(\omega_0,\varepsilon_\omega)$ and all $t\in (t_0-\varepsilon_t,t_0+\varepsilon_t)\cap[0,M]$, it holds that
\begin{align*}
|\Delta(\omega,t)|
&\ge |\Delta(\omega_0,t)| - |\Delta(\omega,t)-\Delta(\omega_0,t)|\\
&\ge \delta_0/2-(K_\mu+K_{\mu_0})\varepsilon_\omega\\
&= \delta_0/4
\defin c_0 >0.
\end{align*}
This concludes the result.
\end{proof}

\begin{proof}[Proof of Lemma \ref{lem:F_H_differ_at_t}]
Suppose that they agree on $\supp F$, and take a maximal open interval $(a,b)$ of $[0,M]\setminus\supp F$, so that $a,b\in \supp{F}$. Continuity of $F$ gives $H(a)=F(a)=F(b)=H(b).$ Monotonicity of $H$ implies that $H=F$ in $(a,b)$. An analogous argument applies to maximal intervals of the form $[0,b)$ or $(a,M]$. Thus, $F=H$ on $[0,M]$, a contradiction.
\end{proof}

\begin{proof}[Proof of Lemma \ref{lem:positive_mass_region}]
Choose $0<\eta<\min\{\varepsilon_\omega,\varepsilon_t/2\}$. Since $\omega_0\in\supp{\mu}$ and $t_0\in\supp{F_{\omega_0}^\mu}$,
$$
\mu\big(B(\omega_0,\eta)\big)>0,
\qquad
\Probbig{|d(\omega_0,Y)-t_0|<\varepsilon_t/2}>0.
$$
Independence of $X$ and $Y$ therefore gives
$$
\Probbig{X\in B(\omega_0,\eta),\ |d(\omega_0,Y)-t_0|<\varepsilon_t/2}>0.
$$
On the intersection of these two events,
$$
|d(X,Y)-t_0|
\le d(X,\omega_0)+|d(\omega_0,Y)-t_0|
<\varepsilon_t.
$$
Moreover, $X\in B_0$, and hence
$$
\Probbig{X\in B_0,\ d(X,Y)\in I_0}>0.
$$
\end{proof}

\begin{proof}[Proof of Lemma~\ref{ap:lem_DCT_P_hat_local}]
For $g\in\mathcal G_{B,w}$, define
$$
r_{n,g}(\omega) \defin \int_{[0,M]} g(\omega, t) \, \rd F_{n,\omega}^{Q_n}(t),
\qquad
r_g(\omega) \defin \int_{[0,M]} g(\omega, t) \, \rd F_{\omega}^{\mu_0}(t).
$$
Then,
\begin{equation*}
\sup_{g\in\mathcal G_{B,w}}
\left|\int_{\Omega} r_{n,g}(\omega) \, \rd P_n(\omega) - \int_{\Omega} r_g(\omega) \, \rd \mu_0(\omega)\right|
\le
A_n+B_n,
\end{equation*}
where
$$
A_n \defin \sup_{g\in\mathcal G_{B,w}}\sup_{\omega\in\Omega}|r_{n,g}(\omega)-r_g(\omega)|
$$
and
$$
B_n \defin
\sup_{g\in\mathcal G_{B,w}}
\left|
\int_{\Omega} r_g(\omega) \, \rd(P_n-\mu_0)(\omega)
\right|.
$$

We first control $A_n$. Fix $\varepsilon>0$, and choose $\delta>0$ such that $w(\delta)\le \varepsilon$. Repeating the regularization argument in the proof of Lemma~\ref{ap:lem_sup_r_n_r}, with
$$
g_\varepsilon(\omega,t)
:=
\inf_{s\in[0,M]}
\left\{
g(\omega,s)+\frac{2B}{\delta}|t-s|
\right\},
$$
gives, uniformly over $g\in\mathcal G_{B,w}$,
$$
A_n
\le
C\varepsilon
+
C_{B,\delta,M}
\sup_{\omega\in\Omega, t\in[0,M]}
\left|
F_{n,\omega}^{Q_n}(t)-F_\omega^{\mu_0}(t)
\right|,
$$
for constants $C$ and $C_{B,\delta,M}$ independent of $n$. It remains to show that
\begin{equation}\label{eq:local_alt_DCT_uniform_profiles}
\sup_{\omega \in\Omega, t\in[0,M]}
\left|F_{n,\omega}^{Q_n}(t)-F_\omega^{\mu_0}(t)\right|
\convp 0.
\end{equation}
The decomposition
$$
F_{n,\omega}^{Q_n}(t)-F_\omega^{\mu_0}(t)
=
(P_n-Q_n)y_{\omega,t}+(Q_n-\mu_0)y_{\omega,t}
$$
gives
$$
\sup_{\omega \in\Omega, t\in[0,M]}
\left|F_{n,\omega}^{Q_n}(t)-F_\omega^{\mu_0}(t)\right|
\le
\frac{1}{\sqrt{n}}
\sup_{\omega \in\Omega, t\in[0,M]}
\left|\G_{n,Q_n}^{Q_n}(\omega,t)\right|
+
\frac{1}{\sqrt{n}}
\sup_{\omega \in\Omega, t\in[0,M]}
\left|(G-\mu_0)y_{\omega,t}\right|.
$$
By Corollary~\ref{cor:bernoulli} applied to $\cY$, one has $\G_{n,Q_n}^{Q_n}\convl \G_0$ in $\ell^\infty(\cY)$, and therefore
$$
\sup_{\omega \in\Omega, t\in[0,M]}
\left|\G_{n,Q_n}^{Q_n}(\omega,t)\right|
=
O_{\mathsf{P}}(1).
$$
Moreover, one has $\sup_{\omega,t}|(G-\mu_0)y_{\omega,t}|\le 2$. Thus, \eqref{eq:local_alt_DCT_uniform_profiles} holds and hence $A_n=o_{\mathsf{P}}(1)$.

We next control $B_n$. Write
\begin{align}\label{eq:local_alt_DCT_Bn_split}
B_n
\le
\sup_{g\in\mathcal G_{B,w}} |(P_n-Q_n)r_g|
+
\sup_{g\in\mathcal G_{B,w}} |(Q_n-\mu_0)r_g|.
\end{align}
Since $|r_g(\omega)|\le B$ for every $g\in\mathcal G_{B,w}$ and $\omega\in\Omega$, one has
$$
\sup_{g\in\mathcal G_{B,w}} |(Q_n-\mu_0)r_g|
\le
\frac{2B}{\sqrt{n}}.
$$

For the first term in \eqref{eq:local_alt_DCT_Bn_split}, note that
$$
r_g(\omega)=\int_\Omega g(\omega,d(\omega,x))\,\rd\mu_0(x).
$$
Hence, for $\omega,\omega^\prime\in\Omega$,
$$
|r_g(\omega)-r_g(\omega^\prime)|
\le
\int_\Omega
\left|
g(\omega,d(\omega,x))-g(\omega^\prime,d(\omega^\prime,x))
\right|\,\rd\mu_0(x).
$$
The triangle inequality gives
$$
d_{\Omega\times\mathbb R_{\ge0}}
\big((\omega,d(\omega,x)),(\omega^\prime,d(\omega^\prime,x))\big)
\le 2d(\omega,\omega^\prime),
$$
and therefore
$$
|r_g(\omega)-r_g(\omega^\prime)|
\le
w\big(2d(\omega,\omega^\prime)\big),
$$
uniformly in $g\in\mathcal G_{B,w}$. Thus, the family $\{r_g:g\in\mathcal G_{B,w}\}$ is uniformly bounded by $B$ and equicontinuous with common modulus $u\mapsto w(2u)$.

Fix $\eta>0$. Choose $\delta>0$ such that $w(2\delta)\le \eta/3$, and let $\omega_1,\dots,\omega_m$ be a finite $\delta$-net of $\Omega$. Consider the set
$$
S:=\left\{
\big(r_g(\omega_1),\dots,r_g(\omega_m)\big):g\in\mathcal G_{B,w}
\right\}\subset[-B,B]^m.
$$
Since $S$ is totally bounded, it admits a finite $\eta/3$-net with centers belonging to $S$. Hence, there exist $g_1,\dots,g_N\in\mathcal G_{B,w}$ such that, for every $g\in\mathcal G_{B,w}$, there exists $k\in\{1,\dots,N\}$ satisfying
$$
\max_{1\le j\le m}|r_g(\omega_j)-r_{g_k}(\omega_j)|\le \eta/3.
$$
If $d(\omega,\omega_j)<\delta$, then
$$
|r_g(\omega)-r_{g_k}(\omega)|
\le
|r_g(\omega)-r_g(\omega_j)|
+
|r_g(\omega_j)-r_{g_k}(\omega_j)|
+
|r_{g_k}(\omega_j)-r_{g_k}(\omega)|
\le \eta.
$$
Thus, $\|r_g-r_{g_k}\|_\infty\le \eta$, and so
$$
\sup_{g\in\mathcal G_{B,w}} |(P_n-Q_n)r_g|
\le
\max_{1\le k\le N}|(P_n-Q_n)r_{g_k}|+2\eta.
$$

For each fixed $k$,
$$
\Var{(P_n-Q_n)r_{g_k}}
=
\frac{\Var{r_{g_k}(X_1)}}{n}
\le
\frac{B^2}{n},
$$
so $(P_n-Q_n)r_{g_k}=o_{\mathsf{P}}(1)$ by Chebyshev's inequality. Since $N$ is finite,
$$
\max_{1\le k\le N}|(P_n-Q_n)r_{g_k}|=o_{\mathsf{P}}(1).
$$
Letting $\eta\to 0^+$, we obtain
$$
\sup_{g\in\mathcal G_{B,w}} |(P_n-Q_n)r_g|=o_{\mathsf{P}}(1).
$$
Therefore, $B_n=o_{\mathsf{P}}(1)$, and the result follows.
\end{proof}

\begin{proof}[Proof of Lemma~\ref{lem:aux_P_n_Q_n_tilde_hat}]
Define, for a law $P$ and a measurable function $f$ taking values in $\R^p$, the pseudo-norms
$$
\rho_{\mathsf{P}}(f) \defin \left(P\|f-Pf\|^2\right)^{1/2}, \quad \|f\|_{P,2} \defin \left(P\|f\|^2\right)^{1/2}.
$$
Fix $\rho>0$ such that $B(\btheta_0,\rho)$ is contained in a neighborhood of $\btheta_0$ on which condition \ref{cp1} holds, and define
$$
A_n\defin \{\hat\btheta\in B(\btheta_0,\rho),\ \tilde\btheta_n\in B(\btheta_0,\rho)\}.
$$
Since $\hat\btheta\convp\btheta_0$ and $\tilde\btheta_n\to\btheta_0$, $\Prob{A_n}\to1$.
Thus, it holds that
\begin{align}\label{eq:relation_pseudo_norms}
   \rho_{\mu_{\btheta_0}}(f) \le \|f\|_{\mu_{\btheta_0},2} \le \Big(1-\frac{1}{\sqrt{n}}\Big)^{-1/2}\|f\|_{Q_n,2}.
\end{align}

On $A_n$, by condition \ref{cp1}, one has 
\begin{align*}
    \|\bpsi_{\hat{\btheta}}-\bpsi_{\tilde{\btheta}_n}\|_{Q_n,2} \le \|L\|_{Q_n,2}\|\hat{\btheta}-\tilde{\btheta}_n\| \convp 0,
\end{align*}
Since $\Probbig{A_n}\to1$, this bound holds with probability increasing to one as $n\to\infty$. Hence, for every $\delta>0$, one has
\begin{equation}\label{eq:bound_asympt_equicont}
\begin{aligned}
\|\G_{n,Q_n}^{Q_n}\big(\bpsi_{\hat{\btheta}}-\bpsi_{\tilde\btheta_n}\big)\|
\le &    
\sup_{\substack{\btheta_1,\btheta_2\in B(\btheta_0,\rho)\\ \|\bpsi_{\btheta_1}-\bpsi_{\btheta_2}\|_{Q_n,2}<\delta}}
\|\G_{n,Q_n}^{Q_n}\big(\bpsi_{\btheta_1}-\bpsi_{\btheta_2}\big)\| + o_{\mathsf{P}}(1) \\
\le &
\sup_{\substack{\btheta_1,\btheta_2\in B(\btheta_0,\rho)\\ \|\bpsi_{\btheta_1}-\bpsi_{\btheta_2}\|_{Q_n,2}<\delta}}
\|\G_{n,\mu_{\btheta_0}}^{\mu_{\btheta_0}}\big(\bpsi_{\btheta_1}-\bpsi_{\btheta_2}\big)\|
\\
&+
2\sup_{\btheta\in B(\btheta_0,\rho)}
\left\|\G_{n,Q_n}^{Q_n}\bpsi_{\btheta} - \G_{n,\mu_{\btheta_0}}^{\mu_{\btheta_0}}\bpsi_{\btheta}\right\|
+o_{\mathsf{P}}(1).
\end{aligned}
\end{equation}

To show $\G_{n,Q_n}^{Q_n}(\bpsi_{\hat{\btheta}}-\bpsi_{\tilde\btheta_n})\convp0$, we first show that
\begin{align}\label{eq:coupling_score_bound}
\sup_{\btheta\in B(\btheta_0,\rho)}\left\|\G_{n,Q_n}^{Q_n}(\bpsi_{\btheta}) - \G_{n,\mu_{\btheta_0}}^{\mu_{\btheta_0}}(\bpsi_{\btheta})\right\|
\le
\sqrt{p}\,\max_{1\le j\le p}\sup_{\btheta\in B(\btheta_0,\rho)}
\left|\G_{n,Q_n}^{Q_n}(\psi_{\btheta,j}) - \G_{n,\mu_{\btheta_0}}^{\mu_{\btheta_0}}(\psi_{\btheta,j})\right|
\convp 0.
\end{align} 
For each $j=1,\ldots,p$, let $\cF_{\rho,j}\defin\{\psi_{\btheta,j} : \btheta \in B(\btheta_0,\rho)\}$. The assumptions of Proposition~\ref{prop:bernoulli} hold for each $\cF_{\rho,j}$ because each $\cF_{\rho,j}$ is $\mu_{\btheta_0}$- and $G$-Donsker, and
\begin{align*}
\sup_{f\in\cF_{\rho,j}}|(G-\mu_{\btheta_0})f|
\le
\sup_{\btheta\in B(\btheta_0,\rho)}\|(G-\mu_{\btheta_0})\bpsi_{\btheta}\|
\le
2\Big(\|G\bpsi_{\btheta_0}\|+\rho\,(GL+\mu_{\btheta_0}L)\Big)<\infty.
\end{align*}
The fact that each $\cF_{\rho,j}$ is $\mu_{\btheta_0}$- and $G$-Donsker follows from conditions \ref{cp1} and \ref{cp2}, $\mu_{\btheta_0}L^2<\infty$, $GL^2<\infty$, and $G\|\bpsi_{\btheta_0}\|^2<\infty$ \citepSM[see Example 19.7 in][]{vandervaart1998SM}.
By Proposition~\ref{prop:bernoulli} applied to the class $\bigcup_{j=1}^p \cF_{\rho,j}$, there exists a version of $\G_{n,\mu_{\btheta_0}}^{\mu_{\btheta_0}}$ such that the RHS of \eqref{eq:coupling_score_bound} converges to zero in probability. 

For the first term on the RHS of \eqref{eq:bound_asympt_equicont}, since each class $\cF_{\rho,j}$ is $\mu_{\btheta_0}$-Donsker, it is asymptotically equicontinuous under $\mu_{\btheta_0}$ \citepSM[see pp. 138--139 in][]{VanderVaartWellner2023SM}. Hence, for every $\varepsilon>0$ and every $j=1,\ldots,p$,
\begin{equation*}
\lim_{\delta \to 0^+} \limsup_{n \to \infty} \mathsf{P}\Bigg(\sup_{\substack{\btheta_1,\btheta_2\in B(\btheta_0,\rho)\\ \rho_{\mu_{\btheta_0}}(\psi_{\btheta_1,j}-\psi_{\btheta_2,j})<\delta}}\left|\G_{n,\mu_{\btheta_0}}^{\mu_{\btheta_0}}\big(\psi_{\btheta_1,j}-\psi_{\btheta_2,j}\big)\right|>\varepsilon\Bigg) = 0.
\end{equation*}
Moreover, applying \eqref{eq:relation_pseudo_norms}, one obtains
\begin{align*}
\sup_{\substack{\btheta_1,\btheta_2\in B(\btheta_0,\rho)\\ \|\bpsi_{\btheta_1}-\bpsi_{\btheta_2}\|_{Q_n,2}<\delta}}\left\|\G_{n,\mu_{\btheta_0}}^{\mu_{\btheta_0}}\big(\bpsi_{\btheta_1}-\bpsi_{\btheta_2}\big)\right\|
\le
\sqrt{p}\,\max_{1\le j\le p}\sup_{\substack{\btheta_1,\btheta_2\in B(\btheta_0,\rho)\\ \rho_{\mu_{\btheta_0}}(\psi_{\btheta_1,j}-\psi_{\btheta_2,j})<2\delta}}
\left|\G_{n,\mu_{\btheta_0}}^{\mu_{\btheta_0}}\big(\psi_{\btheta_1,j}-\psi_{\btheta_2,j}\big)\right|.
\end{align*}
Since \eqref{eq:bound_asympt_equicont} holds for every $\delta>0$, taking $\limsup_{n\to\infty}$ and then letting $\delta\to 0^+$ yields
$\G_{n,Q_n}^{Q_n}(\bpsi_{\hat{\btheta}}-\bpsi_{\tilde\btheta_n})\convp0$, and hence $(P_n - Q_n)(\bpsi_{\hat{\btheta}} - \bpsi_{\tilde\btheta_n}) = o_{\mathsf{P}}(n^{-1/2})$.
\end{proof}

\begin{proof}[Proof of Lemma~\ref{lem:aux_P_n_Q_n_tilde}]
Fix $\rho>0$ such that $B(\btheta_0,\rho)$ is contained in a neighborhood of $\btheta_0$ on which condition \ref{cp1} holds.
Since $\tilde{\btheta}_n\to\btheta_0$ by assumption, there exists $n(\rho)>0$ such that $\tilde{\btheta}_n\in B(\btheta_0,\rho)$ for all $n>n(\rho)$.
Observe that also $\|\sqrt{n}(P_n-Q_n)\bpsi_{\tilde{\btheta}_n}\| \le \sup_{\btheta\in B(\btheta_0,\rho)}\|\sqrt{n}(P_n-Q_n)\bpsi_{\btheta}\|$ for $n>n(\rho)$.

For each $j=1,\ldots,p$, let $\cF_{\rho,j}\defin \{\psi_{\btheta,j} : \btheta \in B(\btheta_0,\rho)\}$.  The assumptions of Corollary~\ref{cor:bernoulli} hold by the same arguments as in Lemma \ref{lem:aux_P_n_Q_n_tilde_hat}. An application of Corollary~\ref{cor:bernoulli} to each $\cF_{\rho,j}$ gives $\sup_{\btheta\in B(\btheta_0,\rho)}|\sqrt{n}(P_n-Q_n)\psi_{\btheta,j}|=O_{\mathsf{P}}(1)$. Hence, since $p<\infty$,
\begin{align*}
\sup_{\btheta\in B(\btheta_0,\rho)}\|\sqrt{n}(P_n-Q_n)\bpsi_{\btheta}\|
\le
\sqrt{p}\,\max_{1\le j\le p}\sup_{\btheta\in B(\btheta_0,\rho)}|\sqrt{n}(P_n-Q_n)\psi_{\btheta,j}|
=O_{\mathsf{P}}(1),
\end{align*}
and hence $\left(P_n - Q_n\right)\bpsi_{\tilde{\btheta}_n} = O_{\mathsf{P}}(n^{-1/2})$. This concludes the result.
\end{proof}

\begin{proof}[Proof of Lemma~\ref{lem:absorbtion_OP}]
Since $C_n=o_{\mathsf{P}}(A_n)$, one has $\mathsf{P}(C_n>A_n/2)\to0$, that is, $\mathsf{P}\left(C_n \le A_n/2\right)\to1$.
If $C_n \le A_n/2$, then
$A_n \le B_n + C_n \le B_n + A_n/2$, so
$A_n \le 2 B_n.$
Therefore, for every $n$,
$$
\mathsf{P}(A_n>2B_n)
\le
\mathsf{P}\left(A_n>2B_n,\ C_n \le \tfrac12 A_n\right)
+
\mathsf{P}\left(C_n > \tfrac12 A_n\right)
=
\mathsf{P}\left(C_n > \tfrac12 A_n\right),
$$
where the equality follows from $A_n>2B_n$. As $\mathsf{P}(C_n > A_n/2)\to0$, we conclude that $\mathsf{P}(A_n>2B_n)\to0$, hence $A_n = O_{\mathsf{P}}(B_n)$.
\end{proof}

\section{Computational details and model-specific derivations}\label{ap:sec_computational_details}

\subsection{Theoretical aspects for rotationally symmetric distributions}\label{subsec:additional_dp_s2}
We provide details on the representations of the distance profiles for the rotationally symmetric models on $\Sp^2$ used in the numerical evaluations presented in Sections~\ref{sec:numerical_experiments} and \ref{sec:real_data}.

For rotationally symmetric models on $\Sp^2$, the distance profile depends on $(\bmu,\bomega)$ only through $s=\bmu^\top\bomega$. We present the representations used in the numerical evaluations of distance profiles. Under the geodesic distance on $\Sp^2$, for every $\bomega\in\Sp^2$ and $0\le t\le \pi$,
the distance profile is determined by the distribution function of the projections
$$
H_s(x)=\Probbig{\bomega^\top \bX\le x},
\qquad
s=\bmu^\top \bomega.
$$
The models below differ only in the representation and numerical evaluation of $H_s$.

\noindent \textit{Uniform--Beta mixture}.
Let $\bmu\in\Sp^2$, $w\in(0,1)$, and $\alpha,\beta>0$. We refer to the model with density
$$
f_{\bX}(\bx)
=
\frac{1}{4\pi}
\left[
w+(1-w)\frac{y(\bx)^{\alpha-1}(1-y(\bx))^{\beta-1}}{B(\alpha,\beta)}
\right]
$$
as the uniform--beta mixture and write
$
\bX\sim \mathrm{UB}(\bmu,w,\alpha,\beta),
$
where $y(\bx)=(1+\bmu^\top \bx)/2$. Writing $Y=y(\bX)$, we have $Y\sim w \operatorname{Unif}(0,1) + (1 - w)\operatorname{Beta}(\alpha,\beta).$ The projected distribution is thus
$$
H_s(x)
=
\frac{x+1}{2}
+\sum_{\ell\ge 1}
a_\ell P_\ell(s)\,\frac{P_{\ell+1}(x)-P_{\ell-1}(x)}{2(2\ell+1)},
\qquad -1\le x\le 1,
$$
with coefficients
$$
a_\ell=(1-w)(2\ell+1)\,\E{P_\ell(2 Y_\beta-1)},\qquad \ell \ge 1,
$$
where $Y_\beta\sim\operatorname{Beta}(\alpha,\beta)$.
It holds that
$$
\E{P_\ell(2Y_\beta-1)}
=
\frac{1}{B(\alpha,\beta)}
\int_0^1
P_\ell(2y-1)\,y^{\alpha-1}(1-y)^{\beta-1}\,\rd y,
$$
which, after the change of variable $z=2y-1$, becomes
$$
\frac{1}{2^{\alpha+\beta-1}B(\alpha,\beta)}
\int_{-1}^1
P_\ell(z)\,(1-z)^{\beta-1}(1+z)^{\alpha-1}\,\rd z.
$$
We truncate the Legendre expansion at $\ell=150$ and compute the corresponding coefficients using $100$-point Gauss--Jacobi quadrature, which is exact for polynomials up to degree $199$.

\medskip
\noindent
\textit{Small-circle \citepSM{Bingham1978SM}.}
Let $\bX\sim \mathrm{SC}(\bmu,\kappa,\nu)$, where $\bmu\in\Sp^2$, $\kappa\ge 0$, and $\nu\in(-1,1)$. The model is rotationally symmetric around $\bmu$, conditionally uniform on the small circles $\{\bx\in\Sp^2:\bmu^\top \bx=z\}$, and its axial density is
$$
h_{\kappa,\nu}(z)
=
\frac{\exp\{-\kappa(z-\nu)^2\}}{c(\kappa,\nu)},
\qquad -1\le z\le 1,
$$
where
$$
c(\kappa,\nu)=\int_{-1}^1 \exp\{-\kappa(z-\nu)^2\}\, \rd z.
$$
Since $c(\kappa,\nu)=c(\kappa,-\nu),$ it follows that $\mathrm{SC}(\bmu,\kappa,-\nu) = \mathrm{SC}(-\bmu,\kappa,\nu).$ Hence, without loss of generality, one may restrict to $\nu\ge 0$. The distribution of the projections is
$$
H_s(x)
=
\frac{x+1}{2}
+\sum_{\ell\ge 1}
a_\ell P_\ell(s)\,\frac{P_{\ell+1}(x)-P_{\ell-1}(x)}{2(2\ell+1)},
$$
with
$$
a_\ell
=
(2\ell+1)\int_{-1}^1 P_\ell(z)\,h_{\kappa,\nu}(z)\, \rd z.
$$
The coefficients are computed using $400$-point Gauss--Legendre quadrature. 

\medskip
\noindent
\textit{Spherical Cardioid \citepSM{Garcia-Portugues2026SM}}
Let $\bX \sim C_k(\bmu,\rho)$, where $k\ge 1$ is an integer order, $\bmu\in\Sp^2$ is the symmetry axis, and $\rho\in[-1,1]$ controls concentration around $\bmu$. For this model, the projected distribution was derived in \citetSM[Theorem~3.2]{Garcia-Portugues2026SM}.

\medskip
\noindent

\medskip
\noindent
\textit{Hyperbolic von Mises–Fisher.} Consider $\bX\sim \mathrm{HvMF}(\bmu,\kappa)$ on $\mathbb{H}^2$, where $\bmu\in\mathbb{H}^2$ and $\kappa>0$. By Proposition~\ref{prop:dp_examples}\ref{prop:dp_examples_hvmf}, the distance profile is $F_{\bomega}(t)=\mathsf{P}(Z\le t)$, where $Z=d_{{\mathbb H}^2}(\bX,\bomega)$ has density given in \eqref{eq:density_z_hvmf}.

\subsection{Maximum likelihood estimators and scores in the simulation scenarios}

The scenario numbering in this section is that of Table~\ref{tab:simulation_scenarios}. Given a sample $\bX_1,\ldots,\bX_n$, let $a_i\geq 0$ with $\sum_{i=1}^n a_i=1$ to cover at once the ordinary fit and the weighted fit used by the reestimated multiplier bootstrap. The ordinary Maximum Likelihood Estimators (MLEs) follow by taking $a_i=1/n$. Only the model-specific steps are derived below; standard derivations of sample means are omitted.

For every score $\boldsymbol\psi_{\boldsymbol\vartheta}$ below, the convention in the paper is
$$
\bV_{\boldsymbol\vartheta}
=
\mathsf{E}_{\boldsymbol\vartheta}
\left[
\frac{\partial\boldsymbol\psi_{\boldsymbol\vartheta}(\bX)}
{\partial\boldsymbol\vartheta^\top}
\right]
=-\mathcal I_{\boldsymbol\vartheta},
$$
where $\mathcal I_{\boldsymbol\vartheta}$ is the Fisher information.

\subsubsection{Scenario 1}

The fitted family is
$$
\bX\sim\mathcal N_d\left(\btheta,
\bI_d+\lambda\bu(\btheta)\bu(\btheta)^\top\right),
\qquad
\bu(\btheta)=\frac{\btheta}{\|\btheta\|},
\qquad \lambda>0.
$$
Thus, the mean determines both the location and the direction of the covariance spike. The log-likelihood is, up to an additive constant,
$$
\ell(\btheta,\lambda;\bX)
=-\frac12\log(1+\lambda)
-\frac12\left(
\|\bX-\btheta\|^2
-\frac{\lambda}{1+\lambda}
\{\bu(\btheta)^\top(\bX-\btheta)\}^2
\right).
$$
The scores are:
\begin{align*}
\boldsymbol\psi_{\btheta}(\bX)
&=
\bX-\btheta
+\frac{\lambda}{1+\lambda}\{\bu(\btheta)^\top(\bX-\btheta)\}
\left\{
\frac{\{\bI_d-\bu(\btheta)\bu(\btheta)^\top\}(\bX-\btheta)}{\|\btheta\|}
-\bu(\btheta)
\right\},\\
\psi_\lambda(\bX)
&=
-\frac{1}{2(1+\lambda)}
+\frac{\{\bu(\btheta)^\top(\bX-\btheta)\}^2}{2(1+\lambda)^2}.
\end{align*}
The corresponding information matrix is block diagonal, with
\begin{align*}
\mathcal I_{\btheta\btheta}
=
\frac{\bu(\btheta)\bu(\btheta)^\top}{1+\lambda}
+\left\{
1+\frac{\lambda^2}{(1+\lambda)\|\btheta\|^2}
\right\}\{\bI_d-\bu(\btheta)\bu(\btheta)^\top\},
\quad 
\mathcal I_{\lambda\lambda}
=
\frac{1}{2(1+\lambda)^2}.
\end{align*}

The ordinary and weighted MLEs differ only in the weights $a_i$. Let
$$
\overline{\bX}=\sum_{i=1}^n a_i\bX_i,
\qquad
\widehat\bSigma_{\bX}
=\sum_{i=1}^n a_i(\bX_i-\overline{\bX})(\bX_i-\overline{\bX})^\top,
$$
Define
\[
\boldsymbol{A}_{\lambda}
\defin
\overline{\boldsymbol X}\overline{\boldsymbol X}^{\top}
+\frac{\lambda}{1+\lambda}\widehat{\boldsymbol\Sigma}_{\boldsymbol X}.
\]
The estimate $\widehat\lambda$ is obtained by maximizing
$$
L(\lambda)
=\operatorname{eig}_{\max}\left(
\boldsymbol{A}_{\lambda}
\right)-\log(1+\lambda),
\qquad \lambda\geq0.
$$
Here $\operatorname{eig}_{\max}(\boldsymbol A)$ denotes the largest eigenvalue of $\boldsymbol A$. Let $\widehat\bu$ be a unit eigenvector associated with the largest eigenvalue of $\boldsymbol A_{\widehat\lambda}$, with $\widehat\bu^\top\overline{\bX}\geq0$. Then, $\widehat\btheta=(\widehat\bu^\top\overline{\bX})\widehat\bu.$

\subsubsection{Scenario 2}
Here,
$ \bX\sim\mathcal N_d(\bmu,\sigma^2\bI_d)$, $\sigma>0$. The estimator of $\bmu$ is the usual weighted sample mean. The variance estimator is given by
$$
\widehat\sigma^2
=
\frac{1}{d}\sum_{i=1}^n a_i\|\bX_i-\widehat\bmu\|^2.
$$
For the parameterization $(\bmu,\sigma)$, the scores are given by
\begin{align*}
\boldsymbol\psi_{\bmu}(\bX)
=\frac{\bX-\bmu}{\sigma^2}, \qquad
\psi_\sigma(\bX)
=-\frac d\sigma+\frac{\|\bX-\bmu\|^2}{\sigma^3},
\end{align*}
and the Fisher information matrix is
$$
\mathcal I_{(\bmu,\sigma)}
=
\begin{pmatrix}
\sigma^{-2}\bI_d & \boldsymbol 0\\
\boldsymbol 0^\top & 2d\,\sigma^{-2}
\end{pmatrix}.
$$

\subsubsection{Scenarios 3 and 4}
For $i=1,\ldots,n$, set $\bZ_i=\mathrm{ilr}(\bX_i)\in\R^d$. All fitted quantities are computed from the transformed sample $\bZ_1,\ldots,\bZ_n$. This reduces the likelihood calculations to the corresponding Gaussian models and avoids introducing redundant simplex coordinates.

\noindent
\textit{Scenario 3.}
The random vector corresponding to these transformed observations satisfies $\bZ\sim\mathcal N_d(\bmu,\bI_d)$, with only $\bmu$ unknown. The fit and score are those of the standard Gaussian location model, so no additional derivation is required.

\medskip
\noindent
\textit{Scenario 4.}
The fitted covariance is the AR(1) correlation matrix
$$
\bR_d(\rho)=\left(\rho^{|j-k|}\right)_{j,k=1}^d,
\qquad -1<\rho<1,
$$
and both $\bmu$ and $\rho$ are unknown. Let $\boldsymbol E_i=\bZ_i-\widehat\bmu$, with $E_{ij}$ denoting its $j$th coordinate, and define
\begin{align*}
A\defin\sum_{i=1}^n a_i\|\boldsymbol E_i\|^2, \quad 
B\defin\sum_{i=1}^n a_i\sum_{j=1}^{d-1}E_{ij}E_{i,j+1}, \quad  
C\defin\sum_{i=1}^n a_i\sum_{j=2}^{d-1}E_{ij}^2,
\end{align*}
where $C=0$ for $d=2$. Since
$\det\{\bR_d(\rho)\}=(1-\rho^2)^{d-1}$ and $\bR_d(\rho)^{-1}$ is tridiagonal, the
log-likelihood is, up to an additive constant,
$$
\ell_p(\rho)
=-\frac12\left\{
(d-1)\log(1-\rho^2)
+\frac{A+\rho^2C-2\rho B}{1-\rho^2}
\right\}.
$$
Its interior stationary points are the real roots of
$$
-(d-1)\rho^3+B\rho^2
+\{(d-1)-(A+C)\}\rho+B=0.
$$
The implementation evaluates the profile at every real root in the admissible interval and at both endpoints, and then takes the global maximizer. For numerical stability the interval used in the simulations was $[-0.995,0.995]$.

Writing $\bR_d'(\rho)=\partial\bR_d(\rho)/\partial\rho$, the scores are given by $\boldsymbol\psi_{\bmu}(\bZ) =\bR_d(\rho)^{-1}(\bZ-\bmu),$
and
\begin{align*}
\psi_\rho(\bZ)
=-\frac12\operatorname{tr}\left\{\bR_d(\rho)^{-1}\bR_d'(\rho)\right\}
+\frac12(\bZ-\bmu)^\top\bR_d(\rho)^{-1}\bR_d'(\rho)
\bR_d(\rho)^{-1}(\bZ-\bmu).
\end{align*}
The off-diagonal entries of $\bR_d'(\rho)$ are $|j-k|\rho^{|j-k|-1}$, and its diagonal entries are zero. The information matrix is block diagonal, with
$$
\mathcal I_{\bmu\bmu}=\bR_d(\rho)^{-1},
\qquad
\mathcal I_{\rho\rho}
=\frac12\operatorname{tr}
\left[
\left\{\bR_d(\rho)^{-1}\bR_d'(\rho)\right\}^2
\right].
$$

\subsubsection{Scenario 5}
The concentration is fixed at $\kappa=2$ and only $\bmu\in\Sp^d$ is fitted. The MLE is the normalized weighted sample resultant and is not rederived here. The implementation-specific point is that the score is represented in the tangent space, rather than in singular ambient coordinates. If $\bB_{\bmu}$ is any $(d+1)\times d$ matrix satisfying $\bB_{\bmu}^\top \bB_{\bmu}=\bI_d$ and $\bB_{\bmu}^\top\bmu=0$, the local score and information are
$$
\boldsymbol\psi_{\bmu}(\bX)
=\kappa \bB_{\bmu}^\top \bX,
\qquad
\mathcal I_{\bmu\bmu}
=\kappa\mathsf E_{\bmu,\kappa}(\bmu^\top \bX)\bI_d.
$$
The code constructs $\bB_{\bmu}$ deterministically by projecting coordinate vectors onto the tangent space and applying a QR decomposition.

\subsubsection{Scenario 6}
Both location and concentration are fitted by the standard vMF MLE for the canonical parameter $\bxi=\kappa\bmu\in\R^{d+1}$. The non-weighted fit is computed using \texttt{movMF} \citepSM{Hornik2014SM}, whereas the weighted fit is computed directly from the weighted sample mean. The MLE and canonical score are standard and are already derived in the manuscript.

\subsubsection{Scenarios 7 and 8}
These scenarios differ only in the alternatives used to generate the data. The model is fitted using the canonical parameter $\boldsymbol\xi=\kappa\boldsymbol\mu$. \citetSM{Jensen1981SM} shows that the MLE on $\mathbb H^q$ exists if and only if $R>n$, where $\bS:=\sum_{i=1}^n\bX_i$ and $R:=\sqrt{-(\bS,\bS)}$, and is given by
\[
\widehat\bmu=\frac{\bS}{R},
\qquad
\frac{\mathcal K_{(q+1)/2}(\widehat\kappa)}
{\mathcal K_{(q-1)/2}(\widehat\kappa)}
=\frac{R}{n}.
\]
Here $\mathcal K_\nu$ denotes the modified Bessel function of the second kind. For the weighted fit, let $\bS_{\bzeta}:=\sum_{i=1}^n\zeta_i\bX_i$, $Z_{\bzeta}:=\sum_{i=1}^n\zeta_i$, and $R_{\bzeta}:=\sqrt{-(\bS_{\bzeta},\bS_{\bzeta})}$. The estimated parameters are defined by
\[
\widehat{\bmu}_{\bzeta}
=\frac{\bS_{\bzeta}}{R_{\bzeta}},
\qquad
\frac{\mathcal K_{(q+1)/2}(\widehat\kappa_{\bzeta})}
{\mathcal K_{(q-1)/2}(\widehat\kappa_{\bzeta})}
=\frac{R_{\bzeta}}{Z_{\bzeta}},
\]
provided that $R_{\bzeta}>Z_{\bzeta}$. For $q=2$, these equations reduce to $\widehat\kappa=n/(R-n)$ and $\widehat\kappa_{\bzeta}=Z_{\bzeta}/(R_{\bzeta}-Z_{\bzeta})$. The derivation of the score is standard since the HvMF model is an exponential family.

\subsection{Evaluation of the profile derivatives}

For the regular models considered here, differentiation under the integral sign gives
\begin{equation}\label{eq:profile_derivative_score}
\dot F_{\bomega}^{\boldsymbol\vartheta}(t)
=
\mathsf E_{\boldsymbol\vartheta}
\left[
1_{\{d(\bomega,\bX)\leq t\}}
\boldsymbol\psi_{\boldsymbol\vartheta}(\bX)
\right].
\end{equation}
In Scenarios 1--4, $\dot F_{\bomega}^{\boldsymbol\vartheta}(t)$ was approximated by Monte Carlo using \eqref{eq:profile_derivative_score} and an auxiliary sample of size $10,000$ from the fitted model. We next give the one-dimensional integrals used in Scenarios 5--8. 

For the vMF model on $\Sp^d$, let $\bxi=\kappa\bmu$ and $a=\bxi^\top\bomega$. The density of $\bomega^\top\bX$ is
$$
g_{\bomega,\bxi}(s)
=
\frac{c_d^{\mathrm{vMF}}(\kappa)}
{c_{d-1}^{\mathrm{vMF}}(\sqrt{\kappa^2-a^2}\sqrt{1-s^2})}
(1-s^2)^{(d-2)/2}e^{as},
\qquad -1<s<1.
$$
Conditional on $\bomega^\top\bX=s$, write
$$
\bX=s\bomega+\sqrt{1-s^2}\,\boldsymbol V.
$$
On the unit sphere in the orthogonal complement of $\bomega$, $\boldsymbol V$ follows a vMF distribution with mean direction along $\bxi-a\bomega$ and concentration $\sqrt{\kappa^2-a^2}\sqrt{1-s^2}$. Using its mean gives
\begin{align*}
\mathsf E_{\bxi}(\bX\mid\bomega^\top\bX=s)
&=
\{s-aC(s)\}\bomega+\kappa C(s)\bmu,\\
C(s)
&=
(1-s^2)
\frac{A_{d-1}(\sqrt{\kappa^2-a^2}\sqrt{1-s^2})}
{\sqrt{\kappa^2-a^2}\sqrt{1-s^2}},
\end{align*}
where
$$
A_j(z) \defin \frac{\mathcal I_{(j+1)/2}(z)}{\mathcal I_{(j-1)/2}(z)},
$$
where $\mathcal I_\nu$ is the modified Bessel function of the first kind.

Since the canonical score is $\bX-A_d(\kappa)\bmu$, it follows that
$$
\dot F_{\bomega}^{\bxi}(t)
=
\int_{\cos t}^{1}
\left[
\{s-aC(s)\}\bomega
+\{\kappa C(s)-A_d(\kappa)\}\bmu
\right]
g_{\bomega,\bxi}(s)\,\rd s.
$$
For Scenario 5, the derivative in the local coordinates of $\bmu$ is obtained as $\kappa\bB_{\bmu}^\top\dot F_{\bomega}^{\bxi}(t)$.

For the HvMF model, let $(\cdot,\cdot)$ denote the Minkowski inner product, and let
$$
\bJ=\operatorname{diag}(-1,1,\ldots,1),
\qquad
\bxi=\kappa\bmu,
\qquad
a=(\bxi,\bomega).
$$
The density of $d(\bomega,\bX)$ is
$$
h_{\bomega,\bxi}(r)
=
\frac{c_d^{\mathrm{HvMF}}(\kappa)}
{c_{d-1}^{\mathrm{vMF}}(\sqrt{a^2-\kappa^2}\sinh r)}
(\sinh r)^{d-1}e^{a\cosh r},
\qquad r\geq0.
$$
Writing
$$
B_d(\kappa)
=\frac{\mathcal K_{(d+1)/2}(\kappa)}{\mathcal K_{(d-1)/2}(\kappa)},
\qquad
D(r)
=(\sinh r)^2
\frac{A_{d-1}(\sqrt{a^2-\kappa^2}\sinh r)}
{\sqrt{a^2-\kappa^2}\sinh r},
$$
where $\mathcal K_\nu$ is the modified Bessel function of the second kind. Conditional on $d(\bomega,\bX)=r$, write $\bX=\cosh(r)\bomega+\sinh(r)\boldsymbol V.$ On the unit sphere in the tangent space at $\bomega$, $\boldsymbol V$ follows a vMF distribution with mean direction along $\bxi+a\bomega$ and concentration $\sqrt{a^2-\kappa^2}\sinh r$. Using its mean gives
$$
\mathsf E_{\bxi}(\bX\mid d(\bomega,\bX)=r)
=
(\cosh r+aD(r))\bomega+\kappa D(r)\bmu.
$$
The canonical score is $\bJ(\bX-B_d(\kappa)\bmu)$, and hence
$$
\dot F_{\bomega}^{\bxi}(t)
=
\int_0^t
\bJ\left(
(\cosh r+aD(r))\bomega
+(\kappa D(r)-B_d(\kappa))\bmu
\right)
h_{\bomega,\bxi}(r)\, \rd r.
$$
The ratios in $C(s)$ and $D(r)$ are evaluated by their continuous limits when their arguments are zero. In the reported computations, for each center $\bomega=\bX_i$, the values $\dot F_{\bomega}^{\bxi}(t_{ij})$ were evaluated at the observed distances $t_{ij}=d(\bX_i,\bX_j)$, $j=1,\ldots,n$. For vMF, the integral defining $\dot F_{\bomega}^{\bxi}(t)$ was evaluated by cumulative trapezoidal integration on a uniform $4097$-point grid on $[-1,1]$, followed by linear interpolation at $\cos(t_{ij})$. For HvMF, for each $i$, the integral defining $\dot F_{\bomega}^{\bxi}(t)$ was evaluated by cumulative trapezoidal integration on a uniform $4097$-point grid on $[0,\max_j t_{ij}]$, followed by linear interpolation at $t_{ij}$.

\section{Computational validation of the null asymptotic distributions}\label{ap:sec_computational_validation}

Figures~\ref{fig:convergence_process_vmf_simple_mu_100} and \ref{fig:convergence_process_vmf_comp_mu_100} illustrate the null convergence of the KS statistics for the vMF model, for simple and composite nulls. To approximate the limiting distribution, the covariance function was evaluated on the same fixed grid consisting of 10 directions and 10 radii, and a centred $100$-dimensional multivariate normal vector with this covariance matrix was simulated. For each setting, we simulate $M=10,000$ realizations of $\sup_{\omega, t}\big|\G_{0}(\omega, t)\big|$ (or $\sup_{\omega, t}\big|\tilde{\G}_{0}(\omega, t)\big|$) and of $T_n^{\mathrm{KS}}$ (or $\tilde{T}_n^{\mathrm{KS}}$) for each sample size $n=50,100,500$. In the composite setting, the unknown parameter $\bxi = \kappa \bmu$ was estimated by maximum likelihood in each simulated sample. The first row of each figure displays kernel density estimates of the simulated distributions, while the second row gives the corresponding quantile-quantile plots of the approximation errors with respect to the weak limits.

\ifincludeimages
\begin{figure}[!htbp]
\begin{subfigure}[htbp]{0.32\textwidth}
    \centering
    \includegraphics[width=\textwidth]{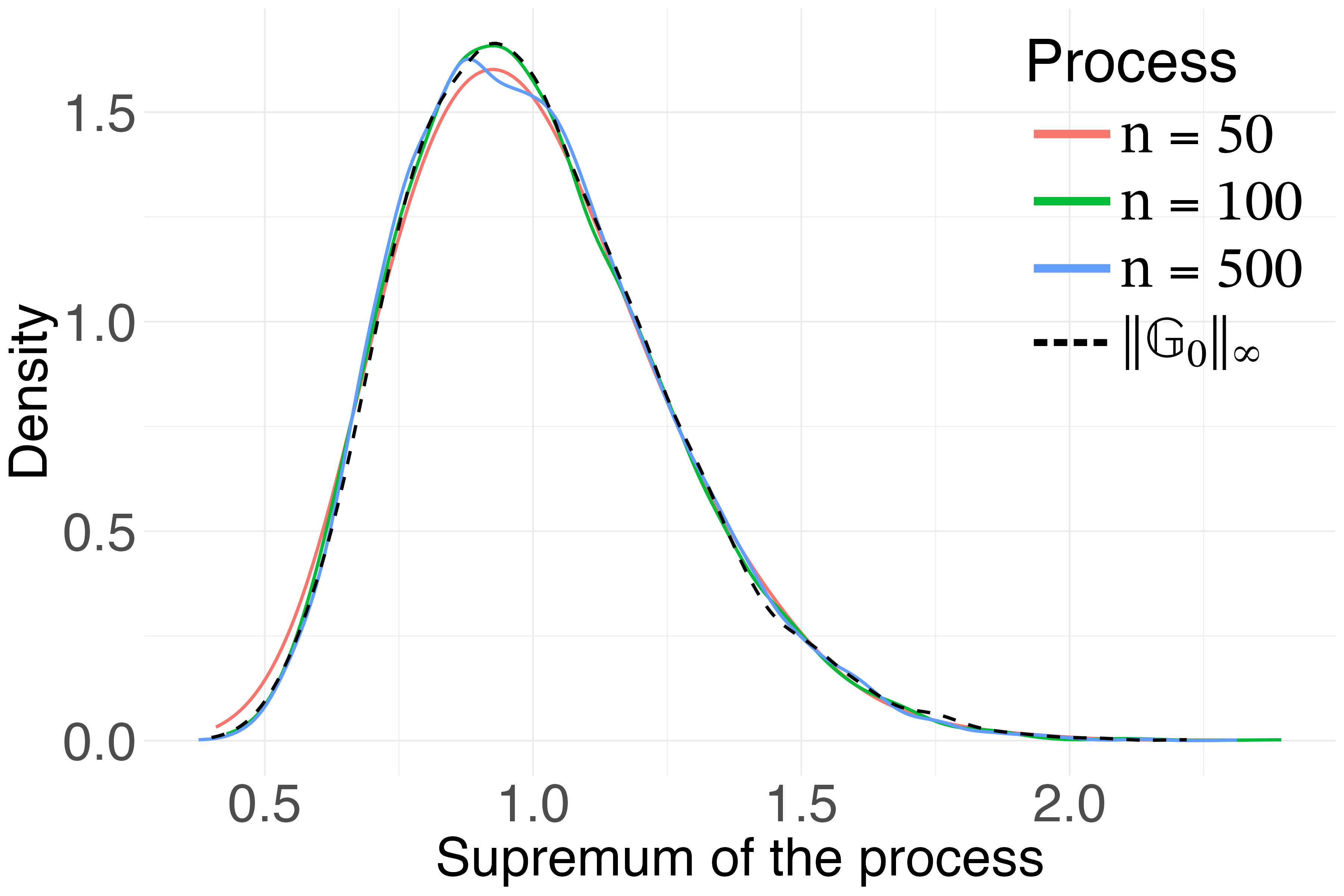}
    \caption{$\mathrm{vMF}(\bmu,1/2)$}
\end{subfigure}
\hfill
\begin{subfigure}[htbp]{0.32\textwidth}
    \centering
    \includegraphics[width=\textwidth]{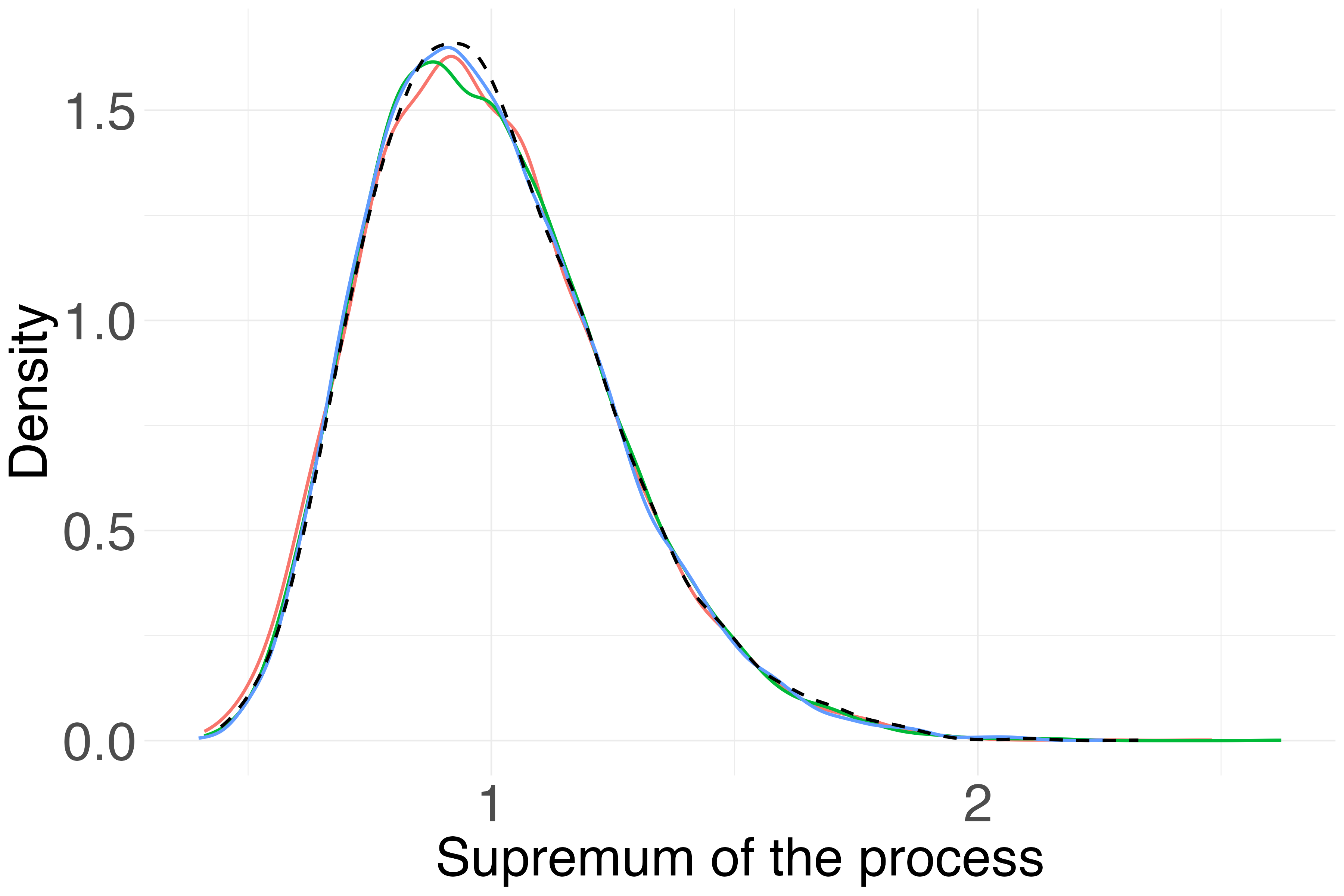}
    \caption{$\mathrm{vMF}(\bmu,1)$}
\end{subfigure}
\hfill
\begin{subfigure}[htbp]{0.32\textwidth}
    \centering
    \includegraphics[width=\textwidth]{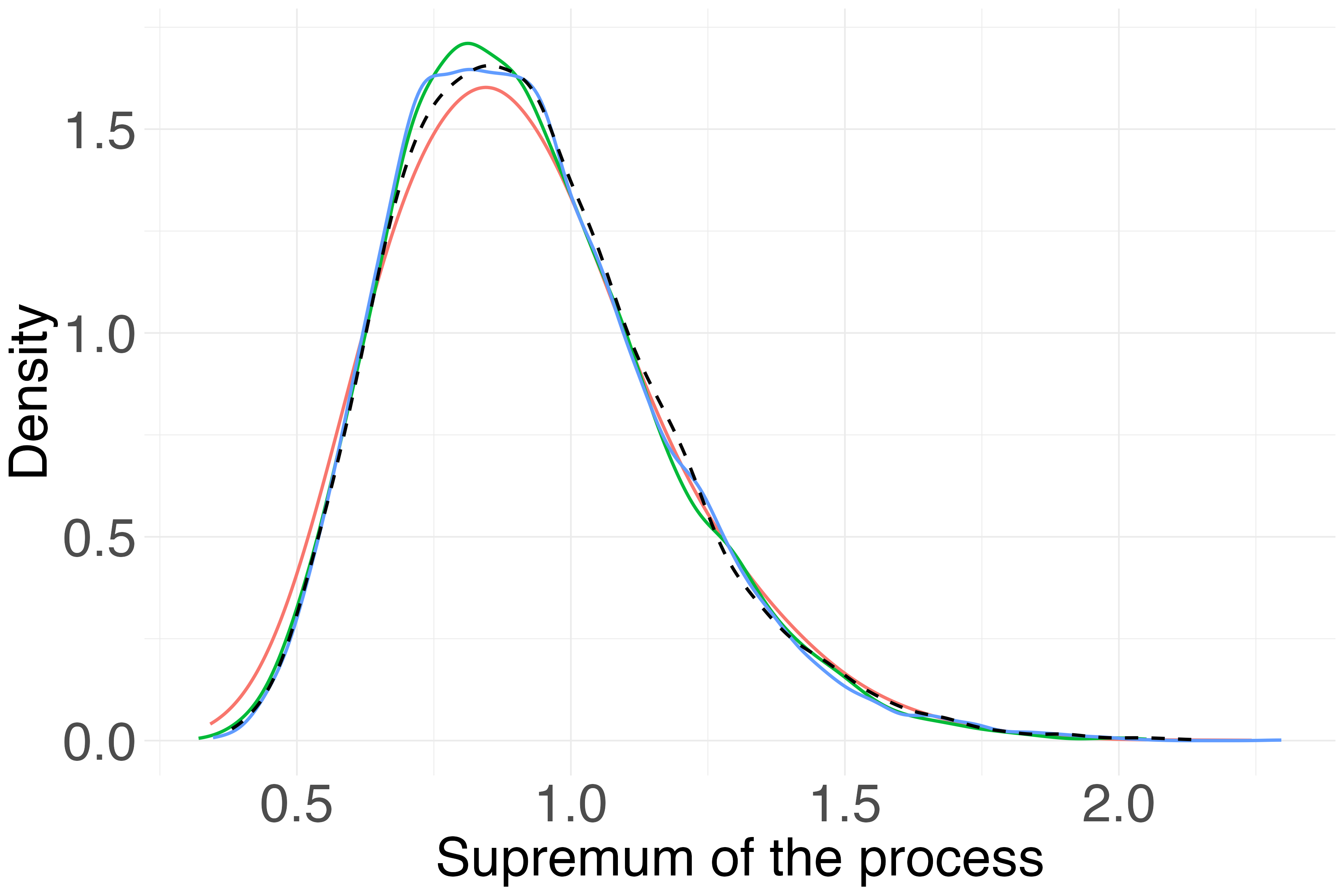}
    \caption{$\mathrm{vMF}(\bmu,5)$}
\end{subfigure}
\par\medskip
\begin{subfigure}[htbp]{0.32\textwidth}
    \centering
    \includegraphics[width=\textwidth]{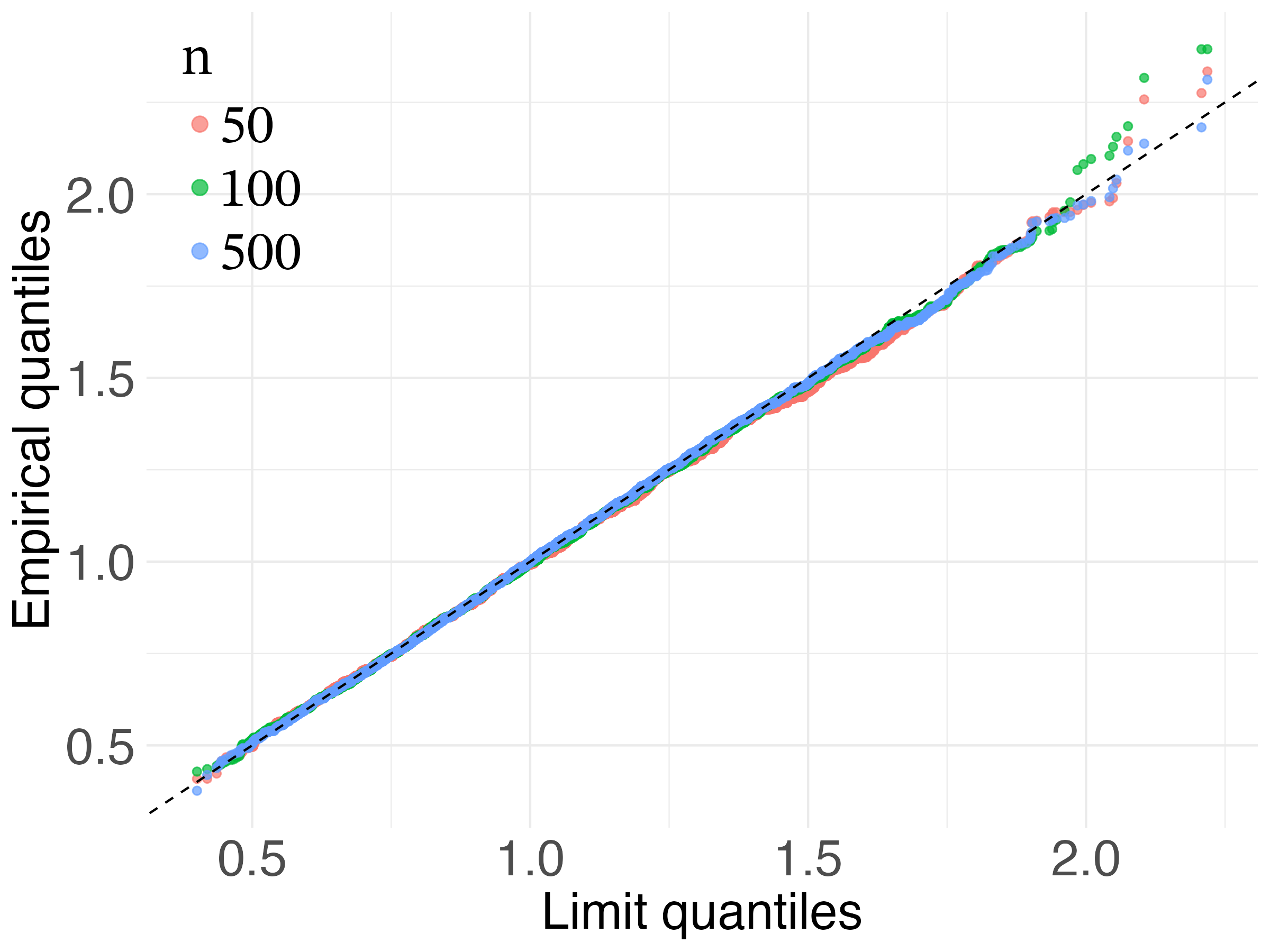}
    \caption{QQ plot for $\mathrm{vMF}(\bmu,1/2)$}
\end{subfigure}
\hfill
\begin{subfigure}[htbp]{0.32\textwidth}
    \centering
    \includegraphics[width=\textwidth]{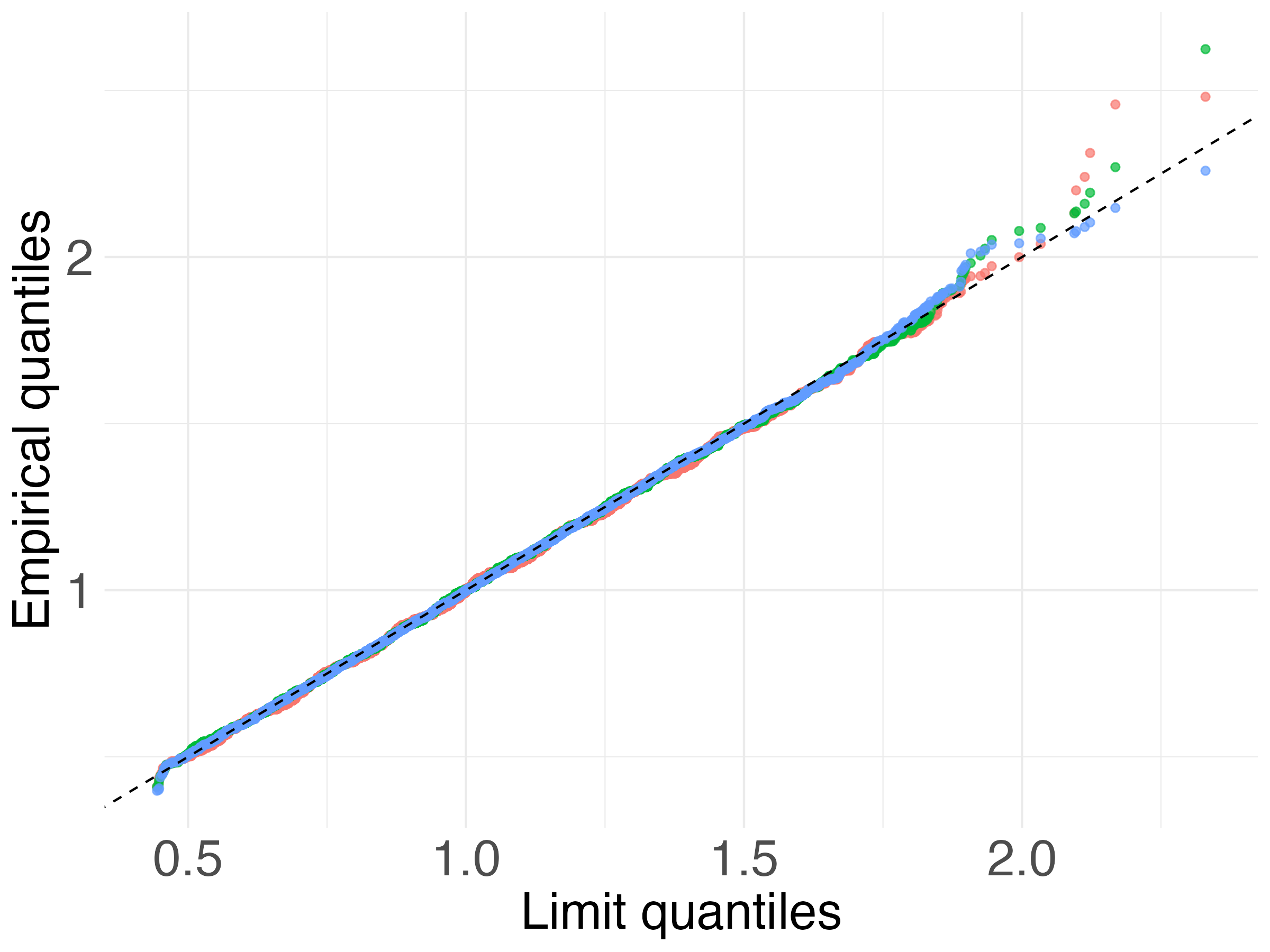}
    \caption{QQ plot for $\mathrm{vMF}(\bmu,1)$}
\end{subfigure}
\hfill
\begin{subfigure}[htbp]{0.32\textwidth}
    \centering
    \includegraphics[width=\textwidth]{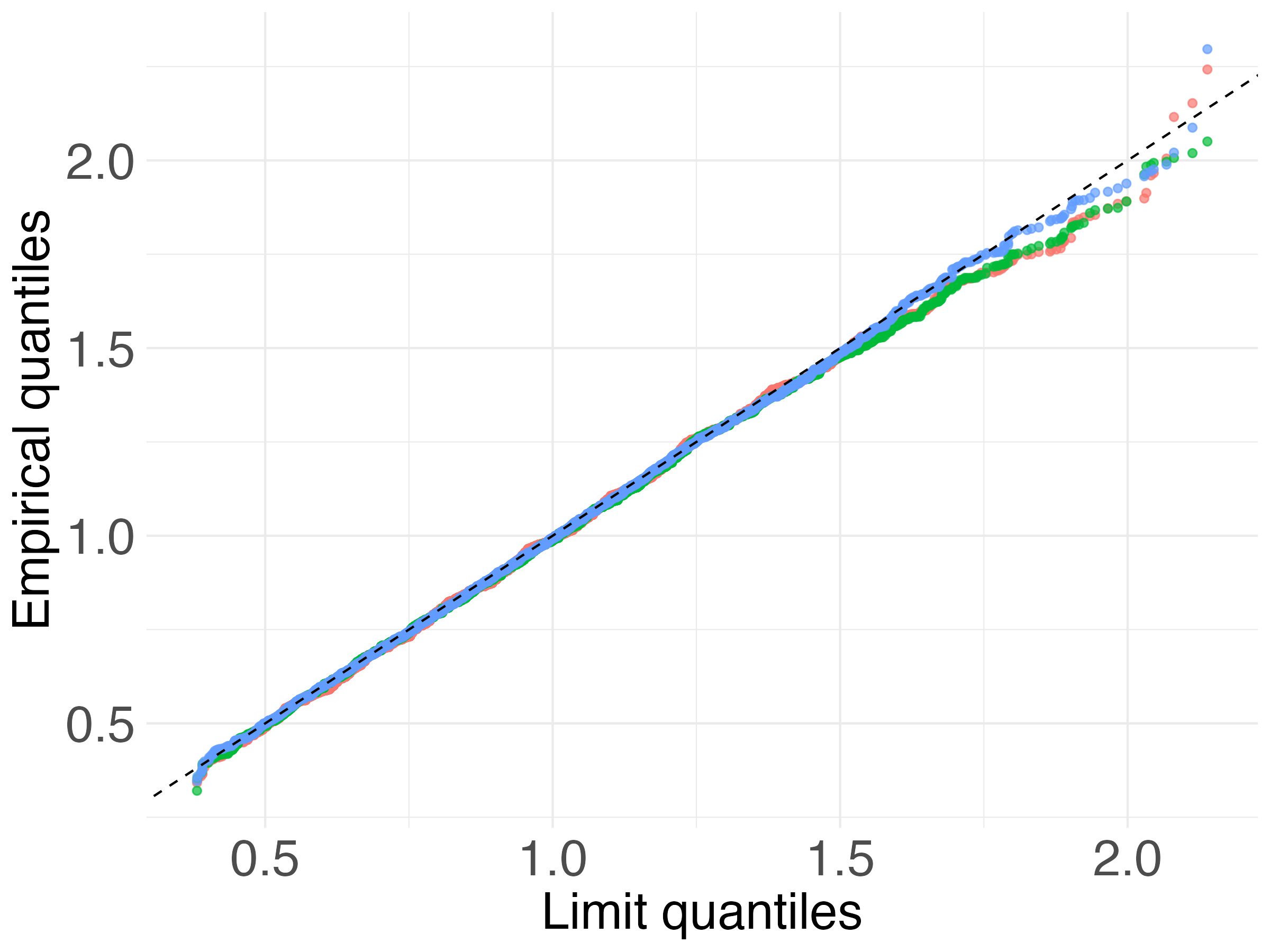}
    \caption{QQ plot for $\mathrm{vMF}(\bmu,5)$}
\end{subfigure}
\caption{Illustration of the convergence of the statistic $T_n^{\mathrm{KS}}$ defined in \eqref{eq:ks_test_statistic}. Specifically, using the geodesic distance, the test contrasts whether the data come from a $\mathrm{vMF}$ with canonical parameter $\bxi = \bmu/2$ (left), $\bxi = \bmu$ (center), and $\bxi = 5\bmu$ (right), with $\bmu = (1,0,0)^\top$.}
\label{fig:convergence_process_vmf_simple_mu_100}
\end{figure}
\fi
\ifincludeimages
\begin{figure}[!htbp]
\begin{subfigure}[htbp]{0.32\textwidth}
    \centering
    \includegraphics[width=\textwidth]{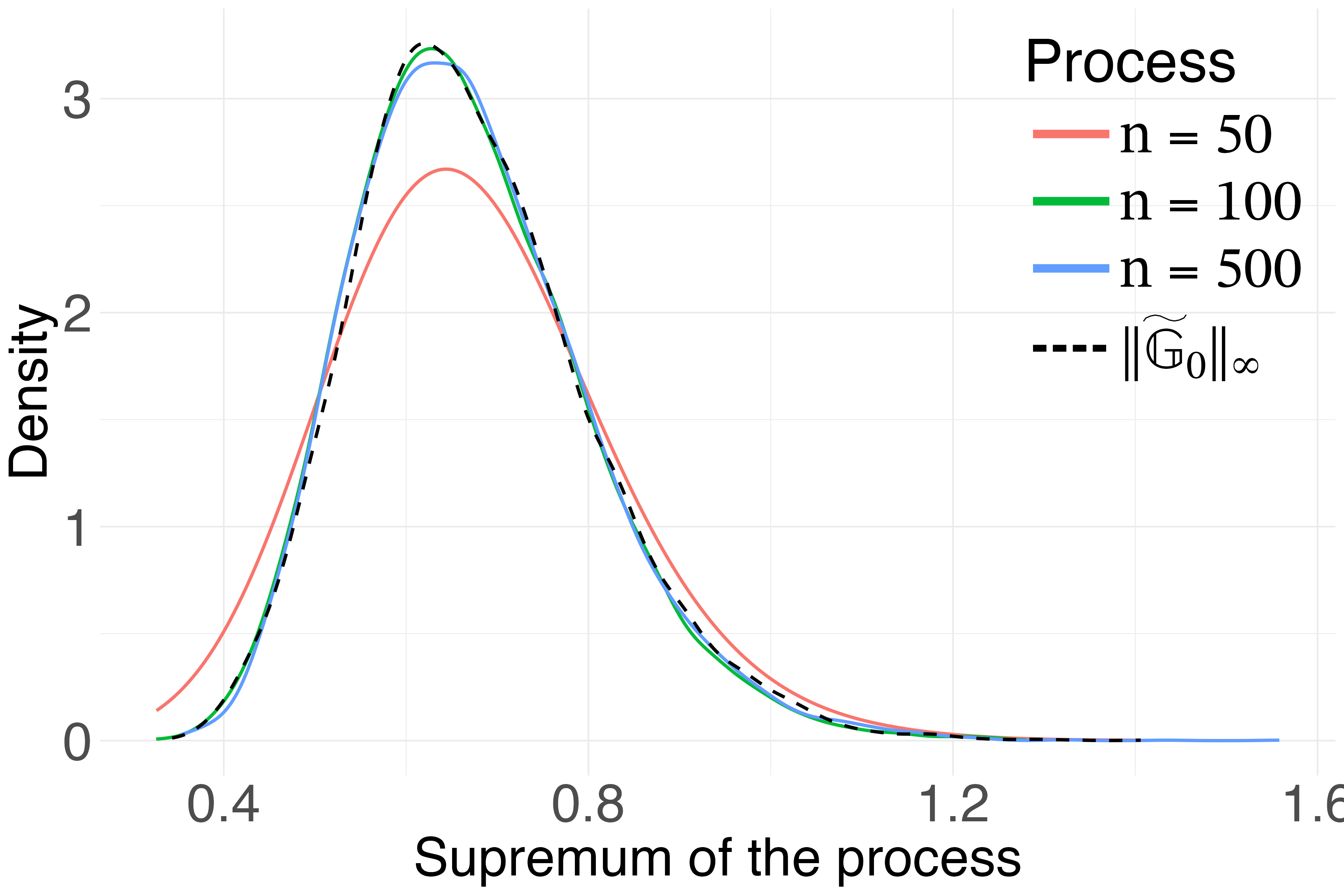}
    \caption{$\mathrm{vMF}(\bmu, 1/2)$}
\end{subfigure}
\hfill
\begin{subfigure}[htbp]{0.32\textwidth}
    \centering
    \includegraphics[width=\textwidth]{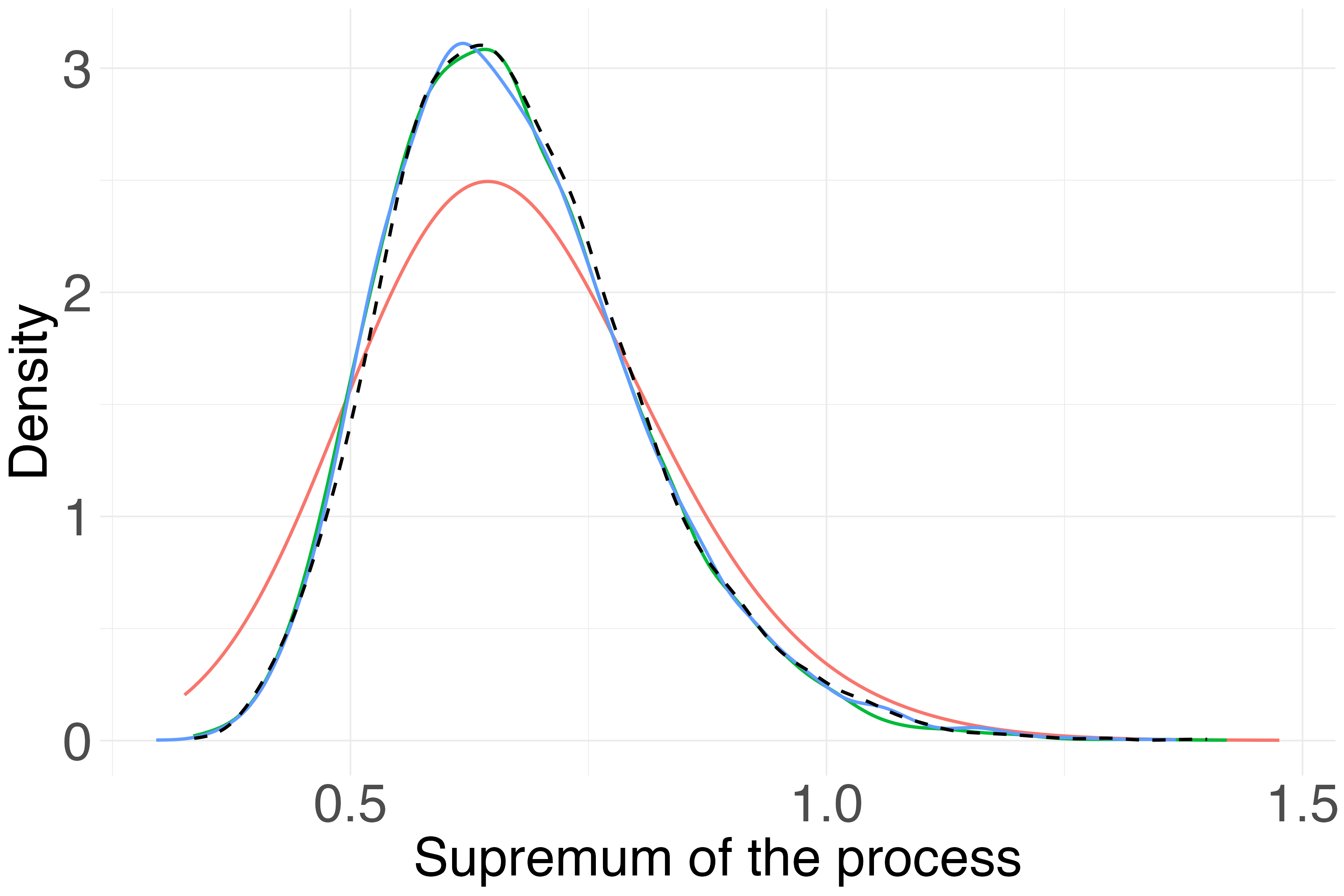}
    \caption{$\mathrm{vMF}(\bmu, 1)$}
\end{subfigure}
\hfill
\begin{subfigure}[htbp]{0.32\textwidth}
    \centering
    \includegraphics[width=\textwidth]{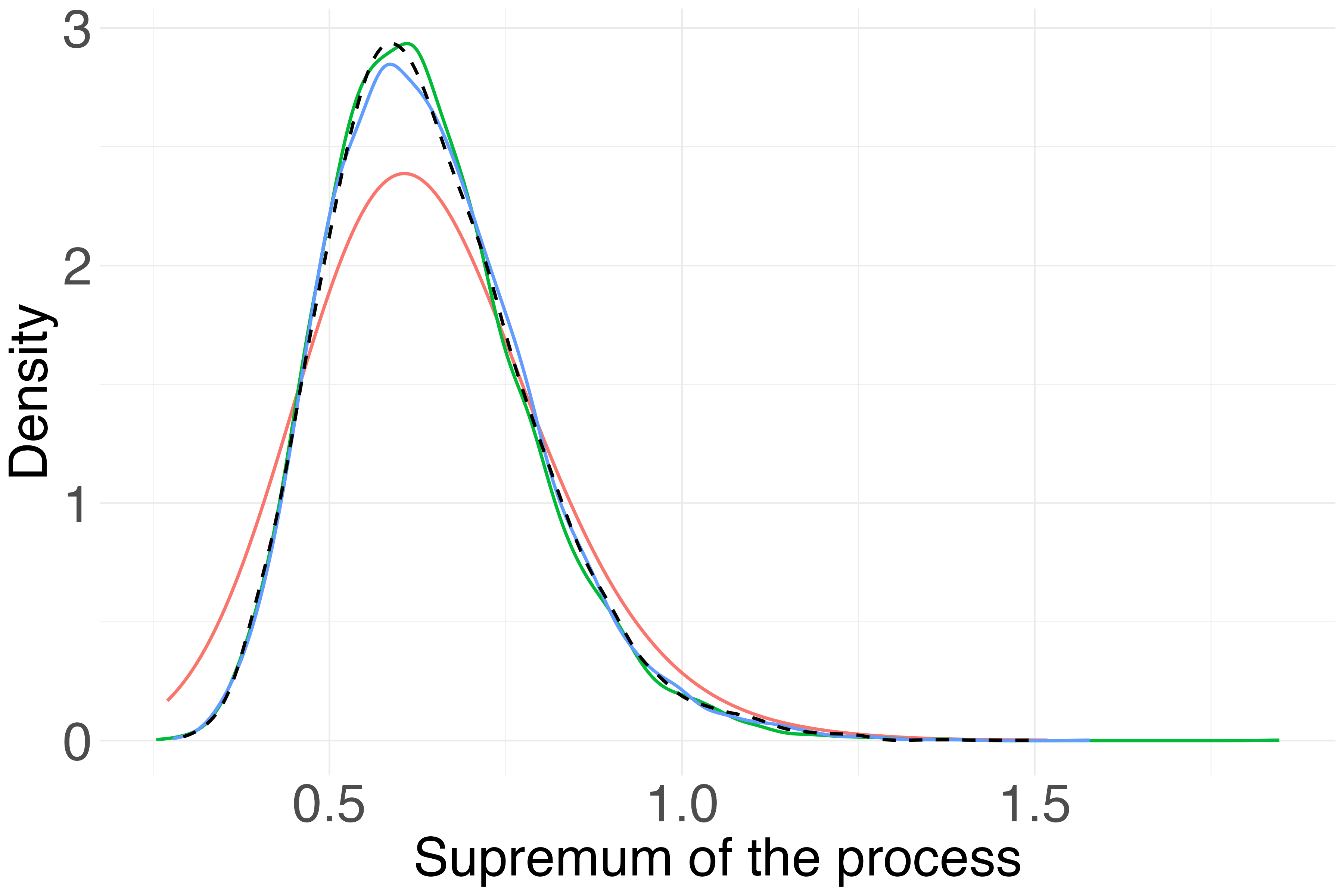}
    \caption{$\mathrm{vMF}(\bmu, 5)$}
\end{subfigure}
\par\medskip
\begin{subfigure}[htbp]{0.32\textwidth}
    \centering
    \includegraphics[width=\textwidth]{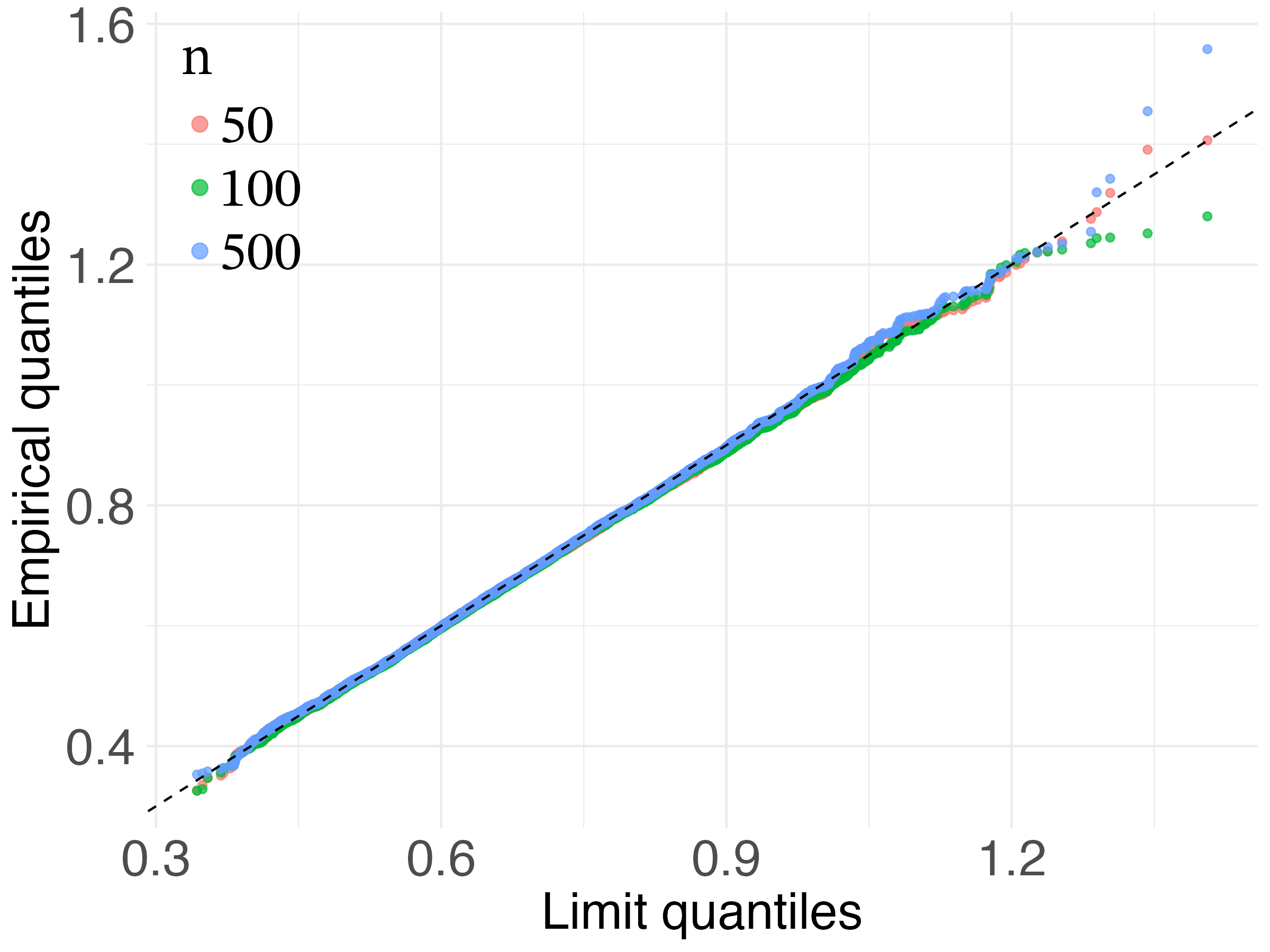}
    \caption{QQ plot for $\mathrm{vMF}(\bmu, 1/2)$}
\end{subfigure}
\hfill
\begin{subfigure}[htbp]{0.32\textwidth}
    \centering
    \includegraphics[width=\textwidth]{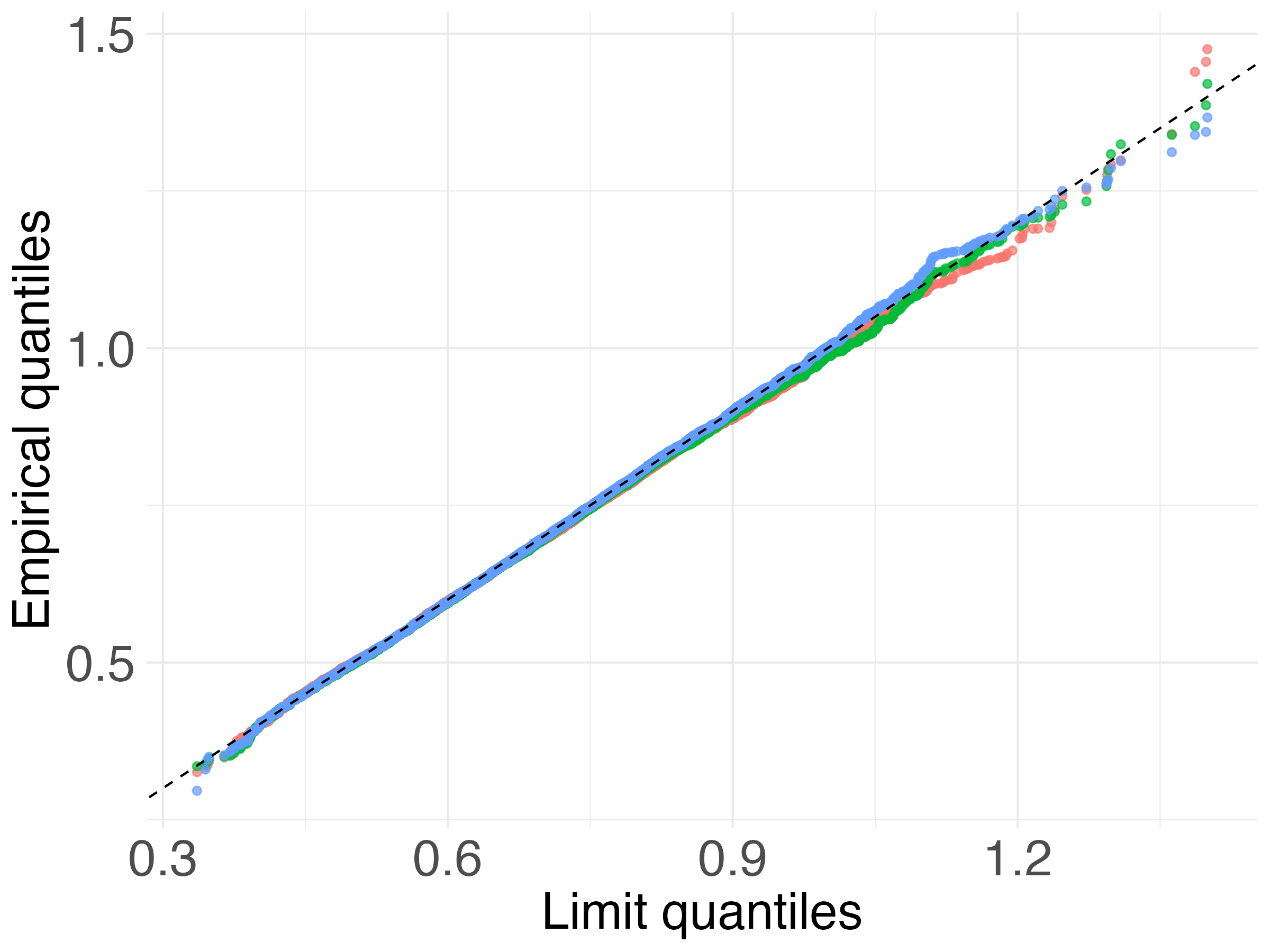}
    \caption{QQ plot for $\mathrm{vMF}(\bmu, 1)$}
\end{subfigure}
\hfill
\begin{subfigure}[htbp]{0.32\textwidth}
    \centering
    \includegraphics[width=\textwidth]{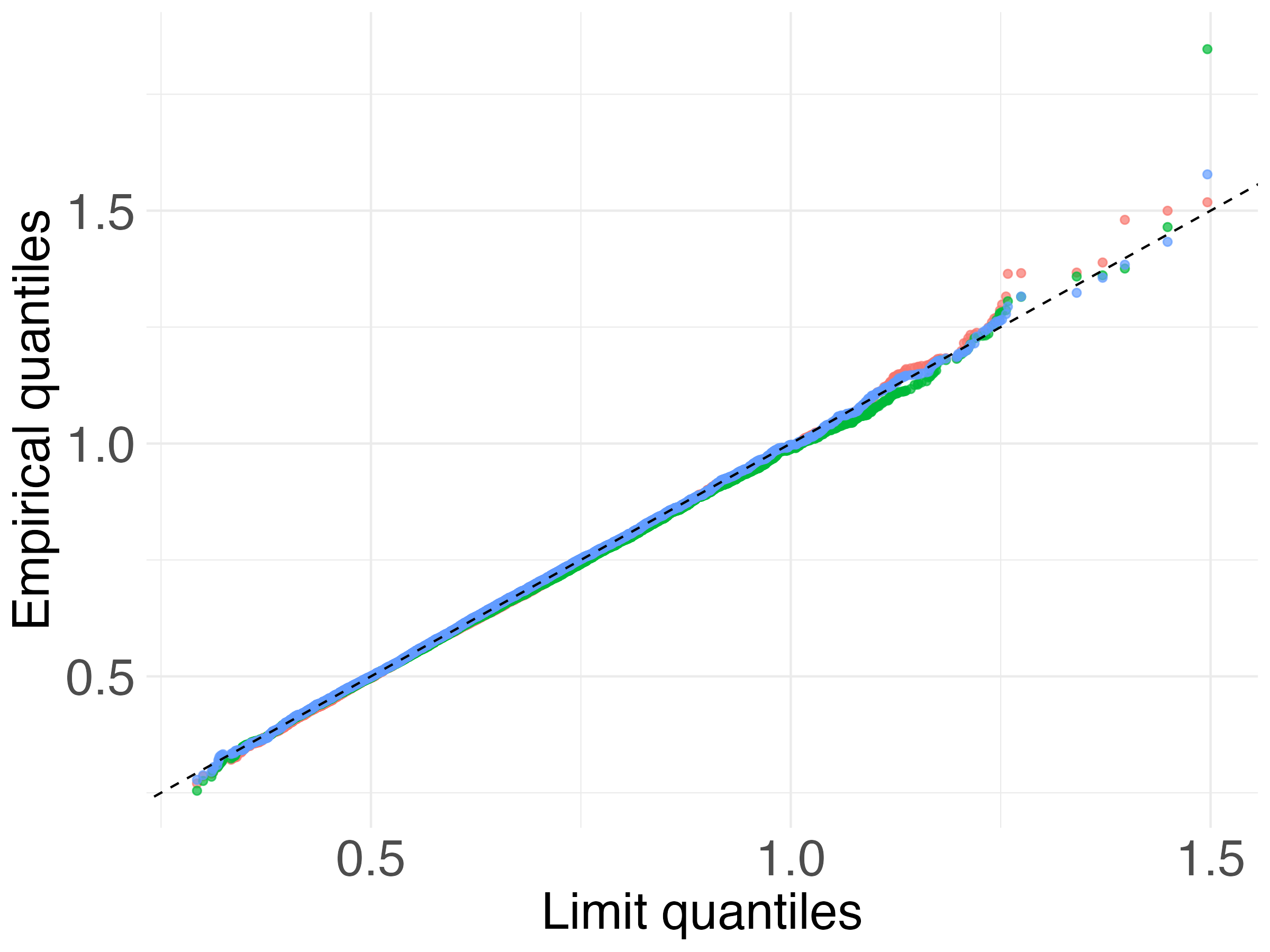}
    \caption{QQ plot for $\mathrm{vMF}(\bmu, 5)$}
\end{subfigure}
\caption{%
Same description as Figure~\ref{fig:convergence_process_vmf_simple_mu_100}, but for the statistic $\tilde{T}_n^{\mathrm{KS}}$ given in \eqref{eq:test_statistic_CM_tilde} (composite null).}
\label{fig:convergence_process_vmf_comp_mu_100}
\end{figure}
\fi

\bibliographystyleSM{apalike-custom}
\bibliographySM{bibliography_SM}

\end{document}